\documentclass{article}

\usepackage[latin1]{inputenc}
\usepackage[T1]{fontenc}
\usepackage[english]{babel}
\usepackage{textcomp}
\usepackage{mathtools,amssymb,amsmath}
\usepackage{amsthm}
\usepackage{lmodern}

\usepackage[a4paper]{geometry}
\usepackage{graphicx}
\usepackage[table]{xcolor}
\usepackage{array}
\usepackage{calc}
\usepackage{microtype}

\usepackage{titlesec}
\usepackage{titletoc}
\usepackage{fancyhdr}
\usepackage{titling}
\usepackage{enumitem}
\usepackage{stmaryrd}
\usepackage{array}%pour les tableaux
\usepackage{pdfpages}
\usepackage{tikz}
\usepackage{vmargin}%pour fixer les marges
\usepackage{mathrsfs}
\usepackage{float}

\usetikzlibrary{decorations}
\usetikzlibrary{decorations.pathmorphing}
\usetikzlibrary{decorations.pathreplacing}
\usetikzlibrary{decorations.shapes}
\usetikzlibrary{decorations.text}
\usetikzlibrary{decorations.markings}
\usetikzlibrary{decorations.fractals}
\usetikzlibrary{decorations.footprints}

\usepackage{hyperref}
\hypersetup{pdfstartview=XYZ}
\newcommand{\diff}{\mathop{}\mathopen{}\mathrm{d}}
\newcommand{\mtext}[1]{\quad\text{#1}\quad}
\newcommand{\abs}[1]{\left\lvert #1 \right\rvert}
\newcommand{\enstq}[2]{\left\{#1\mathrel{}:\mathrel{}#2\right\}}

\newcommand{\proba}[1]{\mathbb{P}\left[#1\right]}
\newcommand{\probacond}[2]{\mathbb{P}\left[#1\mathrel{}\middle|\mathrel{} #2\right]}

\newcommand{\E}{{\mathbb{E}}}
\newcommand{\nn}{\mathbb{N}}
\newcommand{\zz}{\mathbb{Z}}
\newcommand{\rr}{{\mathbb{R}}}
\newcommand{\dd}{{\mathbb{D}}}
\newcommand{\pp}{{\mathbb{P}}}

\newcommand{\zuz}{{\{0,1,2\}^\zz}}
\newcommand{\ba}{{\mathbf{a}}}
\newcommand{\bn}{{\mathbf{n}}}
\newcommand{\bm}{{\mathbf{m}}}
\newcommand{\bk}{{\mathbf{k}}}
\newcommand{\bv}{{\mathbf{v}}}
\newcommand{\bd}{{\mathbf{d}}}
\newcommand{\be}{{\mathbf{e}}}
\newcommand{\bdelta}{{\boldsymbol{\delta}}}

\newcommand{\cA}{{\mathcal A}}
\newcommand{\cB}{{\mathcal B}}
\newcommand{\cC}{{\mathcal C}}

\newcommand{\cE}{{\mathcal E}}

\newcommand{\cI}{{\mathcal I}}

\newcommand{\cL}{{\mathcal L}}
\newcommand{\cM}{{\mathcal M}}
\newcommand{\cP}{{\mathcal P}}
\newcommand{\cR}{{\mathcal R}}
\newcommand{\cS}{{\mathcal S}}
\newcommand{\cT}{{\mathcal T}}
\newcommand{\cU}{{\mathcal U}}
\newcommand{\cV}{{\mathcal V}}

\newcommand{\fP}{{\mathfrak P}}

\newcommand{\ttau}{{\tilde \tau}}

\newcommand{\tT}{{\tilde T}}

\newcommand{\tX}{{\tilde X}}

\newcommand{\tH}{{\tilde H}}
\newcommand{\tOmega}{{\tilde \Omega}}

\newcommand{\al}{{\ba_\la}}
\newcommand{\nl}{{\bn_\la}}
\newcommand{\ml}{{\bm_\la}}
\newcommand{\mla}{\bm_\la^\alpha}
\newcommand{\kl}{{\bk_\la}}
\newcommand{\klp}{{\bk_{\la,\pi}}}
\newcommand{\vlp}{{\bv_{\la,\pi}}}
\newcommand{\elp}{{\be_{\la,\pi}}}
\newcommand{\indiq}[1]{\textrm{\textbf{1}}_{\{#1\}}}
\newcommand{\la}{{\lambda}}
\newcommand{\intot}{{\int_0^t}}
\newcommand{\sm}{{{s-}}}
\newcommand{\tm}{{{t-}}}
\newcommand{\e}{{\varepsilon}}

\newcommand{\proco}{{\check{\zeta}}}

\newcommand{\eg}{{\it e.g. }}
\newcommand{\ie}{{\it i.e. }}
\newcommand{\md}{\medskip}
\newcommand{\intervalle}[4]{\mathopen{#1}#2
                            \mathclose{}\mathpunct{},#3
                            \mathclose{#4}}
\newcommand{\intervalleff}[2]{\intervalle{[}{#1}{#2}{]}}
\newcommand{\intervalleof}[2]{\intervalle{(}{#1}{#2}{]}}
\newcommand{\intervallefo}[2]{\intervalle{[}{#1}{#2}{)}}
\newcommand{\intervalleoo}[2]{\intervalle{(}{#1}{#2}{)}}

\newcommand{\intervalleentier}[2]{\intervalle\llbracket{#1}{#2}\rrbracket}

\makeatletter

\@addtoreset{equation}{section}
\makeatother

\newtheorem{theo}{\indent Theorem}[section]
\newtheorem{prop}[theo]{\indent Proposition}
\newtheorem{rem}[theo]{\indent Remark}
\newtheorem{lem}[theo]{\indent Lemma}
\newtheorem{defin}[theo]{\indent Definition}
\newtheorem{cor}[theo]{\indent Corollary}

\title{Asymptotics of one dimensional forest fire process with non instantaneous propagation}
\author{Jean-Maxime Le Cousin}
\date{\today}
\begin{document}
\maketitle
\begin{abstract}
Consider the following forest fire model where the possible locations of trees are
the sites of $\zz$. Each site has three possible states: 'vacant', 'occupied' or
'burning'. Vacant sites become occupied at rate $1$. At each site, ignition (by
lightning) occurs at rate $\la$. When a site is ignited, a fire starts and
propagates to neighbors at rate $\pi$. We study the asymptotic behavior of this process as $\la\to0$ and $\pi\to\infty$. We show that there are three possible classes of scaling limits, according to the regime in which $\la\to0$ and $\pi\to\infty$.
\end{abstract}

\tableofcontents
\newpage

\section{Introduction}
This section is devoted to preliminaries. We first define the $(\la,\pi)-$forest fire process with non instantaneous propagation. We then recall some known results about forest fire processes. Finally, we give the plan of the present paper.

\subsection{The discrete model}
Here we introduce the forest fire model with non instantaneous propagation.

\begin{defin}\label{definition FFP}
Let $\lambda\in\intervalleof{0}{1}$ and $\pi\geq1$ be fixed. For each $i\in\zz$, we
consider three Poisson processes, $N^S(i)=(N^S_t(i))_{t\geq0},
N^M(i)=(N^M_t(i))_{t\geq0}$ and $N^P(i)=(N^P_t(i))_{t\geq0}$ with respective parameters
$1,\la$ and $\pi$, all of these processes being independent. Consider a
$\{0,1,2\}$-valued process $(\eta^{\lambda,\pi}_t(i))_{t\geq0,i \in\zz}$ such that
a.s., for all $i \in\zz$, $(\eta^{\lambda,\pi}_t(i))_{t\geq0}$ is c\`adl\`ag. We
say that $(\eta^{\lambda,\pi}_t(i))_{t\geq0,i \in\zz}$ is a $(\lambda,\pi)-$forest
fire process ($(\la,\pi)-$FFP in short) if a.s., for all $i \in\zz$, all $t\geq0$,

\begin{align*}
\eta^{\lambda,\pi}_t(i)=&\int_0^t\textrm{\textbf{1}}_{\{\eta_{s-}^{\la,\pi}(i)=0\}}\diff
N_s^S(i)+\int_0^t \textrm{\textbf{1}}_{\{\eta_{s-}^{\la,\pi}(i)=1\}}\diff N_s^M(i)\\
&+\int_0^t
\textrm{\textbf{1}}_{\{\eta_{s-}^{\la,\pi}(i+1)=2,\eta_{s-}^{\la,\pi}(i)=1\}}\diff
N_s^P(i+1)+\int_0^t
\textrm{\textbf{1}}_{\{\eta_{s-}^{\la,\pi}(i-1)=2,\eta_{s-}^{\la,\pi}(i)=1\}}\diff
N_s^P(i-1)\\
&- 2\int_0^t \textrm{\textbf{1}}_{\{\eta_{s-}^{\la,\pi}(i)=2\}}\diff N_s^P(i).
\end{align*}
\end{defin}
Formally, we say that $\eta_t^{\lambda,\pi}(i)=0$ if there is no tree at site $i$ at
time $t$ and $\eta_t^{\lambda,\pi}(i)=1$ if the site $i$ is occupied. The case
$\eta_t^{\lambda,\pi}(i)=2$ means that the site $i$ is burning. Thus, the forest fire process starts from an empty initial
configuration, seeds fall according to some i.i.d. Poisson processes of parameter $1$
and matches fall according to some i.i.d. Poisson processes of parameter $\lambda$.
When a seed falls on an empty site, a tree appears immediately. When a match falls on
an occupied site, a fire starts and waits for an exponential time of parameter $\pi$
before it propagates to its neighbors and vanishes. If its right (resp. left) neighbor is occupied then it becomes
burning. Seeds falling on occupied sites, matches falling on vacant sites and fires propagating to vacant sites have no
effect.

This process can be shown to exist and to be unique (for almost every realization of
$N^S,N^M,N^P$) by using a \textit{graphical construction}. Indeed, to build the
process until a given time $T>0$, it suffices to work between sites $i$ which are
vacant until time $T$ [because $N^S_T(i)=0$]. Interaction cannot
cross such sites. Since such sites are a.s. infinitely many, this allows us to
handle a graphical construction. It should be pointed out that this construction
only works in dimension~$1$. 
% see ligett for many examples of graphical construction

For $a,b\in \zz$, we set $\intervalleentier{a}{b}=\{a,\dots,b\}\subset \zz$. For
$\eta \in \zuz$ and $i \in \zz$, we define the occupied connected component around
$i$ as
\[C(\eta,i)=\begin{cases} \emptyset& \mtext{if} \eta(i)=0 \text{ or } 2,\\
                                                \intervalleentier{l(\eta,i)}{r(\eta,i)}
& \mtext{if} \eta(i)=1,
        \end{cases}\]
where $l(\eta,i)=\sup\{k< i:\; \eta(k)=0\mtext{or}2\}+1$ and $r(\eta,i)=\inf\{k >
i:\; \eta(k)=0\mtext{or}2\}-1$.

\subsection{Motivation and references}
Consider a graph $G=(S,A)$, $S$ being the set of vertices and $A$ the set of edges.
Introduce the space of configurations $E=\{0,1,2\}^S$. For $\eta\in E$, we say that
$\eta(i)=0$ if the site $i\in S$ is vacant, $\eta(i)=1$ if the site $i$ is occupied by a tree and
$\eta(i)=2$ if the tree in $i$ is burning. Two sites are neighbors if there is an
edge between them. We call forests the connected components of occupied sites. For
$i\in S$ and $\eta\in E$, we denote by $C(\eta,i)$ the forest around $i$ in the
configuration $\eta$ (with $C(\eta,i)=\emptyset$ if $\eta(i)=0$ or $\eta(i)=2$). We
consider the following rules
\begin{itemize}
        \item vacant sites become occupied (a seed falls and a tree immediately
grows) at rate $1$;
        \item occupied sites take fire (a match falls) at rate $\la>0$;
        \item fires propagate to neighbors (inside the forest) at rate $\pi>0$.
\end{itemize}
Such a model was introduced by Henley \cite{h} and Drossel and Schwabl \cite{ds} as a toy model for
forest fire propagation and as an example of a simple model intended to clarify the
concept of {\itshape self-organized criticality}.

The study of self-organized critical systems has become rather popular
in physics since the end of the 80's. These are simple models supposed to clarify
temporal
and spatial randomness observed in a variety of natural phenomena showing long range
correlations, like sand piles, avalanches, earthquakes, stock market crashes, forest
fires, shapes of
mountains, clouds, etc. It is remarkable that such phenomena, reminiscent of critical
behavior,
arise so frequently in nature where nobody is here to finely tune the parameters to
critical
values. The most classical model is the sand pile model introduced in 1987 in \cite{btw1}, but many variants or related models have been proposed and studied more or less rigorously, describing earthquakes (see \cite{ofc}) or forest fires (see \cite{h}).

The features of the model depend on the geometry of the graph; we only consider in
this paper the case $S=\zz$ (with its natural set of edges). They also depend on the
laws of the processes governing seeds, matches and propagation. We work here in the
classical case where all processes are Poisson processes.

From the point of view of self-organized criticality, the interesting regime is the
asymptotic behavior of the forest-fire process as $\la\to0$ and $\pi\to\infty$:
then fires are very rare, but concern huge occupied components. We present three
possible limit processes (depending on the regime at which $\la\to0$ and $\pi\to\infty$)
arising when we suitably rescale space and accelerate time.

\subsection*{Forest fire on $\zz$}
All the available results concern the limit case where the
propagation is instantaneous ($\pi=\infty$): when a tree takes fire, the whole forest (to which it
belongs) is {\it destroyed immediately}. The model is thus:
\begin{itemize}
        \item vacant sites become occupied (a seed falls and a tree immediately
grows) at rate 1;
        \item matches fall on occupied sites at rate $\la$ and then burn
instantaneously the corresponding forest.
\end{itemize}
We denote $\eta^\la_t \in \{0, 1\}^\zz$ the configuration at time $t$. Observe that (possible) infinite clusters in the initial configuration would
immediately disappear.

The following results are related to this model.
\subsubsection*{Asymptotic density}
Van den Berg and J\'arai study in \cite{vdbj} the asymptotic density of vacant sites
in the limit $\la \to 0$. Their result states that there are two constants $0<c<C$
such that for any initial configuration, for any $\la >0$ small enough, for $t$
large enough (of order $\log(1/\la)$), 
\[\frac{c}{\log(1/\la)} \leq \proba{\eta^\la_t(0) = 0} \leq \frac{C}{\log(1/\la)}.\]
This is coherent with the intuition that the rarer fires are, the more space is
occupied by trees (although because of the lack of monotonicity, this is not
straightforward). We mention that such a result was stated in Drossel-Clar-Schwabl
\cite{dcs}. But the proof in \cite{dcs} is not rigorous: it is based on  the {\it
ansatz}  that the cluster sizes were following a cutoff power law, for cluster-sizes
up to some $s_{max}^\la$ defined by $s_{max}^\la \log{s_{max}^\la} = 1/\la$, i.e.
\[s_{max}^\la \simeq \frac 1 {\la \log(1/\la)}.\]  
In \cite{vdbj}, van den Berg and J\'arai also show that the cluster sizes cannot
follow the predicted power law.

\subsubsection*{Sizes of clusters, first results}
\label{clusterssize}
In \cite{bp}, Brouwer and Pennanen show that this last {\it ansatz} holds true up to
$s_{max}^{1/3}$. More specifically, they show that there are some constants $0<c<C$ such
that for all $0 < \la < 1$ and all stationary measures $\mu_\la$ (invariant by
translation) of the forest fire model on $\zz$ with parameter $\la$, for all
$x<(s_{max}^\la)^{1/3}$, 
\[\frac{c}{(1+x)\log{(1/\la)}} \leq \mu_\la \left(|C(\eta,0)| = x \right) 
\leq \frac{C}{(1+x) \log{(1/\la)}}.\]
Observe that this estimate is valid for relatively small clusters 
that will not be seen after rescaling (microscopic clusters).  

\subsubsection*{Scaling limits}
Still in the limit case where the propagation is instantaneous, Bressaud and Fournier have proved in \cite{bf} that in the asymptotic of rare
matches, the forest fire process converges, under suitable normalization, to some
limit forest fire process. They described precisely the dynamics of this limit
process and have shown that it is unique, that it can be built by using a graphical
construction and thus can be perfectly simulated. Using the limit process, they have
also estimated the size of clusters. Very roughly, they have proved that in a very
weak sense, for $\la$ small enough and for $t$ large enough (of order
$\log(1/\la)$), the cluster-size distribution resembles 
\[\proba{ C(\eta^\la_{t}, 0) = x} \simeq \frac{a}{(x+1) \log(1/\la)}\indiq{x \ll
1/(\la\log(1/\la))} + b\la\log(1/\la) e^{-x\la\log(1/\la)},\]
where $a,b$ are two positive constants. This means that there are two types of
clusters: {\it microscopic clusters}, described by a power-like law and {\it
macroscopic} clusters, described by an exponential-like law. This shows a {\it phase
transition} around the {\it critical size} $1/(\la\log(1/\la))$.

In \cite{bfnew}, Bressaud and Fournier have extended their results by replacing Poisson processes
by the case where seeds (respectively matches) fall on each site of $\zz$
independently, according to some stationary renewal processes, with stationary delay
distributed according to some law $\nu_S$ (respectively $\nu_M^\la$). This means
that for any time $t \geq 0$ and on any site $i \in\zz$, the time we have to wait
for the next seed is a $\nu_S-$distributed random variable. They also assume that
$\nu_S$ has a bounded support or a tail with fast or regular or slow variations.
They prove that, after rescaling, the corresponding forest fire process converges,
as $\la\to0$, to a limit process. They show that there are four classes of limit
processes, according to the fact that
\begin{itemize}
        \item $\nu_S$ has a bounded support,
        \item $\nu_S$ has a tail with fast decay,
        \item $\nu_S$ has a tail with polynomial decay,
        \item $\nu_S$ has a tail with logarithmic decay.
\end{itemize}
They see that the limit forest fire process build in \cite{bf} is quite universal: it
describes the asymptotics of a large class (roughly exponential decay for $\nu_S$)
of forest fire processes. A similar limit process arises when $\nu_S$ has bounded
support. But some quite different limit processes arise when $\nu_S$ has a heavy
tail.

\subsubsection*{Main idea of the present paper}

From the modelling point of view, the instantaneously destroying of clusters is
not clearly justified. The goal of this paper is to extend the result in \cite{bf} to the case
where fires need a random time to propagate to neighbors. 

We thus consider the case where seeds (resp. matches) fall on each site of $\zz$
independently, according to some Poisson processes with parameter $1$ (resp. $\la$)
and where a burning tree has to wait for an exponential time of parameter $\pi$ to propagate
to neighbors. Since the scaling in \cite{bf} depends only on the seed and match processes (i.e. only on $1$ and $\la$), the time and space scales will be the same here. We will separate three cases :
\begin{itemize}
       \item the case where
fires propagate very fast;
	   \item the case
where fires propagate very slowly;
        \item the intermediate case.
\end{itemize}
The first case is the most physically realistic and the most widely used. We will show that, if $\pi$ is large, then everything happens as if $\pi=\infty$ (instantaneous propagation). The other cases are even though mathematically interesting.
\subsection{Plan of the paper}
In Section 2, we start by explaining the heuristic scales and the relevant quantities (rescaled macroscopic clusters and measure of microscopic clusters).
We then give our main results (scaling limits and cluster-size distribution) together with heuristic proof. In Section 3, we study the existence and uniqueness of the limit process. In Section 4, we study the effect of fires in the discrete process, which will be usefull in the rest of the paper (propagation through an occupied zone). In Section 5, we give a discrete version of Section 3. The rest of the paper is devoted to the rigorous proof of our results: we treat the convergence in the regime $\cR(\infty,z_0)$ in Section 7, in the regime $\cR(p)$, for some $p\in\intervalleoo{0}{\infty}$ in Section 8 and finally in the regime $\cR(0)$ in Section 9. In the end of each two last sections, we deduce estimates on the cluster size distribution for the process.

\section{Main results}
\subsection{Notation}\label{notations}
In the whole paper, we use the convention $1/\infty=0$ and $1/0=\infty$.

We denote, for $J=[a,b]$ an interval of $\rr$, by $|J|=b-a$ the length of $J$ and for
$\alpha>0$, we set $\alpha J = [\alpha a, \alpha b]$. 

For $I\subset \zz$, $|I|=\#I$ stands for the number of
elements in $I$. For $I=\intervalleentier{a}{b}=\{a,\dots,b\} \subset \zz$ and
$\alpha>0$, we will set $\alpha I := \intervalleff{\alpha a}{\alpha b}\subset \rr$. For
$\alpha>0$, we of course take the convention that $\alpha \emptyset=\emptyset$.

For $x\in \rr$, $\lfloor x \rfloor$ stands for the integer part of $x$.

We denote by $\cI=\{[a,b],a\leq b\}$ the set of all closed finite intervals of
$\rr$. For two intervals $[a,b]$ and $[c,d]$, we set
\[\bdelta([a,b],[c,d])=|a-c|+|b-d|, \quad \bdelta([a,b],\emptyset)=|b-a|.\]

For $(x,I), (y,J)$ in $\dd([0,T], \rr_+ \times \cI\cup\{\emptyset\})$, the set of
c\`adl\`ag functions from $[0,T]$ into $\rr_+\times \cI\cup\{\emptyset\}$,  we
define 
\[\bd_T((x,I),(y,J))=\int_0^T\Big[|x(t)-y(t)|+\bdelta(I_t,J_t)\Big]\diff t.\]
For two functions $I,J\colon\intervalleff{0}{T}\to\cI\cup\{\emptyset\}$, we set
\[\bdelta_T(I,J)=\int_0^T\bdelta(I_t,J_t)\diff t.\]

For $(x,t)\in\rr\times\intervalleff{0}{T}$ we also set, for $p\geq0$,
\[\Lambda^ p_{(x,t)}  \coloneqq \enstq{(x+z,t-p\abs{z})}{\abs{z}\leq t/p}\]
($(r,v)\in\Lambda^ p_{(x,t)}\iff v=t-p\abs{r-x}$) and its part which joins $(y,s)$ to $(x,t)$
\[\Lambda^ p_{(x,t)}(y,s)=\begin{cases}
\enstq{(z,t-p\abs{z-x})}{z\in\intervalleff{x}{y}} & \text{if } (y,s)\in\Lambda^p_{(x,s)}\text{ and }y>x, \\
\enstq{(z,t-p\abs{z-x})}{z\in\intervalleff{y}{x}} & \text{if } (y,s)\in\Lambda^p_{(x,s)}\text{ and }y<x,\\
\emptyset &\text{else.}
\end{cases}\]
Similarly, we define
\begin{align*}
\mathcal{V}^ p_{(x,t)} &= \enstq{(x+z,t+p\abs{z})}{z\in\rr}\\
\mathcal{V}^ p_{(x,t)}(y,s) &=\begin{cases}
\enstq{(z,t+p\abs{z-x})}{z\in\intervalleff{x}{y}} & \text{if } (y,s)\in\mathcal{V}^ p_{(x,t)}\text{ and }y>x, \\
\enstq{(z,t+p\abs{z-x})}{z\in\intervalleff{y}{x}} & \text{if } (y,s)\in\mathcal{V}^ p_{(x,t)}\text{ and }y<x, \\
\emptyset &\text{else},
\end{cases}
\end{align*}
see Figure \ref{lap}. Observe that $\Lambda^ p_{(x,t)}(y,s)=\mathcal{V}^ p_{(y,s)}(x,t)$. Also observe that $\Lambda^0_{(x,t)}=\mathcal{V}^0_{(x,t)}=\enstq{(z,t)}{z\in\rr}$.
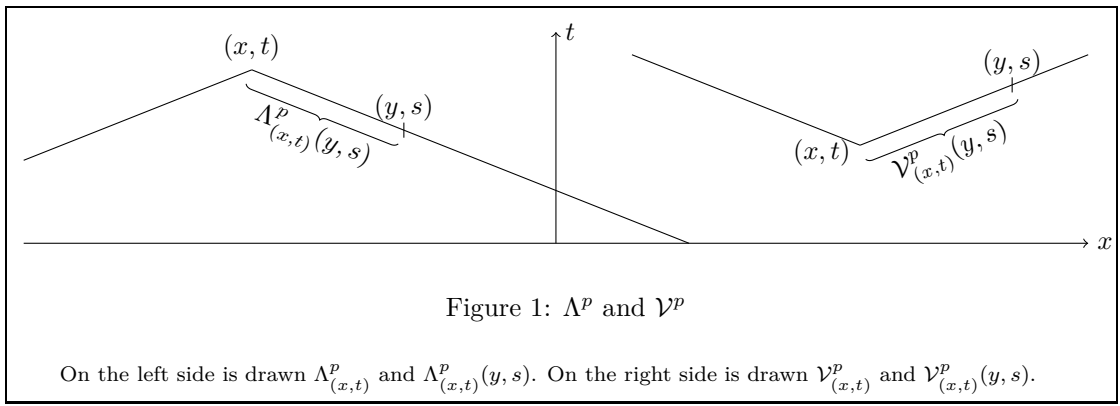
\begin{figure}[h!]
\fbox{
\begin{minipage}[c]{0.95\textwidth}
\centering
\begin{tikzpicture}
\draw[->] (0,-2.3)--(0,.5) node[below,right] {$t$};
\draw[->] (-7,-2.3)--(7,-2.3) node[below,right] {$x$};

\draw (-4,0)--(-7,-1.2);
\draw (-4,0)--(1.75,-2.3) node[above] at (-4,0) {$(x,t)$};
%\draw[dashed] (-4,0)--(-4,-2);
%\draw[dashed] (-2,-.8)--(-2,-2);
\draw (-2,-.7)--(-2,-.9) node[above] at (-2,-.8) {$(y,s)$};
\draw[decorate,decoration={brace,raise=0.2cm,mirror}] (-4,0) -- (-2,-.8) node[below=0.2cm,pos=0.5,sloped] {$\Lambda^ p_{(x,t)}(y,s)$};

\draw (4,-1)--(7,.2);
\draw (4,-1)--(1,.2) node[below,left] at (4,-1.1) {$(x,t)$};
%\draw[dashed] (4,-1)--(4,-2);
%\draw[dashed] (6,-2)--(6,-.2);
\draw (6,-.1)--(6,-.3) node[above] at (6,-.2) {$(y,s)$};
\draw[decorate,decoration={brace,raise=0.2cm,mirror}] (4,-1) -- (6,-.2) node[below=0.2cm,pos=0.5,sloped] {$\mathcal{V}^ p_{(x,t)}(y,s)$};

\end{tikzpicture}\caption{$\Lambda^p$ and $\mathcal{V}^p$}\label{lap}
\vspace{.5cm}
\parbox{13.3cm}{
\footnotesize{
On the left side is drawn $\Lambda^ p_{(x,t)}$ and $\Lambda^ p_{(x,t)}(y,s)$. On the right side is drawn $\mathcal{V}^ p_{(x,t)}$ and $\mathcal{V}^ p_{(x,t)}(y,s)$.
}}
\end{minipage}}
\end{figure}

\subsection{Heuristic scales and relevant quantities}

We look for some time scale for which tree clusters see about one fire per unit of
time. But for $\la$ very small, clusters will be very large before a match falls inside.
We thus also have to rescale space. Since this does not depend on $\pi$, these scales are the same as in \cite{bf}. We also have to find the different regimes at which $\la\to0$ and $\pi\to\infty$.

\subsubsection*{Time scale}
For $\la>0$ very small and for $t$ not too large, one might neglect fires, so that
roughly, each site is vacant with probability $e^{-t}$. Indeed, the time we have to
wait for the first seed follows, on each site, the law $\cE(1)$. Thus
$C(\eta_t^{\la,\pi},0)\simeq \intervalleentier{-X}{Y}$, where $X,Y$ are geometric
random
variables with parameter $e^{-t}$. Consequently, for $t$ not too large,
\[\abs{C(\eta_t^{\la,\pi},0)}\simeq e^t.\]
On the other hand, the rate that at which matches fall in the cluster $C(\eta^{\la,\pi}_t,0)$ is ${\lambda}|C(\eta_t^{\la,\pi},0)|$. So we decide to accelerate time by a factor 
\begin{equation}\label{def al}
\al=\log(1/{\lambda}).
\end{equation}
In this way,
${\lambda}|C(\eta_{\ba_\la}^{\la,\pi},0)|\simeq1$.

\subsubsection*{Space scale}
We now rescale space in such a way that during a time interval of order
$\ba_\la=\log(1/{\lambda})$, something like one match falls per unit of (space) length.
Since fires occur at rate ${\lambda}$, our space scale has to be of order 
\begin{equation}\label{def nl}
\nl=\left\lfloor\frac{1}{\lambda\al}\right\rfloor=\left\lfloor\frac{1}{\lambda\log(1/\lambda)}\right\rfloor.
\end{equation}
This means that we will identify
$\intervalleentier{0}{\bn_\la}\subset\zz$ with
$[0,1]\subset{\rr}$.

\subsubsection*{Propagation velocity}
The time needed for a fire to destroy a macroscopic cluster (which contains about
$\nl$ sites) is of order $\frac{\nl}{\pi}$. Indeed, a burning tree waits for an
exponential time of parameter $\pi$ before it propagates to neighbors. Thus, if a fire
starts at $0$, it needs roughly a time
$\nl/\pi$ to reach $\nl$. We have to
compare
the time $\nl/\pi$ to the characteristic time $\al$. Thus we have to separate
the three following regimes, as $\la\to0$ and $\pi\to\infty$ (observe that $\frac{\nl}{\al\pi}\simeq\frac1{\la\log^2(1/\la)\pi}$):
\begin{itemize}
        \item $\frac1{\la\log^2(1/\la)\pi}\to0$, which corresponds to the case where
fires propagate very fast;
        \item $\frac1{\la\log^2(1/\la)\pi}\to p$, for some
$p\in\intervalleoo{0}{\infty}$,
which is an intermediate case;
        \item $\frac1{\la\log^2(1/\la)\pi}\to\infty$, which corresponds to the case
where fires propagate very slowly.
\end{itemize}
Recall that, when neglecting fires and for $t<1$, $1/\la^t$ is the order of magnitude of the occupied cluster around $0$ at time $\al t$.
Thus a match falling in $0$ at time $\al t$ needs a time of order $1/(\la^t\pi)$ to destroy the whole component. In
order to treat the last case, we suppose that there exists
$z_0\in\intervallefo{0}{1}$ such that
\begin{equation}\label{inftyseuil}
\frac1{\la^{t}\pi}\to\begin{cases}
0 & \text{if } t<z_0,\\
\infty & \text{if } t>z_0.
\end{cases}\end{equation}
This means that if the match falls at time $\al t<\al z_0$, there are few occupied sites around $0$. Thus the fire destroys the whole component in a time of order $1/(\la^t\pi)\ll{\al}$. On the other hand, if the match falls a time $\al t>\al z_0$ then the component is
too big to be destroyed before $\al T$, for all $T>0$.

To summarize, we will treat separately the three following regimes, as $\la\to0$ and $\pi\to\infty$.
\begin{enumerate}
	\item $\cR(0)$: $\frac{\nl}{\al\pi}\ll1$, the fast regime;
	\item $\cR(p)$: $\frac{\nl}{\al\pi}\sim p\in\intervalleoo{0}{\infty}$, the intermediate regime;
	\item $\cR(\infty,z_0)$: $\frac{\nl}{\al\pi}\gg1$ and $\frac{\log(\pi)}{\log(1/\la)}\to z_0\in\intervalleff{0}{1}$, the slow regime.
\end{enumerate}

\subsubsection*{Rescaled clusters}
We thus set, for ${\lambda}\in(0,1)$, $\pi\geq 1$, $t\geq0$ and $x\in{\mathbb{R}}$, recalling Subsection \ref{notations},

\begin{equation}\label{dlambda}
D^{\lambda,\pi}_t(x)\coloneqq\frac1{\nl} C\left(\eta^{\la,\pi}_{\al t},\lfloor
\nl x\rfloor\right).
\end{equation}

However, this creates an immediate difficulty: recalling that $C(\eta^{\la,\pi}_t,0)
\simeq e^t$ for $t$ not too large, we see that for each site $x$, 
$|D^{\lambda,\pi}_t(x)|\simeq{\lambda}\log(1/{\lambda}) e^{t
\log(1/{\lambda})}={\lambda}^{1-t} \log(1/{\lambda})$, of which the limit when
${\lambda}\to0$ is
$0$ for $t<1$ and $+\infty$ for $t\geq1$.

For $t\geq1$, there might be fires in effect and one hopes that this will make the
possible limit of $|D^{\lambda,\pi}_t(x)|$ finite.
However, fires can only reduce the size of clusters so that for $t<1$, the limit of
$|D^{\lambda,\pi}_t(x)|$ will really be $0$. This cannot be a Markov process because it remains at $0$ during a time
interval of length exactly $1$. We thus need to keep track of more information in
order to control when
it exits from $0$.

To have an idea of the sizes of microscopic clusters, we keep some information about {\it the degree of smallness} of microscopic clusters. We
consider 
\begin{equation}\label{def ml}
{\bf m_\la}=\left\lfloor\frac{1}{\la\ba_\la^2}\right\rfloor=\left\lfloor\frac{1}{\lambda\log^2(1/\la)}\right\rfloor.
\end{equation}
Remark that $\ml\ll\nl$ but $\ml\gg\la^ {-t}$, for all $t\in\intervallefo{0}{1}$.
We introduce, for $\la>0$, $\pi\geq1$, $x\in \rr$, $t\geq0$,
\begin{align}\label{zlambda}
K_t^{\la,\pi}(x) &=\frac{\abs{\left\{ i\in\intervalleentier{\lfloor
 \nl x\rfloor-\bm_\la}{\lfloor \nl x\rfloor+\bm_\la} :
\eta^{\la,\pi}_{\al t}(i)=1\right\}}}{2\bm_\la+1}\in\intervalleff{0}{1},\\
Z_t^{\la,\pi}(x) &=\frac{-\log(1-K_t^{\la,\pi}(x))}{\log(1/\la)}\wedge 1\in\intervalleff{0}{1}.
\end{align}
Observe that $K^{\la,\pi}_t(x)$ stands for the {\it local density of occupied sites}
around $ \lfloor  \nl x \rfloor $ at time $\al t$. This density is {\it local}
because $\bm_\la \ll \bn_\la$.
We hope that for $t<1$, neglecting fires, $K^{\la,\pi}_t(x)\simeq 1-\la^t$,
whence $Z^{\la,\pi}_t(x)\simeq t$.

For all $\la>0$ small enough (we need that $2\bm_\la+1<1/\la$), it also holds that
 $Z^{\la,\pi}_t(x)=1$ if and only if $K^{\la,\pi}_t(x)=1$, i.e. if and only if  all
the sites are occupied around $\lfloor \nl x \rfloor$. Indeed,
$Z^{\la,\pi}_t(x)=1$
implies that $-\log(1-K^{\la,\pi}_t(x))\geq \log(1/\la)$, so that $K^{\la,\pi}_t(x) \geq
1-\la> 1-1/(2\ml+1)$, whence $K^{\la,\pi}_t(x)=1$. 

\subsubsection*{Final description}
We will study the $(\lambda,\pi)-$FFP through
$(D^{\la,\pi}_t(x),Z_t^{\la,\pi}(x))_{t\geq 0,x\in \rr}$. The main idea is that for
$\la>0$ very small and $\pi\geq1$ large enough:
\begin{itemize}
        \item if $Z^{\la,\pi}_t(x)=z\in(0,1)$, then $|D^{\la,\pi}_t(x)|\simeq 0$
and the
(rescaled) cluster containing $x$ is microscopic (in the sense that the non-rescaled
cluster containing $\lfloor\nl x\rfloor$ is small when compared to $\bn_\la$), but we control the local
density of occupied sites around $x$, which resembles $1-\la^ z$. Observe that
this density tends to $1$ as $\la \to 0$ for all $z\in (0,1)$;
        \item if $Z^{\la,\pi}_t(x)=1$ and $D_t^{\la,\pi}(x)=[a,b]$, then the
(rescaled)
cluster containing $x$ is macroscopic and has a length equal to $\abs{b-a}$ (or
$|C(\eta^{\la,\pi}_{\al t},\lfloor \nl x \rfloor)|\simeq \nl\abs{b-a}$ in
the original scales).
\end{itemize}

\begin{defin}
Let $(E,d)$ be a metric space. 

Let $p\geq0$. In the rest of the paper, we will say that $f(\la,\pi)\in E$ tends to $\ell\in E$ when $\la\to0$ and $\pi\to\infty$ in the regime $\cR(p)$ if for all $\delta>0$, there are $\e>0$ and $\la_0\in\intervalleof{0}{1}$ such that for all $\la\in\intervalleoo{0}{\la_0}$ and all $\pi\geq1$ in such a way that $\abs{\frac{\nl}{\al\pi}-p}<\e$, there holds $d(f(\la,\pi),\ell)<\delta$.

Let $z_0\in\intervalleff{0}{1}$. Similarly, we will say that $f(\la,\pi)\in E$ tends to $\ell\in E$ when $\la\to0$ and $\pi\to\infty$ in the regime $\cR(\infty,z_0)$ if for all $\delta>0$, there are $\e>0$, $K_0>0$ and $\la_0\in\intervalleof{0}{1}$ such that for all $\la\in\intervalleoo{0}{\la_0}$ and all $\pi\geq1$ in such a way that $\frac{\nl}{\al\pi}\geq K_0$ and $\abs{\frac{\log(\pi)}{\log(1/\la)}-z_0}<\e$, there holds $d(f(\la,\pi),\ell)<\delta$.
\end{defin}

\subsection{Main results when $p\in\intervallefo{0}{\infty}$}
In this section, we are interested in the regime $\cR(p)$, for some $p\in\intervallefo{0}{\infty}$.
We treat together the cases $p=0$ and $p\in\intervalleoo{0}{\infty}$. There are just few differences between these two cases: see Remark \ref{rem p=0} for an alternative definition in the case $p=0$.

\subsubsection{Definition of the limit forest fire process}
We now describe the limit process. We want this process to be Markov and this forces
us to add some variables. We consider a Poisson measure $\pi_M(\diff x,\diff t)$ on
$\rr\times[0,\infty)$, with intensity measure $\diff x\diff t$, whose marks
correspond to matches. Recall Notation \ref{notations}.

\begin{defin}\label{dfplffp}
Let $p\geq0$. A process $(Z_t(x),H_t(x),F_t(x))_{t\geq 0,x\in \rr}$ with
values in $\rr_+\times \rr_+\times\nn$ such that a.s., for all $x\in\rr$,
$(Z_t(x),H_t(x))_{t\geq 0}$ is c\`adl\`ag, is said to be a $p-$limit-forest-fire-process (or LFFP$(p)$ in short), if a.s., for all $t\geq 0$, all $x \in \rr$,
\begin{align}\label{eqplffp}
Z_t(x)&= \intot \indiq{Z_s(x)<1}\diff s-\sum_{s\leq t}(F_s(x)\wedge 1),\notag\\
H_t(x)&= \intot Z_{s-}(x)\indiq{Z_{s-}(x)<1}\pi_M(\{x\}\times\diff s)-\intot
\indiq{H_{s}(x)>0}\diff s,\\
F_t(x)&= \iint_{(y,s)\in\Lambda^ p_{(x,t)}} \indiq{\forall (r,v)\in\Lambda^
p_{(x,t)}(y,s)\, ,\,
Z_{v-}(r)=1\text{ and }H_{v-}(r)=0}\pi_M(\diff y,\diff
s).\notag
\end{align}
\end{defin}
To the LFFP$(p)$, we associate the process $D_t(x) = [L_t(x),R_t(x)]$, with
\begin{align*}
L_t(x) =& \sup\{ y\leq x:\; Z_t(y)<1 \hbox{ or } H_t(y)>0\},\\
R_t(x) =& \inf\{ y\geq x:\; Z_t(y)<1 \hbox{ or } H_t(y)>0\}.
\end{align*}
A typical path of the finite box version of the LFFP$(p)$ is drawn and commented in
Figure \ref{plffpdraw} and a simulation algorithm is explained in the
proof of Proposition \ref{algo}.

\begin{rem}\label{rem p=0}
If $p=0$, we can rewrite the process $(Z_t(x),H_t(x),F_t(x))_{t\geq 0,x\in \rr}$ as follow
\begin{align*}
Z_t(x)&= \intot \indiq{Z_s(x)<1}\diff s-\int_0^t\int_\rr\indiq{Z_{s-}(x)=1,y\in
D_{s-}(x)}\pi_M(\diff y,\diff s),\\
H_t(x) &= \intot Z_{s-}(x)\indiq{Z_{s-}(x)<1}\pi_M(\{x\}\times\diff s)-\intot
\indiq{H_{s}(x)>0}\diff s,\\
F_t(x) &= \int_\rr\indiq{Z_{t-}(x)=1,y\in D_{t-}(x)}\pi_M(\diff
y\times\{t\}),
\end{align*}
where $D_{t-}(x)$ is defined as above. Indeed, for all $x\in\rr$, all $t\geq0$, %the following equivalence is obvious
\[\enstq{(y,s)}{\forall (r,v)\in\Lambda^0_{(x,t)}(y,s) : Z_{v-}(r)=1\text{ and }H_{v-}(r)=0}=D_t(x)\times\{t\}\]
%\[\{\forall (r,v)\in\Lambda^0_{(x,t)}(y,s) : Z_{v-}^p(r)=1\mtext{and}H_{v-}^p(r)=0\}\iff \{s=t \text{ and } y\in D_\tm(x)\}.\]
With a slightly different formulation, this limit process is the same as
in \cite{bf} where the propagation is instantaneous. This relationship is very natural. Indeed, the case $p=0$ corresponds to the case where the propagation velocity is very high.
\end{rem}

\subsubsection{Formal dynamics}\label{formal dyn}

Let us explain the dynamics of this process. For
$p\in\intervallefo{0}{\infty}$, we consider $T>0$ fixed and set
$\cA_T=
\{x \in \rr:\; \pi_M(\{x\}\times\intervalleff{0}{T})> 0\}$.
For each $t\geq 0$, $x\in \rr$, $D_t(x)$ stands for the occupied cluster containing
$x$. We call this cluster {\it microscopic} if $D_t(x)=\{x\}$. Otherwise, we
call it {\itshape macroscopic}.

\md

{\it 1. Initial condition.}
We have $Z_0(x)=H_0(x)=F_0(x)=0$
and $D_0(x)=\{x\}$ for all $x\in \rr$.

\md

{\it 2. Occupation of vacant zones.}
We consider here $x\in \rr\setminus \cA_T$. Then we have $H_t(x)=0$ for all $t\in
[0,T]$. When $Z_t(x)<1$, $D_t(x)=\{x\}$ and $Z_t(x)$ stands for the {\it
local density of occupied sites} around $x$. Then $Z_t(x)$ grows linearly until it reaches $1$, as described by
the first term on the RHS of the first equation in (\ref{eqplffp}). When $Z_t(x)=1$,
the cluster containing $x$ is macroscopic and is described by $D_t(x)$.

\md

{\it 3. Microscopic fires.} Here we assume that $x\in\cA_T$ and that the
corresponding mark of $\pi_M$ happens at some time $t$ where $Z_\tm(x)<1$. In such a
case, the cluster containing $x$ is microscopic. Then we set $H_t(x)=Z_\tm(x)$, as
described by the first term on the RHS of the second equation of (\ref{eqplffp}) and
we leave unchanged the value of $Z_t(x)$ and $F_t(x)$. We then let $H_t(x)$ decrease linearly
until it reaches $0$, see the second term on the RHS of the second equation in
(\ref{eqplffp}). At all times where $H_t(x)>0$, that is during
$\intervallefo{t}{t+Z_\tm(x)}$, the site $x$ acts like a barrier (see Point 4. below).

\md

{\it 4. Macroscopic fires.} 
Here we assume that $y\in\cA_T$ and that the corresponding mark of
$\pi_M$ happens
at some time $s$ where $Z_\sm(y)=1$. This means that the cluster containing $y$ is
macroscopic. Thus this mark creates 2 fires: one goes to the left, the other
to the right. These fires propagates along of $\mathcal{V}^ p_{(y,s)}$, until they are stopped by a microscopic zone or a barrier or an other fire.

In other words, for all $(x,t)\in\rr\times\rr_+$, we set $F_t(x)=0$ unless there exists one (or two) mark $(y,s)$ of $\pi_M$ such that $(y,s)\in\Lambda^p_{(x,t)}$ (or equivalently $(x,t)\in\cV^p_{(y,s)}$) and for all $(r,v)\in\Lambda^p_{(x,t)}(y,s), Z_ {v-}(r)=1$ and $H_ {v-}(r)=0$, in which case we set $F_t(x)=1$ (or $F_t(x)=2$). When $x$ is crossed by a fire, $Z_t(x)$ jumps from $1$ to $0$, see the second term on the RHS of the first equation in \eqref{eqplffp}.

\md

{\it 5. Clusters.} Finally the definition of the clusters $(D_t(x))_{x\in\rr}$
becomes more clear: these clusters are delimited by zones with local density smaller
than $1$ (i.e. $Z_t(y)<1$)
or by sites where a microscopic fire has (recently) started (i.e. $H_t(y)>0$).

\subsubsection{Well posedness}
The existence and uniqueness of the LFFP$(0)$ has been proved in \cite{bf}. The proof in the case $p\in\intervalleoo{0}{\infty}$ is in the same spirit.
\begin{theo}\label{well posedness}
For any Poisson measure $\pi_M(\diff x,\diff t)$ on
$\rr\times\intervallefo{0}{\infty}$ with intensity measure $\diff x\diff t$, there
a.s. exists a unique LFFP$(p)$. Furthermore, it can be constructed graphically and
its restriction to any finite box $[0,T]\times\intervalleff{-n}{n}$ can be perfectly simulated.
\end{theo}

The LFFP$(p)$ $(Z_t(x),H_t(x),F_t(x))_{t\geq 0,x\in \rr}$ is furthermore
Markov, since it solves a well-posed time homogeneous Poisson-driven S.D.E.

\subsubsection{The convergence result}

\begin{theo}\label{converge}
Consider for each $\lambda\in\intervalleof{0}{1}, \pi\geq1$, the process
$(Z^{\la,\pi}_t(x),D_t^{\la,\pi})_{t\geq0,x\in\rr}$ associated to the
$(\lambda,\pi)-$FFP. Consider also the LFFP$(p)$
$(Z_t(x),H_t(x),F_t(x))_{t\geq
0,x\in \rr}$ and the associated $(D_t(x))_{t\geq0, x\in\rr}$. We assume that $\la\to0$ and $\pi\to\infty$ in the regime $\cR(p)$, for some $p\in\intervallefo{0}{\infty}$.
\begin{enumerate}
        \item For any $T>0$, any finite subset $\{x_1,\dots,x_q\}\subset\rr$,
$(Z^{\la,\pi}_t(x_i),D_t^{\la,\pi}(x_i))_{t\in\intervalleff{0}{T},i=1,\dots,q}$ goes in law to
$(Z_t(x_i),D_t(x_i))_{t\in\intervalleff{0}{T},i=1,\dots,q}$ in
$\dd(\intervalleff{0}{T},\rr\times(\cI\cup\{\emptyset\}))$. Here $\dd(\intervalleff{0}{T},\rr\times(\cI\cup\{\emptyset\}))$ is endowed with
the distance $\bd_T$.
        \item For any finite subset
$\{(x_1,t_1),\dots,(x_q,t_q)\}\subset\rr\times\intervallefo{0}{\infty}$,
$(Z^{\la,\pi}_{t_i}(x_i),D_{t_i}^{\la,\pi}(x_i))_{i=1,\dots,q}$ goes in law to
$(Z_{t_i}(x_i),D_{t_i}(x_i))_{i=1,\dots,q}$ in $(\rr\times(\cI\cup\{\emptyset\}))^q$.
Here $\cI\cup\{\emptyset\}$ is endowed with $\bdelta$.
        \item\label{estim cluster size} For all $t>0$, 
\[\left( \frac{\log(|C(\eta_{\al t}^{\lambda,\pi},0)|)}{\log(1/\la)} \indiq{|C(\eta_{\al
t}^{\lambda,\pi},0)|\geq1}\right)\wedge 1\]
goes in law to $Z_t(0)$.
\end{enumerate}
\end{theo}

Point \ref{estim cluster size} will allow us to check some estimates on the
cluster-size distribution. Since we deal with finite-dimensional marginals in space,
it is quite clear that the processes $H$ and $F$ do not appear in the limit, since
for each $x\in\rr$, for all $t\geq0$, a.s., $H_t(x)=F_t(x)=0$. (of course, it is false
that a.s., for all $x\in\rr$, all $t\geq0, H_t(x)=F_t(x)=0$). We obtain the
convergence of $D^{\lambda,\pi}$ (resp. $Z^{\lambda,\pi}$) to $D$ (resp. $Z$) only when
integrating in time. We cannot hope for a Skorokhod convergence since the limit
process $D(x)$ (resp. $Z(x)$) jumps instantaneously from $\{x\}$ (resp. $1$) to some
interval with positive length (resp. $0$), while $D^{\la,\pi}(x)$ (resp.
$Z^{\la,\pi}(x)$) needs
many small jumps, in a very short interval, to become macroscopic (resp. empty). 

The space $(\dd(\intervalleff{0}{T},\rr\times(\cI\cup\{\emptyset\})),\bd_T)$ is not a complete metric space since $\bd_T$ is too weak. However, it seems that it is not really a problem because in the proof, we use a coupling argument and obtain a convergence in probability.

\subsubsection{Heuristics argument}\label{heuristics}
We now explain roughly the reasons why Theorem~\ref{converge} holds. We consider a
$({\lambda},\pi)-$FFP $(\eta^{\lambda,\pi}_t(i))_{t\geq0,i\in\zz}$ and the associated process
$(Z^{\lambda,\pi}_t(x)$, $D^{\lambda,\pi}_t(x))_{t\geq 0,x\in\rr}$. We
assume below that ${\lambda}$ is very small, $\pi$ very large and $\nl/(\al\pi)$ close to $p$.

0. \textit{Scales.} With our scales, there are $\nl=\lfloor 1/({\lambda}\log(1/{\lambda}))\rfloor$
sites per unit of length. Approximately one fire starts per unit of time per unit of
length. A vacant site
becomes occupied at rate $\al=\log(1/{\lambda})$.

1. \textit{Initial condition.}
We have, for all $x\in{\mathbb{R}}$, $(Z^{\lambda,\pi}_0(x),D^{\lambda,\pi}_0(x))=
(0,\emptyset) \simeq(0,\{x\})$.

2. \textit{Occupation of vacant zones.} Assume that no match falls in a zone
$[a,b]$ (which correspond to the zone $\intervalleentier{\nl a}{\nl b}$ before rescaling) during $\intervalleff{0}{1}$ (or $\intervalleff{0}{\al}$ before rescaling).
\begin{enumerate}[label=\alph*.]
        \item For $s\in[0,1)$, we have $D^{\lambda,\pi}_{s}(x) \simeq[x \pm{\lambda}^{1-s}]\simeq\{x\}$ and $Z^{\lambda,\pi}_{s}(x)\simeq s$ for all $x\in[a,b]$.

Indeed, each site is occupied with probability $1-e^{-\al s}=1-\la^ s$. Thus the local density is roughly  $K_t^{\la,\pi}\simeq 1-\la^ s$, whence $Z_t^{\la,\pi}(x)\simeq s$, while the typical size of occupied clusters is $\la^ s$, whence $D_s^{\la,\pi}(x)\simeq\left[ x\pm\la^ s/\nl\right]\simeq\left[x\pm\la^ {1-s}\right]$.
        \item At time $s=1$, $Z^{\lambda,\pi}_{1}(x)\simeq1$ and all the sites in $[a,b]$ are occupied (with
very high probability).

Indeed, we have $(b-a)\nl$ sites and each of
them is occupied at time $1$ with probability
$1-e^{-\al}=1-{\lambda}$ so that all of them are occupied with
probability $(1-{\lambda})^{(b-a)\nl}\simeq e^{-(b-a)/\log(1/{\lambda})}$, which
goes to $1$ as ${\lambda}\to0$.
\end{enumerate}
Assume now that the zone around $x$ (i.e. the zone $\intervalleentier{\lfloor\nl x\rfloor-\ml}{\lfloor\nl x\rfloor+\ml}$ before rescaling) has been destroyed at time $t$ (or at time $\al t$ before rescaling) by a fire. Then, observations 2a. and 2b. above still hold:
\begin{enumerate}[label=\roman*.]
	\item for $s\in\intervallefo{0}{1}$ and if no fire starts in $\intervalleentier{\lfloor\nl x\rfloor-\ml}{\lfloor\nl x\rfloor+\ml}$ during $\intervalleff{\al t}{\al (t+s)}$, we have $D^{\lambda,\pi}_{t+s}(x) \simeq[x \pm{\lambda}^{1-s}]\simeq\{x\}$ and
$Z^{\lambda,\pi}_{t+s}(x)\simeq s$;
	\item $Z^{\lambda,\pi}_{t+1}(x)\simeq1$ and all the sites around $x$ are occupied at time $t+1$ with very high probability.
\end{enumerate}

3. \textit{Microscopic fires.}
Assume that a fire starts at some location $x$ (i.e. $\lfloor\nl x\rfloor$ before
rescaling) at some time $t$ (or $\al t$ before rescaling) with
$Z_{t-}^{\lambda,\pi}(x)=z\in(0,1)$. The possible clusters on the left and right of
$x$ cannot be connected during (approximately) $[t,t+z]$, but they can be
connected after (approximately) $t+z$. In other words, $x$ acts like a barrier
during $[t,t+z]$.

Indeed, the connected component $A$ of $x$ (or $\lfloor\nl x\rfloor$ before rescaling) at time $t$ (or $\al t$ before rescaling) has a size of order $\la^{1-z}$ (which
thus contains approximately ${\lambda}^{1-z}\nl \simeq {\lambda}^{-z}$ sites). The fire destroys the component $A$ in a time of order $1/(\la^ {z}\al\pi)\ll1$ (or $1/(\la^z\pi)\ll\al$ in original scale). Thus this fire
crosses very
fast the component $A$ and each site of $A$ becomes burning and then empty (i.e. $\eta^{\la,\pi}(i)$ jumps from $1$ to $2$ then from $2$ to $0$) during the time interval $\intervalleff{t}{t+1/(\la^
z\al\pi)}\simeq\{t\}$ (or $\intervalleff{\al t}{\al t+1/(\la^
z\pi)}\simeq\{\al t\}$ before rescaling). The probability that a fire starts again in $A$ is very small. Thus, using the same computation as in point
2, we
observe that $\mathbb{P}$[A is completely occupied at time $t+s$]$\simeq (1-\la^
s)^{\la^ {-z}}\simeq e^{-\la^ {s-z}}$. When $\la\to0$, this quantity tends to $0$ if
$s<z$ and to $1$ if $s>z$.

4. \textit{Macroscopic fires.}
Assume, now, that a fire starts at some place $x$ (i.e. $\lfloor\nl x\rfloor$
before rescaling) at some time $t$ (or $\al t$ before rescaling) and that
$Z^{\lambda,\pi}_{t-}(x) \simeq1$. Thus,
$D^{\lambda,\pi}_{t-}(x)$ is macroscopic (i.e. its length is of order $1$ in our
scales). Then the match creates two fires: one propagates to the left and one to the right at speed $p$ ($p$ unit times per unit space). There are only two burning trees at each instant with very high probability. Of course, these fires are stopped when they meet a vacant site (i.e. a microscopic zone or a barrier) or another fire.

Indeed, we have to wait for an exponential time of
parameter $\pi$ between each propagation in the original scales. It then produces two independent Poisson processes of parameter
$\pi$ which stand for the location of the fires. Then, for $b>x$, this Poisson process is at $\lfloor\nl
b\rfloor$ in the original scale (or in $b$ after rescaling) roughly at time $\al t+(\nl/\pi)(b-x)$ (or at time $t+(\nl/(\al\pi))(b-x)\simeq t+p(b-x)$ after rescaling). All sites $i\in\intervalleentier{\lfloor\nl x\rfloor}{\lfloor\nl
b\rfloor}$ becomes successively burning and empty roughly at time $\al t+(i-\lfloor\nl x\rfloor)/\pi$ in the original scale (or the site $y=i/\nl\in\rr$ is burning at time $t+p(y-x)$ after rescaling).

5. \textit{Clusters.} For $t\geq0$, $x\in{\mathbb{R}}$, the cluster
$D_t^{\lambda,\pi}(x)$ resembles $[x\pm{\lambda}^{1-z}]\simeq\{x\}$ if
$Z_t^{\lambda,\pi}(x)=z\in(0,1)$. We then say that $x$ is microscopic. Now,
macroscopic clusters are delimited either by microscopic zones or by sites where
there has been recently a microscopic fire (see point 3) or by a burning tree. 

Comparing the arguments above to the rough description of the LFFP$(p)$ (see Section
\ref{formal dyn}), our hope is that the $({\lambda},\pi)$-FFP resembles the LFFP$(p)$
for
${\lambda}>0$ very small, $\pi$ very large and $1/(\la\al^2\pi)$ close to $p$.

\begin{rem}\label{rem explain p=0}
Remark \ref{rem p=0} is now more clear. Consider the regime $\cR(0)$. If a fire starts at $x$ (or $\lfloor \nl x\rfloor$ before rescaling) at time $t$ (or $\al t$ before rescaling), the time needed to reach a point $b$ (or $\lfloor \nl b\rfloor$ before rescaling) is roughly $\nl |b-x|/(\al\pi)\simeq 0$ (or $\nl(b-x)/\pi\ll\al$ before rescaling). It means that if $b\in D^0_\tm(x)$ (or $\lfloor \nl b\rfloor\in C(\eta^{\la,\pi}_{\al t-},\lfloor \nl x\rfloor)$ before rescaling) the fire reaches $b$ at time $t+\nl|b-x|/(\al\pi)\simeq t$. In the scaling limit, the cluster containing $x$ is thus destroyed instantaneously.
\end{rem}

\subsubsection{Cluster size distribution}
We will deduce from Theorem \ref{converge} the following estimates on the cluster-size distribution.
\begin{cor}\label{cor1}
Let $p\in\intervallefo{0}{\infty}$ be fixed.
Let
$(Z_t(x),H_t(x),F_t(x))_{t\geq 0, x\in \rr}$ be 
a LFFP$(p)$ and $(D_t(x))_{t\geq0, x\in\rr}$ the associated process.
For each $\la\in (0,1]$ and $\pi\geq1$, let $(\eta^{\la,\pi}_t(i))_{t\geq 0, i \in \zz}$ be a 
$(\la,\pi)-$FFP.
\begin{enumerate}[label=\alph*.]
	\item For all $t\geq (5+p)/2$, all $0< a < b < 1$, for some $0<c_1<c_2$ depending on $p$, as $\la\to0$ and $\pi\to\infty$ in the regime $\cR(p)$,
\[\lim_{\la,\pi}
\proba{\abs{C(\eta^{\la,\pi}_{\al t},0)} \in \intervalleff{1/\la^ a}{1/\la^ b}} \\
=\proba{Z_t(0)\in \intervalleff{a}{b}}\in \intervalleff{c_1(b-a)}{c_2 (b-a)}.\]
	\item For all
$t \geq 3/2$, all $B>0$, for some $0<c_1<c_2$ and $0< \kappa_1 <\kappa_2$ depending on $p$, as $\la\to0$ and $\pi\to\infty$ in the regime $\cR(p)$, 
\[\lim_{\la,\pi}
\proba{\abs{C(\eta^{\la,\pi}_{\al t},0)}
\geq B \nl }=\proba{\abs{D_t(0)}\geq B}
\in [c_1 e^{-\kappa_2 B}, c_2 e^{-\kappa_1 B}].\]
\end{enumerate}
\end{cor}

This result shows that there is a {\it phase transition} around the critical size $\nl$: the cluster-size distribution changes of shape at $\nl$. The main idea is that two types of clusters are present: macroscopic clusters, of which the size is of order $\nl$ and microscopic clusters, of which the size is smaller than $\nl$.

\subsection{Main results for $p=\infty$}
In this section, we are interested in the regime $\cR(\infty,z_0)$, for some $z_0\in\intervalleff{0}{1}$.
\subsubsection{Definition of the limit process}
In this regime, the limit process is much simpler, in the sense that fires only have a local (in space) effect (but can have long time effect). This is due to the fact that a fire can't go too far away in a finite time.

We consider a Poisson measure $\pi_M(\diff x,\diff t)$ on $\rr\times[0,\infty)$,
with intensity measure $\diff x\diff t$, whose marks correspond to matches.
\begin{defin}\label{dfinflffp}
Let $z_0\in\intervalleff{0}{1}$. A process $(Y_t(x))_{t\geq 0,x\in \rr}$ with
values in $\rr_+$ such that a.s., for all $x\in\rr$, $(Y_t(x))_{t\geq 0}$ is
c\`adl\`ag, is said to be a LFFP$(\infty,z_0)$ if a.s., for all
$t\geq 0$, all $x \in \rr$,
\begin{equation}\label{eqinflffp}
Y_t(x)= \int_0^{t\wedge z_0} s\,\pi_M(\{x\}\times\diff s)-\intot\indiq{Y_{s}(x)\in\intervallefo{0}{1}}\diff s + \indiq{t\geq z_0}\pi_M(\{ x\}\times\intervalleff{z_0}{t}).
\end{equation}
\end{defin}
The process $Y$ takes its values in $\intervalleff{0}{1}$ and can be non-zero only at locations where $\pi_M(\{x\}\times\rr)\neq0$. If the mark of $\pi_M$ happens at time $t<z_0$, then the (microscopic) cluster containing $x$ is destroyed instantaneously and  $Y_s(x)\in\intervalleoo{0}{1}$ during $\intervallefo{t}{2t}$: $x$ acts like a barrier during this time interval. If the mark happens at time $t>z_0$ then the cluster containing $x$ is too big to be destroyed and $Y_s(x)=1$ for ever: there is always a burning tree close to $x$. We then naturally associate the process $D_t(x) = \intervalleff{L_t(x)}{R_t(x)}$, with
\begin{align*}
L_t(x) =&\begin{cases}
x &\mtext{if} t<1,\\
\sup\{ y\leq x:\; Y_t(y)>0\} & \mtext{if} t\geq1;
\end{cases}\\
R_t(x) =&\begin{cases}
x &\mtext{if} t<1,\\
\inf\{ y\geq x:\; Y_t(y)>0\} & \mtext{if} t\geq1.
\end{cases}
\end{align*}

A typical path of the finite box version of the LFFP$(\infty,z_0)$ is drawn and commented in
Figure \ref{llffpdraw}.

\begin{rem}
The process $Y$ is a time inhomogeneous Markov process. To make it homogeneous, we can add a second variable $Z$ as in the first equation \eqref{eqplffp} in the Definition \ref{dfplffp}.
\end{rem}
\begin{figure}[h!]
\fbox{
\begin{minipage}[c]{0.95\textwidth}
\centering
\begin{tikzpicture}
\fill[fill=gray!50!white] (-6,4) rectangle (6,6.8);
\draw[-] (-6,0)--(6,0);
\draw[->] (-6,0)--(-6,7) node[left] {t};
\draw[-] (-6.1,4)--(-5.9,4) node[left, pos=0.5] {$1$};
\draw[dashed] (-6,3.3)--(6.2,3.3) node[left=12.2cm] {$z_0$};

\draw[thick] (-1,1)--(-1,2) node[below] at (-1,1) {{\tiny $(X_2,T_2)$}}; 
\draw[red] (-1,1) node {$\bullet$};
\draw[thick] (-5,2.15)--(-5,4.2) node[below] at (-5,2.15) {{\tiny $(X_4,T_4)$}};
\draw[red] (-5,2.1) node {$\bullet$};
\draw[thick] (4,1.6)--(4,3.2) node[below] at (4,1.6) {{\tiny $(X_3,T_3)$}};
\draw[red] (4,1.6) node {$\bullet$};
\draw[thick]  (0,.7)--(0,1.4) node[below] at (0,.7) {{\tiny $(X_1,T_1)$}};
\draw[red] (0,.7) node {$\bullet$};
\draw[thick] (-3,2.9)--(-3,5.8) node[below] at (-3,2.9) {{\tiny $(X_6,T_6)$}};
\draw[red] (-3,2.9) node {$\bullet$};
\draw[thick] (5,3.2)--(5,6.4) node[below] at (5,3.2) {{\tiny $(X_7,T_7)$}};
\draw[red] (5,3.2) node {$\bullet$};
\draw[thick] (2,2.4)--(2,4.8) node[below] at (2,2.4) {{\tiny $(X_5,T_5)$}}
node[below] at (-4.3,3.8) {{\tiny $(X_{10},T_{10})$}}
node[below, right] at (-2.8,3.4) {{\tiny $(X_{9},T_{9})$}}
node[below] at (4.5,3.9) {{\tiny $(X_{11},T_{11})$}}
node[below] at (-2,4.5) {{\tiny $(X_{13},T_{13})$}}
node[below] at (-5.8,3.4) {{\tiny $(X_{8},T_{8})$}}
node[below] at (-1.3,6.3) {{\tiny $(X_{16},T_{16})$}}
node[below] at (2.8,4.7) {{\tiny $(X_{14},T_{14})$}}
node[below] at (4.3,5.2) {{\tiny $(X_{15},T_{15})$}}
node[below] at (.7,4.2) {{\tiny $(X_{12},T_{12})$}};
\draw[red] (2,2.4) node {$\bullet$};

\draw[thick, red, dashed] (-4.3,3.8)--(-4.3,6.8);
\draw[red] (-4.3,3.8) node {$\bullet$};
\draw[thick, red, dashed] (-2.8,3.4)--(-2.8,6.8);
\draw[red] (-2.8,3.4) node {$\bullet$};
\draw[thick, red, dashed] (4.5,3.9)--(4.5,6.8);
\draw[red] (4.5,3.9) node {$\bullet$};

\draw[thick, red, dashed] (-2,4.5)--(-2,6.8);
\draw[red] (-2,4.5) node {$\bullet$};
\draw[thick, red, dashed] (-5.8,3.4)--(-5.8,6.8);
\draw[red] (-5.8,3.4) node {$\bullet$};
\draw[thick, red, dashed] (-1.3,6.3)--(-1.3,6.8);
\draw[red] (-1.3,6.3) node {$\bullet$};
\draw[thick, red, dashed] (2.8,4.7)--(2.8,6.8);
\draw[red] (2.8,4.7) node {$\bullet$};
\draw[thick, red, dashed] (4.3,5.2)--(4.3,6.8);
\draw[red] (4.3,5.2) node {$\bullet$};
\draw[thick, red, dashed] (.7,4.2)--(.7,6.8);
\draw[red] (.7,4.2) node {$\bullet$};
\end{tikzpicture}\caption{LFF$(\infty,z_0)-$process in a finite box.}\label{llffpdraw}
\vspace{.5cm}
\parbox{13.3cm}{
\footnotesize{
The marks of $\pi_M$ are represented by $\color{red}{\bullet}$'s. The filled zones represents zones in which $\abs{D(x)}>0$. The plain vertical segments represent the sites where $Y_t(x)\in\intervalleoo{0}{1}$ and the dashed vertical segments represent the sites where $Y_t(x)=1$. In the rest of the space, we always have $Y_t(x)=0$. Until time $1$, all the particles are microscopic. Matches $1$ to $7$ falls before $z_0$. At each of these marks, a process $Y$ starts and its life-time equals the instant where it has started. This creates a barrier with height $T_k$ (the segment above $T_k$ ends at time $2T_k$). The other matches falls after $z_0$. At each of these marks, a process $Y$ starts and remains equal to $1$ forever.

Thus, for each $x\in\intervalleff{-A}{A}$, $D^ A_t(x)=\{x\}$ for $t\in\intervallefo{0}{1}$ and merge at $t=1$. Here we have at time $1$ the clusters $[-A,X_8]$, $[X_8,X_4]$, $[X_4,X_{10}]$, $[X_{10},X_6]$, $[X_6,X_9]$, $[X_9,X_5]$, $[X_5,X_{11}]$, $[X_{11},X_7]$ and , $[X_7,A]$.

Remark that $t\mapsto\abs{D_t(x)}$ is non-increasing on $\intervallefo{2z_0}{\infty}$ for all $x$.
}}
\end{minipage}}
\end{figure}
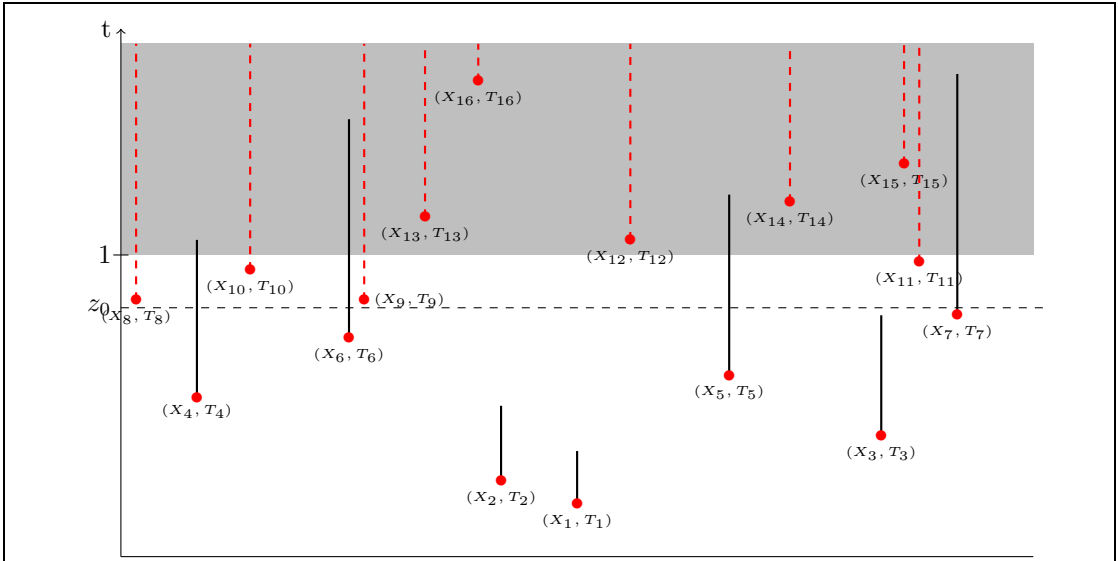

\subsubsection{Formal dynamics}\label{formal dyn infty}

Let us explain the dynamics of this process. We consider $\cA=\enstq{x \in \rr}{\pi_M(\{x\}\times\intervallefo{0}{\infty})> 0}$.
For each $t\geq 0$, $x\in \rr$, $D_t(x)$ stands for the occupied cluster containing
$x$. We call this cluster {\it microscopic} if $D_t(x)=\{x\}$. Otherwise, we
call it {\itshape macroscopic}.

{\it 1. Initial condition.}
We have $Y_0(x)=0$
and $D_0(x)=\{x\}$ for all $x\in \rr$.

{\it 2. Occupation of vacant zones.}
We consider here $x\in \rr\setminus \cA$. Then we have $Y_t(x)=0$ for all $t\in
\intervallefo{0}{\infty}$. When $t<1$, $D_t(x)=\{x\}$. When $t\geq1$,
the cluster containing $x$ is macroscopic and is described by $D_t(x)$.

{\it 3. First kind of fires.} Here we assume that $x\in\cA$ and that the
corresponding mark of $\pi_M$ happens at some time $t<z_0$. We set $Y_t(x)=t$, as
described by the first term on the RHS of the equation of \eqref{eqinflffp}. We then let $Y_t(x)$ decrease linearly
until it reaches $0$, see the second term on the RHS of the equation in
\eqref{eqinflffp} (i.e. $Y_s(x)=\min(2t-s,0)\indiq{s\geq t}$).

{\it 4. Second kind of fires.} 
Here we assume that $x\in\cA$ and that the corresponding mark of
$\pi_M$ happens
at some time $t$ where $t>z_0$. Then we set $Y_s(x)=1$ for all $s\in\intervallefo{t}{\infty}$ see the third term of the RHS of the equation \eqref{eqinflffp}.

{\it 5. Clusters.} Finally the definition of the clusters $(D_t(x))_{x\in\rr}$
becomes more clear: these clusters remain microscopic until $t=1$. For $t\geq 1$, $(D_t(x))_{x\in\rr,t\geq1}$ is delimited by sites where a fire of first kind has (recently) started (i.e. $Y_t(y)\in\intervalleoo{0}{1}$) or by sites where a fire of second kind has started (i.e. $Y_t(y)=1$). Remark that for $t\geq 2z_0$, only fires of second kind delimit the clusters.

\subsubsection{Well posedness}
The following proposition is obvious from the definition, see Figure \ref{llffpdraw}. 

\begin{prop}
Let $\pi_M$ be a Poisson measure on $\rr\times\intervallefo{0}{\infty}$ with intensity measure $\diff x\diff t$. There a.s. exists a unique LFFP$(\infty,z_0)$ $(Y_t(x))_{t\geq0,x\in\rr}$. It can be simulated exactly on any finite box $\intervalleff{0}{T}\times\intervalleff{-n}{n}$. 
\end{prop}

\subsubsection{The convergence result}
We will prove the following result.
\begin{theo}\label{theoinfty}
Let $z_0\in\intervalleff{0}{1}$. Consider for each $\la\in\intervalleof{0}{1}$ and $\pi\geq1$ the process $(D_t^{\la,\pi}(x))_{t\geq0,x\in\rr}$ associated with the $(\la,\pi)-$FFP. Consider also the LFFP$(\infty,z_0)$ $(Y_t(x))_{t\geq0,x\in\rr}$ and the associated $(D_t(x))_{t\geq0,x\in\rr}$ process. We assume that $\la\to0$ and $\pi\to\infty$ in the slow regime $\cR(\infty,z_0)$.
\begin{enumerate}
	\item For any $T>0$, any finite subset $\{x_1,\dots,x_q\}\subset\rr$, $(D_t^{\la,\pi}(x_i))_{t\in\intervalleff{0}{T},i=1,\dots,q}$ goes in law to $(D_t(x_i))_{t\in\intervalleff{0}{T}, i=1,\dots,q}$ in $\dd(\intervalleff{0}{T},\cI)^q$. Here $\dd(\intervalleff{0}{T},\cI)^q$ is endowed with $\bdelta_T$.
	\item For any finite subset $\{(x_1,t_1),\dots,(x_q,t_q)\}\subset\rr\times\intervallefo{0}{\infty}$, $(D_{t_i}^{\la,\pi}(x_i))_{i=1,\dots,q}$ goes in law to $(D_{t_i}(x_i))_{t\in\intervalleff{0}{T}, i=1,\dots,q}$ in $\cI^q$, $\cI$ being endowed with $\bdelta$.
\end{enumerate}
\end{theo}

\subsubsection{Heuristics arguments}
We assume below that $\la>0$ is very small, $\pi\geq1$ is very large, $\la\ba_\la^2\pi$ is close to $0$ and $\log(\pi)/\log(1/\la)$ is close to $z_0$.

0. \textit{Scales.} With our scales, there are $\nl=\lfloor 1/({\lambda}\log(1/{\lambda}))\rfloor$
sites per unit of length. Approximately one fire starts per unit of time per unit of
length. A vacant site
becomes occupied at rate $\al=\log(1/{\lambda})$.

1. {\it Initial condition.}
We have, for all $x\in{\mathbb{R}}$, $D^{\lambda,\pi}_0(x)=\emptyset \simeq \{x\}$ and $D_0(x)=\{x\}$.

2. {\it Occupation of vacant zones.} Exactly as in the regime $\cR(p)$, $D^{\lambda,\pi}_{t}(x)\simeq[x \pm{\lambda}^{1-t}]\simeq\{x\}$ for $t<1$ and the clusters become macroscopic at time $1$.

3. {\it First kind of fires.} Assume that a match falls at some place $x$ (or $\lfloor\nl x\rfloor$ in the original scales) at some time $t<z_0$ (or $\al t<\al z_0$ in the original scales). Then the fire burns almost immediately the occupied cluster and it needs roughly a time $t$ (or $\al t$ in the original scales) to be filled again. Thus $x$ acts like a barrier during $\intervallefo{t}{2t}$.
 
Indeed, the connected component $A$ of $x$ (or $\lfloor\nl x\rfloor$ before rescaling) at time $t$ (or $\al t$ before rescaling) has a size of order $\la^{1-t}$ (which
thus contains approximately ${\lambda}^{1-t}\nl \simeq {\lambda}^{-t}$ sites). The fire destroys the component $A$ in a time of order $1/(\la^ {t}\al\pi)\ll1$ (or $1/(\la^t\pi)\ll\al$ in original scales) due to $\cR(\infty,z_0)$. Thus this fire
crosses very
fast the component $A$ and each site of $A$ becomes burning and then empty (i.e. $\eta^{\la,\pi}(i)$ jumps from $1$ to $2$ then from $2$ to $0$) during the time interval $\intervalleff{t}{t+1/(\la^
t\al\pi)}\simeq\{t\}$ (or $\intervalleff{\al t}{\al t+1/(\la^
t\pi)}\simeq\{\al t\}$ before rescaling). The probability that a fire starts again in $A$ is very small. Thus, we
observe that $\mathbb{P}$[A is completely occupied at time $t+s$]$\simeq (1-\la^
s)^{\la^ {-z}}\simeq e^{-\la^ {s-z}}$. When $\la\to0$, this quantity tends to $0$ if
$s<t$ and to $1$ if $s>t$.

4. {\it Second kind of fires.} Assume that a match falls at some place $x$ (or $\lfloor\nl x\rfloor$ in the original scales) at some time $t>z_0$ (or $\al t>\al z_0$ in the original scales). Then the fire needs an infinite time (in our scales) to burn the occupied cluster, so that there is a burning site close to $x$ forever. 

Indeed, $D^{\lambda,\pi}_{t}(x)$ contains roughly $\la^ {-t}$ sites if $t\in\intervalleoo{z_0}{1}$ and $\nl$ sites if $t\geq1$. In any case, the time needed for the fire to cross this cluster is of order $\abs{D^{\la,\pi}_t(x)}/\pi$, which is very large when compared to $\al$ in the regime $\cR(\infty,z_0)$. Thus, the fire cannot reach the rim of $D^{\la,\pi}_t(x)$.

5. {\it Clusters.} For $t\geq0$, $x\in{\mathbb{R}}$, the cluster
$D_t^{\lambda,\pi}(x)$ resembles $[x\pm{\lambda}^{1-z}]\simeq\{x\}$ if
$t<1$. Now,
macroscopic clusters emerge when $t\geq1$ and are delimited either by a burning tree or by sites where
there has been recently a microscopic fire (see point 3).

Comparing the arguments above to the rough description of the LFFP$(\infty,z_0)$ (see Section
\ref{formal dyn infty}), our hope is that the $({\lambda},\pi)-$FFP resembles the LFFP$(\infty,z_0)$
in the regime $\cR(\infty,z_0)$.

\subsubsection{Cluster-size distribution}
The following corollary is easily deduced from the Theorem \ref{theoinfty}.
\begin{cor}\label{cor2}
Let $z_0\in\intervalleff{0}{1}$.
Let
$(Y_t(x))_{t\geq 0, x\in \rr}$ be 
a LFFP$(\infty,z_0)$ and $(D_t(x))_{t\geq0, x\in\rr}$ the associated process.
For each $\la\in (0,1]$ and $\pi\geq1$, let $(\eta^{\la,\pi}_t(i))_{t\geq 0, i \in \zz}$ be a 
$(\la,\pi)-$FFP.

For all $t > 2z_0$, as $\la\to0$ and $\pi\to\infty$ in the regime $\cR(\infty,z_0)$,
\[\frac1{\nl}\abs{ C(\eta^{\la,\pi}_{\al t},0)}\xrightarrow{\cL}\abs{D_t(0)}\sim\Gamma(2,t-z_0).\]
\end{cor}

This result shows that for $t$ large enough, there are only macroscopic clusters, that is clusters with size of order $\nl$.

We immediately give the proof of Corollary \ref{cor2}. For $t\geq0$, Theorem \ref{theoinfty} shows that, when $\la\to0$ and $\pi\to\infty$ in the regime $\cR(\infty,z_0)$,
\[\frac1{\nl}\abs{ C(\eta^{\la,\pi}_{\al t},0)}\xrightarrow{\cL}\abs{D_t(0)}.\]
Furthermore, if $t>2z_0$, only fires of the second kind (i.e. matches falling after $z_0$) still have an effect. Indeed, when a match falls in $x$ at time $t<z_0$, it creates a barrier in $x$ during $\intervallefo{t}{2t}\subset\intervalleff{0}{2z_0}$. Thus, $D_t(0)$ is only delimited by sites where a match has fallen during $\intervalleff{z_0}{t}$. This is a Poisson process on $\rr$ with intensity $t-z_0$. Consequently,
\[\abs{D_t(0)}\sim\Gamma(2,t-z_0).\]
\subsubsection{Irreversibility}
It might look surprising at the first glance that the limit process is non-reversible while the discrete process is reversible. Indeed, for $t\geq 1\wedge2z_0$, clusters in the limit process are macroscopic and the sizes are non-increasing. On the other hand, in the discrete process, it is quite clear that, when working in a finite box, the process returns to its original state. This is due to the time scale: we have to wait a very long time to observe again the original state.

\section{Existence and uniqueness of the limit process}
The goal of this section is to show that the limit processes are well-defined, unique, can
be obtained from a graphical construction and can be restricted to a finite box.

\subsection{Restriction of the LFFP$(\infty,z_0)$ to a finite box}
Let $z_0\in\intervalleff{0}{1}$ be fixed. In this subsection, we study the LFFP$(\infty,z_0)$.
\begin{prop}\label{restriction limite infini}
Let $\pi_M$ a Poisson measure on $\rr\times\intervallefo{0}{\infty}$ with intensity measure $\diff x\diff t$ and $A>0$.
\begin{enumerate}
	\item  The values of $(Y_t(x))_{t\geq0,x\in\intervalleff{-A}{A}}$ are entirely determined by $\pi_M|_{\intervalleff{-A}{A}\times\rr_+}$. Actually, for all $x\in\rr$, the values of $(Y_t(x))_{t\geq0}$ are entirely determined by $\pi_M|_{\{x\}\times\rr_+}$.
	\item There exists some constants $\alpha>0$ and $C>0$ not depending on $A>0$ such that
\begin{equation}\label{proba restriction infty}
\proba{(D_t(x))_{t\geq 0, x\in\intervalleff{-A/2}{A/2}}\subset\intervalleff{-A}{A}}\geq 1-Ce^{-\alpha A}.
\end{equation}
\end{enumerate}
\end{prop}
\begin{proof}
 The first part of  Proposition \ref{restriction limite infini} is obvious from the definition of the process $(Y_t(x))_{t\geq0,x\in\rr}$. In order to prove the second part, consider the event $\Omega_A^+$ on which $\pi_M$ has at least one mark $(X_1,\tau_1)$ in $\intervalleff{A/2}{A}\times \intervalleoo{3/4}{1}$ and at least one mark $(X_2,\tau_2)$ in $\intervalleff{A/2}{A}\times \intervalleoo{1}{3/2}$.

Observe now that on $\Omega_A^+$, $Y_t(X_1)>0$ for all $t\in\intervallefo{\tau_1}{2\tau_1}\supset\intervalleff{1}{3/2}$, because it is a either a fire of first kind (if $\tau_1\leq z_0$) or $X_1$ burns for ever (if $\tau_1>z_0$), and $Y_t(X_2)=1$ for all $t\in\intervallefo{\tau_2}{\infty}\supset\intervalleff{3/2}{\infty}$ because $\tau_2>1\geq z_0$, then $X_2$ burns for ever, because it is necessarily a fire of second kind.

Similarily, we define the event $\Omega_A^-$ on which $\pi_M$ has at least one mark $(\tX_1,\tilde{\tau}_1)$ in $\intervalleff{-A}{-A/2}\times \intervalleoo{3/4}{1}$ and at least one mark $(\tX_2,\tilde{\tau}_2)$ in $\intervalleff{-A}{-A/2}\times \intervalleoo{1}{3/2}$.

Thus, on $\Omega_A^+\cap\Omega_A^-$, $D_t(x)\subset\intervalleff{-A}{A}$ for all $t\geq 0$ and all $x\in\intervalleff{-A/2}{A/2}$. Finally, we can bound from below the left hand side of \eqref{proba restriction infty} by
\[\proba{\Omega_A^+\cap\Omega_A^-}\geq 1-2(e^{-A/8}+e^{-A/4})\geq1-4e^{-A/8}\]
whence \eqref{proba restriction infty} with $C=4$ and $\alpha=1/8$.
\end{proof}

\begin{defin}
Let $z_0\in\intervalleff{0}{1}$ and $(Y_t(x))_{x\in\rr,t\geq}$ be a LFFP$(\infty,z_0)$. For all $A>0$ and for $x\in\intervalleff{-A}{A}$, we define the process $D_t^A(x) = \intervalleff{L_t^A(x)}{R_t^A(x)}$, with
\begin{align*}
L_t^A(x) =&\begin{cases}
x &\mtext{if} t<1,\\
\sup\{ y\leq x:\; Y_t(y)>0\}\vee(-A) & \mtext{if} t\geq1;
\end{cases}\\
R_t^A(x) =&\begin{cases}
x &\mtext{if} t<1,\\
\inf\{ y\geq x:\; Y_t(y)>0\}\wedge\phantom{(-} A\phantom{)} & \mtext{if} t\geq1.
\end{cases}
\end{align*}
\end{defin}
As a corollary of Proposition \ref{restriction limite infini}, we have, for $A>0$, \[\proba{(D_t(x))_{t\geq 0, x\in\intervalleff{-A/2}{A/2}}=(D_t^A(x))_{t\geq 0, x\in\intervalleff{-A/2}{A/2}}}\geq1-Ce^{-\alpha A}.\]
\subsection{Restriction of the LFFP$(p)$ to a finite box}
The aim of this subsection is to prove Theorem \ref{well posedness}. We define an analogous process of LFFP$(p)$ on a finite space interval, which can be perfectly simulated. We then show that these two processes are equal with very high probability.
\subsubsection{Algorithm}
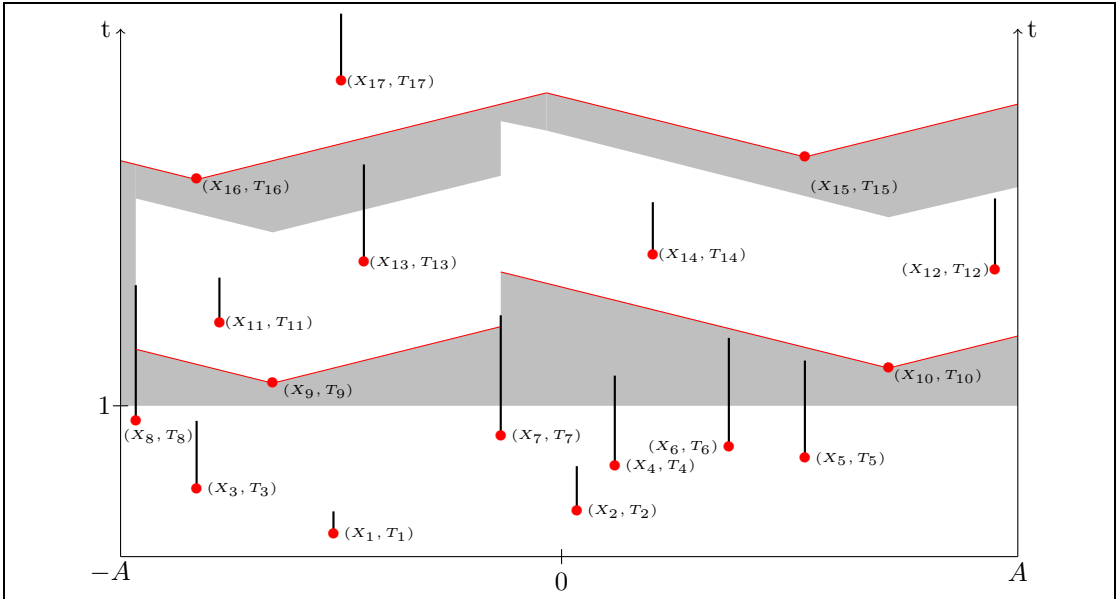
\begin{figure}[h!]
\fbox{
\begin{minipage}[c]{0.95\textwidth}
\centering
\begin{tikzpicture} % p=.25
\fill[fill=gray!50!white] (-6,2) -- (-6,5.25)--(-5.8,5.2)--(-5.8,2)--cycle;

\fill[fill=gray!50!white] (-5.8,2) --(-5.8,2.75)-- (-4,2.3)--(-1,3.05)--(-1,2)--cycle;
\draw[thick] (-5.8,1.8)--(-5.8,3.6)node at (-5.5,1.6) {{\tiny $(X_8,T_8)$}};
\draw[red] (-5.8,1.8) node {$\bullet$};
\draw[red] (-4,2.3) node {$\bullet$};
\draw node at (-3.4,2.2) {{\tiny $(X_9,T_9)$}};
\draw[red] (-4,2.3)--(-1,3.05);
\draw[red] (-4,2.3)--(-5.8,2.75);

\fill[fill=gray!50!white]
(-5,5)--(-.4,6.15)--(-.4,5.65)--(-1,5.775)--(-1,5.05)--(-4,4.3)--(-5.8,4.75)--(-5.8,5.2)--cycle;
\draw[red] (-5,5) node {$\bullet$};
\draw node at (-4.35,4.9) {{\tiny $(X_{16},T_{16})$}};
\draw[red] (-5,5)--(-.4,6.15);
\draw[red] (-5,5)--(-6,5.25);

\fill[fill=gray!50!white] (-1,2) -- (-1,3.775) -- (4.1,2.5) --
(5.8,2.925)--(5.8,2)--cycle;
\draw[thick] (3,1.3)--(3,2.6);
\draw[red] (3,1.3) node {$\bullet$};
\draw node at (3.6,1.3) {{\tiny $(X_{5},T_{5})$}};
\draw[red] (4.1,2.5) node {$\bullet$};
\draw[red] (4.1,2.5)--(-1,3.775);
\draw[red] (4.1,2.5)--(5.8,2.925);
\draw node at (4.75,2.4) {{\tiny $(X_{10},T_{10})$}};
\draw[thick] (-1,1.6)--(-1,3.2)node at (-.4,1.6) {{\tiny $(X_{7},T_7)$}};
\draw[red] (-1,1.6) node {$\bullet$};

\fill[fill=gray!50!white]
(4.1,4.5)--(-.4,5.65)--(-.4,6.15)--(3,5.3)--(5.8,6)--(5.8,4.9)--cycle;
\draw[red] (3,5.3) node {$\bullet$};
\draw[red] (3,5.3)--(5.8,6);
\draw[red] (3,5.3)--(-.4,6.15);
\draw node at (3.65,4.9) {{\tiny $(X_{15},T_{15})$}};

\draw[-] (-6,0)--(5.8,0) node at (-6,-.2) {$-A\phantom{-}$} node at (5.8,-.2) {$A$};
\draw[-] (-.2,0.1)--(-.2,-.1) node[below] {$0$};
\draw[->] (-6,0)--(-6,7) node[left] {t};
\draw[->] (5.8,0)--(5.8,7) node[right] {t};
\draw[-] (-6.1,2)--(-5.9,2) node[left, pos=0.5] {$1$};

\draw[thick] (-3.2,.3)--(-3.2,.6) node at (-2.6,.3) {{\tiny $(X_1,T_1)$}};
\draw[red] (-3.2,.3) node {$\bullet$};
\draw[thick] (0,.6)--(0,1.2) node at (.6,.6) {{\tiny $(X_2,T_2)$}};
\draw[red] (0,.6) node {$\bullet$};
\draw[thick] (2,1.45)--(2,2.9)node at (1.4,1.45) {{\tiny $(X_6,T_6)$}};
\draw[red] (2,1.45) node {$\bullet$};
\draw[thick] (-5,.9)--(-5,1.8)node at (-4.4,.9) {{\tiny $(X_3,T_3)$}};
\draw[red] (-5,.9) node {$\bullet$};
\draw[thick] (.5,1.2)--(.5,2.4)node at (1.1,1.2) {{\tiny $(X_4,T_4)$}};
\draw[red] (.5,1.2) node {$\bullet$};
\draw[thick] (-4.7,3.1)--(-4.7,3.7)node at (-4.05,3.1) {{\tiny $(X_{11},T_{11})$}};
\draw[red] (-4.7,3.1) node {$\bullet$};
\draw[thick] (-2.8,3.9)--(-2.8,5.2)node at (-2.15,3.9) {{\tiny $(X_{13},T_{13})$}};
\draw[red] (-2.8,3.9) node {$\bullet$};
\draw[thick] (5.5,3.8)--(5.5,4.75)node at (4.85,3.8) {{\tiny $(X_{12},T_{12})$}};
\draw[red] (5.5,3.8) node {$\bullet$};
\draw[thick] (1,4)--(1,4.7)node at (1.65,4) {{\tiny $(X_{14},T_{14})$}};
\draw[red] (1,4) node {$\bullet$};
\draw[thick] (-3.1,6.3)--(-3.1,7.2)node at (-2.45,6.3) {{\tiny $(X_{17},T_{17})$}};
\draw[red] (-3.1,6.3) node {$\bullet$};
\end{tikzpicture} \caption{LFFP$(p)$ in a finite box}\label{plffpdraw}
\vspace{.5cm}
\parbox{13.3cm}{
\footnotesize{
The marks of $\pi_M$ (matches) are represented as $\color{red}{\bullet}$'s. The filled zones represent zones in which $Z_t^A(x)=1$, that is
macroscopic clusters. In the rest of the
space, we always have $Z_t^A(x)<1$. The plain vertical segments represent the sites where
$H_t^A(x)>0$. $F_t^A(x)=0$ except on the lines with slope $p$ where $F^A_t(x)=1$ or $F_t^A(x)=2$ in the crossing point of the fires starting in $(X_{15},T_{15})$ and $(X_{16},T_{16})$. Until time $1$, all of the clusters are
microscopic. The first eigth marks of the Poisson measure fall in that zone. As a
consequence, at each of these marks, a process $H^A$ starts. Their lifetime is
equal to the instant where they have started (e.g., the segment above $(X_1,T_1)$
ends at time $2T_1$). At time $1$, all clusters where there has been no mark
become macroscopic and merge together. However, this is limited by vertical
segments. Here, at time $1$, we have the clusters $[-A,X_8]$, $[X_8,X_7]$,
$[X_7,X_4]$, $[X_4,X_6]$, $[X_6,X_5]$ and $[X_5,A]$. The segment above $(X_4,T_4)$
ends at time $2T_4$ and thus, at this time, the clusters $[X_7,X_4]$ and $[X_4,X_6]$
merge into $[X_7,X_6]$. The ninth mark falls in the (macroscopic) zone $[X_8,X_7]$
and thus two fires start. They cross the cluster $[X_8,X_7]$ at speed $p$, i.e. cross
$[X_8,X_7]$ with a slope $p$. A process $H^A$ then starts at $X_{11}$ at time $T_{11}$. Since $Z^A_{T_{11}-}(X_{11})=T_{11}-(T_9+p\abs{X_9-X_{11}})$
[because $Z^A_{T_9+p\abs{X_9-X_{11}}}(X_{11})$ has been set to $0$], the segment above
$(X_{11},T_{11})$ will end at time
$2T_{11}-(T_9+p\abs{X_9-X_{11}})$. 
On the other hand, a fire starts at $X_{10}$ at time $T_{10}$ and crosses the cluster
of $X_{10}$ at speed $p$. A site $x$ in $[X_7,A]$ remains microscopic from time $T_{10}+p\abs{X_{10}-x}$ until
time $T_{10}+p\abs{X_{10}-x}+1$. The two matches $14$ and $12$ create microscopic
fires (because they fall on sites where $Z^A_t(x)<1$). Observe finally that the 15th and the 16th fires are stopped by each oher.

With this realization, we have $0 \in(X_{7},X_{2})$ and, thus, 
$Z^A_t(0)= t$ for $t \in[0,1]$, then $Z^A_t(0)=1$ for $t \in [1,T_{10}+pX_{10})$,
then $Z^A_t(0)= t-(T_{10}+pX_{10})$ for $t \in[T_{10}+pX_{10},T_{10}+pX_{10}+1)$,
then $Z^A_t(0)=1$ for $t \in[T_{10}+pX_{10}+1,T_{16}+pX_{15})$, etc. We also
see that $D^A_t(0)=\{0\}$ for $t \in[0,1)$, $D^A_t(0)=[X_7,X_4]$ for $t\in[1,2T_4)$,
$D^A_t(0)=[X_7,X_6]$ for $t\in[2T_4,2T_{6})$,
$D^A_t(0)=[X_7,X_{10}+\frac{T_{10}-t}{p})$ for $t\in[2T_{6},T_{10}+pX_{10})$,
$D^A_t(0)=\{0\}$ for $t\in[T_{10}+pX_{10},T_{10}+pX_{10}+1)$, etc. We finally have $F_t^A(0)=0$ for all $t\neq\{T_{10}+pX_{10},T_{15}+pX_{15}\}$ and $F_{T_{10}+pX_{10}}^A(0)=F_{T_{15}+pX_{15}}^A(0)=1$.
}}
\end{minipage}}
\end{figure}

Let $p\in\intervallefo{0}{\infty}$. Here we show that when working on a
finite space interval, the LFFP$(p)$ is somewhat discrete. We consider a Poisson measure $\pi_M(\diff x,\diff t)$ on
$\rr\times\intervallefo{0}{\infty}$ with intensity measure $\diff x\diff t$.
\begin{defin}
Let $A>0$. A process $( Z^A_t(x), H^A_t(x), F^A_t(x))_{t\geq0, x\in\intervalleff{-A}{A}}$ with values in $\rr_+\times\rr_+\times\nn-$ such that
a.s., for all $x\in\intervalleff{-A}{A}$, $(Z^{A}_t(x),H^{A}_t(x))_{t\geq0}$ is
c\`adl\`ag, is a $A-$LFFP$(p)$ if a.s., for all $t\geq0$, all
$x\in\intervalleff{-A}{A}$,
\begin{align}
Z_t^{A}(x)&= \intot \indiq{Z_s^{A}(x)<1}\diff s-\sum_{s\leq t}(F^{A}_s\wedge 1),\notag\\
H_t^{A}(x)&= \intot Z_{s-}^{A}(x)\indiq{Z_{s-}^{A}(x)<1}\pi_M(\{x\}\times\diff
s)-\intot
\indiq{H_{s}^{A}(x)>0}\diff s,\\
F_t^{A}(x)&= \iint_{(y,s)\in\Lambda^p_{(x,t)}\cap(\intervalleff{-A}{A}\times\intervallefo{0}{\infty})} \indiq{\forall (r,v)\in \Lambda^p_{(x,t)}(y,s) ,\, Z_{v-}^{A}(r)=1\text{ and }H_{v-}^{A}(r)=0}\pi_M(\diff y,\diff s).\notag
\end{align}
\end{defin}
To the $A-$LFFP$(p)$, as usual, we associate the process $D_t^{A}(x) = [L_t^{A}(x),R_t^{A}(x)]$, with
\begin{align*}
L_t^{A}(x) =& (-A)\vee\sup\{ y\in\intervalleff{-A}{x}:\; Z_t^{A}(y)<1 \text{ or
} H_t^{A}(y)>0\},\\
R_t^{A}(x) =& A\wedge\inf\{ y\in\intervalleff{x}{A}:\; Z_t^{A}(y)<1 \text{ or }
H_t^{A}(y)>0\}.
\end{align*}

A typical path of
$(Z^{A}_t(x),H^{A}_t(x),F^{A}_t(x))_{t\geq0,x\in\intervalleff{-A}{A}}$ is drawn
in figure \ref{plffpdraw}.

The proof of the following proposition shows the construction of the
$A-$LFFP$(p)$ in an algorithmic way.

\begin{prop}\label{algo}
Consider a Poisson measure $\pi_M(\diff x,\diff t)$ on
$\rr\times\intervallefo{0}{\infty}$ with intensity measure $\diff x\diff t$. For any
$A>0$ and $p\ge0$, there a.s. exists a unique $A-$LFFP$(p)$ which can be perfectly simulated.
\end{prop}

\begin{proof}[Algorithm] Here we only treat the case $p>0$. The case $p=0$ is much easier and has been treated in \cite{bf}, as mentioned in Remark \ref{rem p=0}.

Consider the marks $(X_k,T_k)_{k=1,\dots,n}$ of $\pi_M$ in
$\intervalleff{-A}{A}\times\intervalleff{0}{T}$, ordered chronologically and set
$T_0=0$. We describe the construction via an algorithm, which also shows uniqueness, in the sense that there is no choice in the construction.

Suppose that we have built the process  $(Z^{A}_t(x),H^{A}_t(x),F^{A}_t(x))_{x\in\intervalleff{-A}{A}}$ at some time $t\geq0$. We then can set
\begin{align*}
\chi_t^+ &=\enstq{x\in\intervalleff{-A}{A}}{F_t^{A}(x)=1\text{ and }Z^{A}_t(x+)=1},\\
\chi_t^- &=\enstq{x\in\intervalleff{-A}{A}}{F_t^{A}(x)=1\text{ and }Z^{A}_t(x-)=1},\\
\chi_t^0
&=\enstq{x\in\intervalleff{-A}{A}}{H_t^{A}(x)>0\text{ or }Z_t^{A}(x+)\neq Z_t^{A}(x-)}\cup\{-A,A\},\\
\chi_t &=\chi_t^+\cup\chi_t^-\cup\chi_t^0,
\end{align*}
where $Z_t^A(x+)=\lim_{\substack{y\to x,\\y>x}}Z_t^A(y)$ (resp. $Z_t^A(x-)=\lim_{\substack{y\to x,\\y<x}}Z_t^A(y)$). Observe that $\chi_t^+$ (resp. $\chi_t^-$) is the set of fires at time $t$ that spread to the right (resp. to the left) and that $\chi^0_t$ is the set of sites where a fire can be stopped (barrier or microscopic zone).  We also define, for $r>t$, 
%\begin{align}
%\cE_{t}^{r}\coloneqq&\enstq{\mathcal{V}^p_{(x,t)}\Big(x+\frac{r-t}{p},r\Big)\cap\mathcal{V}^p_{(y,t)}\Big(y-\frac{r-t}{p},r\Big)}{x\in\chi_{t}^+,y\in\chi_{t}^-}\label{meeting fire}\\
%&\cup\enstq{\mathcal{V}^p_{(x,t)}\Big(x+\frac{r-t}{p},r\Big)\cap(\{y\}\times\intervalleff{t}{r})}{x\in\chi_{t}^+,y\in\chi_{t}^0}\label{stop fire left}\\
%&\cup\enstq{\mathcal{V}^p_{(x,t)}\Big(x-\frac{r-t}{p},r\Big)\cap(\{y\}\times\intervalleff{t}{r})}{x\in\chi_{t}^-,y\in\chi_{t}^0}.\label{stop fire right}
%\end{align}
\begin{align}
\cE_{t}^{r}\coloneqq&\bigcup_{x\in\chi_t^ +, y\in\chi_t^ -}\mathcal{V}^p_{(x,t)}\cap\mathcal{V}^p_{(y,t)}\cap(\intervalleff{-A}{A}\times\intervalleff{t}{r})\label{meeting fire}\\
&\cup\bigcup_{x\in\chi_t^+\cup\chi_t^-, y\in\chi_t^0}\cV^p_{(x,t)}\cap(\{y\}\times\intervalleff{t}{r}). \label{stop fire}%\\
% &\cup\bigcup_{x\in\chi_t^+, %y\in\chi_t^0}\mathcal{V}^p_{(x,t)}\cap(\{y\}\times\intervalleff{t}{r}).%\label{stop fire right}
\end{align}
The set \eqref{meeting fire} is the possible locations $(y,s)$ where two fires may meet during $\intervalleff{t}{r}$. The set \eqref{stop fire} is the possible locations $(y,s)$ where a fire may be stopped by a microscopic zone or a barrier during $\intervalleff{t}{r}$. Thus, $\cE_{t}^{r}$ is the set of possible locations $(y,s)$ where a fire may be stopped during $\intervalleff{t}{r}$, when no match falls in $\intervalleff{-A}{A}$ during $\intervalleff{t}{r}$.

\md

\noindent{\bf Step 0.} Put $Z_0^{A}(x)=H_0^{A}(x)=F^A_0(x)=0$ for all
$x\in\intervalleff{-A}{A}$.

Assume that, for some $q\in\{0,\dots,n-1\}$, the process 
$(Z^{A}_t(x),H^{A}_t(x),F^{A}_t(x))_{t\in\intervalleff{0}{T_q},x\in\intervalleff{-A}{A}}$
has been built.

\md

\noindent{\bf Step $q+1$.} We build $(Z^{A}_t(x),H^{A}_t(x),F^{A}_t(x))_{t\in\intervalleof{T_q}{T_{q+1}},x\in\intervalleff{-A}{A}}$
in the following way: for $x\in\intervalleff{-A}{A}$ and
$t\in\intervalleoo{T_q}{T_{q+1}}$, we set
$H_t^{A}(x)=\max(0,H_{T_q}^{A}(x)-(t-T_q))$. We then set, recall \eqref{meeting fire} and \eqref{stop fire},
\[\cE_{T_q}^{T_{q+1}}=\left\{(X^1_q,T^1_q),\dots,(X_q^{N},T_q^{N})\right\}\]
ordered chronologically, and put $(X^0_q,T_q^0)=(X_q,T_q)$ and $(X^{N+1}_q,T_q^{N+1})=(X_{q+1},T_{q+1})$. Observe that a.s. $T_q=T_q^0<T_q^1<\dots <T_q^N<T_q^ {N+1}=T_{q+1}$. Assume that the process has been built until $T^k_q$, for some $k\in\{0,\dots,N\}$. We then build the process on $\intervalleof{T^k_q}{T^{k+1}_q}$. Recall that no match falls in $\intervalleff{-A}{A}$ during the time interval $\intervalleoo{T_q^k}{T_q^{k+1}}$.

We first compute $(F_t^A(x))_{t\in\intervalleoo{T_q^k}{T_q^{k+1}},x\in\intervalleff{-A}{A}}$. Since a fire can't be stopped during $\intervalleoo{T^k_q}{T^{k+1}_q}$, if $x\in\chi^+_{T_q^k}$, we set $F_s^A(y)=1$ for all $(y,s)\in \cV^p_{(x,T_q^k)}(x+\frac{T_q^{k+1}-T_q^k}{p},T_q^{k+1})$, recall Subsection \ref{notations}, while, if $x\in\chi^-_{T_q^k}$, we set $F_s^A(y)=1$ for all $(y,s)\in  \cV^p_{(x,T_q^k)}(x-\frac{T_q^{k+1}-T_q^k}{p},T_q^{k+1})$. Otherwise, that is if $(y,s)\not\in\left(\bigcup_{x\in\chi_{T_q^k}^+} \cV^p_{(x,T_q^k)}(x+\frac{T_q^{k+1}-T_q^k}{p},T_q^{k+1})\right)\cup \left(\bigcup_{x\in\chi_{T_q^k}^-} \cV^p_{(x,T_q^k)}(x-\frac{T_q^{k+1}-T_q^k}{p},T_q^{k+1})\right)$, we set $F_s^A(y)=0$. To summarize, for all $(y,s)\in\intervalleff{-A}{A}\times\intervalleoo{T_q^k}{T_q^{k+1}}$, we have
\[F_s^A(y)=\begin{cases}
1 &\text{if } y-\frac{s-T_q^k}{p}\in\chi_{T_q^k}^+\\
1 &\text{if } y+\frac{s-T_q^k}{p}\in\chi_{T_q^k}^-\\
0 &\text{else.}
\end{cases}\]

We then compute $(Z_t^A(x))_{t\in\intervalleoo{T_q^k}{T_q^{k+1}},x\in\intervalleff{-A}{A}}$. Let us fix $x\in\intervalleff{-A}{A}$. We set $N_x\coloneqq\#\enstq{s\in\intervalleoo{T_q^k}{T_q^{k+1}}}{F^A_s(x)=1}$ and $\tau_0\coloneqq T_q^k$. If $N_x\geq1$, for $j=0,\dots,N_x-1$, we set $\tau_{j+1}\coloneqq\inf\enstq{s\in\intervalleoo{\tau_j}{T^{k+1}_q}}{F^{A}_s(x)=1})$. While $x$ isn't crossed by a fire,  $Z_s^A(x)$ grows linearly. We thus have, for all $s\in\intervalleoo{T^k_q}{T^{k+1}_q}$
\[Z_{s}^{A}(x)=\begin{cases}
\min(Z_{T_q^k}^{A}(x)+s-T_q^k,1) &\text{if } s\in\intervalleoo{T_q^k}{\tau_{1}},\\
\min(s-\tau_{j},1) &\text{if } s\in\intervallefo{\tau_{j}}{\tau_{j+1}}\text{ and }N_x\geq j\geq1,\\
\min(s-\tau_{N_x},1) &\text{if } s\in\intervallefo{\tau_{N_x}}{T_q^{k+1}}.
\end{cases}\]
if $N_x\geq1$, whereas
\[Z_{s}^{A}(x)=\min(Z_{T_q^k}^{A}(x)+s-T_q^k,1)\]
if $N_x=0$.

We finally compute $F_{T^{k+1}_q}^A(x)$, $Z^{A}_{T^ {k+1}_q}(x)$ and $H^A_{T_{q+1}}(x)$ for all $x\in\intervalleff{-A}{A}$.

\md

{\it Case 1.} If $x\neq X_q^{k+1}$, observe that at most one fire can reach $x$ at time ${T^ {k+1}_q}$ (else $x\in\cE_{T_q^k}^{T_q^{k+1}}$). If $x-\frac{T_q^{k+1}-T_q^k}{p}\in\chi_{T_q^ k}^ +$ or $x+\frac{T_q^{k+1}-T_q^k}{p}\in\chi_{T_q^ k}^-$, that is if a fire reaches $x$ at time $T_q^{k+1}$, we set $F_{T_q^{k+1}}^A(x)=1$ and $Z_{T_q^{k+1}}^A(x)=0$. Else, we set $F_{T_q^{k+1}}^A(x)=0$ and $Z_{T_q^{k+1}}^A(x)=Z_{T_q^{k+1}-}^A(x)$.

\md

{\it Case 2.} If $x=X_q^{k+1}$ and $k<N$, observe that $X_q^{k+1}$ isn't crossed by a fire during $\intervalleoo{T^k_q}{T^{k+1}_q}$ i.e. $N_{X^ {k+1}_q}=0$. If $X_q^{k+1}-\frac{T_q^{k+1}-T_q^k}{p}\not\in\chi_{T_q^k}^+$ and $X_q^{k+1}+\frac{T_q^{k+1}-T_q^k}{p}\not\in\chi_{T_q^k}^-$ (i.e. if the fire which might have reached $X_q^{k+1}$ has been stopped before $T_q^k$) or if $H^{A}_{T_q^{k+1}-}(X_{q}^{k+1})>0$ or $Z^{A}_{T_q^{k+1}-}(X_{q}^{k+1})<1$ (i.e. if there has been recently a microscopic fire), then put $F_{T_q^{k+1}}^A(X_q^{k+1})=0$. Else, there is one (or two) fire that reaches $X_q^{k+1}$ at time $T_q^{k+1}$ and we set $F^A_{T_q^{k+1}}(X_q^{k+1})=1$ (or $2$). To summarize, we put
\begin{multline*}
F_{T_q^{k+1}}^A(X_q^{k+1})=\indiq{H^{A}_{T_q^{k+1}-}(X_{q}^{k+1})=0\text{ and }Z^{A}_{T_q^{k+1}-}(X_{q}^{k+1})=1}\\
\times\left(\indiq{X_q^{k+1}-\frac{T_q^{k+1}-T_q^k}{p}\in\chi_{T_q^k}^+}+\indiq{X_q^{k+1}+\frac{T_q^{k+1}-T_q^k}{p}\in\chi_{T_q^k}^-}\right).
\end{multline*}
We finally put
\[Z_{T^{k+1}_q}^{A}(X_q^{k+1})=Z_{T^{k+1}_q-}^{A}(X_q^{k+1})\indiq{F_{T_q^{k+1}}^A(X_q^{k+1})=0}.\]

\md

{\it Case 3.} If $x=X_{q+1}=X_q^{N+1}$ and $k=N$, a match falls in $X_{q+1}$ at  time $T_{q+1}=T_q^{N+1}$. We then set
\[Z_{T_{q+1}}^{A}(X_{q+1})=Z_{T_{q+1}-}^{A}(X_{q+1}) \indiq{Z_{T_{q+1}-}^{A}(X_{q+1})<1}\]
and
\[F^A_{T_{q+1}}(X_{q+1})=\indiq{Z_{T_{q+1}-}^{A}(X_{q+1})=1}.\]

To conclude the construction, we set, for all $x\in\intervalleff{-A}{A}$
\[H^A_{T_{q+1}}(x)=\begin{cases}
H^A_{T_{q+1}-}(x) &\text{if } x\neq X_{q+1},\\
Z^A_{T_{q+1}-}(X_{q+1})\indiq{Z^A_{T_{q+1}-}(X_{q+1})<1} & \text{if } x=X_{q+1}.\qedhere
\end{cases}\]
\end{proof}

\subsubsection{Restriction of the LFFP$(p)$ to a finite box}
We now prove a refined version of Theorem \ref{well posedness}.
\begin{prop}\label{restriction limite}
Let $p\in\intervallefo{0}{\infty}$ and $\pi_M$ be a Poisson measure on $\rr\times\intervallefo{0}{\infty}$ with
intensity measure $\diff x\diff t$.
\begin{enumerate}
        \item There exists a unique LFFP$(p)$
$(Z_t(x),H_t(x),F_t(x))_{t\geq 0,x\in\rr}$.
        \item It can be perfectly simulated on
$\intervalleff{-n}{n}\times\intervalleff{0}{T}$ for any $T>0$, any $n>0$.
        \item For $A>0$, let
$(Z^{A}_t(x),H^{A}_t(x),F^{A}_t(x))_{t\geq0,x\in\intervalleff{-A}{A}}$ be the
unique $A-$LFFP$(p)$ and the associated $(D^{A}_t(x))_{t\geq0,x\in\intervalleff{-A}{A}}$. There holds
\begin{multline} 
\mathbb{P}\left[
(Z_t(x),H_t(x),F_t(x),D_t(x))_{t\in\intervalleff{0}{T},x\in\intervalleff{-A/2}{A/2}}\right.\\
        \left.
=(Z^{A}_t(x),H^{A}_t(x),F^{A}_t(x),D^{A}_t(x))_{t\in\intervalleff{0}{T},x\in\intervalleff{-A/2}{A/2}}\right]\geq
1-C_Te^{-\alpha_T A}\label{proba restriction}
\end{multline}
for some constants $\alpha_T>0$ and $C_T>0$ not depending on $A>0$.
\end{enumerate}
\end{prop}
\begin{proof}%[Proof for $p\in\intervallefo{0}{\infty}$]
We divide the proof into several step. We work on $\intervalleff{0}{T}$.

\md

\noindent{\bf Step 1.} We observe that for a mark $(X,\tau)$ of $\pi_M$ with $X\in\intervalleff{-A}{A}$, we have $H^A_t(X)>0$ or $Z^A_t(X)<1$ for all $t\in\intervallefo{\tau}{\tau+1/2}$.
	
Indeed, assume first that $Z^A_{\tau-}(X)\in\intervallefo{0}{1/2}$. Then $Z_t^A(X)=Z^A_{\tau-}(X)+t-\tau<1$ for all $t\in\intervalleff{\tau}{\tau+1/2}$.

Assume next that $Z^A_{\tau-}(X)\in\intervallefo{1/2}{1}$. Then $H_\tau^A(X)=Z^A_{\tau-}\geq1/2$, so that $H^A_t(X)=H_\tau^A(X)-t+\tau>0$ for all $t\in\intervallefo{\tau}{\tau+1/2}$.

If finally $Z^A_{\tau-}(X)=1$, then $Z^A_{\tau}(X)=0$, whence $Z^A_{t}(X)=t-\tau<1$ for $t\in\intervallefo{\tau}{\tau+1}$.

\md

\noindent{\bf Step 2.} For $a\in\rr$, we consider the event $\Omega_a^l$ defined as follows: for $\{(X_k,T_k)\}_{k=1,\dots,n}$ the marks of $\pi_M$ restricted to $\intervallefo{a}{a+1}\times\intervalleff{0}{T}$ ordered chronologically, for $T_0=0$, $T_{n+1}=T$, we put $\Omega_a^l=\{\max_{i=0,\dots,n}(T_{i+1}-T_i)<1/4\}\cap\{\min_{i=1,\dots,n-1}(X_{i+1}-X_i)>0\}$.
	
We immediately deduce from Step 1 that for any $a\in\rr$, any $A>\abs{a}+1$,
\begin{multline*}
\Omega_a^l\subset\{\exists x:\intervalleff{0}{T}\to\intervalleoo{a}{a+1},t\mapsto
x_t\text{ non decreasing}\\
\text{and for all } t\in\intervalleff{0}{T}, H_t^A(x_t)>0\text{ or }Z_t^A(x_t)<1\}.
\end{multline*}
Thus, on $\Omega_a^l$, clusters on the left of $a$ cannot be connected to clusters on the right of $a+1$ during $\intervalleff{0}{T}$. Furthermore, since the function $x$ is non decreasing, a fire starting from the left of $a$ can't cross the zone $\intervalleoo{a}{a+1}$ (i.e. it necessarily would be stopped by some $x_{t_0}$). Thus, matches falling at the left of $a$ do not affect the zone $\intervalleoo{a+1}{\infty}$.

In the same way, we put $\Omega_a^r=\{\max_{i=0,\dots,n}(T_{i+1}-T_i)<1/4\}\cap\{\max_{i=1,\dots,n-1}(X_{i+1}-X_i)<0\}$. %, for $\{(X_k,T_k)\}_{k=1,\dots,n}$ the marks of $\pi_M$ restricted to $\intervallefo{a}{a+1}\times\intervalleff{0}{T}$ ordered chronologically such that $T_0=0$, $T_{n+1}=T$.
We of course have, for any $a\in\rr$, $A>\abs{a}+1$,
\begin{multline*}
\Omega_a^r\subset\{\exists y:\intervalleff{0}{T}\to\intervalleoo{a}{a+1},t\mapsto
y_t\text{ non increasing }\\
\text{and for all } t\in\intervalleff{0}{T}, H_t^A(y_t)>0\text{ or }Z_t^A(y_t)<1\}.
\end{multline*}
As above, on $\Omega_a^r$, clusters on the right of $a+1$ cannot be connected to clusters on the left of $a$ during $\intervalleff{0}{T}$ and the fact that $y$ is non increasing ensures us that matches falling on the right on $a+1$ do not affect the zone $\intervalleoo{-\infty}{a}$.

\md

\noindent{\bf Step 3.} Obviously, $q_T=\proba{\Omega_a^l}=\proba{\Omega_a^r}$ is positive and does not depend on $a$. Furthermore, $\Omega_a^l$ (resp.  $\Omega_a^r$) is independent of $\Omega_b^l$ (resp.  $\Omega_b^r$) for all $a,b\in\zz$ with $a\neq b$. Hence there are a.s. infinitely many $a\in\zz$ (resp. $b\in\zz$) such that $\Omega_a^l$ (resp. $\Omega_b^r$) is realized.
	
Then it is routine to deduce the well-posedness of the LFFP$(p)$. The perfect simulation algorithm on a finite-box $\intervalleff{-n}{n}\times\intervalleff{0}{T}$ is also easy: find $a_1<a_2$ with $a_1+1<-n<n<a_2$ such that $\Omega_{a_1}^l\cap\Omega_{a_2}^r$ is realized. Then apply the same rules as for the $A-$LFFP$(p)$ to simulate the process in $\intervalleff{a_1}{a_2+1}$. This will give the true LFFP$(p)$ inside $\intervalleff{a_1+1}{a_2}$ during $\intervalleff{0}{T}$.

Finally, we can clearly bound from below the left hand side of \eqref{proba restriction} by
\[\proba{(\cup_{a\in\intervalleff{-A}{-A/2-1}\cap\zz}\,\Omega_a^l)\cap(\cup_{a\in\intervalleff{A/2}{A-1}\cap\zz}\,\Omega_a^r)}\geq 1-2(1-q_T)^{A/2-2}\]
whence \eqref{proba restriction} with $C_T=2/(1-q_T)^2$ and $\alpha_T=-\log(1-q_T)/2$.
\end{proof}

\section{Propagation Lemmas}\label{section propagation lemma}
Here we study the
propagation of a fire through an occupied cluster. When a match falls on an occupied
cluster, two fires start: one goes to the left and one goes to the right. This
propagation is not necessarily linear, it sometimes can regress. However there are
few 'sparks'.

Consider two families of Poisson processes $(N_t^S(i))_{t\geq0,i\in\zz}$ and
$(N_t^P(i))_{t\geq0,i\in\zz}$ with respective rates $1$ and $\pi$, all these
processes being independent. We consider {\it the propagation process ignited at $(0,0)$} defined by
\begin{align*}
\check{\zeta}^{\la,\pi}_t(i) =& 1+\indiq{i=0}+\int_0^t\indiq{{\check{\zeta}}^{\la,\pi}_{s-}(i)=0}\diff
N_s^S(i)\\
&+\int_0^t \indiq{{\check{\zeta}}_{s-}^{\la,\pi}(i+1)=2,{\proco}_{s-}^{\la,\pi}(i)=1}\diff
N_s^P(i+1)+\int_0^t
\indiq{{\proco}_{s-}^{\la,\pi}(i-1)=2,\proco_{s-}^{\la,\pi}(i)=1}\diff N_s^P(i-1)\\
&- 2\int_0^t \indiq{{\proco}_{s-}^{\la,\pi}(i)=2}\diff N_s^P(i).
\end{align*}
Roughly, the process $(\proco^{\la,\pi}_t(i))_{t\geq0,i\in\zz}$ starts from an occupied initial configuration and a match falls on the site $0$ at time $0$. Afterwards the fire spreads into $\zz$. We are interested in the space-time position of burning trees (i.e. $(i,t)\in\zz\times\intervallefo{0}{\infty}$ such that $\proco^{\la,\pi}_t(i) =2$), when $\la\to0$ and $\pi\to\infty$ in the different regimes.

We set, for $t\geq0$,
\begin{align}
i_t^+ &=\max\enstq{i\geq0}{\proco_t^{\la,\pi}(i)=2}\label{right front}\\
i_t^- &=\min\enstq{i\leq0}{\proco_t^{\la,\pi}(i)=2}\label{left front}
\end{align}
the right and the left fronts at time $t$. Observe that $(i_t^+)_{t\geq0}$ and $(-i_t^-)_{t\geq0}$ are two Poisson processes with intensity $\pi$. For $i\in\zz$, we set
\begin{align}
T_i &=\inf\enstq{s\geq0}{\check{\zeta}^{\la,\pi}_s(i)=2}\label{def burning time}\\
&=\begin{cases}
\inf\enstq{s\geq0}{i_s^+=i} &\text{if }i\geq0,\\
\inf\enstq{s\geq0}{i_s^-=i} &\text{if }i\leq0,
\end{cases}\notag
\end{align}
which represents the first time that the site $i\in\nn$ is burning. We clearly have for all $t\geq0$, 
\[\proco^{\la,\pi}_t(i_t^-)=2=\proco^{\la,\pi}_t(i_t^+)\]
and for all $i\not\in\intervalleentier{i_t^-}{i_t^+}$,
\[\proco^{\la,\pi}_t(i)=1.\]
In this section, we will show that burning trees at some time $t$ are \emph{concentrated} around $i_t^+$ and $i_t^-$. We say that a site $i$ is a  \emph{spark} at time $t$ if it is a burning tree such that $i\not\in\{i^-_t,i^+_t\}$.

We recall that $\al=\log(1/\la)$, $\nl=\left\lfloor\frac1{\la\al}\right\rfloor$ and we introduce $\e_\la=\frac{1}{\ba_\la^{3}}$. For $B>0$, we finally set $B_\la=\lfloor B\nl\rfloor$.

The following Definition will be usefull.
\begin{defin}\label{close to}
Let $p\geq0$. In the rest of the paper, we will say that \emph{a statement $\cS(\la,\pi)$ holds for all $(\la,\pi)$ sufficiently close to the regime $\cR(p)$} if there are $\e_0>0$ and $\la_0\in\intervalleoo{0}{1}$ such that for all $\la\in\intervalleoo{0}{\la_0}$ and all $\pi\geq1$ such that $\abs{\frac{\nl}{\al\pi}-p}<\e_0$, the statement $\cS(\la,\pi)$ holds.

Similarly, let $z_0\in\intervalleff{0}{1}$. We will say that \emph{a statement $\cS(\la,\pi)$ holds for all $(\la,\pi)$ sufficiently close to the regime $\cR(\infty,z_0)$} if there are $\e_0>0$, $\la_0\in\intervalleoo{0}{1}$ and $K_0>0$ such that for all $\la\in\intervalleoo{0}{\la_0}$ and all $\pi\geq 1$ such that $\frac{\nl}{\al\pi}>K_0$ and $\abs{\frac{\log(\pi)}{\log(1/\la)}-z_0}<\e_0$, the statement $\cS(\la,\pi)$ holds.
\end{defin}

\subsection{Propagation lemma in the regime $\cR(p)$, for some $p\in\intervalleoo{0}{\infty}$}
We first study the propagation in the regime $\cR(p)$, for some $p>0$.
\begin{lem}\label{propagation lemma p}
Let $p>0, T>0$. There exists an event $\Omega^{P,T}_{\la,\pi}$ depending only on the Poisson processes $(N_t^S(i),N^P_t(i))_{t\in\intervalleff{0}{\al (T+\e_\la)},i\in\intervalleentier{-\lfloor\al\pi(T+\e_\la)\rfloor}{\lfloor\al\pi(T+\e_\la)\rfloor}}$ such that
\begin{multline*}
\Omega^{P,T}_{\la,\pi}\subset\{\text{At any time } t\in\intervalleff{0}{\al T},
\text{any burning tree belongs to }\\
\intervalleentier{-\lfloor(t+\al\e_\la)\pi\rfloor}{-\lfloor(t-\al\e_\la)\pi\rfloor}\cup\intervalleentier{\lfloor(t-\al\e_\la)\pi\rfloor}{\lfloor(t+\al\e_\la)\pi\rfloor}\\
\text{ and is either }i_t^+ \text{ or } i_t^-\text{ or has vacant neighbors}\},
\end{multline*}
where the event on the right concerns $(\proco^{\la,\pi}_t(i))_{i\in\zz,t\geq0}$, and 
\[\lim_{\la,\pi}\proba{\Omega^{P,T}_{\la,\pi}}=1\]
when $\la\to0$ and $\pi\to\infty$ in the regime $\cR(p)$.
\end{lem}

\begin{proof}
Recall that a \emph{spark} at time $t$ is a burning tree $i$ such that $i\not\in\{i^-_t,i^+_t\}$. We say that a site $i$ \emph{propagates for the first time} when the first fire at $i$ extinguishes and spreads to its neighbors (if they are occupied). Observe that for $i\geq0$, this happens at time $T_{i+1}$, while for $i\leq0$, this happens at time $T_{i-1}$.

Consider, for $i\geq0$, the events
\begin{multline}\label{omega1}
\Omega_i^1 =\{i\text{ remains vacant from the instant at which it propagates for the first time}\\
\text{until the instant at which the fire in $i+1$ propagates for the first time}\}
\end{multline}
and 
\begin{multline}\label{omega2}
\Omega_i^2=\{i\text{ is occupied when the fire in $i+1$ propagates for the first time,}\\
\text{ but then, $i$ burns for the second time during less than }\al\e_\la/4\\
\text{ and no seed has fallen on its neighbors $i-1$, $i+1$}\\
\text{from the instant they burnt for the first time until $i$ propagates for the second time}\}
\end{multline}
and similar events for $i\leq0$ (replace $i+1$ by $i-1$). Recall \eqref{right front}, \eqref{left front} and remark that the event on the right hand side in Lemma \ref{propagation lemma p} contains the event
\begin{multline*}
\Omega^{P,T}_{\la,\pi}=\left\{\sup_{t\in\intervalleff{0}{\al T}}\abs{i_t^+-\pi t}\leq \frac{\al\pi\e_\la}{2}\right\}\cap\left\{\sup_{t\in\intervalleff{0}{\al T}}\abs{i_t^-+\pi t}\leq \frac{\al\pi\e_\la}{2}\right\}\\
\cap\{\forall i\in\intervalleentier{i_{\al T}^-+1}{i_{\al T}^+-1}, \Omega_i^1\text{ or }\Omega_i^2\text{ is realized}\}.
\end{multline*}
Indeed, the two first terms ensure that the right (resp. left) front at time $t\in\intervalleff{0}{\al T}$ belongs to $\intervalleentier{\lfloor(t-\al\e_\la/2)\pi\rfloor}{\lfloor(t+\al\e_\la/2)\pi\rfloor}$ (resp. $\intervalleentier{-\lfloor(t+\al\e_\la/2)\pi\rfloor}{-\lfloor(t-\al\e_\la/2)\pi\rfloor}$). This in particular implies that for all $i\in\intervalleentier{-\lfloor(T-\e_\la/2)\al\pi\rfloor}{\lfloor(T-\e_\la/2)\al\pi\rfloor}$, 
\[T_i\in\left[\frac{|i|}{\pi}-\frac{\al\e_\la}{2},\,\frac{|i|}{\pi}+\frac{\al\e_\la}{2}\right].\]
The last term says that either $i$ remains vacant until $i+1$ propagates (i.e. there is no spark) or a seed has fallen on $i$ but then $i$ has vacant neighbors when it propagates for the second time (i.e. the spark has a size $1$). Finally remark that on $\Omega^{P,T}_{\la,\pi}$, for $t\in\intervalleff{0}{\al T}$,
\[\enstq{0\leq i\leq i_t^+}{T_{i+2}+\frac{\al\e_\la}{4}\geq t}\subset\intervalleentier{\lfloor(t-\al\e_\la)\pi\rfloor}{i_t^+}\]
and 
\[\enstq{0\geq i\geq i_t^-}{T_{i-2}+\frac{\al\e_\la}{4}\geq t}\subset\intervalleentier{i_t^-}{-\lfloor(t-\al\e_\la)\pi\rfloor},\]
thus a burning tree (i.e. a front or a spark) necessarily belongs to $\intervalleentier{\lfloor(t-\al\e_\la)\pi\rfloor}{\lfloor(t+\al\e_\la)\pi\rfloor}\cup\intervalleentier{-\lfloor(t+\al\e_\la)\pi\rfloor}{-\lfloor(t-\al\e_\la)\pi\rfloor}$, as desired.

Clearly, $\Omega^{P,T}_{\la,\pi}$ depends only on the Poisson processes $(N_t^S(i),N^P_t(i))_{t\geq0,i\in\zz}$ through $t\in\intervalleff{0}{\al (T+\e_\la)}$ and $i\in\intervalleentier{-\lfloor\al\pi(T+\e_\la)\rfloor}{\lfloor\al\pi(T+\e_\la)\rfloor}$. It remains to prove that $\proba{\Omega^{P,T}_{\la,\pi}}$ tends to $1$ when $\la\to0$ and $\pi\to\infty$ in the regime $\cR(p)$.

Since $(i_t^+)_{t\geq0}$ and $(-i_t^-)_{t\geq0}$ are two Poisson processes with intensity $\pi$, the maximal inequality for martingales gives
\begin{align}
\proba{\sup_{t\in\intervalleff{0}{\al T}}\abs{i_t^-+\pi t}> \frac{\al\pi\e_\la}{2}}&=\proba{\sup_{t\in\intervalleff{0}{\al T}}\abs{i_t^+-\pi t}> \frac{\al\pi\e_\la}{2}}\notag\\
& \leq\left(\frac2{\al\pi\e_\la}\right)^4\times(3(\al\pi T)^2+\al\pi T)\notag\\
&\leq \frac{16T^2}{(\al\pi\e_\la^2)^2}=\frac{16T^2\ba_\la^{10}}{\pi^2}\label{max ineq}
\end{align}
which tends to $0$ when $\la\to0$ and $\pi\to\infty$ in the regime $\cR(p)$.

Next, for all $i\geq0$, we have 
\begin{equation}\label{proba omega1}
\proba{\Omega_i^1}=\frac{\pi}{1+\pi}
\end{equation}
because seeds fall on $i$ at rate $1$ while the fire on $i+1$ propagates at rate $\pi$.

Now, for all $i\geq0$, we set
\begin{align*}
X_i &=\inf\enstq{s>T_{i+1}}{N^S_s(i)-N^S_{T_{i+1}}(i)>0}-T_{i+1},\\
Y^1_i &=  T_{i+1}-T_{i},\\
Y^2_i &= \inf\enstq{s>T_{i+2}}{N^P_s(i)-N^P_{T_{i+2}}(i)>0}-T_{i+2}.
\end{align*}
Let $i\geq0$. At time $T_i$, the site $i$ is burning and propagates to neighbors at time $T_{i+1}$. Thus, $X_i$ is the time we have to wait for a seed to fall again on $i$ after it propagates for the first time. Furthermore, $Y_i^1$ stands for the duration that $i$ is burning for the first time. If a seed falls on $i$ before $T_{i+2}$, that is before the burning tree $i+1$ propagates, then $i$ becomes again burning at time $T_{i+2}$ and burns during $\intervallefo{T_{i+2}}{T_{i+2}+Y_i^2}$.

The random variables $(X_i)_{i\in\nn}$ are exponential random variables with parameter $1$ and the random variables $(Y_i^1)_{i\in\nn}$ and $(Y_i^2)_{i\in\nn}$ are exponential random variables with parameter $\pi$. All these random variables are independent.

Then observe that 
\begin{equation}\label{star}
\Omega_i^2=\left(\{X_i\leq Y_{i+1}^1\}\cap\{Y_i^2<\frac{\al\e_\la}{4}\}\cap\{X_{i-1}>Y_i^1+Y_{i+1}^1+Y_i^2\}\cap\{X_{i+1}>Y_i^2\}\right).
\end{equation}
We have by independence
\begin{align*}
\probacond{\Omega_i^2}{Y_i^1,Y_{i+1}^1,Y_i^2} &=(1-e^{-Y_{i+1}^1})\times\indiq{Y_i^2\leq\frac{\al\e_\la}{4}}\times e^{-(Y_i^1+Y_{i+1}^1+Y_i^2)}\times e^{-Y_i^2}\\
&=(1-e^{-Y_{i+1}^1})\times e^{-Y_{i+1}^1}\times e^{-Y_i^1}\times e^{-2Y_i^2}\times\indiq{Y_i^2\leq\frac{\al\e_\la}{4}}.
\end{align*}
Integrating,
\begin{align}
\proba{\Omega^2_i}&=\pi^3\int_0^\infty(1-e^{-x})e^{-(\pi+1)x}\diff x\times\int_0^\infty e^{-(\pi+1)y}\diff y\times\int_0^{\al\e_\la/4} e^{-(\pi+2)z}\diff z\notag\\
&=\frac{\pi^3}{(1+\pi)^2(2+\pi)^2}(1-e^{-(2+\pi)\al\e_\la/4}).\label{proba omega 2}
\end{align}
Finally, note that, in the regime $\cR(p)$,
\begin{align*}
\proba{\Omega_i^1\cup\Omega_i^2}=\proba{\Omega_i^1}+\proba{\Omega_i^2} &=\frac{\pi}{1+\pi}+\frac{\pi^3}{(1+\pi)^2(2+\pi)^2}(1-e^{-(2+\pi)\al\e_\la/4})\\
&=1-\frac{5\pi^2+8\pi+4+\pi^3e^{-(2+\pi)\al\e_\la/4}}{(1+\pi)^2(2+\pi)^2}\\
&\geq 1-\frac{\alpha}{\pi^2}
\end{align*}
for some constant $\alpha>0$, because $e^{-(2+\pi)\al\e_\la/4}\ll 1/\pi$ when $\la\to0$ and $\pi\to\infty$ in the regime $\cR(p)$ (indeed, $\pi\sim1/(p\la\log^2(1/\la))$ whence $(2+\pi)\al\e_\la\simeq 1/(p\la\log^3(1/\la))$). Similar computations hold for $i\leq0$.

Consequently, the probability of $\{\forall i\in\intervalleentier{i_{\al T}^-+1}{i_{\al T}^+-1}, \Omega_i^1\text{ or }\Omega_i^2\text{ is realized}\}$ knowing $\left\{\sup_{t\in\intervalleff{0}{\al T}}\abs{i_t^+-\pi t}\leq \frac{\al\pi\e_\la}{2}\right\}\cap\left\{\sup_{t\in\intervalleff{0}{\al T}}\abs{i_t^-+\pi t}\leq \frac{\al\pi\e_\la}{2}\right\}$ is bounded from below by
\begin{align}
1-\sum_{i=-\lfloor\al\pi (T+\e_\la)\rfloor}^{\lfloor\al\pi (T+\e_\la)\rfloor}\proba{(\Omega_i^1\cup\Omega_i^2)^c}&= 1-\sum_{i=-\lfloor\al\pi (T+\e_\la)\rfloor}^{\lfloor\al\pi (T+\e_\la)\rfloor}(1-\proba{\Omega_i^1}-\proba{\Omega_i^2})\notag\\
&\geq 1-\alpha\frac{\al\pi(T+1)}{\pi^2}= 1-\alpha_T\frac{\al}{\pi}\label{ineq propa p}
\end{align}
which tends to $1$ when $\la\to0$ and $\pi\to\infty$ in the regime $\cR(p)$. Gathering \eqref{max ineq} and \eqref{ineq propa p} concludes the proof of Lemma \ref{propagation lemma p}.
\end{proof}

\subsection{Propagation lemma in the regime $\cR(0)$}
For all $A>0$, we set
\begin{equation}
\varkappa^{A}_{\la,\pi}=\frac{\nl A}{\al\pi}+\e_\la
\end{equation}
which tends to $0$ when $\la\to0$ and $\pi\to\infty$ in the regime $\cR(0)$.
\begin{lem}\label{propagation lemma 0}
Let $A, B>0$. There exists an event $\Omega^{P,A,B}_{\la,\pi}$ depending only on the Poisson processes $(N_t^S(i),N^P_t(i))_{t\in\intervalleff{0}{\al\varkappa_{\la,\pi}^{A\vee B}},i\in\intervalleentier{-A_\la-\ml}{B_\la+\ml}}$ such that
\begin{multline*}
\Omega^{P,A,B}_{\la,\pi}\subset\{\text{There is no more burning tree in } \intervalleentier{-A_\la}{B_\la}
\text{ at time }\al\varkappa_{\la,\pi}^{A\vee B}\\
\text{ and a burning tree in }\intervalleentier{-A_\la}{B_\la}\text{ at some time }0\leq t\leq \al\varkappa_{\la,\pi}^{A\vee B}\\
\text{ is either }i_t^+ \text{ or } i_t^-\text{ or has vacant neighbors}\}
\end{multline*}
where the event on the right concerns $(\proco^{\la,\pi}_t(i))_{t\geq0,i\in\zz}$, and 
\[\lim_{\la,\pi}\proba{\Omega^{P,A,B}_{\la,\pi}}=1\]
when $\la\to0$ and $\pi\to\infty$ in the regime $\cR(0)$.
\end{lem}

\begin{proof}
Recall \eqref{def burning time}, \eqref{omega1} and \eqref{omega2}. We set
\begin{multline*}
\Omega^{P,A,B}_{\la,\pi}=\left\{T_{B_\la+\ml}\leq \frac{\nl B}{\pi}+\frac{\al\e_\la}{2}\right\}\cap\left\{T_{-A_\la-\ml}\leq \frac{\nl A}{\pi}+\frac{\al\e_\la}{2}\right\}\\
\cap\bigcap_{i\in\intervalleentier{-A_\la-\ml+1}{B_\la+\ml-1}}(\Omega_i^1\cup\Omega_i^2)\\
\cap\left\{\exists i\in\intervalleentier{-A_\la-\ml+1}{-A_\la}, N^S_{\al\varkappa_{\la,\pi}^{A\vee B}}(i)=0\right\}\cap\left\{\exists i\in\intervalleentier{B_\la}{B_\la+\ml-1}, N^S_{\al\varkappa_{\la,\pi}^{A\vee B}}(i)=0\right\}.
\end{multline*}
Observe now that the event on the right hand side in Lemma \ref{propagation lemma 0} contains the event $\Omega^{P,A,B}_{\la,\pi}$. Indeed, the two first terms ensure that the left and the right fronts are outside $\intervalleentier{-A_\la}{B_\la}$ at time $\al\varkappa_{\la,\pi}^{A\vee B}$ whereas the third term ensures that a spark burns not for a long time and has vacants neighbors. The two last terms prevent from a return of a fire.

It remains to prove that $\proba{\Omega^{P,A,B}_{\la,\pi}}$ tends to $1$. First, observe that $T_{B_\la+\ml}$ is a sum of $B_\la+\ml$ i.i.d. exponential random variables with parameter $\pi$, then, Chebyshev's inequality implies
\begin{align*}
\proba{T_{B_\la+\ml}>\frac{\nl B}{\pi}+\frac{\al\e_\la}{2}}\leq \proba{\abs{T_{B_\la+\ml}-\frac{\nl B}{\pi}}>\frac{\al\e_\la}{2}} &\leq \frac{4}{(\al\e_\la)^2}\frac{B_\la+\ml}{\pi^2}\\
&\leq C_B\frac{\nl}{\al\pi}\frac{1}{\al\pi\e_\la^ 2}
\end{align*}
which tends to $0$ when $\la\to0$ and $\pi\to\infty$ in the regime $\cR(0)$. Similar computation holds for $T_{-A_\la-\ml}$.

A basic calculation, as in \eqref{ineq propa p}, shows that (because it also holds true that $e^{-(2+\pi)\al\e_\la/4}\ll 1/\pi$ in the regime $\cR(0)$)
\[\proba{\bigcap_{i\in\intervalleentier{-A_\la-\ml+1}{B_\la+\ml-1}}(\Omega_i^1\cup\Omega_i^2)} \geq1-\alpha_T\frac{\al}{\pi}\quad(\text{for some }\alpha_T>0),\]
which tends to $1$ when $\la\to0$ and $\pi\to\infty$ in the regime $\cR(0)$.

Finally, as soon as $\varkappa_{\la,\pi}^{A\vee B}\leq\frac{1}{2}$, it holds that, using space stationarity,
\begin{align*}
\proba{\exists i\in\intervalleentier{B_\la}{B_\la+\ml-1}, N^S_{\al\varkappa_{\la,\pi}^{A\vee B}}(i)=0}\geq&\,\proba{\exists i\in\intervalleentier{0}{\ml-1}, N^S_{\al/2}(i)=0}\\
&=1-(1-e^{-\al/2})^{\ml-1}\simeq1-e^{-\sqrt{\la}(\ml-1)}
\end{align*}
which tends to $1$ when $\la\to0$ and $\pi\to\infty$ in the regime $\cR(0)$.
\end{proof}

\subsection{Propagation lemma in the regime $\cR(\infty,z_0)$}
We first introduce, for $\la\in\intervalleof{0}{1}$ and $\gamma\in\intervalleoo{0}{1}$,
\[\bm_\la^\gamma=\left\lfloor\frac{\gamma}{\la^ {\gamma+(1-\gamma)z_0}\al}\right\rfloor.\]
For $z_0=1$, $\bm_\la^ \gamma$ is nothing but $\lfloor\gamma\nl\rfloor$. For $z_0\in\intervallefo{0}{1}$ and $\gamma\in\intervalleoo{0}{1}$, observe that 
\[z_0<\gamma+(1-\gamma)z_0<1,\]
so that $\bm_\la^ \gamma\ll\nl$. In any cases, we have $\bm_\la^ \gamma/\nl\leq\gamma$.

\begin{lem}\label{propagation lemma infty}%\quad
Let $T>0$. For all $z_0\in\intervalleff{0}{1}$ and all $\gamma\in\intervalleoo{0}{1}$, there exists an event $\Omega^{P,T,\gamma}_{\la,\pi}$ depending only on the Poisson processes $(N_t^S(i),N^P_t(i))_{t\in\intervalleff{0}{\al T},i\in\intervalleentier{-\bm_\la^ \gamma}{\bm_\la^ \gamma}}$, such that
\[\Omega^{P,T,\gamma}_{\la,\pi}\subset\{i_{\al T}^+ \text{ and } i_{\al T}^- \text{ belong to }\intervalleentier{-\bm_\la^ \gamma}{\bm_\la^ \gamma}\},\]
where the event on the right concerns the process $(\proco^{\la,\pi}_t(i))_{t\geq0,i\in\zz}$, and 
\[\lim_{\la,\pi}\proba{\Omega^{P,T,\gamma}_{\la,\pi}}=1\]
when $\la\to0$ and $\pi\to\infty$ in the regime $\cR(\infty,z_0)$.
\end{lem}
\begin{proof}
Recall \eqref{right front} and \eqref{left front}. We define
\[\Omega^{P,T,\gamma}_{\la,\pi}=\{0\leq i_{\al T}^+\leq\bm_\la^\gamma\}\cap\{-\bm_\la^\gamma\leq i_{\al T}^-\leq0\},\]
which clearly implies that $i_{\al T}^+$ and $i_{\al T}^-$ belong to $\intervalleentier{-\bm_\la^ \gamma}{\bm_\la^ \gamma}$. Markov's inequality shows that
\[\proba{i^-_{\al T}<-\bm_\la^ \gamma}=\proba{i^+_{\al T}>\bm_\la^ \gamma}\leq \frac{\al\pi T}{\bm_\la^ \gamma}\simeq \frac{T}{\gamma}\ba_\la^ 2\pi\la^ {\gamma+(1-\gamma)z_0},\]
which tends to $0$ when $\la\to0$ and $\pi\to\infty$ in the regime $\cR(\infty,z_0)$. Indeed, for $z_0=1$, then $\frac{T}{\gamma}\ba_\la^ 2\pi\la=\frac{T}{\gamma}\frac{\al\pi}{\nl}$ tends to $0$ (it is the definition of the regime $\cR(\infty,1)$), while, for $z_0\in\intervallefo{0}{1}$, $z_0<\gamma+(1-\gamma)z_0<1$, then $\frac{T}{\gamma}\ba_\la^ 2\pi\la^ {\gamma+(1-\gamma)z_0}=\frac{T}{\gamma}\frac{\ba_\la^2\pi}{\la^ {z_0}}\la^ {(1-z_0)\gamma}$ tends to $0$, because $\log(\pi)/\log(1/\la)$ tends to $z_0$.
\end{proof}

For $z\in\intervalleoo{0}{1}$, we next define
\[\kappa_{\la,\pi}^z=\frac{1}{\la^{z}\al\pi}+\e_\la.\]
Observe that, if $0<z<z_0$, then $\al\kappa_{\la,\pi}^z$ tends to $0$ when $\la\to0$ and $\pi\to\infty$ in the regime $\cR(\infty,z_0)$.
\begin{lem}\label{propagation lemma micro infty}%\quad
For all $z_0\in\intervalleof{0}{1}$ and all $z\in\intervalleoo{0}{z_0}$, there exists an event $\Omega^{P,z}_{\la,\pi}$, depending only on the Poisson processes $(N_t^S(i),N^P_t(i))_{t\in\intervalleff{0}{\al T},i\in\intervalleentier{-\bm_\la^ \gamma}{\bm_\la^ \gamma}}$, such that
\begin{multline*}
\Omega^{P,z}_{\la,\pi}\subset\{i_{\al\kappa_{\la,\pi}^z}^+\text{ and }-i_{\al\kappa_{\la,\pi}^z}^-\text{are greater than }\lfloor\la^ {-z}\rfloor\\
\text{and all }i\in\intervalleentier{i_{\al\kappa_{\la,\pi}^z}^-+1}{i_{\al\kappa_{\la,\pi}^z}^+-1}\text{ burns exactly once before }\al\kappa_{\la,\pi}^z\},
\end{multline*}
where the event on the right concerns the process $(\proco^{\la,\pi}_t(i))_{t\geq0,i\in\zz}$, and	
\[\lim_{\la,\pi}\proba{\Omega^{P,z}_{\la,\pi}}=1\]
when $\la\to0$ and $\pi\to\infty$ in the regime $\cR(\infty,z_0)$.
\end{lem}

\begin{proof}
Let $z\in\intervalleoo{0}{z_0}$. Recall \eqref{right front}, \eqref{left front}, \eqref{omega1} and remark that $\al\pi\kappa_{\la,\pi}^z=\la^ {-z}+\al\pi\e_\la$. We define
\begin{multline*}
\Omega^{P,z}_{\la,\pi}=\left\{i_{\al \kappa_{\la,\pi}^z}^+\in\intervalleentier{\lfloor\la^ {-z}\rfloor}{\lfloor\la^ {-z}+2\al\pi\e_\la\rfloor}\right\}\cap\left\{i_{\al \kappa_{\la,\pi}^z}^-\in\intervalleentier{-\lfloor\la^ {-z}-2\al\pi\e_\la\rfloor}{\lfloor\la^ {-z}\rfloor}\right\}\\
\cap\bigcap_{i\in\intervalleentier{-\lfloor\la^ {-z}+2\al\pi\e_\la\rfloor}{\lfloor\la^ {-z}+2\al\pi\e_\la\rfloor}}\Omega_i^1.
\end{multline*}
Observe that the event on the right hand side in Lemma \ref{propagation lemma micro infty} contains the event $\Omega^{P,z}_{\la,\pi}$. Indeed, as in the proof of Lemma \ref{propagation lemma p}, the two first terms situate the left and the right fronts. The third term ensures that there is no spark in the zone $\intervalleentier{-\lfloor\al\pi(\kappa_{\la,\pi}^z+\e_\la)\rfloor}{\lfloor\al\pi(\kappa_{\la,\pi}^z+\e_\la)\rfloor}\supset\intervalleentier{i^-_{\al \kappa_{\la,\pi}^z}}{i^+_{\al \kappa_{\la,\pi}^ z}}\supset\intervalleentier{-\lfloor\la^{-z}\rfloor}{\lfloor\la^{-z}\rfloor}$. 

Since $i_{\al\kappa_{\la,\pi}^z}^+$ and $-i_{\al\kappa_{\la,\pi}^z}^-$ are two Poisson random variables with parameter $\al\pi\kappa_{\la,\pi}^z$, Chebyshev's inequality shows
\begin{multline*}
\proba{i_{\al\kappa_ {\la,\pi}^z}^-\not\in\intervalleentier{-\lfloor\la^ {-z}-2\al\e_\la\pi\rfloor}{-\lfloor\la^ {-z}\rfloor}}=\proba{|i_{\al\kappa_ {\la,\pi}^z}^-+\al\pi\kappa_ {\la,\pi}^z|> \al\pi\e_\la}\\
=\proba{i_{\al\kappa_ {\la,\pi}^z}^+\not\in\intervalleentier{\lfloor\la^ {-z}\rfloor}{\lfloor\la^ {-z}+2\al\e_\la\pi\rfloor}}= \proba{\abs{i_{\al\kappa_ {\la,\pi}^z}^+-\al\pi\kappa_ {\la,\pi}^z}> \al\pi\e_\la}
\\
\leq\frac{\al\pi\kappa_{\la,\pi}^z}{(\al\e_\la\pi)^2}=\frac{\kappa_{\la,\pi}^z}{\al\pi\e_\la^2}=\kappa_{\la,\pi}^z\frac{\ba_\la^ 3}{\pi}
\end{multline*}
which again tends to $0$ when $\la\to0$ and $\pi\to\infty$ in the regime $\cR(\infty,z_0)$ (because $\log(\pi)\sim z_0\al$). %The last inequality together with
%\[|i_{\al\kappa_ {\la,\pi}^z}^+-\al\pi\kappa_ {\la,\pi}^z|> \al\pi\e_\la\iff i_{\al\kappa_ {\la,\pi}^z}^+\in\intervalleentier{\lfloor\la^ {-z}\rfloor}{\lfloor\la^ {-z}\rfloor+2\al\e_\la\pi}\]
%show that $\proba{i_{\al \kappa_{\la,\pi}^z}^+\in\intervalleentier{\lfloor\la^ {-z}\rfloor}{\lfloor\la^ {-z}\rfloor+2\al\pi\e_\la}}$ tends to $1$ when $\la\to0$ and $\pi\to\infty$ in the regime $\cR(\infty,z_0)$.

Finally, we still have $\proba{\Omega_i^1}=\frac{\pi}{1+\pi}$, recall \eqref{proba omega1}, whence
\[\proba{\bigcap_{i\in\intervalleentier{-\lfloor\al\pi(\kappa_{\la,\pi}^z+\e_\la)\rfloor}{\lfloor\al\pi(\kappa_{\la,\pi}^z+\e_\la)\rfloor}}\Omega_i^1}=\left(\frac{\pi}{1+\pi}\right)^{2\lfloor\al\pi(\kappa_{\la,\pi}^z+\e_\la)\rfloor+1}\simeq e^{-2\al(\kappa_{\la,\pi}^z+\e_\la)}\]
which tends to $1$ when $\la\to0$ and $\pi\to\infty$ in the regime $\cR(\infty,z_0)$. This concludes the proof of Lemma \ref{propagation lemma micro infty}.
\end{proof}

\subsection{Application to the $(\la,\pi)-$FFP}\label{application ffp}

We next give some useful definitions.
\begin{defin}\label{definition1 application}
Consider two families of Poisson processes $(N_t^S(i))_{t\geq0,i\in\zz}$ and
$(N_t^P(i))_{t\geq0,i\in\zz}$ with respective rates $1$ and $\pi$, all these
processes being independent. Let $(x_0,t_0)\in\rr\times\rr_+$. We call
\begin{itemize}
	\item \emph{propagation process ignited at $(x_0,t_0)$} the process $(\proco^{\la,\pi,0}_t(i))_{t\geq0,i\in\zz}$ built using the seed processes family $(N_t^{S,0}(i))_{t\geq0,i\in\zz}=(N_{t+\al t_0}^S(i+\lfloor\nl x_0\rfloor)-N_{\al t_0}^S(i+\lfloor\nl x_0\rfloor))_{t\geq0,i\in\zz}$ and the propagation processes family $(N_t^{P,0}(i))_{t\geq0,i\in\zz}=(N_{t+\al t_0}^P(i+\lfloor\nl x_0\rfloor)-N_{\al t_0}^P(i+\lfloor\nl x_0\rfloor))_{t\geq0,i\in\zz}$;
	\item \emph{right and left fronts of the propagation process ignited at $(x_0,t_0)$} the   processes $(i_t^{0,+})_{t\geq0}$ and $(i_t^{0,-})_{t\geq0}$, where for $t\geq0$
\begin{align*}
i_t^{0,+} &=\max\enstq{i\geq0}{\proco_t^{\la,\pi,0}(i)=2},\\
i_t^{0,-} &=\min\enstq{i\leq0}{\proco_t^{\la,\pi,0}(i)=2}.
\end{align*}
The processes $(i_t^{0,+})_{t\geq0}$ and $(-i_t^{0,-})_{t\geq0}$ are Poisson processes with parameter $\pi$;
	\item \emph{burning times of the propagation process ignited at $(x_0,t_0)$} the sequence $(T_i^0)_{i\in\zz}$ where, for $i\in\zz$,
\begin{align*}
T_i^ 0 &=\inf\enstq{s\geq0}{\check{\zeta}^{\la,\pi,0}_s(i)=2}\\
&=\begin{cases}
\inf\enstq{s\geq0}{i_s^{0,+}=i} &\text{if }i\geq0,\\
\inf\enstq{s\geq0}{i_s^{0,-}=i} &\text{if }i\leq0.
\end{cases}\notag
\end{align*}
\end{itemize}
Observe that $(T_i^0)_{i\in\zz}$,$(i_t^{0,+})_{t\geq0}$ and $(-i_t^{0,-})_{t\geq0}$ only depend on the propagation processes family $(N^P_t(i))_{t\geq0,i\in\zz}$.
\end{defin}
We then reformulate Lemmas \ref{propagation lemma p}, \ref{propagation lemma 0}, \ref{propagation lemma infty} and \ref{propagation lemma micro infty} with the previous definition.
\begin{defin}\label{definition2 application}
Consider two families of Poisson processes $(N_t^S(i))_{t\geq0,i\in\zz}$ and
$(N_t^P(i))_{t\geq0,i\in\zz}$ with respective rates $1$ and $\pi$, all these
processes being independent. Let $(x_0,t_0)\in\rr\times\rr_+$ and  $(\proco^{\la,\pi,0}_t(i))_{t\geq0,i\in\zz}$ be the propagation process ignited at $(x_0,t_0)$, recall Definition \ref{definition1 application}.
\begin{itemize}
	\item We define, for $T>0$, $\Omega^{P,T}_{\la,\pi}(x_0,t_0) \coloneqq \Omega^{P,T}_{\la,\pi}$, where $\Omega^{P,T}_{\la,\pi}$ is defined as in Lemma \ref{propagation lemma p}, using the process $(\proco^{\la,\pi,0}_t(i))_{t\geq0,i\in\zz}$.
	
Lemma \ref{propagation lemma p} implies that for all $\delta>0$, $\proba{\Omega^{P,T}_{\la,\pi}(x_0,t_0)}\geq1-\delta$ for all $(\la,\pi)$ sufficiently close to the regime $\cR(p)$.
	\item We define, for $A,B>0$, $\Omega^{P,A,B}_{\la,\pi}(x_0,t_0) \coloneqq \Omega^{P,A,B}_{\la,\pi}$, where $\Omega^{P,A,B}_{\la,\pi}$ is defined as in Lemma \ref{propagation lemma 0}, using the process $(\proco^{\la,\pi,0}_t(i))_{t\geq0,i\in\zz}$.

Lemma \ref{propagation lemma 0} implies that for all $\delta>0$, $\proba{\Omega^{P,A,B}_{\la,\pi}(x_0,t_0)}\geq1-\delta$ for all $(\la,\pi)$ sufficiently close to the regime $\cR(0)$.
	\item We define, for $z_0\in\intervalleff{0}{1}$ and $\gamma\in\intervalleoo{0}{1}$, $\Omega^{P,T,\gamma}_{\la,\pi}(x_0,t_0) \coloneqq \Omega^{P,T,\gamma}_{\la,\pi}$, where $\Omega^{P,T,\gamma}_{\la,\pi}$ is defined as in Lemma \ref{propagation lemma infty}, using the process $(\proco^{\la,\pi,0}_t(i))_{t\geq0,i\in\zz}$.

Lemma \ref{propagation lemma infty} implies that for all $\delta>0$, $\proba{\Omega^{P,T,\gamma}_{\la,\pi}(x_0,t_0)}\geq1-\delta$ for all $(\la,\pi)$ sufficiently close to the regime $\cR(\infty,z_0)$.
	\item We define, for $z_0\in\intervalleof{0}{1}$ and $z\in\intervalleoo{0}{z_0}$, $\Omega^{P,z}_{\la,\pi}(x_0,t_0) \coloneqq \Omega^{P,z}_{\la,\pi}$, where $\Omega^{P,z}_{\la,\pi}$ is defined as in Lemma \ref{propagation lemma micro infty}, using the process $(\proco^{\la,\pi,0}_t(i))_{t\geq0,i\in\zz}$.

Lemma \ref{propagation lemma micro infty} implies that for all $\delta>0$, $\proba{\Omega^{P,z}_{\la,\pi}(x_0,t_0)}\geq1-\delta$ for all $(\la,\pi)$ sufficiently close to the regime $\cR(\infty,z_0)$.
\end{itemize}
\end{defin}
Finally, we define the destroyed component by a fire starting on $\lfloor\nl x_0\rfloor$ at time $\al t_0$. Indeed, knowing the sequence of burning times $(T_i)_{i\in\zz}$ and conditionally on a suitable event defined above, we can localize the set of sites which are burning by a fire.
\begin{defin}\label{def destroyed comp}
Consider a family of independent Poisson processes $(N_t^P(i))_{t\geq0,i\in\zz}$ with rate $\pi$. Let $(x_0,t_0)\in\rr\times\intervalleff{0}{T}$ and let $(T_i^0)_{i\in\zz}$ be the burning times of the propagation process ignited at $(x_0,t_0)$. For a $\nn-$valued process $(\eta_t(i))_{t\geq0,i\in\zz}$, we define
\begin{equation}\label{destroyed comp}
C^P((\eta_{t}(i))_{i\in\zz,t\geq0},(x_0,t_0))=\intervalleentier{\lfloor\nl x_0\rfloor+i^g}{\lfloor\nl x_0\rfloor+i^d}
\end{equation}
where 
\begin{align*}
i^g &=\max\enstq{i\leq0}{\eta_{\al t_0+T^0_i-}(\lfloor\nl x_0\rfloor+i)=0}+1,\\
i^d &=\min\enstq{i\geq0}{\eta_{\al t_0+T^0_i-}(\lfloor\nl x_0\rfloor+i)=0}-1.
\end{align*}
We will use this definition with the $(\la,\pi)-$FFP: on a suitable event, $C^P((\eta_{t}^{\la,\pi}(i))_{i\in\zz,t\geq0},(x_0,t_0))$ is exactly the component destroyed by a match falling in $\lfloor\nl x_0\rfloor$ at time $\al t_0$, see the comments below.
\end{defin}

Let now $(\eta^{\la,\pi}_t(i))_{t\geq0,i\in\zz}$ be the $(\la,\pi)-$FFP. Let $(x_0,t_0) \in\rr\times\intervallefo{0}{\infty}$ be fixed in the rest of the section. Assume that a match falls in $\lfloor\nl x_0\rfloor$ at some time $\al t_0$. Then, on an appropriate event and regardless of the other phenomena, fires propagate with the good speed while they spread in occupied zones. 
Indeed, consider $(\proco^{\la,\pi,0}_t(i))_{t\geq0,i\in\zz}$ the propagation process ignited at $(x_0,t_0)$, the associated right front $(i_t^{0,+})_{t\geq0}$ and left front $(i_t^{0,-})_{t\geq0}$  and the associated burning times $(T_i^0)_{i\in\zz}$. Remark that $T^0_{i-\lfloor\nl x_0\rfloor}$ is the time needed for the fire starting in $\lfloor\nl x_0\rfloor$ at time $\al t_0$ to reach $i$.
\paragraph{Microscopic fire:} we describe here the effect of a microscopic fire in the discrete process in the different regimes. Let $\la\in\intervalleof{0}{1}$ and $\pi\geq1$.
%\begin{enumerate}[label=$\bullet$]
\begin{description}
	\item[Micro$(p)$:] here we focus on the regime $\cR(p)$, for some $p>0$. Set $\kappa_{\la,\pi}^0=\frac{\ml}{\al\pi}+\e_\la$. Assume that 
\begin{enumerate}[label=$\rhd$]
	\item there are $-\ml<i_1<0<i_2<\ml$ such that $\eta^{\la,\pi}_{\al t}(\lfloor\nl x_0\rfloor+i_1)=\eta^{\la,\pi}_{\al t}(\lfloor\nl x_0\rfloor+i_2)=0$ for all $t\in\intervalleff{t_0}{t_0+\kappa_{\la,\pi}^0}$,
	\item there is no burning tree in $\intervalleentier{\lfloor\nl x_0\rfloor+i_1}{\lfloor\nl x_0\rfloor+i_2}$ at time $\al t_0-$,
	\item no other match falls in $\intervalleentier{\lfloor\nl x_0\rfloor+i_1}{\lfloor\nl x_0\rfloor+i_2}$ during $\intervalleff{\al t_0}{\al (t_0+\kappa_{\la,\pi}^0)}$.
\end{enumerate}	
Then, on $\Omega^{P,T}_{\la,\pi}(x_0,t_0)$, we have 
\[C^P((\eta_{t}^{\la,\pi}(i))_{t\geq0,i\in\zz},(x_0,t_0))\subset\intervalleentier{\lfloor\nl x_0\rfloor+i_1}{\lfloor\nl x_0\rfloor+i_2}.\]
Furthermore, $\eta_{\al(t_0+\kappa_{\la,\pi}^0)}^{\la,\pi}(i)\leq1$ for all $i\in\intervalleentier{\lfloor\nl x_0\rfloor+i_1}{\lfloor\nl x_0\rfloor+i_2}$ and the fire destroys exactly the component $C^P((\eta_{t}^{\la,\pi}(i))_{t\geq0,i\in\zz},(x_0,t_0))$.

Indeed, since $\ml=\al\pi(\kappa_{\la,\pi}^0-\e_\la)$, on $\Omega^{P,T}_{\la,\pi}(x_0,t_0)$,  there holds that $T_{i_1}^0\leq \al\kappa_{\la,\pi}^0$ and $T_{i_2}^0\leq\al\kappa_{\la,\pi}^0$ (the left front satisfies $i^-_{\al\kappa_{\la,\pi}^0}\leq i_1$ and the right front satifies $i^+_{\al\kappa_{\la,\pi}^0}\geq i_2$, thanks to Lemma \ref{propagation lemma p}). Consequently, 
\[C^P((\eta_{t}^{\la,\pi}(i))_{t\geq0,i\in\zz},(x_0,t_0))\coloneqq \intervalleentier{\lfloor\nl x_0\rfloor+i^g}{\lfloor\nl x_0\rfloor+i^d}\subset\intervalleentier{\lfloor\nl x_0\rfloor+i_1}{\lfloor\nl x_0\rfloor+i_2}\]
where $i^g$ and $i^d$ are defined in Definition \ref{def destroyed comp}. Observe now that, by construction, for all $i\in\intervalleentier{i^g}{i^d}$
\[\eta^{\la,\pi}_{\al t_0+T^0_i}(\lfloor\nl x_0\rfloor+i)=2=\proco^{\la,\pi,0}_{T_i^0}(i)\]
and $\eta^{\la,\pi}_{\al t_0+T^0_{i^g-1}}(\lfloor\nl x_0\rfloor+i^g-1)=0=\eta^{\la,\pi}_{\al t_0+T^0_{i^d+1}}(\lfloor\nl x_0\rfloor+i^d+1)$. Recall that on $\Omega_{\la,\pi}^{T,P}(x_0,t_0)$, a spark at time $t\in\intervalleff{0}{\al T}$ for the process $(\proco^{\la,\pi,0}_t(i))_{t\geq0,i\in\zz}$ has vacant neighbors. Since for all $i\in\intervalleentier{i^g}{i^d}$, the processes $(\proco^{\la,\pi,0}_t(i))_{t\geq0}$ and $(\eta^{\la,\pi}_{\al t_0 +t}(\lfloor\nl x_0\rfloor+i))_{t\geq0}$ evolve with the same seed processes and the same propagation processes after burning for the first time until $\al\kappa_{\la,\pi}^0$, a straightforward observation shows that for all $i\in\intervalleentier{i^g+1}{i^d-1}$,
\[\eta^{\la,\pi}_{\al(t_0+\kappa_{\la,\pi}^0)}(\lfloor\nl x_0\rfloor+i)=\proco^{\la,\pi,0}_{\al \kappa_{\la,\pi}^0}(i)\]
and a site $i\in\intervalleentier{\lfloor\nl x_0\rfloor+i_1}{\lfloor\nl x_0\rfloor+i_2}\setminus  C^P((\eta_{t}^{\la,\pi}(i))_{t\geq0,i\in\zz},(x_0,t_0))$ can't be burnt during $\intervalleff{\al t_0}{\al(t_0+\kappa_{\la,\pi}^0)}$. Observe also that $i^g$ and $i^d$ burn exactly once during $\intervalleff{\al t_0}{\al(t_0+\kappa_{\la,\pi}^0)}$ (because the site $i^d+1$ is vacant at time $T_{i^d+1}^0$ and $i^g-1$ is vacant at time $T_{i^g-1}^0$ with $T_{i^g}^0\vee T_{i^d}^0\leq \al\kappa_{\la,\pi}^0$).

On $\Omega^{P,T}_{\la,\pi}(x_0,t_0)$, there is no more burning tree in $\intervalleentier{-\ml}{\ml}\supset\intervalleentier{i^g}{i^d}$ at time $\al\kappa_{\la,\pi}^0$ for the process $(\proco^{\la,\pi,0}_t(i))_{t\geq0,i\in\zz}$ (because $\ml=\al\pi(\kappa_{\la,\pi}^0-\e_\la)$) and  consequently, its also holds true in $\intervalleentier{\lfloor\nl x_0\rfloor+i^g}{\lfloor\nl x_0\rfloor+i^d}$ at time $\al(t_0+\kappa_{\la,\pi}^0)$ for the process $(\eta_{t}^{\la,\pi}(i))_{t\geq0,i\in\zz}$.

All this implies that, on $\Omega_{\la,\pi}^{P,T}(x_0,t_0)$,  $\eta_{\al(t_0+\kappa_{\la,\pi}^0)}^{\la,\pi}(i)\leq1$ for all $i\in C^P((\eta_{t}^{\la,\pi}(i))_{t\geq0,i\in\zz},(x_0,t_0))$ and therefore for all $i\in\intervalleentier{\lfloor\nl x_0\rfloor+i_1}{\lfloor\nl x_0\rfloor+i_2}$.
	\item[Micro$(0)$:] here we focus on the regime $\cR(0)$. Let $A,B>0$ and recall that, for $A>0$, $\varkappa_{\la,\pi}^A=\frac{\nl A}{\al\pi}+\e_\la$. Assume that 
\begin{enumerate}[label=$\rhd$]
	\item there are $-\ml<i_1<0<i_2<\ml$ such that $\eta^{\la,\pi}_{\al t}(\lfloor\nl x_0\rfloor+i_1)=\eta^{\la,\pi}_{\al t}(\lfloor\nl x_0\rfloor+i_2)=0$ for all $t\in\intervalleff{t_0}{t_0+\varkappa_{\la,\pi}^{A\vee B}}$,
	\item there is no burning tree in $\intervalleentier{\lfloor\nl x_0\rfloor+i_1}{\lfloor\nl x_0\rfloor+i_2}$ at time $\al t_0-$,
	\item no other match falls in $\intervalleentier{\lfloor\nl x_0\rfloor+i_1}{\lfloor\nl x_0\rfloor+i_2}$ during $\intervalleff{\al t_0}{\al (t_0+\varkappa_{\la,\pi}^{A\vee B})}$.
\end{enumerate}	
Then, on $\Omega^{P,A,B}_{\la,\pi}(x_0,t_0)$, we have 
\[C^P((\eta_{t}^{\la,\pi}(i))_{t\geq0,i\in\zz},(x_0,t_0))\subset\intervalleentier{\lfloor\nl x_0\rfloor+i_1}{\lfloor\nl x_0\rfloor+i_2}.\]
Furthermore, $\eta_{\al(t_0+\varkappa_{\la,\pi}^{A\vee B})}^{\la,\pi}(i)\leq1$ for all $i\in\intervalleentier{\lfloor\nl x_0\rfloor+i_1}{\lfloor\nl x_0\rfloor+i_2}$ and the fire destroys exactly the zone $C^P((\eta_{t}^{\la,\pi}(i))_{t\geq0,i\in\zz},(x_0,t_0))$.

Indeed, this can be checked exactly as above (replace $\kappa_{\la,\pi}^0$ by $\varkappa_{\la,\pi}^{A\vee B}$ and $\Omega_{\la,\pi}^{P,T}(x_0,t_0)$ by $\Omega^{P,A,B}_{\la,\pi}(x_0,t_0)$).
	\item[Micro$(\infty,z_0)$:] here we focus on the regime $\cR(\infty,z_0)$, for some $z_0\in\intervalleof{0}{1}$ (in the case $z_0=0$, there are only fires of the second kind). Let $0<z<z_0$ and recall that $\kappa_{\la,\pi}^z=\frac{1}{\la^ z\al\pi}+\e_\la$. Assume that
\begin{enumerate}[label=$\rhd$]
	\item there are $-\lfloor\la^ {-z}\rfloor<i_1<0<i_2<\lfloor\la^ {-z}\rfloor$ such that $\eta^{\la,\pi}_{\al t}(\lfloor\nl x_0\rfloor+i_1)=\eta^{\la,\pi}_{\al t}(\lfloor\nl x_0\rfloor+i_2)=0$ for all $t\in\intervalleff{t_0}{t_0+\kappa_{\la,\pi}^z}$,
	\item there is no burning tree in $\intervalleentier{\lfloor\nl x_0\rfloor+i_1}{\lfloor\nl x_0\rfloor+i_2}$ at time $\al t_0-$,
	\item no other match falls in $\intervalleentier{\lfloor\nl x_0\rfloor+i_1}{\lfloor\nl x_0\rfloor+i_2}$ during $\intervalleff{\al t_0}{\al (t_0+\kappa_{\la,\pi}^z)}$.
\end{enumerate}
Then, on $\Omega_{\la,\pi}^{P,z}(x_0,t_0)$, as above (replace $\kappa_{\la,\pi}^0$ by $\kappa_{\la,\pi}^z$ and $\Omega_{\la,\pi}^{P,T}(x_0,t_0)$ by $\Omega_{\la,\pi}^{P,z}(x_0,t_0)$)
\[C^P((\eta_{t}^{\la,\pi}(i))_{t\geq0,i\in\zz},(x_0,t_0))\coloneqq\intervalleentier{\lfloor\nl x_0\rfloor+i^g}{\lfloor\nl x_0\rfloor+i^d}\subset\intervalleentier{\lfloor\nl x_0\rfloor+i_1}{\lfloor\nl x_0\rfloor+i_2}.\]
Furthermore, $\eta_{\al(t_0+\kappa_{\la,\pi}^z)}^{\la,\pi}(i)\leq1$ for all $i\in\intervalleentier{\lfloor\nl x_0\rfloor+i_1}{\lfloor\nl x_0\rfloor+i_2}$ and the fire destroys exactly the zone $C^P((\eta_{t}^{\la,\pi}(i))_{t\geq0,i\in\zz},(x_0,t_0))$.

More precisely, on $\Omega_{\la,\pi}^{P,z}(x_0,t_0)$, for the process $(\proco^{\la,\pi,0}_t(i))_{t\geq0,i\in\zz}$, all site $i\in\intervalleentier{i_{\al\kappa_{\la,\pi}^z}^{0,-}+1}{i_{\al\kappa_{\la,\pi}^z}^{0,+}-1}$ burns exactly once before $\al\kappa_{\la,\pi}^z$. Thus, for the process $(\proco^{\la,\pi,0}_t(i))_{t\geq0,i\in\zz}$, there is no spark in $\intervalleentier{i_{\al\kappa_{\la,\pi}^z}^{0,-}+1}{i_{\al\kappa_{\la,\pi}^z}^{0,+}-1}$ at any time   $t\in\intervalleff{0}{\al\kappa_{\la,\pi}^z}$.

Since, for all $i\in\intervalleentier{i^g}{i^d}$, the processes $(\proco^{\la,\pi,0}_t(i))_{t\geq0}$ and $(\eta^{\la,\pi}_{\al t_0 +t}(\lfloor\nl x_0\rfloor+i))_{t\geq0}$ evolve with the same seed processes and the same propagation processes after burning for the first time until $\al\kappa_{\la,\pi}^z$, a straightforward observation shows that, for all $t\in\intervalleff{\al t_0}{\al (t_0+\kappa_{\la,\pi}^z)}$, and all $i\in C^P((\eta^{\la,\pi}_t(i))_{t\geq0,i\in\zz},(x_0,t_0))$, for $i\geq\lfloor\nl x_0\rfloor$,
\[\eta^{\la,\pi}_{t}(i)=\begin{cases}
\min(\eta^{\la,\pi}_{\al t_0}(i)+N^S_{t+\al t_0-}(i)-N^S_{\al t_0}(i),1) &\text{if } \al t_0\leq t<\al t_0+T^0_{i-\lfloor\nl x_0\rfloor}\\
2 &\text{if } \al t_0+T^0_{i-\lfloor\nl x_0\rfloor}\leq t<\al t_0 +T^0_{i+1-\lfloor\nl x_0\rfloor}\\
\min(N^S_{t}(i)-N^S_{T_{i+1-\lfloor\nl x_0\rfloor}}(i),1) &\text{if } \al t_0+T^0_{i+1-\lfloor\nl x_0\rfloor}\leq t\leq\al (t_0+\kappa_{\la,\pi}^z),
\end{cases}\]
and, for $i\leq\lfloor\nl x_0\rfloor$,
\[\eta^{\la,\pi}_{t}(i)=\begin{cases}
\min(\eta^{\la,\pi}_{\al t_0-}(i)+N^S_{t+\al t_0}(i)-N^S_{\al t_0}(i),1) &\text{if } \al t_0\leq  t<\al t_0+T^0_{i-\lfloor\nl x_0\rfloor}\\
2 &\text{if } \al t_0+T^0_{i-\lfloor\nl x_0\rfloor}\leq t<\al t_0 +T^0_{i-1-\lfloor\nl x_0\rfloor}\\
\min(N^S_{t}(i)-N^S_{T^0_{i-1-\lfloor\nl x_0\rfloor}}(i),1) &\text{if } \al t_0+T^0_{i-1-\lfloor\nl x_0\rfloor}\leq t\leq\al (t_0+\kappa_{\la,\pi}^z),
\end{cases}\]
Finally, for $i\in\intervalleentier{\lfloor\nl x_0\rfloor+i_1}{\lfloor\nl x_0\rfloor+i_2}\setminus C^P((\eta^{\la,\pi}_t(i))_{t\geq0,i\in\zz},(x_0,t_0))$ and $t\in\intervalleff{\al t_0}{\al(t_0+\kappa_{\la,\pi}^z)}$, $\eta^{\la,\pi}_{t}(i)$ is nothing but $\min(\eta^{\la,\pi}_{\al t_0}(i)+N^S_{t+\al t_0}(i)-N^S_{\al t_0}(i),1)$.
\end{description}
%\end{enumerate}
\paragraph{Macroscopic fire:} let $\la\in\intervalleof{0}{1}$ and $\pi\geq1$. Recall that, for $x>x_0$, $T^0_{\lfloor\nl x\rfloor-\lfloor\nl x_0\rfloor}$ is the time needed for the fire starting in $\lfloor\nl x_0\rfloor$ at time $\al t_0$ to reach $\lfloor\nl x\rfloor$. 
\begin{description}
	\item[Macro$(p)$:] here we focus on the regime $\cR(p)$, for some $p>0$. On $\Omega_{\la,\pi}^{P,T}(x_0,t_0)$, if $0\leq x-x_0\leq (T-t_0-\e_\la)\frac{\al\pi}{\nl}$,  there holds that
\begin{align*}
\frac{\al t_0+T_{\lfloor\nl x\rfloor-\lfloor\nl x_0\rfloor}^0}{\al}\in&\intervalleff{t_0+\frac{\lfloor\nl x\rfloor-\lfloor\nl x_0\rfloor}{\al\pi}-\e_\la}{t_0+\frac{\lfloor\nl x\rfloor-\lfloor\nl x_0\rfloor}{\al\pi}+\e_\la}
\end{align*}
and observe that, when $\la\to0$ and $\pi\to\infty$ in the regime $\cR(p)$,
\[\intervalleff{t_0+\frac{\lfloor\nl x\rfloor-\lfloor\nl x_0\rfloor}{\al\pi}-\e_\la}{t_0+\frac{\lfloor\nl x\rfloor-\lfloor\nl x_0\rfloor}{\al\pi}+\e_\la}\simeq\{t_0+p(x-x_0)\}.\]
This is just a rewriting of Lemma \ref{propagation lemma p}.
	\item[Macro$(0)$:] here we focus on the regime $\cR(0)$. On $\Omega^{P,A,B}_{\la,\pi}(x_0,t_0)$, for some $B>x-x_0$ and $A>0$, there holds that
\[\frac{\al t_0+T_{\lfloor\nl x\rfloor-\lfloor\nl x_0\rfloor}^0}{\al}\in\intervalleff{t_0}{t_0+\varkappa_{\la,\pi}^{B}}\]
and observe that $\intervalleff{t_0}{t_0+\varkappa_{\la,\pi}^{B}}\simeq\{t_0\}$ when $\la\to0$ and $\pi\to\infty$ in the regime $\cR(p)$.

Besides, assume that
\begin{enumerate}[label=$\rhd$]
	\item there are $\lfloor\nl (x_0-A)\rfloor<i_1<\lfloor\nl x_0\rfloor<i_2<\lfloor\nl (x_0+B)\rfloor$ such that $\eta_{\al s}^{\la,\pi}(i_1)=\eta_{\al s}^{\la,\pi}(i_1)=0$ for all $s\in\intervalleff{t_0}{t_0+\varkappa_{\la,\pi}^{A\vee B}}$,
	\item there is no burning tree in $\intervalleentier{i_1}{i_2}$ at time $\al t_0$,
	\item no other match falls in $\intervalleentier{i_1}{i_2}$ during $\intervalleff{\al t_0}{\al (t_0+\varkappa_{\la,\pi}^{A\vee B})}$.
\end{enumerate}
Then, on $\Omega^{P,A,B}_{\la,\pi}(x_0,t_0)$, we have 
\[C^P((\eta_{t}^{\la,\pi}(i))_{t\geq0,i\in\zz},(x_0,t_0))\subset\intervalleentier{\lfloor\nl x_0\rfloor+i_1}{\lfloor\nl x_0\rfloor+i_2}.\]
Furthermore, $\eta_{\al(t_0+\varkappa_{\la,\pi}^{A\vee B})}^{\la,\pi}(i)\leq1$ for all $i\in\intervalleentier{\lfloor\nl x_0\rfloor+i_1}{\lfloor\nl x_0\rfloor+i_2}$ and the fire destroys exactly the zone $C^P((\eta_{t}^{\la,\pi}(i))_{t\geq0,i\in\zz},(x_0,t_0))$.

This can be shown exactly as in the case {\bf Micro$(p)$} (the two statement are very similar).
	\item[Macro$(\infty,z_0)$:] here we focus on fires of second kind in the regime $\cR(\infty,z_0)$, for some $z_0\in\intervalleff{0}{1}$. Let $\gamma\in\intervalleoo{0}{1}$, on $\Omega^{P,T,\gamma}_{\la,\pi}(x_0,t_0)$, there holds that
\[x_0-\frac{\bm_\la^ \gamma}{\nl}\leq\frac{\lfloor\nl x_0\rfloor+i^{0,-}_{\al T}}{\nl}\leq x_0\leq\frac{\lfloor\nl x_0\rfloor+1+i^{0,+}_{\al T}}{\nl}
\leq x_0+\frac{\bm_\la^ \gamma}{\nl}\]
and observe that $\bm_\la^ \gamma/\nl\leq\gamma$: this is just a rewriting of Lemma \ref{propagation lemma infty}. Thus, since $\gamma$ can be chosen arbitrarily small, in the regime $\cR(\infty,z_0)$, fires have only a local effect.
\end{description}

\section{Localization of the $(\la,\pi)-$FFP}
Recall that $\al,\nl$ and $\ml$ are defined in \eqref{def al}, \eqref{def nl} and \eqref{def ml}. For $A>0$, we set $A_\la=\lfloor A\nl\rfloor$ and $I_A^\la=\intervalleentier{-A_\la}{A_\la}$. For $i\in\zz$, we set $i_\la=\intervallefo{i/\nl}{(i+1)/\nl}$ and $\e_\la=1/\ba_\la^ 3$.

We first introduce the $(\la,\pi,A)-$FFP.
\begin{defin}
Let $\lambda\in\intervalleof{0}{1}, \pi\geq1$ and $A>0$ be fixed. For each $i\in
I_A^\la$, we consider three independent Poisson processes,
$N^S(i)=(N^S_t(i))_{t\geq0}, N^M(i)=(N^M_t(i))_{t\geq0}$ and
$N^P(i)=(N^P_t(i))_{t\geq0}$ of respective parameters $1,\la$ and $\pi$, all these
processes being independent. Consider a $\{0,1,2\}$-valued process
$(\eta^{\lambda,\pi,A}_t(i))_{t\geq0,i \in I_A^\la}$ such that a.s., for all $i \in
I_A^\la$, $(\eta^{\lambda,\pi,A}_t(i))_{t\geq0}$ is c\`adl\`ag. We say that
$(\eta^{\lambda,\pi,A}_t(i))_{t\geq0,i \in I_A^\la}$ is a $(\lambda,\pi,A)-$FFP if
a.s., for all $i \in I_A^\la$, all $t\geq0$
\begin{align*}
\eta^{\lambda,\pi,A}_t(i)=&\int_0^t\textrm{\textbf{1}}_{\{\eta_{s-}^{\la,\pi,A}(i)=0\}}\diff
N_s^S(i)+\int_0^t \textrm{\textbf{1}}_{\{\eta_{s-}^{\la,\pi,A}(i)=1\}}\diff
N_s^M(i)\\
&+\int_0^t
\textrm{\textbf{1}}_{\{\eta_{s-}^{\la,\pi,A}(i+1)=2,\eta_{s-}^{\la,\pi,A}(i)=1\}}\diff
N_s^P(i+1)+\int_0^t
\textrm{\textbf{1}}_{\{\eta_{s-}^{\la,\pi,A}(i-1)=2,\eta_{s-}^{\la,\pi,A}(i)=1\}}\diff
N_s^P(i-1)\\
&- 2\int_0^t \textrm{\textbf{1}}_{\{\eta_{s-}^{\la,\pi,A}(i)=2\}}\diff N_s^P(i)
\end{align*}
with the convention $N^S_t(A_\la+1)=N^S_t(-A_\la-1)=0$ for all $t\geq0$.
\end{defin}

For $\eta \in \{0,1,2\}^{I_A^\la}$ and $i \in I_A^\la$, we define the occupied 
connected component around $i$ as
\[C_A(\eta,i)=\begin{cases} \emptyset& \text{if } \eta(i)=0\text{ or } 2,\\
                            \intervalleentier{l_A(\eta,i)}{r_A(\eta,i)} & \text{if } \eta(i)=1,
\end{cases}\] 
where 
\begin{align*}
l_A(\eta,i) &=(-A_\la)\vee(\sup\{k< i:\; \eta(k)=0\text{ or }2\}+1),\\
r_A(\eta,i) &=A_\la\wedge\left(\inf\{k > i:\; \eta(k)=0\text{ or }2\}-1\right).
\end{align*}

For $x\in\intervalleff{-A}{A}$ and $t\geq0$, we also introduce
\begin{align}
D^{\lambda,\pi,A}_t(x) &= \frac1{\nl} C_A\left(\eta^{\la,\pi,A}_{\al t},\lfloor\nl x\rfloor\right)\subset\intervalleff{-A_\la /\nl}{A_\la/\nl}\simeq\intervalleff{-A}{A},\\
K_t^{\la,\pi,A}(x) &=\frac{\abs{\left\{ i\in\intervalleentier{\lfloor\nl x\rfloor-\ml}{\lfloor\nl x\rfloor+\ml}\cap I_A^\la :\eta^{\la,\pi,A}_{\al t}(i)=1\right\}}}{\abs{\intervalleentier{\lfloor\nl x\rfloor-\ml}{\lfloor\nl x\rfloor+\ml}\cap I_A^\la}}\in\intervalleff{0}{1},\\ 
Z_t^{\la,\pi,A}(x) &=\frac{-\log\left(1-K_t^{\la,\pi,A}(x)\right)}{\log(1/\la)}\wedge
1\in\intervalleff{0}{1}.
\end{align}

We now give a discrete version of Proposition \ref{restriction limite}. Recall Definition \ref{close to}.
\begin{prop}\label{restriction limite loc}
Let $T>0,\la\in\intervalleof{0}{1}$ and $\pi\geq1$. For each $i\in\zz$, we consider
three Poisson processes $N^S(i)=(N^S_t(i))_{t\geq0}, N^M(i)=(N^M_t(i))_{t\geq0}$ and
$N^P(i)=(N^P_t(i))_{t\geq0}$, all these processes being independent. Let
$(\eta^{\lambda,\pi}_t(i))_{t\geq0,i \in \zz}$ be the corresponding $(\la,\pi)-$FFP
and for each $A>0$, let $(\eta^{\lambda,\pi,A}_t(i))_{t\geq0,i \in I_{A}^\la}$ be the
corresponding $(\la,\pi,A)-$FFP. There are some constants $\alpha_T>0$
and $C_T>0$ such that for all $A\geq1$, all $(\la,\pi)$ sufficiently close to the regime $\cR(p)$, for some $p\geq0$ (or to the regime $\cR(\infty,z_0)$, for some $z_0\in\intervalleff{0}{1}$),
\begin{multline*}
\mathbb{P}\left[ (\eta^{\lambda,\pi}_t(i))_{t\in\intervalleff{0}{\al T},i \in
I_{A/2}^\la}=(\eta^{\lambda,\pi,A}_t(i))_{t\in\intervalleff{0}{\al T},i \in
I_{A/2}^\la},\right.\\
        \left.
(Z_t^{\la,\pi}(x),D^{\lambda,\pi}_t(x))_{t\in\intervalleff{0}{T},x\in\intervalleff{-A/2}{A/2}}=(Z_t^{\la,\pi,A}(x),D^{\lambda,\pi,A}_t(x))_{t\in\intervalleff{0}{T},x\in\intervalleff{-A/2}{A/2}}\right]\\
        \geq 1-C_T e^{-\alpha_T A}.
\end{multline*}
\end{prop}
Observe that the Proposition \ref{restriction limite loc} holds for the three regimes, with the same scales but for different reasons. We thus distinguish the three regimes. The proof given for $p=0$ can be adapted in order to work for $p>0$, as in Proposition \ref{restriction limite}, but the proof given here for $p>0$ is much simpler. 
\begin{proof}[Proof in the regime $\cR(p)$ for some $p>0$.] 
Consider the true $(\la,\pi)-$FFP $(\eta^{\la,\pi}_t(i))_{t\geq0,i\in\zz}$. It of course suffices to prove the result for $A$ large enough.
Temporarily assume that for $a\in\rr$, there is an event $\Omega_{a,T}^{\la,\pi}$,
depending only on the Poisson processes $N_t^S(i),N^M_t(i)$ and $N^P_t(i)$ for
$t\in\intervalleff{0}{\al (T+2)}$ and 
\[i\in\bar{J}_a^\la
\coloneqq\intervalleentier{\lfloor (a-1-2\frac{T-1}{p})\nl\rfloor}{\lfloor
(a+1+2\frac{T-1}{p})\nl\rfloor-1},\]
such that 
\begin{enumerate}[label=(\roman*)]
        \item on $\Omega_{a,T}^{\la,\pi}$, a.s., there are $\iota^{+}\colon\intervalleff{0}{\al T}
\mapsto \bar{J}_a^\la$ non
decreasing and $\iota^{-}\colon\intervalleff{0}{\al T} \mapsto \bar{J}_a^\la$ non increasing such that $\eta^{\la,\pi}_t(\iota^{+}_t)=0$ or $2$ and
$\eta^{\la,\pi}_t(\iota^{-}_t)=0$ or $2$ for all $t\in\intervalleff{0}{\al T}$,
        \item there exists $q_T>0$ such that for all $a\in\rr$, we have $\proba{\Omega_{a,T}^{\la,\pi}}\geq q_T$, for all $(\la,\pi)$ sufficiently close to the regime $\cR(p)$.
\end{enumerate}
The proof is then concluded using similar argument as Step 3 in the proof of Proposition \ref{restriction limite}: thanks to point (ii), the probability that there are $-A+1+2\frac{T-1}{p}<a_1<-A/2-1-2\frac{T-1}{p}$ and $A/2+1+2\frac{T-1}{p}<a_2<A-1-2\frac{T-1}{p}$ with $\Omega_{a_1,T}^{\la,\pi}$ and $\Omega_{a_2,T}^{\la,\pi}$ realized is easily bounded from below by $1-C_Te^{-\alpha_T A}$. Next, on this event, a fire starting at the left of $\lfloor(a_1-1-2\frac{T-1}{p})\nl\rfloor$ will never cross $\lfloor (a_1+1+2\frac{T-1}{p})\nl\rfloor\leq\lfloor-A\nl/2\rfloor$ (thanks to $\iota^+$). Same thing holds on the right: a fire starting at the right of $\lfloor(a_2+1+2\frac{T-1}{p})\nl\rfloor$ will never cross $\lfloor (a_2-1-2\frac{T-1}{p})\nl\rfloor\geq\lfloor A\nl/2\rfloor$ (thanks to $\iota^-$). Finally, the clusters $D^{\la,\pi}_t(x)$ and $D^{\la,\pi,A}_t(x)$ clearly coincide for all $x\in\intervalleff{-\frac{A}{2}}{\frac{A}{2}}$ and all $t\in\intervalleff{0}{T}$.

\md

\noindent{\bf Step 1.} Fix some $\alpha>0$ small enough, say $\alpha=0.001$. Define $\kappa_{\la,\pi}^0=\ml/(\al\pi)+\e_\la$ and assume that $\kappa_{\la,\pi}^0\leq\alpha/2$.

For $\la>0, \pi\geq1$ and $a\in\rr$, we define the event
$\tOmega^{\la,\pi}_{a,T}$ on which points 1 and 2 below are satisfied:
\begin{enumerate}
        \item The family of Poisson processes $(N^ M_t(i))_{t
\in\intervalleff{0}{\al T},i
\in \bar{J}_a^{\lambda}}$ has exactly $4$ marks in $\bar{J}_a^\la$, and we call them  $\{(X_1^\la,T_1^\la),(X_2^\la,T_2^\la),(X_3^\la,T_3^\la),(X_4^\la,T_4^\la)\}$,
in such a way 
the match $(X_1^\la,T_1^\la)$ (resp. $(X_2^\la,T_2^\la)$) belongs to 
\begin{align*}
&\intervalleentier{\lfloor(a-\frac{5}{6}-\frac{T-1}{p})\nl\rfloor}{\lfloor(a-\frac{2}{3}-\frac{T-1}{p})\nl\rfloor}\times\intervalleff{\al(\frac{3}{4}+\alpha)}{\al(1-\alpha)}\\
\text{(resp. }&\intervalleentier{\lfloor(a+\frac{2}{3}+\frac{T-1}{p})\nl\rfloor}{\lfloor(a+\frac{5}{6}+\frac{T-1}{p})\nl\rfloor}\times\intervalleff{\al(\frac{3}{4}+\alpha)}{\al(1-\alpha)}),
\end{align*}
and the match $(X_3^\la,T_3^\la)$ (resp. $(X_4^\la,T_4^\la)$) belongs to 
\begin{align*}
&\intervalleentier{\lfloor(a-\frac{1}{2}-\frac{T-1}{p})\nl\rfloor+1}{\lfloor(a-\frac{1}{3}-\frac{T-1}{p})\nl\rfloor}\times\intervalleff{\al(1+\alpha)}{\al(\frac{3}{2}-\alpha)}\\
\text{(resp. }&\intervalleentier{\lfloor(a+\frac{1}{3}+\frac{T-1}{p})\nl\rfloor}{\lfloor(a+\frac{1}{2}+\frac{T-1}{p})\nl\rfloor-1}\times\intervalleff{\al(1+\alpha)}{\al(\frac{3}{2}-\alpha)}).
\end{align*}
        \item The family of Poisson processes
$(N^S_t(i))_{t\geq0,i\in \bar{J}_a^\la}$ satisfies 
\begin{enumerate}
	\item for $k=1,2$, for all $i \in\intervalleentier{X_k^ \la-\lfloor\la^{-3/4}\rfloor}{X_k^
\la+\lfloor\la^{-3/4}\rfloor} , N^S_{T^{\lambda}_{k}}(i)>0$;
	\item for $k=1,2$, there are $i_1^k\in \intervalleentier{X^\la_k-\ml+1}{X^\la_k-\lfloor\la^{-3/4}\rfloor-1}$ and $i_2^k\in\intervalleentier{X^\la_k+\lfloor\la^{-3/4}\rfloor+1}{X^\la_k+\ml-1}$ such that
$N^S_{T^{\lambda}_{k}+\al\kappa_{\la,\pi}^0}(i_1^k)=N^S_{T^{\lambda}_{k}+\al\kappa_{\la,\pi}^0}(i_2^k)=0$;
	\item for $k=1,2$, there is $i^k_3 \in \intervalleentier{X_k^ \la-\lfloor\la^{-3/4}\rfloor}{X_k^
\la+\lfloor\la^{-3/4}\rfloor}$ such that
$N^S_{3\al/2}(i^k_3) -
N^S_{T^{\lambda}_{k}}(i^k_3) =0$;
	\item for all $i \in \intervalleentier{\lfloor(a-1-\frac{T-1}{p})\nl\rfloor}{\lfloor(a+1+\frac{T-1}{p})\nl\rfloor},
N^S_{\al(1+\alpha)}(i) >0$.
\end{enumerate}
\end{enumerate}
We now introduce the event $\Omega^P_{a,T}(\la,\pi)$ on which all these four fires propagate at the good speed
%\[\Omega^P_{a,T}(\la,\pi)=\Omega^{P,T}_{\la,\pi}\left(\frac{X_1^\la}{\nl},\frac{T_1^\la}{\al}\right)\cap\Omega^{P,T}_{\la,\pi}\left(\frac{X_2^\la}{\nl},\frac{T_2^\la}{\al}\right)\cap{\Omega}^{P,T}_{\la,\pi}\left(\frac{X_3^\la}{\nl},\frac{T_3^\la}{\al}\right)\cap{\Omega}^{P,T}_{\la,\pi}\left(\frac{X_4^\la}{\nl},\frac{T_4^\la}{\al}\right),\]
\[\Omega^{P}_{a,T}(\la,\pi)=\bigcap_{i=1}^4 {\Omega}^{P,T}_{\la,\pi}\left(\frac{X_i^\la}{\nl},\frac{T_i^\la}{\al}\right)\]
recall Definition \ref{definition2 application}. We finally set 
\[\Omega_{a,T}^{\la,\pi}=\tOmega_{a,T}^{\la,\pi}\cap\Omega_{a,T}^P(\la,\pi).\]

\md

\noindent{\bf Step 2.} We now prove that on $\Omega_{a,T}^{\la,\pi}$, there exist $(\iota^+_t)_{t\in\intervalleff{0}{\al T}}$ and $(\iota^-_t)_{t\in\intervalleff{0}{\al T}}$ which satisfy (i).

Indeed, sites $i_1^1$ and $i_2^1$ are vacant until $T_1^\la+\al\kappa_{\la,\pi}^0$ because we start from an vacant initial configuration and 2-(b). On the one hand, they protect the zone $\intervalleentier{i^1_1+1}{i^1_2-1}$ and thus, the zone $\intervalleentier{X_1^ \la-\lfloor\la^{-3/4}\rfloor}{X_1^
\la+\lfloor\la^{-3/4}\rfloor}\subset\intervalleentier{i^1_1+1}{i^1_2-1}$ is completely filled at time $T_1^\la-$, thanks to 2-(a). On the other hand, on $\Omega^{P,T}_{\la,\pi}(X_1^\la/\nl,T_1^\la/\al)$, as seen in {\bf Micro$(p)$} in Subsection \ref{application ffp}, 
\begin{enumerate}[label=$\rhd$]
	\item the match falling on $X_1^\la$ at time $T_1^\la$ destroys entirely the zone $\intervalleentier{X_1^ \la-\lfloor\la^{-3/4}\rfloor}{X_1^
\la+\lfloor\la^{-3/4}\rfloor}$ before $T_1^\la+\al\kappa_{\la,\pi}^0$ (it is still protected by $i_1^1$ and $i_1^2$),
	\item the fire does not affect the zone outside $\intervalleentier{i^1_1}{i^1_2}$,
	\item there is no more burning tree in the zone $\intervalleentier{i^1_1}{i^1_2}$ at time $T_1^\la+\al\kappa_{\la,\pi}^0$.
\end{enumerate}
Then, since no seed fall on $i_3^1$ during $\intervallefo{T_1^\la}{3\al/2}$, $i_3^1$ remains vacant since it burnt (this happened between $T_1^\la$ and $T_1^\la+\al\kappa_{\la,\pi}^0$)  until time $3\al/2$, thanks to 2-(c).

Remark that same considerations holds around $X_2^\la$: the match falling in $X_2^\la$ at time $T_2^\la$ doesn't affect the zone outside $\intervalleentier{i_1^2}{i_2^2}$ (because they remain vacant until time $T^\la_2+\al\kappa_{\la,\pi}^0$), and $i_3^2$ remains vacant during $\intervallefo{T^\la_2+\al\kappa_{\la,\pi}^0}{3\al/2}$.

All this implies that the zone $\intervalleentier{\lfloor(a-\frac{1}{2}-\frac{T-1}{p})\nl\rfloor}{\lfloor(a+\frac{1}{2}+\frac{T-1}{p})\nl\rfloor}$ is protected from all the fire until $3\al/2$ (except possibles those falling at $(X_3^\la,T_3^\la)$ and $(X_4^\la,T_4^\la)$). Thus, thanks to 2-(d), the zone $\intervalleentier{\lfloor(a-\frac{1}{2}-\frac{T-1}{p})\nl\rfloor}{\lfloor(a+\frac{1}{2}+\frac{T-1}{p})\nl\rfloor}$ is completely occupied at time $\al(1+\alpha)$.

Since now, on ${\Omega}^{P,T}_{\la,\pi}\left(\frac{X_3^\la}{\nl},\frac{T_3^\la}{\al}\right)$, the right front $(i_t^{3,+})_{t\geq0}$ of the fire ignited at $(X_3^\la/\nl,T_3^\la/\al)$ statisfies
\[i_{\al T-T_3^\la}^{3,+}\leq \pi(\al T- T_3^\la+\al\e_\la)\leq \al\pi( T-1-\alpha+\e_\la),\]
recall Lemma \ref{propagation lemma p}, then $i_{\al T-T_3^\la}^{3,+}\leq (T-1)\frac{\nl}{p}$ as soon as $\abs{\frac{\nl}{\al\pi}-p}\leq p\frac{\alpha}{2(T-1)}$ (recall that $2\e<\alpha$). This in particular implies that
\[X_3^\la+i^{3,+}_{\al T-T_3^\la}\leq \lfloor(a-\frac{1}{3}-\frac{T-1}{p})\nl\rfloor+(T-1)\frac{\nl}{p}<\lfloor\nl a\rfloor.\]
Similarly, on ${\Omega}^{P,T}_{\la,\pi}\left(\frac{X_4^\la}{\nl},\frac{T_4^\la}{\al}\right)$ and for $\abs{\frac{\nl}{\al\pi}-p}\leq p\frac{\alpha}{2(T-1)}$, we clearly have
\[\lfloor\nl a\rfloor< \lfloor(a+\frac{1}{3}+\frac{T-1}{p})\nl\rfloor-(T-1)\frac{\nl}{p}\leq X_4^\la+i^{4,-}_{\al T-T_4^\la}.\]
We easily deduce that for all $t\in\intervalleff{0}{\al T-T_3^\la}$, $\eta^{\la,\pi}_{t+T_3^\la}(X_3^\la+i^{3,+}_t)=2$ and for all $t\in\intervalleff{0}{\al T-T_4^\la}$, $\eta^{\la,\pi}_{t+T_4^\la}(X_4^\la+i^{4,-}_t)=2$.

Finally, we set, for all $t\in\intervalleff{0}{\al T}$
\[\iota^+_t=\begin{cases}
i_1^1 &\text{if }0\leq t<T_1^\la+\kappa_{\la,\pi}^0,\\
i_3^1 &\text{if }T_1^\la+\kappa_{\la,\pi}^0\leq t<T_3^\la,\\
X_3^\la+i^{3,+}_{t-T_3^\la}&\text{if }T_3^\la\leq t\leq \al T.
\end{cases}\]
Clearly, $(\iota^+_t)_{t\in\intervalleff{0}{\al T}}$ is non decreasing, $\eta^{\la,\pi}_{s}(\iota_s^+)$ is $0$ until $T_3^\la$ and $2$ between $T_3^\la$ and $\al T$.

Similarly, we can choose
\[\iota^-_t=\begin{cases}
i_2^2 &\text{if }0\leq t<T_2^\la+\kappa_{\la,\pi}^0,\\
i_3^2 &\text{if }T_2^\la+\kappa_{\la,\pi}^0\leq t<T_4^\la,\\
X_4^\la+i^{4,-}_{t-T_4^\la}&\text{if }T_4^\la\leq t\leq \al T.
\end{cases}\]
Clearly, $(\iota^-_t)_{t\in\intervalleff{0}{\al T}}$ is non increasing, $\eta^{\la,\pi}_{s}(\iota_s^-)$ is $0$ until $T_4^\la$ and $2$ between $T_4^\la$ and $\al T$.

\md

\noindent{\bf Step 3.} We now prove (ii). The quantity $\proba{\Omega_{a,T}^ {\la,\pi}}$ does obviously not depend on $a\in\rr$ by spatial invariance. Then, we observe that we can construct $N^M$ by 
using a Poisson measure $\pi_M$ on $\rr\times\intervallefo{0}{\infty}$ with
intensity measure $\diff x \diff t$, independent of $N^S$ and $N^P$, by setting, for all $i\in\zz$,
\[N_t^M(i)=\pi_M(i_\la\times\intervalleff{0}{t/\al}).\]
Hence, the event on which $N^M$ satisifies 1. contains the event $\Omega^M_{a,T}$ on which $\pi_M$ has exactly $4$ marks in $\intervalleff{a-1-2\frac{T-1}{p}}{a+1+2\frac{T-1}{p}}\times\intervalleff{0}{T}$, which can be called $(X_1,T_1),(X_2,T_2),(X_3,T_3)$ and $(X_4,T_4)$ in such a way $(X_1,T_1)$ (resp. $(X_2,T_2)$) belongs to
\begin{align*}
&\intervalleff{a-\frac{5}{6}-\frac{T-1}{p}+\alpha}{a-\frac{2}{3}-\frac{T-1}{p}-\alpha}\times\intervalleff{\frac{3}{4}+\alpha}{1-\alpha}\\
\text{(resp. }&\intervalleff{a+\frac{2}{3}+\frac{T-1}{p}+\alpha}{a+\frac{5}{6}+\frac{T-1}{p}-\alpha}\times\intervalleff{\frac{3}{4}+\alpha}{1-\alpha}),
\end{align*}
and $(X_3,T_3)$ (resp. $(X_4,T_4)$) belongs to 
\begin{align*}
&\intervalleff{a-\frac{1}{2}-\frac{T-1}{p}+\alpha}{a-\frac{1}{3}-\frac{T-1}{p}-\alpha}\times\intervalleff{1+\alpha}{\frac{3}{2}-\alpha}\\
\text{(resp. }&\intervalleff{a+\frac{1}{3}+\frac{T-1}{p}+\alpha}{a+\frac{1}{2}+\frac{T-1}{p}-\alpha}\times\intervalleff{1+\alpha}{\frac{3}{2}-\alpha}).
\end{align*}
Clearly, the probability $\proba{\Omega_{a,T}^M}$ does not depend on $a$ nor on $\la$ and $\pi$ and is positive. We then define $q_T>0$ by
\[\proba{\Omega_{a,T}^M}=2q_T.\quad(\star)\]
We then use basic consideration on i.i.d. Poisson processes with rate $1$ (we write $\mathbb{P}_M$ for the conditional probability w.r.t. $\pi_M$) to show that point 2. occurs with high probability.
\begin{itemize}
	\item For $k=1,2$, we have $T_k^\la\geq \al(3/4+\alpha)$ and
\[\mathbb{P}_M\left[\forall i \in\intervalleentier{X_k^ \la-\lfloor\la^{-3/4}\rfloor}{X_k^\la+\lfloor\la^{-3/4}\rfloor} , N^S_{T^{\lambda}_{k}}(i)>0\right]\geq(1-\la^ {3/4+\alpha})^{2\lfloor\la^ {-3/4}\rfloor+1}\]
which tends to $1$ when $\la\to0$.
	\item For $k=1,2$, we have $T_k^\la+\al\kappa_{\la,\pi}^0\leq \al(1-\alpha/2)$ (recall that $\kappa_{\la,\pi}^0\leq\alpha/2$) and
\begin{multline*}
\mathbb{P}_M\left[\exists i_2^k\in \intervalleentier{X^\la_k+\lfloor\la^{-3/4}\rfloor+1}{X^\la_k+\ml-1}, N^S_{T^{\lambda}_{k}+\al\kappa_{\la,\pi}^0}(i_j^k)=0\right]\\
\geq1-(1-\la^ {1-\alpha/2})^{\ml-\lfloor\la^{-3/4}\rfloor-1}
\end{multline*}
which tends to $1$ when $\la\to0$ (and similar computation for $i_1^k$).
	\item For $k=1,2$, we have $T_k^\la\geq \al(3/4+\alpha)$ and
\begin{multline*}
\mathbb{P}_M\left[\exists i \in \intervalleentier{X_k^ \la-\lfloor\la^{-3/4}\rfloor}{X_k^
\la+\lfloor\la^{-3/4}\rfloor},N^S_{3\al/2}(i) -N^S_{T^{\lambda}_{k}}(i) =0\right]\\
\geq 1-(1-\la^{3/4-\alpha})^ {2\lfloor\la^{-3/4}\rfloor+1}
\end{multline*}
which tends to $1$ when $\la\to0$;
	\item Finally, 
\begin{multline*}
\mathbb{P}_M\left[\forall i \in \intervalleentier{\lfloor(a-1-\frac{T-1}{p})\nl\rfloor}{\lfloor(a+1+\frac{T-1}{p})\nl\rfloor},
N^S_{\al(1+\alpha)}(i) >0\right]\\
=(1-\la^{1+\alpha})^{(2+2\frac{T-1}{p})\nl}
\end{multline*}
which tends also to $1$ when $\la\to0$.
\end{itemize}

Next, since $\pi_M$ is independent of the processes family $(N^S_t(i))_{i\in\zz,t\geq0}$ and   $(N^P_t(i))_{i\in\zz,t\geq0}$, Lemma \ref{propagation lemma p} directly imply  that, for all  $k=1,\dots,4$,  $\mathbb{P}_M\left[\Omega^{P,T}_{\la,\pi}(X_k,T_k)\right]$ tends to $1$ when $\la\to0$ and $\pi\to\infty$ in the regime $\cR(p)$. 

All this, together with $(\star)$, implies that $\proba{\Omega_{a,T}^{\la,\pi}}\geq q_T>0$ for all $(\la,\pi)$ sufficiently close to the regime $\cR(p)$.

In the end, for all $(\la,\pi)$ sufficiently close to the regime $\cR(p)$, the event $\Omega_{a,T}^{\la,\pi}$ depend only on the Poisson processes $N_t^S(i),N^M_t(i)$ and $N^P_t(i)$ for
$t\in\intervalleff{0}{\al (T+2)}$ and $i\in\bar{J}_a^\la$. This suffices to conclude the proof.
\end{proof}

\begin{proof}[Proof in the regime $\cR(\infty,z_0)$]
Let $z_0\in\intervalleff{0}{1}$. Consider the true $(\la,\pi)-$FFP $(\eta^{\la,\pi}_t(i))_{t\geq0,i\in\zz}$. We introduce
\[J_a^\la=\intervalleentier{\lfloor a\nl\rfloor}{\lfloor (a+1)\nl\rfloor-1}.\]

As above, for $a\in\rr$, we are going to construct an event $\Omega_{a,T}^{\la,\pi}$ depending only on the Poisson processes $N_t^S(i),N^M_t(i)$ and $N^P_t(i)$ for
$t\in\intervalleff{0}{\al (T+2)}$ and $i\in{J}_a^\la$ such that 
\begin{enumerate}[label=(\roman*)]
        \item on $\Omega_{a,T}^{\la,\pi}$, there exists $\iota^{+}\colon\intervalleff{0}{\al T}
\mapsto {J}_a^\la$ non
decreasing and $\iota^{-}\colon\intervalleff{0}{\al T} \mapsto {J}_a^\la$ non increasing such that $\eta^{\la,\pi}_t(\iota^{+}_t)=0$ or $2$ and
$\eta^{\la,\pi}_t(\iota^{-}_t)=0$ or $2$ for all $t\in\intervalleff{0}{\al T}$,
        \item there exists $q_T>0$ such that for all $a\in\rr$, we have $\proba{\Omega_{a,T}^{\la,\pi}}\geq q_T$ for all $(\la,\pi)$ sufficiently close to the regime $\cR(\infty,z_0)$.
\end{enumerate}
The proof is then concluded as previously. We divide the proof in two cases.

\vspace*{0.3cm}
\noindent{\it Case 1: $z_0\in\intervallefo{0}{1}$.} We fix $\alpha=0.001$ and $\gamma\in\intervalleoo{0}{\frac{1-z_0}{4}}$. Recall that $\bm_\la^ {\gamma}=\lfloor\frac{\gamma}{\la^ {\gamma+(1-\gamma)z_0}\al}\rfloor\ll\ml$ and $\ml\ll\nl$.

\md

\noindent{\bf Step 1.} For $\la>0, \pi\geq1$ and $a\in\rr$, we define the event
$\tOmega^{\la,\pi}_{a,T}$ on which points 1 and 2 below are satisfied:
\begin{enumerate}
        \item The family of Poisson processes $(N^ M_t(i))_{t
\in\intervalleff{0}{\al T},i
\in J_a^{\lambda}}$ has exactly $2$ marks in $J_a^\la$, and we call them  $(X_1^\la,T_1^\la),(X_2^\la,T_2^\la)$,
in such a way that
\begin{align*}
(X_1^\la,T_1^\la)&\in\intervalleentier{\lfloor a\nl\rfloor+\ml}{\lfloor(a+\frac{1}{2})\nl\rfloor-\ml-1}\times\intervalleff{\al(z_0+2\gamma)}{\al(1-\gamma)}\\
\text{and }(X_2^\la,T_2^\la)&\in\intervalleentier{\lfloor(a+\frac{1}{2})\nl\rfloor+\ml}{\lfloor(a+1)\nl\rfloor-\ml-1}\times\intervalleff{\al(z_0+2\gamma)}{\al(1-\gamma)}.
\end{align*}
        \item The family of Poisson processes
$(N^S_t(i))_{t\geq0,i\in {J}_a^\la}$ satisfies, for $k=1,2$, 
\begin{enumerate}
	\item for all $i \in\intervalleentier{X_k^ \la-\bm_\la^ {\gamma}}{X_k^
\la+\bm_\la^ {\gamma}} , N^S_{T^{\lambda}_{k}}(i)>0$;
	\item there are $i_1^k\in \intervalleentier{X^\la_k-\ml+1}{X^\la_k-\bm_\la^ {\gamma}-1}$ and $i_2^k\in\intervalleentier{X^\la_k+\bm_\la^ {\gamma}+1}{X^\la_k+\ml-1}$ such that
$N^S_{\al(1-\gamma)}(i_1^k)=N^S_{\al(1-\gamma)}(i_2^k)=0$.
\end{enumerate}
\end{enumerate}
We now introduce the event on which all of these two fires propagate at the correct speed,
\[\Omega^P_{a,T}(\la,\pi)=\Omega_ {\la,\pi}^{P,T,\gamma}\left(\frac{X_1^\la}{\nl},\frac{T_1^\la}{\al}\right)\cap\Omega_ {\la,\pi}^{P,T,\gamma}\left(\frac{X_2^\la}{\nl},\frac{T_2^\la}{\al}\right).\]
We finally set 
\[\Omega_{a,T}^{\la,\pi}=\tOmega_{a,T}^{\la,\pi}\cap\Omega_{a,T}^P(\la,\pi).\]

\md

\noindent{\bf Step 2.} We now prove that on $\Omega_{a,T}^{\la,\pi}$, (i) holds.

For $k=1,2$, thanks to 2-(b), the sites $i_1^k$ and $i_2^k$ remain vacant until $\al(1-\gamma)>T_k^\la$. Thus, no fire can affect the zone $\intervalleentier{X_k^ \la-\bm_\la^ {\gamma}}{X_k^
\la+\bm_\la^ {\gamma}}$ during $\intervalleff{0}{\al(1-\gamma)}$. Hence, the zone $\intervalleentier{X_k^ \la-\bm_\la^ {\gamma}}{X_k^
\la+\bm_\la^ {\gamma}}$ is completely filled at time $T_k^\la-$, thanks to 2-(a). On $\Omega^{P,T,\gamma}_{\la,\pi}\left(\frac{X_k^\la}{\nl},\frac{T_k^\la}{\al}\right)\subset\Omega_{a,T}^P(\la,\pi)$, the fire starting in $X_k^\la$ at time $T_k^\la$ does not affect the zone outside $\intervalleentier{X_k^\la-\bm_\la^ {\gamma}}{X_k^\la+\bm_\la^ {\gamma}}$ during $\intervalleff{0}{\al T}$, recall {\bf Macro$(\infty,z_0)$} in Subsection \ref{application ffp}. Since $X_2^\la-X_1^\la\geq2\ml\geq 2\bm_\la^ {\gamma}+1$, we deduce that $\eta^{\la,\pi}_{s}(X_1^\la+i_{s-T_1^\la}^{1,+})=2$ for all $s\in\intervalleff{T_1^\la}{\al T}$ and $\eta^{\la,\pi}_{s}(X_2^\la+i_{s-T_2^\la}^{2,-})=2$ for all $s\in\intervalleff{T_2^\la}{\al T}$.

Finally, we set, for all $t\in\intervalleff{0}{\al T}$
\[\iota^+_t=\begin{cases}
i_1^1 &\text{if }0\leq t<T_1^\la,\\
X_1^\la+i^{1,+}_{t-T_3^\la}&\text{if }T_1^\la\leq t\leq \al T.
\end{cases}\]
The process $(\iota^+_t)_{t\in\intervalleff{0}{\al T}}$ is non decreasing, $\eta^{\la,\pi}_{s}(\iota_s^+)$ is $0$ for $s\in\intervallefo{0}{T_1^\la}$ and $2$ for $s\in\intervalleff{T_1^\la}{\al T}$.

Similarly, we set for all $t\in\intervalleff{0}{\al T}$,
\[\iota^-_t=\begin{cases}
i_2^2 &\text{if }0\leq t<T_2^\la,\\
X_2^\la+i^{2,-}_{t-T_2^\la}&\text{if }T_2^\la\leq t\leq \al T,
\end{cases}\]
which also satisfies the requirements.

\md

\noindent{\bf Step 3.} The event $\Omega_{a,T}^{\la,\pi}$ also satisfies point $(ii)$. 

Indeed, the quantity $\proba{\Omega_{a,T}^ {\la,\pi}}$ does obviously not depend on $a\in\rr$ by spatial invariance. As previously, we can construct $N^M$ by 
using a Poisson measure $\pi_M$ on $\rr\times\intervallefo{0}{\infty}$ with
intensity measure $\diff x \diff t$, independent of $N^S$ and $N^P$, by setting, for all $i\in\zz$,
\[N_t^M(i)=\pi_M(i_\la\times\intervalleff{0}{t/\al}).\]
Hence, the event on which $N^M$ satisifies 1. contains the event $\Omega_{a,T}^M$ on which $\pi_M$ has exactly $2$ marks in $\intervalleff{a}{a+1}\times\intervalleff{0}{T}$, which can be called $(X_1,T_1)$ and $(X_2,T_2)$ such that (remark that $\gamma<1/4$)
\begin{align*}
(X_1,T_1)&\in\intervalleff{a+\gamma}{a+\frac{1}{2}-\gamma}\times\intervalleff{z_0+2\gamma}{1-\gamma}\\
\text{and }(X_2,T_2)&\in\intervalleff{a+\frac{1}{2}+\gamma}{a+1-\gamma}\times\intervalleff{z_0+2\gamma}{1-\gamma}.
\end{align*}
Clearly, the probability $\proba{\Omega_{a,T}^M}$ does not depend on $a$ nor on $\la$ and $\pi$ and is positive. We then define $q_T>0$ by
\[\proba{\Omega_{a,T}^M}=2q_T.\quad(\star)\]
We then use basic considerations on i.i.d. Poisson processes with rate $1$ (we write $\mathbb{P}_M$ for the conditional probability w.r.t. $\pi_M$) to show that point 2. occurs with high probability.
\begin{itemize}
	\item For $k=1,2$, we have $T_k^\la\geq \al(z_0+2\gamma)$ and
\begin{multline*}
\mathbb{P}_M\left[\forall i \in\intervalleentier{X_k^ \la-\bm_\la^ {\gamma}}{X_k^\la+\bm_\la^ {\gamma}} , N^S_{T^{\lambda}_{k}}(i)>0\right]\geq(1-\la^ {z_0+2\gamma})^{2\bm_\la^ {\gamma}+1}\\
\simeq \exp(-\la^ {z_0+2\gamma}\frac{\gamma\la^ {-\gamma-(1-\gamma)z_0}}{\al})
=\exp(-\gamma\frac{\la^{\gamma(z_0+1)}}{\al})
\end{multline*}
which tends to $1$ when $\la\to0$.
	\item For $k=1,2$, we have 
\[\mathbb{P}_M\left[\exists i_2^k\in \intervalleentier{X^\la_k+\bm_\la^ {\gamma}+1}{X^\la_k+\ml-1}, N^S_{\al(1-\gamma)}(i_2^k)=0\right]
=1-(1-\la^ {1-\gamma})^{\ml-\bm_\la^ {\gamma}-1}\]
which tends to $1$ when $\la\to0$, because $\bm_\la^ \gamma\ll\ml$ and $\la^ {1-\gamma}\ll\ml$ (similar computation holds for $i_1^k$).
\end{itemize}
Finally, since $\pi_M$ is independent of the processes family $(N^S_t(i))_{t\geq0,i\in\zz}$ and   $(N^P_t(i))_{t\geq0,i\in\zz}$, Lemma \ref{propagation lemma p} directly imply  that, for all  $k=1,2$,  $\mathbb{P}_M\left[\Omega^{P,T}_{\la,\pi}(X_k,T_k)\right]$ tends to $1$ when $\la\to0$ and $\pi\to\infty$ in the regime $\cR(p)$. 

All this, together with $(\star)$, implies that $\proba{\Omega_{a,T}^{\la,\pi}}\geq q_T>0$ for all $(\la,\pi)$ sufficiently close to the regime $\cR(p)$.

In the end, for all $(\la,\pi)$ sufficiently close to the regime $\cR(p)$, the event $\Omega_{a,T}^{\la,\pi}$ depend only on the Poisson processes $N_t^S(i),N^M_t(i)$ and $N^P_t(i)$ for
$t\in\intervalleff{0}{\al (T+2)}$ and $i\in J_a^\la$. This suffices to conclude the proof in the case $z_0\in\intervallefo{0}{1}$.

\vspace*{0.3cm}

\noindent{\it Case 2: $z_0=1$.} Fix some $\alpha>0$ small enough, say $\alpha=0.001$. Recall that 
\[\kappa_{\la,\pi}^{1-\alpha}=\frac{1}{\la^{1-\alpha}\al\pi}+\e_\la\]
and assume that $\kappa_{\la,\pi}^{1-\alpha}<\alpha$. We first define the event
$\tOmega^{\la,\pi}_{a,T}$ on which points 1 and 2 below are satisfied:
\begin{enumerate}
        \item The family of Poisson processes $(N^ M_t(i))_{t
\in\intervalleff{0}{\al T},i
\in J_a^{\lambda}}$ has exactly $4$ marks in $J_a^\la$, and we call them  $(X_k^\la,T_k^\la)_{k=1,\dots,4}$,
in such a way 
the match $(X_1^\la,T_1^\la)$ (resp. $(X_2^\la,T_2^\la)$) belongs to 
\begin{align*}
&\intervalleentier{\lfloor (a+\alpha)\nl\rfloor}{\lfloor(a+\frac{1}{4}-\alpha)\nl\rfloor}\times\intervalleff{\al(\frac{3}{4}+\alpha)}{\al(1-2\alpha)}\\
\text{(resp. }&\intervalleentier{\lfloor(a+\frac{3}{4}+\alpha)\nl\rfloor}{\lfloor(a+1-\alpha)\nl\rfloor}\times\intervalleff{\al(\frac{3}{4}+\alpha)}{\al(1-2\alpha)}),
\end{align*}
and the match $(X_3^\la,T_3^\la)$ (resp. $(X_4^\la,T_4^\la)$) belongs to 
\begin{align*}
&\intervalleentier{\lfloor (a+\frac{1}{4}+\alpha)\nl\rfloor}{\lfloor(a+\frac{1}{2}-\alpha)\nl\rfloor}\times\intervalleff{\al(1+\alpha)}{\al(\frac{5}{4}-2\alpha)}\\
\text{(resp. }&\intervalleentier{\lfloor(a+\frac{1}{2}+\alpha)\nl\rfloor}{\lfloor(a+\frac{3}{4}-\alpha)\nl\rfloor}\times\intervalleff{\al(1+\alpha)}{\al(\frac{5}{4}-2\alpha)}).
\end{align*}
	\item The family of Poisson processes
$(N^S_t(i))_{t\geq0,i\in {J}_a^\la}$ satisfies,  
\begin{enumerate}
	\item for $k=1,2$, $\forall i \in\intervalleentier{X_k^ \la-\lfloor\la^{-3/4}\rfloor}{X_k^
\la+\lfloor\la^{-3/4}\rfloor} , N^S_{T^{\lambda}_{k}}(i)>0$;
	\item for $k=1,2$, there are $i_1^k\in \intervalleentier{X^\la_k-\lfloor\la^{-(1-\alpha)}\rfloor-1}{X^\la_k}$ and $i_2^k\in\intervalleentier{X^\la_k}{X^\la_k+\lfloor\la^{-(1-\alpha)}\rfloor+1}$ such that
$N^S_{T_k^\la+\al\kappa_{\la,\pi}^{1-\alpha}}(i_j^k)=0$.
	\item for $k=1,2$, there exists $i^k_3 \in \intervalleentier{X_k^ \la-\lfloor\la^{-3/4}\rfloor}{X_k^
\la+\lfloor\la^{-3/4}\rfloor}$ such that
$N^S_{3\al/2}(i^k_3) -
N^S_{T^{\lambda}_{k}}(i^k_3) =0$;
	\item $\forall i \in \intervalleentier{\lfloor a\nl}{\lfloor(a+1)\nl},
N^S_{\al(1+\alpha)}(i) >0$.
\end{enumerate}
\end{enumerate}
We now introduce the event on which all these four fires propagate on the good speed
\[\Omega^P_{a,T}(\la,\pi)=\Omega^{P,1-\alpha}_{\la,\pi}(\frac{X_1^\la}{\nl},\frac{T_1^\la}{\al})\cap\Omega^{P,1-\alpha}_{\la,\pi}(\frac{X_2^\la}{\nl},\frac{T_2^\la}{\al})\cap{\Omega}^{P,T,\alpha}_{\la,\pi}(\frac{X_3^\la}{\nl},\frac{T_3^\la}{\al})\cap{\Omega}^{P,T,\alpha}_{\la,\pi}(\frac{X_4^\la}{\nl},\frac{T_4^\la}{\al}),\]
recall Definition \ref{definition2 application}.

We finally set 
\[\Omega_{a,T}^{\la,\pi}=\tOmega_{a,T}^{\la,\pi}\cap\Omega_{a,T}^P(\la,\pi).\]

We deduce that $\Omega_{a,T}^{\la,\pi}$ satisfies (i) as above: the match falling in $X_k^\la$, for $k=1,2$, destroys at least the zone $\intervalleentier{X_k^ \la-\lfloor\la^{-3/4}\rfloor}{X_k^
\la+\lfloor\la^{-3/4}\rfloor}$ (thanks to 2-(a)) but does not affect the zone outside $\intervalleentier{X^\la_k-\lfloor\la^{-(1-\alpha)}\rfloor}{X^\la_k+\lfloor\la^{-(1-\alpha)}\rfloor}$ (thanks to 2-(b) and recall {\bf Micro$(\infty,1)$} in Subsection \ref{application ffp}). Hence, for $k=1,2$, $i_3^k$ remains vacant from $T_k^\la+\al\kappa_{\la,\pi}^{1-\alpha}$ until $3\al/2$. Thus, $i_3^1$ and $i_3^2$ protect the zone $\intervalleentier{\lfloor(a+\frac{1}{4}-\alpha)\nl\rfloor}{\lfloor(a+\frac{3}{4}-\alpha)\nl\rfloor}$, which is completely filled at time $\al(1+\alpha)$, thanks to 2-(d). As previously, and since fires have only a local effect (recall that $\bm_\la^ \alpha=\lfloor\alpha\nl\rfloor$), the right front of the fire 3 and the left front of the fire 4 burn until $\al T$.

We then can set, for all $t\in\intervalleff{0}{\al T}$
\[\iota^+_t=\begin{cases}
i_1^1 &\text{if }0\leq t<T_1^\la+\al\kappa_{\al,\pi}^{1-\alpha},\\
i_3^1 &\text{if }T_1^\la+\al\kappa_{\al,\pi}^{1-\alpha}\leq t< T_3^\la,\\
X_3^\la+i^{3,+}_{t-T_3^\la}&\text{if }T_3^\la\leq t\leq \al T,
\end{cases}\]
and
\[\iota^-_t=\begin{cases}
i_2^2 &\text{if }0\leq t<T_2^\la+\al\kappa_{\al,\pi}^{1-\alpha},\\
i_3^2 &\text{if }T_2^\la+\al\kappa_{\al,\pi}^{1-\alpha}\leq t< T_4^\la,\\
X_4^\la+i^{4,-}_{t-T_4^\la}&\text{if }T_4^\la\leq t\leq \al T.
\end{cases}\]

We can check, as usual, that $\proba{\Omega_{a,T}^{\la,\pi}}\geq q_T$, for all $(\la,\pi)$ sufficiently close to the regime $\cR(\infty,1)$, where $2q_T$ is the probability that a Poisson measure $\pi_M$ has exactly 4 marks $(X_k,T_k)_{k=1,\dots,4}$ in $\intervalleff{a}{a+1}\times\intervalleff{0}{T}$ in such a way that
\begin{align*}
(X_1,T_1) &\in\intervalleff{a+\alpha}{a+\frac{1}{4}-\alpha}\times\intervalleff{\frac{3}{4}+\alpha}{1-2\alpha},\\
(X_2,T_2) &\in\intervalleff{a+\frac{3}{4}+\alpha}{a+1-\alpha}\times\intervalleff{\frac{3}{4}+\alpha}{1-2\alpha},\\
(X_3,T_3) &\in\intervalleff{a+\frac{1}{4}+\alpha}{a+\frac{1}{2}-\alpha}\times\intervalleff{1+\alpha}{\frac{5}{4}-2\alpha},\\
(X_4,T_4) &\in\intervalleff{a+\frac{1}{2}+\alpha}{a+\frac{3}{4}-\alpha}\times\intervalleff{1+\alpha}{\frac{5}{4}-2\alpha}.\qedhere
\end{align*}
\end{proof}

\begin{proof}[Proof in the regime $\cR(0)$]
We fix $T>0$. It of course suffices to prove the result for $A$ large enough. We consider the true $(\la,\pi)-$FFP $(\eta^{\lambda,\pi}_t(i))_{t\geq0,i \in \zz}$ and set $K=\lfloor 4T\rfloor$. For $a\in\rr$, we recall that
\[\varkappa_{\la,\pi}=\varkappa_{\la,\pi}^{2K}=\frac{2K\nl}{\al\pi}+\e_\la\]
and
\[J_a^\la\coloneqq\intervalleentier{\lfloor a \nl\rfloor}{\lfloor (a+1)\nl\rfloor-1}\]
and introduce
\[J_{a,K}^\la\coloneqq\intervalleentier{\lfloor (a-3K) \nl\rfloor}{\lfloor (a+3K+1)\nl\rfloor-1}.\]
As usual, for $a\in\rr$, we are going to build an event $\Omega_{a,T}^{\la,\pi}$ depending only on the Poisson processes $N_t^S(i),N^M_t(i)$ and $N^P_t(i)$ for
$t\in\intervalleff{0}{\al T}$ and $i\in J_{a,K}^\la$ such that 
\begin{enumerate}[label=(\roman*)]
        \item on $\Omega_{a,T}^{\la,\pi}$, there exists $\iota^{+}\colon\intervalleff{0}{\al T}
\mapsto {J}_{a,K}^\la$ (resp. $\iota^{-}\colon\intervalleff{0}{\al T} \mapsto {J}_{a,K}^\la$), non
decreasing (resp. non increasing), such that $\eta^{\la,\pi}_t(\iota^{+}_t)=0$ (resp.
$\eta^{\la,\pi}_t(\iota^{-}_t)=0$) for all $t\in\intervalleff{0}{\al T}$,
        \item there exists $q_T>0$ such that for all $a\in\rr$, we have $\proba{\Omega_{a,T}^{\la,\pi}}\geq q_T$ for all $(\la,\pi)$ sufficiently close to the regime $\cR(0)$.
\end{enumerate}
It is then routine to conclude the proof. We fix $\alpha=0.001$ and assume that $(\la,\pi)$ is sufficienly close to the regime $\cR(0)$ in such a way that $\varkappa_{\la,\pi}\leq \alpha$.

\md

\noindent{\bf Step 1.} Here we show that for all $b\in\rr$, there exists an event $\Omega_{b,0}^{\la,\pi}$, depending only on $(N^S_s(i),N^M_s(i),N^P_s(i))_{s\in\intervalleff{0}{3\al/4},i\in J_b^{\la}}$ such that
\begin{enumerate}[label=(\roman*)]
	\item on $\Omega^{\la,\pi}_{b,0}$, a.s., there is $i\in J_b^{\la}$ such that $\eta^{\la,\pi}_{\al s}(i)=0$ for all $s\in\intervalleff{0}{3/4}$;
	\item $\lim_{\la\to0}\proba{\Omega^{\la,\pi}_{b,0}}=1$.
\end{enumerate}
Simply consider the event $\Omega^{\la,\pi}_{b,0}=\{\exists i\in J_b^{\la}, N^S_{3\al/4}(i)=0\}$. Clearly, point (i) is satisfied, since there is a site in $J_b^\la$ on which no seed falls during $\intervalleff{0}{3\al/4}$. Since $\abs{J_b^\la}\simeq\nl\simeq1/(\la\log(1/\la))$, we deduce that
\[\proba{\Omega^{\la,\pi}_{b,0}}=1-(1-e^{-3\al/4})^{\nl}\simeq 1-e^{-1/(\la^{1/4}\al)}\xrightarrow[\la\to0]{}1,\]
whence (ii).

\md

\noindent{\bf Step 2.} For $\la>0$ and $\pi\geq1$, we put $\kl\coloneqq\lfloor\la^{-3/8}\rfloor$ and observe that $\kl\ll\nl$. For $k\in\{1,\dots,K-1\}$, we set
\[\tau_k=\frac{k+1}{4}\text{ and }\ttau_k=\frac{k+1}{4}+\frac{1}{8}.\]
Consider the event $\tOmega_{a,T}^{\la,\pi}$ on which points 1, 2 and 3 below are satisfied.
\begin{enumerate}
	\item The family of Poisson processes  $(N^ M_t(i))_{t
\in\intervalleff{0}{\al T},i
\in J_{a,K}^{\lambda}}$ has exactly $2(K-1)$ marks in $J_{a,K}^{\lambda}$, and we call them  
\[\{(X_1^\la,T_1^\la),\dots,(X_{K-1}^\la,T_{K-1}^\la)\}\text{ and }\{(\tX_1^\la,\tT_1^\la),\dots,(\tX_{K-1}^\la,\tT_{K-1}^\la)\},\]
in such a way that, for all $k\in\{1,\dots,K-1\}$,
\[(X_k^\la,T_k^\la)\in\intervalleentier{\lfloor(a-K+k+\frac1{3})\nl\rfloor}{\lfloor(a-K+k+\frac2{3})\nl\rfloor}\times\intervalleff{(\tau_k-1/12)\al}{\left(\tau_k-\varkappa_{\la,\pi}\right)\al}\]
and
\begin{multline*}
(\tX_k^\la,\tT_k^\la) \in\intervalleentier{\lfloor(a+K-(k+1)+\frac1{3})\nl\rfloor}{\lfloor(a+K-(k+1)+\frac2{3})\nl\rfloor}\\
\times\intervalleff{(\ttau_k-1/12)\al}{\left(\ttau_k-\varkappa_{\la,\pi}\right)\al}.
\end{multline*}
(See Figure \ref{sweet event} for a graphical example.)
	\item The family of Poisson processes $(N^S_t(i))_{t\geq0,i\in J_{a,K}^\la}$ satisfies, for all $k\in\{1,\dots, K-1\}$,
\begin{enumerate}
	\item there are $j_g\in\intervalleentier{\lfloor (a-K+k)\nl\rfloor}{\lfloor(a-K+k+1/4)\nl\rfloor}$ and $j_d\in\intervalleentier{\lfloor (a-K+k+3/4)\nl\rfloor}{\lfloor(a-K+k+1)\nl-1\rfloor}$ such that 
\[N^S_{\al(\tau_k+1/4)}(j_g)-N^S_{\al(\tau_k-1/2)}(j_g)=N^S_{\al(\tau_k+1/4)}(j_d)-N^S_{\al(\tau_k-1/2)}(j_d)=0;\]
	\item for all $i\in\intervalleentier{X_k^\la-\kl}{X_k^\la+\kl}$, 
\[N^S_{\al(\tau_k-1/12)}(i)-N^S_{\al(\tau_k-1/2)}(i)>0;\]
	\item there is $j_0\in\intervalleentier{X_k^\la-\kl}{X_k^\la+\kl}$ such that 
\[N^S_{\al(\tau_k+1/4)}(j_0)-N^S_{\al(\tau_k-1/12)}(j_0)=0.\]
\end{enumerate}
	\item The family of Poisson processes $(N^S_t(i))_{t\geq,i\in J_{a,K}^\la}$ satisfies, for all $k\in\{1,\dots, K-1\}$,
\begin{enumerate}
	\item there are $j_g\in\intervalleentier{\lfloor (a+K-(k+1))\nl\rfloor}{\lfloor(a+K-(k+1)+1/4)\nl\rfloor}$ and $j_d\in\intervalleentier{\lfloor (a+K-(k+1)+3/4)\nl\rfloor}{\lfloor(a+K-(k+1)+1)\nl-1\rfloor}$ such that 
\[N^S_{\al(\ttau_k+1/4)}(j_g)-N^S_{\al(\ttau_k-1/2)}(j_g)=N^S_{\al(\ttau_k+1/4)}(j_d)-N^S_{\al(\ttau_k-1/2)}(j_d)=0;\]
	\item for all $i\in\intervalleentier{\tX_k^\la-\kl}{\tX_k^\la+\kl}$, 
\[N^S_{\al(\ttau_k-1/12)}(i)-N^S_{\al(\ttau_k-1/2)}(i)>0;\]
	\item there is $j_0\in\intervalleentier{\tX_k^\la-\kl}{\tX_k^\la+\kl}$ such that 
\[N^S_{\al(\ttau_k+1/4)}(j_0)-N^S_{\al(\ttau_k-1/12)}(j_0)=0.\]
\end{enumerate}
\end{enumerate}
We also introduce the event 
\[\Omega^{P,K}_{\la,\pi}=\left(\bigcap_{k=1}^{K-1}\Omega^{P,2K,2K}_{\la,\pi}\left(\frac{X^\la_k}{\nl},\frac{T^\la_k}{\al}\right)\right)\cap\left(\bigcap_{k=1}^{K-1}\Omega^{P,2K,2K}_{\la,\pi}\left(\frac{\tX^\la_k}{\nl},\frac{\tT^\la_k}{\al}\right)\right),\]
recall Definition \ref{definition2 application}.

Finally, we set
\[\Omega_{a,T}^{\la,\pi}=\tOmega_{a,T}^{\la,\pi}\cap\Omega^{P,K}_{\la,\pi}\cap\Omega_{a-K,0}^{\la,\pi}\cap\Omega_{a+K-1,0}^{\la,\pi}.\]

\md

\noindent{\bf Step 3.} Here we prove (ii).

The probability of the event on which $N^M$ satisfies 1. does not depend on $a\in\rr$ by invariance by spatial translation. We also can construct $N^M$ using a Poisson measure $\pi_M$ on $\rr\times\intervallefo{0}{\infty}$ with intensity measure $\diff x\diff t$, independent of $N^S$ and $N^P$, by setting, for all $i\in\zz$
\[N_t^M(i)=\pi_M(i_\la\times\intervalleff{0}{t/\al}).\]
As usual, for all $\la>0$ small enough, the probability of the event on which $N^M$ satisfies 1 is then bounded from below by some constant $2q_T>0$, which does not depend on $a\in\rr$ nor on $\la>0$ and $\pi\geq1$. We write $\pp_M$ for the conditional probability w.r.t. $\pi_M$.

Let now $k\in\{1,\dots,K-1\}$. The probability of 2-(a) tends to $1$. Indeed, treating e.g. the case of $j_g$, there holds, recalling $\nl\simeq1/(\la\al)$ and $\al=\log(1/\la)$,
\begin{multline*}
\proba{\exists j\in\intervalleentier{\lfloor (a-K+k)\nl\rfloor}{\lfloor(a-K+k+1/4)\nl\rfloor}, N^S_{\al(\tau_k+1/4)}(j)-N^S_{\al(\tau_k-1/2)}(j)=0} \\
=1-(1-e^{-(3/4)\al})^{\nl/4}\simeq 1-e^ {-\nl\la^ {3/4}/4}\xrightarrow[\la\to0]{}1.
\end{multline*}
The probability of 2-(b) (conditionally on $\pi_M$) also tends to $1$. Indeed, it equals 
\[(1-e^{-5\al/12})^{2\kl+1}\simeq e^{-2\kl\la^{5/12}}\xrightarrow[\la\to0]{}1\]
since $\kl=\lfloor\la^{-3/8}\rfloor$ and since $3/8<5/12$. Finally, the probability of 2-(c) (conditionally on $\pi_M$) also tends to $1$, since it equals
\[1-(1-e^{-\al/3})^{2\kl+1}\simeq1-e^{-2\kl\la^{1/3}}\]
which tends to $1$ when $\la\to0$, since $1/3<3/8$.

Similar considerations hold for Point 3.

Finally, since $\pi_M$ is independent of the processes  family $(N^S_t(i))_{t\geq0,i\in\zz}$ and $(N^P_t(i))_{t\geq0,i\in\zz}$, Lemma \ref{propagation lemma 0} directly implies that, using space/time stationarity, for all $k\in\{1,\dots,K-1\}$, 
\[\pp_M\left[\Omega^{P,2K,2K}_{\la,\pi}(X_k^\la/\nl,T_k^\la/\al)\right]=\pp_M\left[\Omega^{P,2K,2K}_{\la,\pi}(\tX_k^\la/\nl,\tT_k^\la/\al)\right]\]
tends to $1$ when $\la\to0$ and $\pi\to\infty$ in the regime $\cR(0)$.

All this implies that there exists $q_T>0$ such that $\proba{\Omega_{a,T}^{\la,\pi}}>q_T$ for all $(\la,\pi)$ sufficiently close to the regime $\cR(0)$.

\md

\noindent{\bf Step 4.} Here we work on $\Omega_{a,T}^{\la,\pi}$ and we prove that, for all $k\in\{1,\dots,K-1\}$, if there is no burning tree in $J_{a-K+k}^\la$ at time $(\tau_k-1/2)\al$, then there is $i\in J_{a-K+k}^\la$ such that $\eta^{\la,\pi}_{\al t}(i)=0$ for all $t\in\intervalleff{\tau_k}{\tau_k+1/4}$. We distinguish two cases.
\begin{itemize}
	\item If the zone $\intervalleentier{X_k^\la-\kl}{X_k^\la+\kl}$ is completely occupied at time $T_k^\la-$, then each site burns at least one time (i.e. each site in this zone is ignited and then extinguished) during $\intervalleff{T_k^\la}{T_k^\la+\al \varkappa_{\la,\pi}}$, thanks to $\Omega^{P,2K,2K}_{\la,\pi}(X_k^\la/\nl,T_k^\la/\al)$, recall {\bf Macro$(0)$} in Subsection \ref{application ffp}. Since no seed falls on $j_0$, which belongs to this zone, during 
\[\intervalleff{\al(\tau_k-1/12)}{\al(\tau_k+1/4)}\supset\intervalleff{T_k^\la+\al\varkappa_{\la,\pi}}{\al(\tau_k+1/4)}\supset\intervalleff{\al \tau_k}{\al(\tau_k+1/4)},\]
we deduce that $\eta^{\la,\pi}_{\al s}(j_0)=0$ for all $s\in\intervalleff{\tau_k}{\tau_k+1/4}$.
	\item Assume now that there exists $i_0\in\intervalleentier{X_k^\la-\kl}{X_k^\la+\kl}$ that is vacant at time $T_k^\la-$. Recall that there is no match falling in $J_a^\la$ during $\intervallefo{\al(\tau_k-1/2)}{T_k^\la}$, that on each site of $\intervalleentier{X_k^\la-\kl}{X_k^\la+\kl}$, at least one seed falls during $\intervalleff{\al(\tau_k-1/2)}{\al(\tau_k-1/12)}\subset\intervallefo{\al(\tau_k-1/2)}{T_k^\la}$ and that there is no burning tree in $J_{a-K+k}^\la$ at time $\al(\tau_k-1/2)$. Then necessarily,  a fire starting at some $i'_M\not\in J_{a-K+k}^\la$ at some time $t'_M<T_k^\la$, has made vacant $i_0$. Assume e.g. that $i'_M<\lfloor (a-K+k)\nl\rfloor$ and observe that $i'_M<j_g<i_0$. The fire $(i'_M,t'_M)$ has then also necessarily made vacant $j_g$ during $\intervalleoo{\al(\tau_k-1/2)}{T_k^\la}$. Since no seed falls on $j_g$ during $\intervalleff{\al(\tau_k-1/2)}{\al(\tau_k+1/4)}$, we deduce that $j_g$ remains vacant during $\intervalleff{\al \tau_k}{\al(\tau_k+1/4)}$.
\end{itemize}

\md

\noindent{\bf Step 5.} We can show, exactly as above, that, on $\Omega_{a,T}^{\la,\pi}$, if there is no burning tree in $J_{a+K-(k+1)}^\la$ at time $(\ttau_k-1/2)\al$, for some $k\in\{1,\dots,K-1\}$, then there is $i\in J_{a+K-(k+1)}^\la$ such that $\eta^{\la,\pi}_{\al t}(i)=0$ for all $t\in\intervalleff{\ttau_k}{\ttau_k+1/4}$.

\md

\noindent{\bf Step 6.} To conclude the proof, we now prove by induction (see Figure \ref{sweet event}) that for all $k\in\{1,\dots,K-1\}$ 
	\begin{itemize}
		\item there exists $i_k\in J_{a-K+k}^\la$ such that $\eta_{\al t}^{\la,\pi}(i_k)=0$ for all $t\in\intervalleff{\tau_k}{\tau_k+1/4}$;
		\item there exists $j_k\in J_{a+K-(k+1)}^\la$ such that $\eta_{\al t}^{\la,\pi}(j_k)=0$ for all $t\in\intervalleff{\ttau_k}{\ttau_k+1/4}$;
		\item there is no burning tree in $\intervalleentier{i_k}{j_k}$ at time $\al\tau_k$ nor at time $\al\ttau_k$.
	\end{itemize}
\begin{enumerate}[label=$\rhd$]
	\item At time $0$, all sites are vacant. Thus, there are $i_0\in J_{a-K}^\la$ and $j_0\in J_{a+K-1}^\la$ which remain vacant until time $3\al/4$ (thanks to $\Omega_{a-K,0}^{\la,\pi}\cap\Omega_{a+K-1,0}^{\la,\pi}$). Since no match falls in $\intervalleentier{i_0}{j_0}$ until time $T_1^\la\geq\al(1/2-1/12)=5\al/12$, there is no burning tree in $\intervallefo{0}{5\al/12}$ (no match falling outside $\intervalleentier{i_0}{j_0}$ during $\intervallefo{0}{5\al/12}$ can affect this zone).
	
Thus, Step 4 shows that there are $i_1\in J_{a-K+1}^\la$ which is vacant during $\intervalleff{\al/2}{3\al/4}$ (because $\tau_1-1/2=0$) and $i_2\in J_{a-K+2}^\la$ which is vacant during $\intervalleff{3\al/4}{\al}$ (because $\tau_2-1/2=1/4<5/12$). Similarly, Step 5 above shows that there are $j_1\in J_{a+K-2}^\la$ which is vacant during $\intervalleff{5\al/8}{7\al/8}$ (because $\ttau_1-1/2=1/8<5/12$) and $j_2\in J_{a+K-3}^\la$ which is vacant during $\intervalleff{7\al/8}{9\al/8}$ (because $\ttau_2-1/2=3/8<5/12$).

Since $T_1^\la\leq (1/2-\varkappa_{\la,\pi})\al$ and $|X_1^\la-i_0|\leq|X_1^\la-j_0|\leq2K\nl$, as seen in {\bf Macro$(0)$} in Subsection \ref{application ffp} (recall that we work on $\Omega^{P,2K,2K}_{\la,\pi}(X_1^\la/\nl,T_1^\la/\al)$), there is no more burning tree in $\intervalleentier{i_0}{j_0}$ at time $T_1^\la+\al\varkappa_{\la,\pi}\leq \al/2=\al\tau_1$. Since no match falls in $\intervalleentier{i_0}{j_0}$ during $\intervalleff{T_1^\la+\al\varkappa_{\la,\pi}}{\al/2}$, we deduce that there is also no burning tree in $\intervalleentier{i_0}{j_0}\supset\intervalleentier{i_1}{j_1}$ at time $\al\tau_1$ (because $i_0$ and $j_0$ remain vacant until $\al/2$).

Since no match falls in $\intervalleentier{i_1}{j_0}$ during $\intervallefo{\al\tau_1}{\tT^\la_1}$, we deduce that there is no burning tree in $\intervalleentier{i_1}{j_0}$ at time $\tT_1^\la-$. Since  $\eta^{\la,\pi}_{t}(i_1)=\eta^{\la,\pi}_t(j_0)=0$ for all $t\in\intervalleff{\tT_1^\la}{\tT_1^\la+\al\varkappa_{\la,\pi}}$ and only one match falls in $\intervalleentier{i_1}{j_0}$ during $\intervalleff{\tT_1^\la}{\tT_1^\la+\al\varkappa_{\la,\pi}}$, we deduce, recall {\bf Macro$(0)$} in Subsection \ref{application ffp}, that there is no more burning tree in $\intervalleentier{i_1}{j_0}$ at time $\tT_1^\la+\al\varkappa_{\la,\pi}$. We easily deduce that there is also no burning tree in $\intervalleentier{i_1}{j_1}\subset\intervalleentier{i_1}{j_0}$ at time $\al\ttau_1$.

Similarly, since $i_0<i_1<i_2<j_2<j_1<j_0$ and thanks to $\Omega^{P,K}_{\la,\pi}$, there is no more burning tree in $\intervalleentier{i_1}{j_1}\supset\intervalleentier{i_2}{j_2}$ at time $\tau_2$ nor in $\intervalleentier{i_2}{j_1}\supset\intervalleentier{i_2}{j_2}$ at time $\ttau_2$.
	\item Assume now that there is $k\in\{2,\dots,K-2\}$ such that, for all $l\leq k$,
	\begin{itemize}
		\item there exists $i_l\in J_{a-K+l}^\la$ such that $\eta_{\al t}^{\la,\pi}(i_l)=0$ for all $t\in\intervalleff{\al \tau_l}{\al(\tau_l+1/4)}$;
		\item there exists $j_l\in J_{a+K-(l+1)}^\la$ such that $\eta_{\al t}^{\la,\pi}(j_l)=0$ for all $t\in\intervalleff{\al \ttau_l}{\al(\ttau_l+1/4)}$;
		\item there is no burning tree in $\intervalleentier{i_l}{j_l}$ at time $\al\tau_l$ nor at time $\al\ttau_l$.
	\end{itemize}
Since there is no burning tree in $J_{a-K+k+1}^\la\subset\intervalleentier{i_{k-1}}{j_{k-1}}$ at time $\al\tau_{k-1}=\al(\tau_{k+1}-1/2)$, see Step 4, there is $i_{k+1}\in J_{a-K+k+1}^\la$ which is vacant during $\intervalleff{\al\tau_{k+1}}{\al(\tau_{k+1}+1/4)}$. Furthermore, no match falls in $\intervalleentier{i_k}{j_k}$ during $\intervallefo{\al \ttau_k}{T^\la_{k+1}}\subset\intervalleff{\al \ttau_k}{\al(\tau_{k+1}-\varkappa_{\la,\pi})}$ and there is no burning tree in $\intervalleentier{i_k}{j_k}$ at time $\al \ttau_k$, thus, as seen in {\bf Macro$(0)$} in Subsection \ref{application ffp} and thanks to $\Omega_ {\la,\pi}^{P,2K,2K}\left(\frac{X_{k+1}^\la}{\nl},\frac{T_{k+1}^\la}{\al}\right)$, there is no more burning tree in $\intervalleentier{i_k}{j_k}$ at time $T^\la_{k+1}+\al\varkappa_{\la,\pi}$ nor at time $\al \tau_{k+1}$ (because $i_k$ and $j_k$ remain vacant until $\al\tau_{k+1}$ and no match falls in $\intervalleentier{i_k}{j_k}$ during $\intervalleof{T_{k+1}^\la+\al\varkappa_{\la,\pi}}{\al\tau_{k+1}}$). 

Since there is no burning tree in $J^\la_{a+K-(k+2)}\subset\intervalleentier{i_{k-1}}{j_{k-1}}$ at time $\al\ttau_{k-1}=\al(\ttau_{k+1}-1/2)$, we deduce by Step 5 that there is $j_{k+1}\in J_{a+K-(k+2)}^\la$ which is vacant during $\intervalleff{\al\ttau_{k+1}}{\al(\ttau_{k+1}+1/4)}$. No match falls in $\intervalleentier{i_{k+1}}{j_k}$ during  $\intervallefo{\al \tau_{k+1}}{\tT^\la_{k+1}}\subset\intervalleff{\al\tau_{k+1}}{\al(\ttau_{k+1}-\varkappa_{\la,\pi})}$ and there is no burning tree in $\intervalleentier{i_{k+1}}{j_k}$ at time $\al \tau_{k+1}$, thus, as seen in {\bf Macro$(0)$} in Subsection \ref{application ffp} and thanks to $\Omega_ {\la,\pi}^{P,2K,2K}\left(\frac{\tX_{k+1}^\la}{\nl},\frac{\tT_{k+1}^\la}{\al}\right)$, there is no more burning tree in $\intervalleentier{i_{k+1}}{j_k}$ at time $\tT^\la_{k+1}+\al\varkappa_{\la,\pi}$ nor at time $\al \ttau_{k+1}$, as usual.
\end{enumerate}
By the induction above, we deduce that there are
\[\iota^+\colon\intervalleff{0}{T}\to J_{a,K}^\la\]
non decreasing, such that for all $t\in\intervalleff{0}{T}$, $\eta^{\la,\pi}_{\al t}(\iota_{\al t}^+)=0$ and 
\[\iota^-\colon\intervalleff{0}{T}\to J_{a,K}^\la\]
non increasing, such that for all $t\in\intervalleff{0}{T}$, $\eta^{\la,\pi}_{\al t}(\iota_{\al t}^-)=0$.
This together with Step 3 conclude the proof in the regime $\cR(0)$.\qedhere

\begin{figure}[h!]
\fbox{
\begin{minipage}[c]{0.95\textwidth}
\centering
\begin{tikzpicture}
\begin{scope}
\clip (0,2) -- (1.76,2)--(1.76,2.5)-- (2.4,2.5)--(2.4,3) -- (2.86,3 )--(2.86 ,3.5 ) -- (3.24,3.5 )--(3.24 ,4) -- (3.78,4)--(3.78,4.5) -- (4.38,4.5)--(4.38,5) -- (4.84,5)--(4.84,5.5)-- (5.24,5.5)--(5.24,6)--(5.76,6)--(5.76,6.4)--(0,6.4)--cycle;
\foreach \x in {1,...,40} {
  \draw[dashed] (\x/2,2) -- (0,2+\x/2);
}
\end{scope}

\begin{scope}
\clip (12,2)--(10.8,2)--(10.8,2.25)--(10.24,2.25)--(10.24,2.75)-- (10,2.75)--(10,3.25) -- (9.24,3.25 )--(9.24 ,3.75 ) --(8.76,3.75 )--(8.76 ,4.3) --(8.22,4.25)--(8.22,4.75)-- (7.9,4.75)--(7.9,5.25)--(7.26,5.25)--(7.26,5.75)--(6.76,5.75)--(6.76,6.35) --(6.24,6.35)--(6.24,6.4)--(12.5,6.4)--(12,2);
\foreach \x in {1,...,40} {
  \draw[dashed] (\x/2,2) -- (0,2+\x/2);
}
\end{scope}
\fill[white] (6,2) rectangle (8,4);

\draw[->] (0,0)-- (0,7);
\draw[->] (0,0)-- (12.5,0);
\draw (-.1,2)--(.1,2) node at (0,-.3) {$-12\phantom{-}$} node at (3,-.3) {$-6\phantom{-}$} node at (9,-.3) {$6$} node at (12,-.3) {$12$};
\draw (-.1,6.4)--(.1,6.4) node at (-.3,6.4) {$T$} node at (-.3,2) {$1$};
\draw (-.1,1)--(.1,1)  node at (-.4,1) {$0{,}5$};
\draw (6,-.1)--(6,.1) node[below] at (6,-.1) {$0$};

\foreach \x in {0,...,12} {
\draw (\x,0.1cm) -- (\x,-0.1cm);
}
\foreach \x in {0.5,...,12} {
\draw (\x,0.1cm) -- (\x,-0.1cm);
}
%\foreach \x in {1,...,4} {
%\draw (\x,0.1cm) -- (\x,-0.1cm) node[below] {$\x\strut$};
%}

%%%d'abord le feu
\draw[decorate,decoration={random steps,segment length=0.5mm,amplitude=0.1mm},color=red, fill=gray!50!white]
	(1.76,2.2)--(10.8,2.2)--(10.8,2.1)--(1.76,2.1)--cycle;
%%%puis je complete
\draw[decorate,decoration={random steps,segment length=0.5mm,amplitude=0.1mm},color=white,fill=gray!50!white, color=gray!50!white]
	(1.76,2)--(10.8,2)--(10.8,2.13)--(1.76,2.13)--cycle;
	
%%%%d'abord le feu
\draw[decorate,decoration={random steps,segment length=0.5mm,amplitude=0.1mm},color=red, fill=gray!50!white]
	(3.78,4.25)--(8.76,4.25)--(8.76,4.22)--(3.78,4.22)--cycle;

%%%puis je complete
\draw[decorate,decoration={random steps,segment length=0.5mm,amplitude=0.1mm},color=white,fill=gray!50!white, color=gray!50!white]
	(3.78,4.2)--(8.76,4.2)--(8.76,4.23)--(3.78,4.23)--cycle;

%%%%d'abord le feu
\draw[decorate,decoration={random steps,segment length=0.5mm,amplitude=0.1mm},color=red, fill=gray!50!white]
	(5.76,6.3)--(6.76,6.3)--(6.76,6.27)--(5.76,6.27)--cycle;

%%%puis je complete
\draw[decorate,decoration={random steps,segment length=0.5mm,amplitude=0.1mm},color=white,fill=gray!50!white, color=gray!50!white]
	(5.76,6.23)--(6.76,6.23)--(6.76,6.28)--(5.76,6.28)--cycle;

%%%%%%%les marques de K de hauteur 1/4->.5 dans notre echelle
\draw[thick] (.2,0)--(.2,1);
\draw[thick] (.74,1)--(.74,1.5) node at (.74,.95) {\color{red}{$\bullet$}};
\draw[thick] (1.2,1.5)--(1.2,2) node at (1.2,1.4) {\color{red}{$\bullet$}};

\draw[thick] (1.76,2)--(1.76,2.5) node at (1.76,1.9) {\color{red}{$\bullet$}};
\draw[thick] (2.4,2.5)--(2.4,3) node at (2.2,2.45) {\color{red}{$\bullet$}};
\draw[thick] (2.86,3 )--(2.86 ,3.5 ) node at (2.76,2.9) {\color{red}{$\bullet$}};
\draw[thick] (3.24,3.5 )--(3.24 ,4) node at (3.24,3.45) {\color{red}{$\bullet$}};

\draw[thick] (3.78,4)--(3.78,4.5) node at (3.78,3.9) {\color{red}{$\bullet$}};
\draw[thick] (4.38,4.5)--(4.38,5) node at (4.26,4.5) {\color{red}{$\bullet$}};
\draw[thick] (4.84,5)--(4.84,5.5) node at (4.74,5) {\color{red}{$\bullet$}};
\draw[thick] (5.24,5.5)--(5.24,6) node at (5.24, 5.5) {\color{red}{$\bullet$}};

\draw[thick] (5.76,6)--(5.76,6.5) node at (5.76,6) {\color{red}{$\bullet$}};
%%%%%%%%%%%%%%%%%%%%%%%%%%%%%%%%%%%%%%%
\draw[thick] (11.8,0)--(11.8,1.25);
\draw[thick] (11.26,1.25)--(11.26,1.75) node at (11.26,1.25) {\color{red}{$\bullet$}};
\draw[thick] (10.8,1.75)--(10.8,2.25)node at (10.8,1.75) {\color{red}{$\bullet$}};

\draw[thick] (10.24,2.25)--(10.24,2.75);% node at (10.24,2.25) {\color{red}{$\bullet$}};
\draw[fill=red,color=red] (10.24,2.2) circle (0.05cm);
\draw[thick] (10,2.75)--(10,3.25) node at (9.8,2.75) {\color{red}{$\bullet$}};
\draw[thick] (9.24,3.25 )--(9.24 ,3.75 ) node at (9.24,3.25) {\color{red}{$\bullet$}};
\draw[thick] (8.76,3.75 )--(8.76,4.3) node at (8.76,3.75) {\color{red}{$\bullet$}};

\draw[thick] (8.22,4.25)--(12-3.78,4.75);% node at (8.22,4.25) {\color{red}{$\bullet$}};
\draw[fill=red,color=red] (8.22,4.25) circle (0.05cm);
\draw[thick] (7.9,4.75)--(7.9,5.25) node at (7.74,4.75) {\color{red}{$\bullet$}};
\draw[thick] (12-4.74,5.25)--(12-4.74,5.75) node at (7.26,5.25) {\color{red}{$\bullet$}};
\draw[thick] (6.76,5.75)--(12-5.24,6.35) node at (6.76,5.75) {\color{red}{$\bullet$}};

\draw[thick] (12-5.76,6.35)--(6.24,6.75);% node at (6.24,6.35) {\color{red}{$\bullet$}};
\draw[fill=red,color=red] (6.24,6.3) circle (0.05cm);

\end{tikzpicture}\caption{The sweet event}\label{sweet event}
\vspace{.5cm}
\parbox{13.3cm}{
\footnotesize{Here $T=3.2, K=12$ and $a\in\intervallefo{0}{1}$. 
The marks of $\pi_M$ (matches) are represented as $\color{red}{\bullet}$'s. The filled zones represent macroscopic zones ($Z^{\la,\pi}_{\al t}(x)=1$). In the rest of the space, we always have $Z^{\la,\pi}_{\al t}(x)<1$. The plain vertical segments represent vacants sites i.e. sites where no seed falls after being propagated. Remark that sometimes the vacant site is above the match (that is in an interval with length $2\kl$) and sometimes it is next to the match (that is an $i^g$ or an $i^d$).
}}
\end{minipage}}
\end{figure}

%\end{figure}
%\begin{proof}
%(ii) follows from preliminaries and (iii) is obvious by construction. This is also easy to prove (i) by induction. Indeed, let $r\in\{2,\dots, K\}$ and suppose that we have build
%\begin{align*}
%i^+ \colon\intervalleff{0}{\al r/4} &\to\intervalleentier{\lfloor(a-K)\nl\rfloor}{\lfloor(a-K+(r-2))\nl\rfloor}\\
%i^- \colon\intervalleff{0}{\al(r/4+1/8)} &\to\intervalleentier{\lfloor(a+K-(r-1))\nl\rfloor}{\lfloor(a+K)\nl\rfloor}\\
%\end{align*}
%and there is no burning in 
%\[
%\left(\Omega_{a-K,0}^{\la,\pi} \cap \bigcap_{k=2}^{r-1} \Omega_{a-K+(k-1),k/4}^{\la,\pi}\right)
%\cap\left(\Omega_{a+K-1,0}^{\la,\pi} \cap \bigcap_{k=2}^{r-1} \Omega_{a+K-k,k/4+1/8}^{\la,\pi}\right)
%\]
%the match falling in $J_{a-K+(r-1)}^\la$ during the time interval $\intervalleff{\al(r/4-1/12)}{\al(r/4-\kappa^0_{\la,\pi})}$ have a limited range: it is 
%there is no burning tree in $J_{a-K+(r-1)}^\la$ at time $r/4-1/2$ (this is easy by induction). In the same way, on
%\begin{multline*}
%\left(\Omega_{a-K,0}^{\la,\pi} \cap \bigcap_{k=2}^{r} \Omega_{a-K+(k-1),k/4}^{\la,\pi}(2,2K-(k-1) )\right)\\
%\cap\left(\Omega_{a+K-1,0}^{\la,\pi} \cap \bigcap_{k=2}^{r-1} \Omega_{a+K-k,k/4+1/8}^{\la,\pi}(2K-(k-2),2)\right)
%\end{multline*}
%there is no burning tree in $J_{a+K-r}^\la$ at time $r/4-3/8$.
%\end{proof}

\end{proof}

\section{Localization of the result}
In this section, we localize Theorems \ref{converge} and \ref{theoinfty}.
\subsection{Localization in the regime $\cR(p)$}
The following Theorem will be proved in Section \ref{convergence in the regime p}.
\begin{theo}\label{converge restriction}
Let $A>0$ and $p\geq0$ be fixed. Consider for each $\lambda\in\intervalleof{0}{1}, \pi\geq1$, the process
$(Z^{\la,\pi,A}_t(x),D_t^{\la,\pi,A})_{t\geq0,x\in\rr}$ associated with the
$(\lambda,\pi,A)-$FFP. Consider also the $A-$LFFP$(p)$
$(Z_t^{A}(x),H_t^{A}(x),F_t^{A}(x))_{t\geq
0,x\in \rr}$ and the associated $(D_t^A(x))_{t\geq0, x\in\rr}$. We assume that $\la\to0$ and $\pi\to\infty$ in the regime $\cR(p)$, for some $p\in\intervallefo{0}{+\infty}$.
\begin{enumerate}
        \item For any $T>0$, any finite subset $\{x_1,\dots,x_q\}\subset\rr$,
$(Z^{\la,\pi,A}_t(x_i),D_t^{\la,\pi,A}(x_i))_{t\in\intervalleff{0}{T},i=1,\dots,q}$ goes in law to
$(Z_t^{A}(x_i),D_t^A(x_i))_{t\in\intervalleff{0}{T},i=1,\dots,q}$ in
$\dd(\intervalleff{0}{T},\rr\times(\cI\cup\{\emptyset\}))$. Here $\dd(\intervalleff{0}{T},\rr\times(\cI\cup\{\emptyset\}))$ is endowed with
the distance $\bd_T$.
        \item For any subset
$\{(x_1,t_1),\dots,(x_q,t_q)\}\subset\rr\times\intervallefo{0}{\infty}$,
$(Z^{\la,\pi,A}_{t_i}(x_i),D_{t_i}^{\la,\pi,A}(x_i))_{i=1,\dots,q}$ goes in law to
$(Z_{t_i}^{A}(x_i),D_{t_i}^A(x_i))_{i=1,\dots,q}$ in $(\rr\times(\cI\cup\{\emptyset\}))^q$.
Here $\cI\cup\{\emptyset\}$ is endowed with $\bdelta$.
        \item\label{estim cluster size} For all $t>0$, 
\[\left( \frac{\log(|C(\eta_{\al t}^{\lambda,\pi,A},0)|)}{\log(1/\la)} \indiq{|C(\eta_{\al
t}^{\lambda,\pi,A},0)|\geq1}\right)\wedge 1\]
goes in law to $Z_t^{A}(0)$.
\end{enumerate}
\end{theo}
 Assuming for a moment that this theorem holds true, we conclude the proof of Theorem \ref{converge}.
 
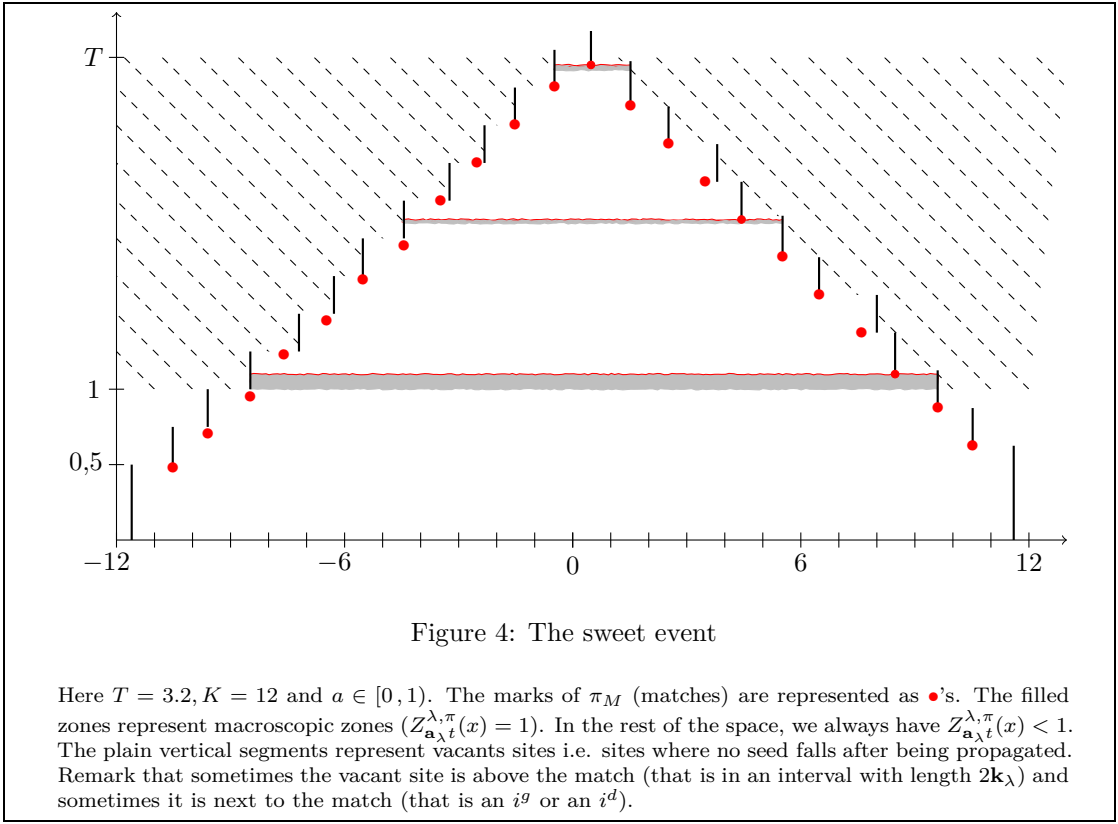
\begin{proof}[Proof of Theorem \ref{converge}]
Let us first 
prove 1. Consider a continuous bounded function 
$\Psi:\dd([0,T], \rr\times\cI\cup\{\emptyset\})^q \mapsto \rr$. 
We have to prove that $G_{\la,\pi}(\Psi)$ tends to $0$ when $\la\to0$ and $\pi\to\infty$ in the regime $\cR(p)$, where
\[G_{\la,\pi}(\Psi)=
\E\left[\Psi\left((Z^{\la,\pi}_t(x_i),D^{\la,\pi}_t(x_i))_{t\in [0,T],i=1,\dots,q} \right)
\right]-
\E\left[\Psi\left((Z_t(x_i),D_t(x_i))_{t\in [0,T],i=1,\dots,q} \right)
\right].\]
Using now Propositions \ref{restriction limite} and \ref{restriction limite loc}, we observe that for any
$A>2 \max_{i=1,\dots,q} |x_i|$, there holds that for all $(\la,\pi)$ sufficiently close to the regime $\cR(p)$,
\begin{align*}
&|G_{\la,\pi}(\Psi)|\\
\leq& 2||\Psi||_\infty\proba{(Z^{\la,\pi,A}_t(x),D^{\la,\pi,A}_t(x))_{t\in 
[0,T],x\in [-A/2,A/2]} \neq (Z^{\la,\pi}_t(x),D^{\la,\pi}_t(x))_{t\in 
[0,T],x\in [-A/2,A/2]}}\\
&+2||\Psi||_\infty \proba{(Z^{A}_t(x),D^{A}_t(x))_{t\in 
[0,T],x\in [-A/2,A/2]} \ne (Z_t(x),D_t(x))_{t\in 
[0,T],x\in [-A/2,A/2]}}\\
&+\left|\E\left[\Psi\left((Z^{\la,\pi,A}_t(x_i),D^{\la,\pi,A}_t(x_i))_{t\in 
[0,T],i=1,\dots,q} \right) \right]-
\E\left[\Psi\left((Z_t^{A}(x_i),D_t^A(x_i))_{t\in [0,T],i=1,\dots,q} \right)
\right]\right|\\
\leq& 4 ||\Psi||_\infty C_T e^{-\alpha_T A}\\
&+\left|\E\left[\Psi\left((Z^{\la,\pi,A}_t(x_i),D^{\la,\pi,A}_t(x_i))_{t\in 
[0,T],i=1,\dots,q} \right) \right]-
\E\left[\Psi\left((Z_t^{A}(x_i),D_t^A(x_i))_{t\in [0,T],i=1,\dots,q} \right)
\right]\right|.
\end{align*}
Thus Proposition \ref{converge restriction}-(1) implies that $|G_{\la,\pi}(\Psi)|
\leq 5 ||\Psi||_\infty C_T e^{-\alpha_T A}$ for all $(\la,\pi)$ sufficiently close to the regime $\cR(p)$. We conclude by making $A$ tend
to infinity.

Point (2) is checked similarly. The proof of (3) is also similar,
since $D^{\la,\pi}_t(0)=D^{\la,\pi,A}_t(0)$ implies that
$C(\eta^{\la,\pi}_{\al t},0)= C_A(\eta^{\la,\pi,A}_{\al t},0)$.
\end{proof}

\subsection{Localization in the regime $\cR(\infty,z_0)$}
The following Theorem will be proved in the next Section.
\begin{theo}\label{theoinfty loc}
Let $z_0\in\intervalleff{0}{1}$ and $A>0$. Consider for each $\la\in\intervalleof{0}{1}$ and $\pi\geq1$ the process $(D_t^{\la,\pi,A}(x))_{t\geq0,x\in\rr}$ associated with the $(\la,\pi,A)-$FFP. Consider also the LFFP$(\infty,z_0)$ $(Y_t(x))_{t\geq0,x\in\rr}$ and the associated $(D_t^A(x))_{t\geq0,x\in\rr}$ process. We assume that $\la\to0$ and $\pi\to\infty$ in the slow regime $\cR(\infty,z_0)$.
\begin{enumerate}
	\item For any $T>0$, any finite subset $\{x_1,\dots,x_q\}\subset\rr$, $(D_t^{\la,\pi,A}(x_i))_{t\in\intervalleff{0}{T},i=1,\dots,q}$ goes in law to $(D_t^A(x_i))_{t\in\intervalleff{0}{T}, i=1,\dots,q}$ in $\dd(\intervalleff{0}{T},\cI)^q$. Here $\dd(\intervalleff{0}{T},\cI)^q$ is endowed with $\bdelta_T$.
	\item For any finite subset $\{(x_1,t_1),\dots,(x_q,t_q)\}\subset\rr\times\intervallefo{0}{\infty}$, $(D_{t_i}^{\la,\pi,A}(x_i))_{i=1,\dots,q}$ goes in law to $(D_{t_i}^A(x_i))_{t\in\intervalleff{0}{T}, i=1,\dots,q}$ in $\cI^q$, $\cI$ being endowed with $\bdelta$.
\end{enumerate}
\end{theo}

\begin{proof}[Proof of Theorem \ref{theoinfty}]
The proof easily follows from Proposition \ref{restriction limite infini}, Proposition \ref{restriction limite loc} and Theorem \ref{theoinfty loc}, as in the proof above.
\end{proof}

\section{Convergence in the regime $\cR(\infty,z_0)$}
The aim of this section is to prove Theorem \ref{theoinfty loc}. We thus fix the parameters $A>0$ and $T>0$.

We recall that $\al=\log(1/\la)$, $\nl=\lfloor 1/(\la\al)\rfloor$, $\ml=\lfloor 1/(\la\ba_\la^ 2)\rfloor$, $\e_\la=1/\ba_\la^3$ and that
\begin{align*}
A_\la &=\lfloor A\nl\rfloor,\\
I^\la_A &= \intervalleentier{-A_\la}{A_\la}.
\end{align*}
For $x\in\rr$, we define
\[(x)_\la =\intervalleentier{\lfloor\nl x\rfloor-\ml}{\lfloor\nl x\rfloor+\ml}.\]
For $\alpha\in\intervalleoo{0}{1}$, we also define
\begin{align*}
\mla &= \left\lfloor\frac{\alpha}{\la^{\alpha+(1-\alpha)z_0}\al}\right\rfloor,\\
(x)_\la^\alpha &=\intervalleentier{\lfloor\nl x\rfloor-\mla}{\lfloor\nl x\rfloor+\mla}.
\end{align*}
Observe that $\mla\leq \lfloor\alpha\nl\rfloor$ for all $z_0\in\intervalleff{0}{1}$.

\subsection{Occupation of vacant zone}
We start with some easy estimates.
\begin{lem}\label{speed infty}
Consider a family of i.i.d. Poisson processes $(N^S_t(i))_{t\geq0,i\in\zz}$. Let $0<z<1$, $\alpha\in\intervalleoo{0}{1}$ and $a<b$.
\begin{enumerate}
        \item For $t<z$, $\proba{\forall i \in
\intervalleentier{\lfloor a\la^{-z}\rfloor}{\lfloor b\la^{-z}\rfloor}, N^S_{\al
t}(i)>0}\xrightarrow[\la\to0]{}0$.
        \item For $t>z$, $\proba{\forall i \in
\intervalleentier{\lfloor a\la^{-z}\rfloor}{\lfloor b\la^{-z}\rfloor},
N^S_{\al t}(i)>0}\xrightarrow[\la\to0]{}1$.
        \item For $t\geq1$, $\proba{\forall i \in
\intervalleentier{\lfloor a\nl\rfloor}{\lfloor b\nl\rfloor}, N^S_{\al
t}(i)>0}\xrightarrow[\la\to0]{}1$.
        \item For $t<1$, $\proba{\forall i \in
\intervalleentier{\lfloor a\ml\rfloor}{\lfloor b\ml\rfloor}, N^S_{\al
t}(i)>0}\xrightarrow[\la\to0]{}0$.
	\item For $t>z_0+\alpha$, $\proba{\forall i \in
\intervalleentier{-\lfloor a\mla\rfloor}{\lfloor b\mla\rfloor}, N^S_{\al
t}(i)>0}\xrightarrow[\la\to0]{}1$.
\end{enumerate}
\end{lem}

\begin{proof}
To check Lemma \ref{speed infty}, observe that, for $k_\la\xrightarrow[\la\to0]{}\infty$,
\begin{equation}\label{oq}
\proba{\forall i \in
\intervalleentier{-\lfloor ak_\la\rfloor}{\lfloor bk_\la\rfloor}, N^S_{\al
t}(i)>0}\simeq(1-e^ {\al t})^{(b-a)k_\la}\simeq e^{-(b-a)k_\la \la^t}.
\end{equation}
In order to prove 1 and 2, use \eqref{oq} with $k_\la=\la^{-z}$ and observe that
\[k_\la\la^t=\la^{-z}\la^ t\xrightarrow[\la\to0]{}\begin{cases}
\infty &\text{if } t<z,\\
0 &\text{if } t>z.
\end{cases}\]
To prove 3, use \eqref{oq} with $k_\la=\nl$ and observe that, if $t\geq1$, $\nl \la^ t\simeq\la^ {t-1}/\al$ tends to $0$ when $\la\to0$. In the same way, 4 can be proved using $k_\la=\ml$ and observing that, if $t<1$, $\ml\la^ {t}\simeq\la^{t-1}/\ba_\la^2$ tends to $\infty$ when $\la\to0$. 

Finally, prove 5 with \eqref{oq} and using $k_\la=\mla$ and observing that $\mla\la^ {t}\simeq\frac{\alpha}{\al}\la^ {t-\alpha-(1-\alpha)z_0}$ tends to $0$ when $\la\to0$ as soon as $t-\alpha-(1-\alpha)z_0>0$ (in particular, for $t\geq z_0+\alpha>\alpha+(1-\alpha)z_0)$).
\end{proof}

\subsection{Height of the barrier}\label{height infty}
We describe here the time needed for a destroyed microscopic cluster to be
regenerated. Assume that a match
falls in the site $0$ at some time $\al t_1\in\intervalleoo{0}{\al z_0}$. As seen in {\bf Micro$(\infty,z_0)$} in Subsection \ref{application ffp}, on a suitable event, the $(\la,\pi)-$FFP is well understood around $0$ during $\intervalleff{\al t_1}{\al(t_1+\kappa_{\la,\pi}^z)}$, for some $0<z<z_0$ (it can be expressed using the sequence $(T^1_i)_{i\in\zz}$). We then denote by $\Theta^{\la,\pi}_{t_1}$ the delay needed for the destroyed cluster to be fully regenerated (after rescaling). We show that $\Theta^{\la,\pi}_{t_1}\simeq t_1$.

\begin{lem}\label{micro fire infty}
Consider two Poisson processes $(N_t^S(i))_{t\geq0,i\in\zz}$ and
$(N_t^P(i))_{t\geq0,i\in\zz}$ with respective rates $1$ and $\pi$, all this
processes being independent. Let $0<t_1<z_0$. We call $(T_i^ 1)_{i\in\zz}$ the burning times of the propagation process ignited in $0$ at time $\al t_1$, recall Definition \ref{definition1 application}.

Put, for all $t\geq0$ and $i\in\zz$, $\zeta_{t}^{\la,\pi}(i)=\min(N_{t}^S(i),1)$ and define 
\[C^P((\zeta_{t}^{\la,\pi}(i))_{t\geq0,i\in\zz},(0,t_1))=\intervalleentier{i^g}{i^d},\]
recall Definition \ref{def destroyed comp}.

We define a process $(\zeta^{\la,\pi}_{t_1,t}(i))_{t\in\intervalleff{0}{T},i\in\zz}$ in the following way (which is inspired by {\bf Micro$(\infty,z_0)$} in Subsection \ref{application ffp}): 
 we put, for all $i\in C^P((\zeta^{\la,\pi}_t(i))_{t\geq0,i\in\zz},(0,t_1))$
\[\zeta^{\la,\pi}_{t_1,t}(i)=\min(N_{\al t}^S(i),1)\text{ for }t\in\intervallefo{0}{t_1+(T^1_i/\al)}\]
and
\[\zeta^{\la,\pi}_{t_1,t}(i)=2\begin{cases}
\text{for }t\in\intervallefo{t_1+(T^1_i/\al)}{t_1+(T^1_{i+1}/\al)} &\text{if }i\geq0,\\
\text{for }t\in\intervallefo{t_1+(T^1_i/\al)}{t_1+(T^1_{i-1}/\al)} &\text{if }i\leq0\\
\end{cases}
\]
and
\[\zeta^{\la,\pi}_{t_1,t}(i)=\begin{cases}
\min(N_{\al (t+t_1)}^S(i)-N_{\al t_1+T_{i+1}^1}^S(i),1) &\text{for }t\in\intervalleff{t_1+(T^1_{i+1}/\al)}{ T} \text{ if }i\geq0,\\
\min(N_{\al (t+t_1)}^S(i)-N_{\al t_1+T_{i-1}^1}^S(i),1) &\text{for }t\in\intervalleff{t_1+(T^1_{i-1}/\al)}{ T} \text{ if }i\leq0.
\end{cases}
\]
For all $i\not\in C^P((\zeta^{\la,\pi}_t(i))_{t\geq0,i\in\zz},(0,t_1))$ and all $t\in\intervalleff{0}{T}$, we put
\[\zeta^{\la,\pi}_{t_1,t}(i)=\min(N^S_{\al t}(i),1).\]

We finally define
\[\Theta_ {t_1}^{\la,\pi}=\inf\enstq{t>t_1}{\forall i\in C^P((\zeta_{t}^{\la,\pi}(i))_{t\geq0,i\in\zz},(0,t_1)), \zeta_{t_1,t}^{\la,\pi}(i)=1}.\]
Then, for all $\delta>0$, as $\la\to0$ and $\pi\to\infty$ in the regime $\cR(\infty,z_0)$, there holds
\[\lim_{\la ,\pi} \proba{
|\Theta^{\la,\pi}_{t_1} - t_1|\geq
\delta} = 0.\]
\end{lem}
The process $(\zeta^{\la,\pi}_{t_1,t}(i))_{i\in\zz,t\geq0}$ is closely related to the process observed in {\bf Micro$(\infty,z_0)$} in Subsection \ref{application ffp} (on a suitable event).
\begin{proof}
We divide the proof in two steps. We first define a simplest process with an instantaneous propagation: if a match falls in a cluster, it destroys
instantaneously the entire connected component. The time needed for a microscopic cluster
to become again occupied is almost $t_1$. Secondly, we flank the killed cluster
$C^P((\zeta_{t}^{\la,\pi}(i))_{t\geq0,i\in\zz},(0,t_1))$ to estimate the time to become again occupied.

\md

\noindent{\bf Step 1.} Let $0<\tau_1<z_0$ be fixed. Put
$\vartheta_{t}^\la(i)=\min(N^{S}_{\al t}(i),1)$ and 
$\vartheta_{\tau_1,t}^\la(i)=\min(N^{S}_{\al(\tau_1+t)}(i)-N^{S}_{\al
\tau_1}(i),1)$ for all $t>0$ and $i\in \zz$. We define the time needed for the destroyed cluster to be fully regenerated
\[\Xi_{\tau_1}^\la=\inf \left\{ t>0 : \forall i \in
C(\vartheta_{\tau_1}^\la,0),\; \vartheta_{\tau_1,t}^\la(i)=1 \right\}.\]
Then for all $\delta>0$, 
\[\lim_{\la \to 0} \proba{ |\Xi_{\tau_1}^\la - \tau_1|\geq \delta} =
0.\]
Indeed, we write, for  $h >0$, 
\begin{multline*}
\proba{\Xi_{\tau_1}^\la \leq h}= \proba{N^S_{\al \tau_1}(0)=0} + \sum_{k\geq 1}
\sum_{j=0}^{k-1} \mathbb{P}\left[N^S_{\al \tau_1}(j-k)=N^S_{\al \tau_1}(j+1)=0,\right.\\
\left. \forall i\in\intervalleentier{j-k+1}{j}, N^S_{\al \tau_1}(i)>0,
N^S_{\al(\tau_1+h)}(i)>N^S_{\al \tau_1}(i) \right],
\end{multline*}
that is
\begin{align*}
\proba{\Xi_{\tau_1}^\la \leq h} &= \la^{\tau_1}+\sum_{k\geq1}\sum_{j=0}^{k-1}\la^{\tau_1}\times\la^{\tau_1}\times\left((1-\la^{\tau_1})(1-\la^h)\right)^{k}\\
&=\la^{\tau_1}+\la^{2\tau_1}\sum_{k\geq1}k(\left((1-\la^{\tau_1})(1-\la^h)\right)^k\\
&=\la^{\tau_1}+ \frac{\la^ {2 \tau_1}}{(1-(1-\la^ {\tau_1})(1-\la^ h))^2}(1-\la^ {\tau_1})(1-\la^ h)\\
&=\la^{\tau_1}+ \frac{\la^ {2 \tau_1}}{(\la^ {\tau_1}+\la^ h-\la^ {\tau_1+h})^2}(1-\la^ {\tau_1})(1-\la^ h).
\end{align*}
This quantity obviously tends to $1$ as $\la\to 0$ if $h>\tau_1$ and to $0$ if $h<\tau_1$.

\md

\noindent{\bf Step 2.} Let $z\in\intervalleoo{t_1}{z_0}$ and define  $\Omega^{P,z}_{\la,\pi}(0,t_1)$, recall Definition \ref{definition2 application}. Set
\begin{multline*}
\tOmega^{P,z}_{\la,\pi}(0,t_1)\coloneqq\Omega^{P,z}_{\la,\pi}(0,t_1)\cap\{\exists i_1\in\intervalleentier{0}{\lfloor\la^ {-z}\rfloor},N^S_{\al (t_1+\kappa_{\la,\pi}^{z})}(i_1)=0\}\\
\cap\{\exists i_2\in\intervalleentier{-\lfloor\la^ {-z}\rfloor}{0},N^S_{\al (t_1+\kappa_{\la,\pi}^{z})}(i_2)=0\}.
\end{multline*}
First, Lemma \ref{propagation lemma infty} together with Lemma \ref{speed infty}-1 show that $\proba{\tOmega^{P,z}_{\la,\pi}(0,t_1)}$ tends to $1$ when $\la\to0$ and $\pi\to\infty$ in the regime $\cR(\infty,z_0)$ (because $t_1+\kappa_{\la,\pi}^z<(z+t_1)/2<z$ for $(\la,\pi)$ sufficiently close to the regime $\cR(\infty,z_0)$).
Next, on $\tOmega^{P,z}_{\la,\pi}(0,t_1)$, there holds that
\[C(\vartheta_{t_1+\kappa_{\la,\pi}^z}^\la,0)\coloneqq\intervalleentier{C^-}{C^+}\subset\intervalleentier{-\lfloor\la^{-z}\rfloor}{\lfloor\la^{-z}\rfloor}.\]
Since $C^+$ and $C^-$ are vacant during $\intervalleff{\al t_1}{\al (t_1+\kappa_{\la,\pi}^z)}\subset\intervalleff{0}{\al (t_1+\kappa_{\la,\pi}^z)}$, there holds that, as seen in {\bf Micro$(\infty,z_0)$} in Subsection \ref{application ffp},
\[C^P((\zeta^{\la,\pi}_t(i))_{t\geq 0,i\in\zz},(0,t_1))\subset C(\vartheta_{t_1+\kappa_{\la,\pi}^z}^\la,0)\subset\intervalleentier{-\lfloor\la^{-z}\rfloor}{\lfloor\la^{-z}\rfloor}\]
and $\zeta^{\la,\pi}_{\al (t_1+\kappa_{\la,\pi}^{z})}(i)\leq 1$ for all $i\in\zz$. Besides, $C^P((\zeta^{\la,\pi}_t(i))_{t\geq 0,i\in\zz},(0,t_1))$ clearly contains $C(\vartheta_{t_1}^\la,0)$, see Figure \ref{fig height barrier}.

We trivially deduce that, conditionaly on $\tOmega^{P,z}_{\la,\pi}(0,t_1)$,
\[t_1+\Xi_{t_1}^\la\leq t_1+\Theta^{\la,\pi}_{t_1}\leq t_1+\kappa_{\la,\pi}^z+\Xi_{t_1+\kappa_{\la,\pi}^z}^\la.\]
Remark now that the function $\colon t\mapsto t+\Xi^\la_t$ is a.s. non decreasing and right-continuous. We thus deduce from Step 1 that
\[t_1+\Theta_{t_1}^{\la,\pi}\xrightarrow[\la\to0]{}2t_1\]
in probability, whence  for all $\delta>0$ and all $\e>0$, there holds  that $\proba{
|\Theta^{\la,\pi}_{t_1} - t_1|\geq
\delta} <\e$ for all $(\la,\pi)$ sufficiently close to the regime $\cR(\infty,z_0)$.\qedhere
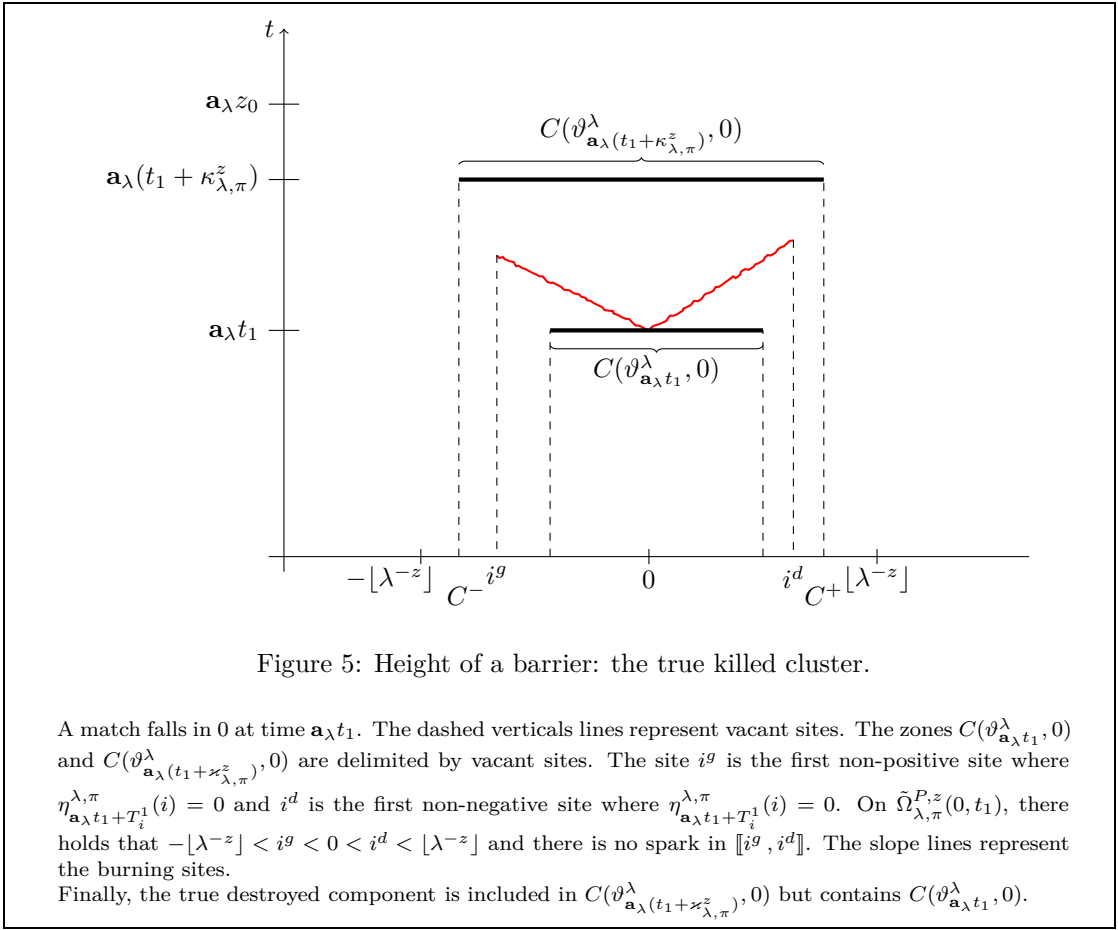
\begin{figure}[h!]
\fbox{
\begin{minipage}[c]{0.95\textwidth}
\centering
\begin{tikzpicture}
\draw (0,0) -- (10,0);
\draw[->] (.2,-.2)--(.2,7) node[left] {$t$};

\draw[thick,decorate,decoration={random steps,segment length=0.5mm,amplitude=0.2mm},color=red]
	(5,3)--(3,4);

\draw[thick,decorate,decoration={random steps,segment length=0.5mm,amplitude=0.2mm},color=red]
	(5,3)--(6.9,4.2);

\draw (5,-.1)--(5,.1) node at (5,-.3) {$0$};
\draw (2,-.1)--(2,.1) node at (1.6,-.3) {$-\lfloor\lambda^{-z}\rfloor$};
%\draw[dashed] (2,0)--(2,7);
\draw (8,-.1)--(8,.1) node at (8,-.3) {$\lfloor\lambda^{-z}\rfloor$};
%\draw[dashed] (8,0)--(8,7);

\draw (.4,6)--(0,6) node[left] {$\al z_0$};

\draw (.4,3)--(0,3) node[left] {$\al t_1$};
\draw[ultra thick] (3.7,3)--(6.5,3);
\draw[decorate,decoration={brace,raise=0.2cm,mirror}] (3.7,3)--(6.5,3) node[below=0.2cm,pos=0.5,sloped] {$C(\vartheta_{\al t_1}^\lambda,0)$};

\draw (.4,5)--(0,5) node[left] {$\al (t_1+\kappa_{\la,\pi}^z)$};
\draw[ultra thick] (2.5,5)--(7.3,5);
\draw[decorate,decoration={brace,raise=0.2cm}] (2.5,5)--(7.3,5) node[above=0.2cm,pos=0.5,sloped] {$C(\vartheta_{\al (t_1+\kappa_{\la,\pi}^z)}^\lambda,0)$};

\draw[dashed] (2.5,0)--(2.5,5) node at (2.6,-.5) {$C^-$};
\draw[dashed] (7.3,0)--(7.3,5) node at (7.3,-.5) {$C^+$};

\draw[dashed] (3.7,0)--(3.7,3);
\draw[dashed] (6.5,0)--(6.5,3);

\draw[dashed] (3,4)--(3,0) node[below] {$i^g$};
\draw[dashed] (6.9,4.2)--(6.9,0) node[below] {$i^d$};;
\end{tikzpicture} \caption{Height of a barrier: the true killed cluster.}\label{fig height barrier}
\vspace{.5cm}
\parbox{13.3cm}{
\footnotesize{
A match falls in $0$ at time $\al t_1$. The dashed verticals lines represent vacant sites. The zones $C(\vartheta_{\al t_1}^\la,0)$ and $C(\vartheta_{\al (t_1+\varkappa_{\la,\pi}^z)}^\la,0)$ are  delimited by vacant sites. The site $i^g$ is the first non-positive site where $\eta_{\al t_1+T_i^1}^{\la,\pi}(i)=0$ and $i^d$ is the first non-negative site where $\eta_{\al t_1+T_i^1}^{\la,\pi}(i)=0$. On $\tOmega_{\la,\pi}^{P,z}(0,t_1)$, there holds that $-\lfloor\la^{-z}\rfloor<i^g<0<i^d<\lfloor\la^{-z}\rfloor$ and there is no spark in $\intervalleentier{i^g}{i^d}$. The slope lines represent the burning sites.

Finally, the true destroyed component is included in $C(\vartheta_{\al(t_1+\varkappa_{\la,\pi}^z)}^\la,0)$ but contains $C(\vartheta_{\al t_1}^\la,0)$.
}}
\end{minipage}}
\end{figure}
\end{proof}

\subsection{Proof of Theorem \ref{theoinfty loc}}
Let us fix $z_0\in\intervalleff{0}{1}$, $x_0 \in \intervalleoo{-A}{A}$, $t_0>0$ 
and $\e>0$. The aim of this Section is to prove the
\begin{lem}\label{lemmaconvinfty}
For all $\delta>0$, there holds that
\begin{align}
\proba{\bdelta(D_{t_0}^{\la,\pi,A}(x_0),D_{t_0}^A(x_0))>\e} &<\delta,\label{convinfty1}\\
\proba{\bdelta_T(D^{\la,\pi,A}(x_0),D^A(x_0))>\e} &<\delta,\label{convinfty2}
\end{align}
for all $(\la,\pi)$ sufficiently close to the regime $\cR(\infty,z_0)$.
\end{lem}
Clearly, \eqref{convinfty1} and \eqref{convinfty2} will imply the result. Let us first show that \eqref{convinfty1} (which holds for an arbitrary value of $t_0\in\intervalleoo{0}{T}$) implies \eqref{convinfty2}. Indeed, we have by construction for any $t\in\intervalleff{0}{T}$, $\bdelta(D_{t}^{\la,\pi,A}(x_0),D_{t}^A(x_0))<4A$. Hence, by dominated convergence, \eqref{convinfty1} implies that $\E\left[\bdelta(D_{t}^{\la,\pi,A}(x_0),D_{t}^A(x_0)\right]<\delta$ for all $(\la,\pi)$ sufficiently close to the regime $\cR(\infty,z_0)$, whence again by dominated convergence, $\E\left[\bdelta_T(D^{\la,\pi,A}(x_0),D^A(x_0))\right]<\delta$.

\subsubsection{The coupling}\label{coupling infty}
We are going to construct a coupling between the $(\la,\pi,A)-$FFP
(on the time interval $\intervalleff{0}{\al T}$) and the LFFP$(\infty,z_0)$ (on $\intervalleff{0}{T}$):  we build the LFFP$(\infty,z_0)$ $(Y_t(x))_{t\in[0,T],x\in\intervalleff{-A}{A}}$ from a Poisson measure $\pi_M$ and we take for the matches for the discrete process the Poisson process
\[N^M_t(i)=\pi_M(\intervallefo{i/\nl}{(i+1)/\nl}\times\intervalleff{0}{t/\al})\]
for all $i\in I^\la_A$ and $t\in\intervalleff{0}{\al T}$.

We next introduce a family of i.i.d. Poisson processes $(N^{S}_t(i))_{t\geq0, i \in \zz}$ and $(N^{P}_t(i))_{t\geq0, i \in \zz}$ with respective parameter $1$ and $\pi$, independent of $\pi_M$.

The $(\la,\pi,A)-$FFP $(\eta^{\la,\pi}_{t}(i))_{t\geq0,i\in I^\la_A}$ is built from the seed processes $(N^{S}_t(i))_{t\geq0, i \in \zz}$, from the match processes $(N^{M}_t(i))_{t\geq0, i \in \zz}$ and from the propagation processes $(N^{P}_t(i))_{t\geq0, i \in \zz}$.

Observe that $(Y_t(x))_{t\in[0,T],x\in\intervalleff{-A}{A}}$ is
independent of $(N^{S}_t(i))_{t\in [0,\al T], i \in I^\la_A}$ and $(N^{P}_t(i))_{t\in [0,\al T], i \in I^\la_A}$.

When a match falls at some $x\in\intervalleff{-A}{A}$ at some time $t\in\intervalleff{0}{T}$ for the LFFP$(\infty,z_0)$, it will fall at $\lfloor\nl x\rfloor$ at time $\al t$ in the discrete process.
\subsubsection{A sweet event}
We call 
\[n:=\pi_M([0,T]\times\intervalleff{-A}{A})\]
and we consider the marks $(T_q,X_q)_{q=1,\dots,n}$ of $\pi_M$ ordered 
in such a way that $0<T_1<\dots<T_n<T$.
We introduce 
\[\cT_M=\{T_1,\dots,T_n\}\text{ and }\cB_M=\{X_1,\dots,X_n\}.\]  

We also introduce 
\[\cS_M=\enstq{2t}{t\in\cT_M, t<z_0},\]
which has
to be seen as the possible limit values of $t+\Theta_{t}^{\la,\pi} \simeq
t+t$, recall Lemma \ref{micro fire infty}.

For $\alpha>0$, we consider the event 
\[\Omega_M^0(\alpha)=\Big\{\min_{\substack{s\in \cT_M\cup \cS_M,\\ t\in\{0,z_0,t_0\}}}\abs{t-s}> 2\alpha, 
\min_{\substack{x,y \in \cB_M\cup\{x_0,-A,A\},\\ x\neq y}}\abs{x-y}> 2\alpha
\Big\},\]
which clearly satisfies $\lim_{\alpha \to 0} \proba{\Omega_M^0(\alpha)}=1$.
For any given $\alpha\in\intervalleoo{0}{1}$, on $\Omega_M^0(\alpha)$,
there holds that for all $x,y\in \cB_M\cup\{x_0\}$ with $x\neq y$, 
$(x)_\la^ \alpha\cap (y)_\la^ \alpha =\emptyset=(x)_\la\cap(y)_\la$.

We set 
\[z_\alpha=(z_0-\alpha)\vee(z_0/2).\]
For $q\in\{1,\dots,n\}$, using the seed processes family $(N^{S}_t(i))_{t\geq0, i \in \zz}$ and the propagation processes family $(N^{P}_t(i))_{t\geq0, i \in \zz}$, we build, recall Definition \ref{definition1 application}, $(\check{\zeta}_t^{\la,\pi,q}(i))_{t\geq0,i\in\zz}$ the propagation process ignited at $(X_q,T_q)$, $(i^{q,+}_t)_{t\geq0}$ and $(i^{q,-}_t)_{t\geq0}$ the corresponding right and  left fronts, and $(T^q_i)_{i\in\zz}$ the associated burning times. We also define $\Omega^{P,T,\alpha}_{\la,\pi}(X_q,T_q)$ and $\Omega_{\la,\pi}^{P,z_\alpha}(X_q,T_q)$, recall Definition \ref{definition2 application}. If $z_0\in\intervalleof{0}{1}$, we set
\[\Omega^{P,T}(\alpha,\la,\pi)=\bigcap_{q=1,\dots,n}(\Omega^{P,T,\alpha}_{\la,\pi}(X_q,T_q)\cap\Omega_{\la,\pi}^{P,z_\alpha}(X_q,T_q)).\]
If $z_0=0$, we simply set
\[\Omega^{P,T}(\alpha,\la,\pi)=\bigcap_{q=1,\dots,n}\Omega^{P,T,\alpha}_{\la,\pi}(X_q,T_q).\]
By Lemma \ref{propagation lemma infty} and since $\pi_M$ is independent of $(N^{S}_t(i))_{t\geq0, i \in \zz}$ and $(N^{P}_t(i))_{t\geq0, i \in \zz}$, we deduce that $\proba{\Omega^{P,T}(\alpha,\la,\pi)}$ tends to $1$ when $\la\to0$ and $\pi\to\infty$ in the regime $\cR(\infty,z_0)$.

Next we introduce the event $\Omega^S_1(\la,\pi)$ on which the following 
conditions hold:
for all $q\in\{1,\dots,n\}$, 
\begin{itemize}
	\item if $T_q<z_\alpha$, there are $-\lfloor\la^ {-z_\alpha}\rfloor<i_1^q<0<i_2^q<\lfloor\la^ {-z_\alpha}\rfloor$ with $N^S_{\al(T_q+\kappa_{\la,\pi}^{z_\alpha})}(\lfloor\nl X_q\rfloor+i_1^q)=N^S_{\al(T_q+\kappa_{\la,\pi}^{z_\alpha})}(\lfloor\nl X_q\rfloor+i_2^q)=0$;
	\item if $T_q>z_0+\alpha$, for all $i\in (X_q)_ \la^\alpha$, $N^S_{\al T_q}(i)>0$.
\end{itemize}
Since $\kappa_{\la,\pi}^{z_\alpha}$ can be made arbitrarily small in the regime $\cR(\infty,z_0)$, Lemma \ref{speed infty} then show that $\proba{\Omega^S_1 (\la,\pi) }$ tends to $1$ when $\la\to0$ and $\to\infty$ in the regime $\cR(\infty,z_0)$.

We also consider the event $\Omega_2^S(\la)$ on which the following conditions holds
\begin{itemize}
	\item if $t_0<1$, there are $\lfloor\nl x_0\rfloor-\ml<i_1^0<\lfloor\nl x_0\rfloor<i_2^0<\lfloor\nl x_0\rfloor+\ml$ with $N^{S}_{\al t_0}(i_1)=N^{S}_{\al t_0}(i_2)=0$;
	\item for all 
$i\in \intervalleentier{-A_\la}{A_\la}$, $N^{S}_{\al}(i)>0$.
\end{itemize}
Lemma \ref{speed infty} together with space/time stationarity implies that $\lim_{\la\to0}\proba{\Omega_2^S(\la)}=1$.

We also need $\Omega^{S,P}_3(\gamma,\la,\pi)$, defined for $\gamma>0$ as follows:
for all $q=1,\dots,n$ with $T_q<z_0$, there holds that $|\Theta^{\la,\pi,q}_{T_q} - T_q|<\gamma$. Here $\Theta^{\la,\pi,q}_{T_q}$ is defined as in 
Lemma \ref{micro fire infty} with the seed processes family
$(N^{S,q}_t(i))_{t\geq 0, i\in\zz}= (N^{S}_t(i+\lfloor\nl X_q\rfloor))_{t\geq 0, i\in\zz}$ and the propagation processes family $(N^{P,q}_t(i))_{t\geq 0, i\in\zz}=(N^P_t(i+\lfloor\nl X_q\rfloor))_{t\geq 0, i\in\zz}$.
Lemma \ref{micro fire infty} directly implies that for any $\gamma>0$, $\proba{\Omega^{S,P}_3(\gamma,\la,\pi)}$ tends to $1$ when $\la\to 0$ and $\pi\to\infty$ in the regime $\cR(\infty,z_0)$.

We finally introduce the event 
\[\Omega(\alpha,\gamma,\la,\pi)=\Omega_M^0(\alpha)\cap\Omega^{P,T}(\alpha,\la,\pi)\cap 
\Omega^S_1(\la,\pi)\cap\Omega^S_2(\la)\cap\Omega^{S,P}_3(\gamma,\la,\pi).\]
We have shown that for any $\delta>0$, there exists 
$\alpha\in\intervalleoo{0}{1}$ such that for any $\gamma>0$, there holds
$\proba{\Omega(\alpha,\gamma,\la,\pi)}>1-\delta$ for all $(\la,\pi)$ sufficiently close to the regime $\cR(\infty,z_0)$.

\subsubsection{Heart of the proof}
The next Lemma is the key of the proof: it guarantees that each fire have a local effect. It will be repeteadly used in Lemmas \ref{lemlocfeu} and \ref{lemloctrou}.
\begin{lem}\label{lemloc}
On $\Omega(\alpha,\gamma,\la,\pi)$, the match falling on $\lfloor\nl X_q\rfloor$ at time $\al T_q$, for some $q\in\{1,\dots,n\}$, does not affect the zone outside $(X_q)_\la^{\alpha}$ during $\intervalleff{\al T_q}{\al T}$.

Consequently, on $\Omega(\alpha,\gamma,\la,\pi)$, for all $i\in I_A^\la\setminus\cup_{q=1,\dots,n}(X_q)_\la^{\alpha}$ and all $t\in\intervalleff{0}{T}$, there holds that
\[\eta^{\la,\pi,A}_{\al t}(i)=\min(N^{S}_{\al t}(i),1).\]
\end{lem}
\begin{proof}
As be seen in {\bf Macro$(\infty,z_0)$} in Subsection \ref{application ffp}, on $\Omega^{P,T,\alpha}_{\la,\pi}(X_q,T_q)\subset\Omega(\alpha,\gamma,\la,\pi)$, there holds that
\[X_q-\frac{\mla}{\nl}\leq\frac{\lfloor\nl X_q\rfloor+i^{q,-}_{\al T}}{\nl}\leq X_q\leq\frac{\lfloor\nl X_q\rfloor+1+i^{q,+}_{\al T}}{\nl}
\leq X_q+\frac{\mla}{\nl}\]
with $\mla/\nl\leq\alpha$. Hence, each fire has only a local effect and does not affect the zone outside $(X_q)_\la^{\alpha}$.
\end{proof}
We now turn to fires of the second kind.
\begin{lem}\label{lemlocfeu}
Let $q\in\{1,\dots,n\}$ such that $T_q>z_0+\alpha$. On $\Omega(\alpha,\gamma,\la,\pi)$, for all $t\in\intervalleff{\al T_q}{\al T}$, there holds that
\[\eta^{\la,\pi,A}_{\al t}(\lfloor\nl X_q\rfloor+i^{q,+}_{\al (t-T_q)})=2=\eta^{\la,\pi,A}_{\al t}(\lfloor\nl X_q\rfloor+i^{q,-}_{\al (t-T_q)}).\]
\end{lem}
\begin{proof}
At time $\al T_q-$, at least one seed has fallen on each site of $(X_q)^\alpha_\la$, thanks to $\Omega_1^S(\la,\pi)$. Thus, the zone $(X_q)_\la^\alpha$ is completely filled at time $\al T_q-$, thanks to Lemma \ref{lemloc} (no fire can affect this zone during $\intervallefo{0}{\al T_q}$). The conclusion is then straightforward, since on $\Omega_{\la,\pi}^{P,T}(X_q,T_q)$ there holds that $i^{q,+}\leq \mla/\nl$ and $i^{q,-}\leq \mla/\nl$ (as seen in {\bf Macro$(\infty,z_0)$} in Subsection \ref{application ffp}) and since no match falling outside $(X_q)_\la^\alpha$ can affect this zone.
\end{proof}
Finally, we treat the case of the fires of the first kind.
\begin{lem}\label{lemloctrou}
Let $q\in\{1,\dots,n\}$ such that $T_q<z_0-\alpha$. On $\Omega(\alpha,\gamma,\la,\pi)$, there holds that
\[(\eta^{\la,\pi,A}_{\al t}(i))_{t\in\intervalleff{0}{T},i\in(X_q)_\la^\alpha}=(\zeta^{\la,\pi,q}_{T_q, t}(i-\lfloor\nl X_q\rfloor))_{t\in\intervalleff{0}{T},i\in(X_q)_\la^\alpha},\]
where the last process is defined as in Lemma \ref{micro fire infty}, using the seed processes family $(N^{S,q}_t(i))_{t\geq 0, i\in\zz}= (N^{S}_t(i+\lfloor\nl X_q\rfloor))_{t\geq 0, i\in\zz}$ and the propagation processes family $(N^{P,q}_t(i))_{t\geq 0, i\in\zz}=(N^P_t(i+\lfloor\nl X_q\rfloor))_{t\geq 0, i\in\zz}$.

Consequently, on $\Omega(\alpha,\gamma,\la,\pi)$, for some $\gamma\in\intervalleoo{0}{\alpha}$,
\begin{enumerate}[label=(\alph*)]
	\item if $t\in\intervalleff{T_q+\alpha}{2T_q-\alpha}$, then there exists $i\in (X_q)_\la^\alpha$ such that $\eta^{\la,\pi}_{\al t}(i)=0$,
	\item if $t\geq(2T_q+\alpha)\vee 1$, then $\eta^{\la,\pi}_{\al t}(i)=1$ for all $i\in(X_q)_\la^\alpha$.
\end{enumerate}
\end{lem}
\begin{proof}
First observe that the process $(\eta^{\la,\pi,A}_{\al t}(\lfloor\nl X_q\rfloor+i))_{t\in\intervalleff{0}{T},i\in\intervalleentier{-\mla}{\mla}}$ and the process $(\zeta^{\la,\pi,q}_{T_q, t}(i))_{t\in\intervalleff{0}{T},i\in\intervalleentier{-\mla}{\mla}}$ evolve according to the same seed processes family and to the same propagation processes family.

Lemma \ref{lemloc} implies that, for all $i\in(X_q)_\la^\alpha$ and all $t\in\intervallefo{0}{ T_q}$,
\[\eta^{\la,\pi}_{\al t}(i)=\min(N^S_{\al t}(i),1),\]
because no match falls in $(X_q)_\la^\alpha$ during $\intervallefo{0}{\al T_q}$. This in particular implies that, for all $i\in(X_q)_\la^\alpha$ and all $t\in\intervallefo{0}{ T_q}$,
\[\eta^{\la,\pi}_{\al t}(i)=\zeta^{\la,\pi,q}_{T_q,t}(i-\lfloor\nl X_q\rfloor).\]

On $\Omega_{\la,\pi}^{P,z_\alpha}(X_q,T_q)\cap\Omega_1^S(\la,\pi)$, as seen in {\bf Micro$(\infty,z_0)$} in Subsection \ref{application ffp}, since the two processes are building using the same seed processes family and the same propagation processes family, there also holds true that for all $i\in(X_q)_\la^\alpha$ and all $t\in\intervalleff{T_q}{T_q+\kappa_{\la,\pi}^{z_\alpha}}$,
\[\eta^{\la,\pi}_{\al t}(i)=\zeta^{\la,\pi,q}_{T_q,t}(i-\lfloor\nl X_q\rfloor).\]

Finally, since there is no more burning tree in $(X_q)_\la^\alpha$ at time $\al(T_q+\kappa_{\la,\pi}^{z_\alpha})$ and since seeds fall according to the same processes, we deduce that, thanks again to Lemma \ref{lemloc}, the two processes remain equal during $\intervalleof{T_q+\kappa_{\la,\pi}^{z_\alpha}}{T}$.

All this implies that
\begin{equation}\label{plouf}
(\eta^{\la,\pi,A}_{\al t}(i))_{t\in\intervalleff{0}{T},i\in(X_q)_\la^\alpha}=(\zeta^{\la,\pi,q}_{T_q, t}(i-\lfloor\nl X_q\rfloor))_{t\in\intervalleff{0}{T},i\in(X_q)_\la^\alpha}.
\end{equation}
Consider now the zone destroyed by the match falling on $\lfloor\nl X_q\rfloor$ at time $\al T_q$
\[C^P\coloneqq C^P((\eta^{\la,\pi}_{t}(i))_{t\geq0,i\in\zz},(X_q,T_q)).\]
As seen in {\bf Micro$(\infty,z_0)$} in Subsection \ref{application ffp}, $C^P\subset\intervalleentier{-\lfloor\la^{-z_\alpha}\rfloor}{\lfloor\la^{-z_\alpha}\rfloor}$ because there are $i_1\in\intervalleentier{-\lfloor\la^{-z_\alpha}\rfloor}{0}$ and $i_2\in\intervalleentier{0}{\lfloor\la^{-z_\alpha}\rfloor}$ which are vacant until $\al(T_q+\kappa_{\la,\pi}^{z_\alpha})$, thanks to $\Omega_1^S(\la,\pi)$.

From \eqref{plouf} and since no match falling outside $(X_q)_\la^\alpha$ can affect this zone, it follows that
\[\Theta^{\la,\pi,q}_{T_q}=\inf\enstq{t>T_q}{\forall i\in C^P((\eta^{\la,\pi}_{t}(i))_{t\geq0,i\in\zz},(X_q,T_q)), \eta_{\al t}^{\la,\pi}(i)=1}.\]
Hence, the zone $C^P$ is not completely occupied during $\intervalleoo{\al (T_q+\kappa_{\la,\pi}^z)}{\al(T_q+\Theta^{\la,\pi,q}_{T_q})}$ but is completely filled at time $\al(T_k+\Theta^{\la,\pi,q}_{T_q})$.

Using $\Omega_3^{S,P}(\gamma,\la,\pi)\cap\Omega^0_M(\alpha)$ and since $\gamma\in\intervalleoo{0}{\alpha}$, we deduce that, 
\[T_q+\alpha<2T_q-\alpha\leq 2T_q-\gamma\leq T_q+\Theta^{\la,\pi,q}_{T_q}\leq 2T_q+\gamma\leq 2T_q+\alpha.\]
We now conclude.
\begin{enumerate}[label=(\alph*)]
	\item If $t\in\intervalleff{T_q+\alpha}{2T_q-\alpha}$, then the zone $C^P$ is not completely occupied at time $t$. Hence, there exists $i\in C^P\subset (X_q)_\la^\alpha$ such that $\eta^{\la,\pi}_{\al t}(i)=0$.
	\item If $t\geq (2T_q+\alpha)\vee 1$, then $C^P$ is completely filled at time $t$ because $t\geq T_q+\alpha$. 
	
Consider now $i\in(X_q)_\la^{\alpha}\setminus C^P$. Then $i$ has not been killed by the fire starting at $\lfloor\nl X_q\rfloor$. Thus $i$ cannot have been killed during $\intervalleff{0}{\al t}\supset\intervalleff{0}{\al}$, thanks to Lemma \ref{lemloc}. We conclude using that $t\geq1$, so that on $\Omega_1^S(\la)$, $i$ is occupied at time $\al t$.\qedhere
\end{enumerate}
\end{proof}

\subsubsection{Conclusion}
First, the case $t_0<1$ is simple.
\begin{lem}\label{lemma0-}
For $t_0<1$, on $\Omega(\alpha,\gamma,\la,\pi)$, there holds that
\[\bdelta(D_{t_0}^{\la,\pi,A}(x_0),D_{t_0}^A(x_0))<\frac{2\ml}{\nl}.\]
\end{lem}
\begin{proof}
Thanks to $\Omega^S_2(\la)$, there are $i_1^0\in\intervalleentier{\lfloor \nl x_0\rfloor-\ml}{\lfloor\nl x_0\rfloor}$ and $i_2^0\in\intervalleentier{\lfloor\nl x_0\rfloor}{\lfloor \nl x_0\rfloor+\ml}$ such that $\eta^{\la,\pi,A}_{\al t_0}(i_1)=\eta^{\la,\pi,A}_{\al t_0}(i_2)=0$. Thus, $C(\eta^{\la,\pi}_{\al t_0},\lfloor\nl x_0\rfloor)\subset\intervalleentier{\lfloor\nl x_0\rfloor-\ml}{\lfloor\nl x_0\rfloor+\ml}$ whence $D^{\la,\pi,A}_{t_0}(x_0)\subset\intervalleff{x_0-\ml/\nl}{x_0+\ml/\nl}$. Since $D_{t_0}^A(x_0)=\{x_0\}$, we deduce that
\[\bdelta(D_{t_0}^{\la,\pi,A}(x_0),D_{t_0}^A(x_0))\leq \frac{2\ml}{\nl}.\qedhere\]
\end{proof}

We now turn to the case $t_0\geq1$.
\begin{lem}\label{lemma0+}
For $t_0\geq1$, on $\Omega(\alpha,\gamma,\la,\pi)$ for some $0<\gamma<\alpha$ and for all $(\la,\pi)$ sufficiently close to the regime $\cR(\infty,z_0)$ in such a way that $\kappa_{\la,\pi}^{z_\alpha}\leq \alpha$ and $\lfloor z^{-\alpha}\rfloor\leq\mla$, there holds that
\[\bdelta(D_{t_0}^{\la,\pi,A}(x_0),D_{t_0}^A(x_0))<\frac{2\mla}{\nl}.\]
\end{lem}
\begin{proof}
Clearly, since $t_0\geq 1$, $D_{t_0}^A(x_0)=\intervalleff{a}{b}$ for some $a,b\in\cB_M\cup\{-A,A\}$. Assume $-A<a<b<A$, the other cases being treated similarly. In the limit process, we then have $Y_{t_0}(a)>0$, $Y_{t_0}(b)>0$ and $Y_{t_0}(x)=0$ for all $x\in\intervalleoo{a}{b}$. We will prove separately that
\begin{enumerate}
	\item there are $i\in(a)_\la^ {\alpha}$ and $j\in(b)_\la^ {\alpha}$ such that $\eta^{\la,\pi,A}_{\al t_0}(i)=0$ or $2$ and $\eta^{\la,\pi,A}_{\al t_0}(j)=0$ or $2$;
	\item for all $x\in\cB_M\cap\intervalleoo{a}{b}$, for all $i\in(x)_\la^ {\alpha}$, $\eta^{\la,\pi,A}_{\al t_0}(i)=1$;
	\item for all $i\in\intervalleentier{\lfloor\nl a\rfloor+\mla+1}{\lfloor\nl b\rfloor-\mla-1}\setminus\cup_{x\in\cB_M\cap\intervalleoo{a}{b}}(x)_\la^ {\alpha}$, we have $\eta^{\la,\pi,A}_{\al t_0}(i)=1$.
\end{enumerate}
Points 1., 2. and 3. imply that,
\[\intervalleentier{\lfloor\nl a\rfloor+\mla+1}{\lfloor\nl b\rfloor-\mla-1}\subset C(\eta^{\la,\pi}_{\al t_0},\lfloor\nl x_0\rfloor)\subset\intervalleentier{\lfloor\nl a\rfloor-\mla-1}{\lfloor\nl b\rfloor+\mla+1}\]
and thus $\intervalleff{a+\mla/\nl}{b-\mla/\nl}\subset D^{\la,\pi,A}_{t_0}(x_0)\subset\intervalleff{a-\mla/\nl}{b+\mla/\nl}$, whence, 
\[\bdelta(D_{t_0}^A(x_0),D_{t_0}^{\la,\pi,A}(x_0))\leq2\mla/\nl.\]
%
%
%\vspace{.3cm}

\noindent{\bf We prove 1.} Let $k\in\{1,\dots,n\}$ such that $a=X_k$. There are two cases.

{\it Case 1.} If $Y_{t_0}(X_k)=1$ in the limit process, then $t_0\geq T_k\geq z_0$ whence $t_0\geq T_k\geq z_0+2\alpha$ due to $\Omega_M^0(\alpha)$. We then use Lemma \ref{lemlocfeu} and conclude that there is a burning tree in $(a)^\alpha_\la$ at time $\al t_0$.

{\it Case 2.} If $Y_{t_0}(a)\in\intervalleoo{0}{1}$ in the limit process, then $T_k\leq z_0\leq1\leq t_0\leq 2T_k$ whence $T_k+4\alpha\leq z_0+2\alpha\leq t_0+2\alpha\leq 2T_k$, due to $\Omega_M^0(\alpha)$. We conclude using Lemma \ref{lemloctrou}-(a) that there is a vacant site in $(a)^\alpha_\la$ at time $\al t_0$.

Similar considerations hold for $b$.

\md

\noindent{\bf We prove 2.} Let $x\in\cB_M\cap\intervalleoo{a}{b}$ and let $k\in\{1,\dots,n\}$ such that $x=X_k$.

{\it Case 1.} If $T_k> t_0$, then no fire has fallen in $(X_k)_\la^\alpha$ during $\intervalleff{0}{\al t_0}$. Using $\Omega_1^S(\la,\pi)$ and Lemma \ref{lemloc}, we conclude that $(X_k)_\la^ \alpha$ is completely occupied at time $\al t_0$ (because no fire can affect this zone).

{\it Case 2.} If $T_k\leq t_0$, since in the limit process $Y_{t_0}(X_k)=0$, necessarily $T_k\leq z_0\leq t_0$ and $2T_k\leq t_0$ whence $T_k\leq z_0-2\alpha$ and $2T_k\leq t_0-2\alpha$ due to $\Omega_M(\alpha)$. Lemma \ref{lemloctrou}-(b) concludes this case since $t_0\geq(2T_q+\alpha)\vee1$.

\md

\noindent{\bf We prove 3.} Let $i\in\intervalleentier{\lfloor\nl a\rfloor+\mla+1}{\lfloor\nl b\rfloor-\mla-1}\setminus\cup_{j=1,\dots,n}(X_j)_\la^ {\alpha}$, using Lemma \ref{lemloc} and  $\Omega_2^S(\la)$, we immediately conclude that $i$ is occupied at time $\al t_0$.
\end{proof}

We now can conclude.
\begin{proof}[Proof of Lemma \ref{lemmaconvinfty}]
Let $\delta>0$ be fixed. We first consider $\alpha_0\in\intervalleoo{0}{\e/2}$, $\gamma_0\in\intervalleoo{0}{\alpha_0}$, $\la_0\in\intervalleof{0}{1}$, $\epsilon_0>0$ and $K_0\geq1$ such that for all $\la\in\intervalleoo{0}{\la_0}$, all $\pi\geq1$ in such a way that $\frac{\nl}{\al\pi}\geq K_0$ and $\abs{\frac{\log(\pi)}{\log(1/\la)}-z_0}<\epsilon_0$, we have 
\[\proba{\Omega(\alpha_0,\gamma_0,\la,\pi)}>1-\delta.\]
Then we consider $\la_1\in\intervalleoo{0}{\la_0}$, $K_1>K_0$ and $\epsilon_1\in\intervalleoo{0}{\epsilon_0}$ such that for all $\la\in\intervalleoo{0}{\la_1}$ and all $\pi\geq1$ in such a way that $\frac{\nl}{\al\pi}\geq K_1$ and $\abs{\frac{\log(\pi)}{\log(1/\la)}-z_0}<\epsilon_1$, we have
\begin{itemize}
	\item $2\ml/\nl<\e$,
	\item $\kappa_{\la,\pi}^ {z_\alpha}<\alpha$,
	\item $2\la^ {-z_\alpha}/\nl<2\mla/\nl<\e$.
\end{itemize}

For all $\la\in\intervalleoo{0}{\la_1}$, all $\pi\geq1$ in such a way that $\frac{\nl}{\al\pi}>K_1$ and $\abs{\frac{\log(\pi)}{\log(1/\la)}-z_0}<\epsilon_1$, Lemma \ref{lemma0-} implies that, if $t_0<1$,
\[\proba{\bdelta(D_{t_0}^A(x_0),D_{t_0}^{\la,\pi,A}(x_0))>\e}\leq\proba{\bdelta(D_{t_0}^A(x_0),D_{t_0}^{\la,\pi,A}(x_0))>\frac{2\ml}{\nl}}\leq\proba{\Omega(\alpha_0,\gamma_0,\la,\pi)^c}<\delta\]
while, if $t_0\geq 1$, Lemma \ref{lemma0+} implies that, (since $\alpha\geq\gamma$ and $\alpha\geq\kappa_{\la,\pi}^{z_\alpha}$)
\[\proba{\bdelta(D_{t_0}^A(x_0),D_{t_0}^{\la,\pi,A}(x_0))>\e}\leq\proba{\bdelta(D_{t_0}^A(x_0),D_{t_0}^{\la,\pi,A}(x_0))>\frac{2\bm_\la^{\alpha_0}}{\nl}}\leq\proba{\Omega(\alpha_0,\gamma_0,\la,\pi)^c}<\delta.\]
This concludes the proof.
\end{proof}

\section{Convergence in the regime $\cR(p)$, for some $p>0$}\label{convergence in the regime p}
The aim of this section is to prove Theorem \ref{converge restriction} for $p>0$ and this will conclude the proof of Theorem \ref{converge} for $p>0$.

In the whole section, we fix the parameters $A>0$, $T>2$ and $p>0$. We omit the subscript/superscript $A$ in the whole proof.

We recall that $\al=\log(1/\la)$, $\nl=\lfloor 1/(\la\al)\rfloor$, $\ml=\lfloor 1/(\la\ba_\la^2)\rfloor$, $\e_\la=1/\ba_\la^ {3}$. We set as usual $A_\la=\lfloor\nl A\rfloor$ and $I_A^\la=\intervalleentier{-A_\la}{A_\la}$. For $i\in\zz$, we set $i_\la=\intervallefo{i/\nl}{(i+1)/\nl}$. For $\intervalleff{a}{b}$ an interval of $\intervalleff{-A}{A}$ and $\la\in\intervalleoo{0}{1}$, we introduce, assuming that $-A<a<b<A$,
\begin{align*}
\intervalleff{a}{b}_\la &= \intervalleentier{\lfloor\nl a+\ml\rfloor+1}{\lfloor\nl b-\ml\rfloor-1}\subset\zz,\\
\intervalleff{-A}{b}_\la &= \intervalleentier{-A_\la}{\lfloor\nl b-\ml\rfloor-1}\subset\zz,\\
\intervalleff{a}{A}_\la &= \intervalleentier{\lfloor\nl a+\ml\rfloor+1}{A_\la}\subset\zz.
\end{align*}
For $\la\in\intervalleoo{0}{1}$ and $\pi\geq1$, we recall that
\[\kappa_{\la,\pi}^0 = \frac{\ml}{\al\pi}+\e_\la\]
and introduce
\begin{align}
\klp &=\left\lfloor\al\pi\left(\e_\la+\mathfrak{v}_{\la,\pi}\right)\right\rfloor,\label{klp}\\%en vrai \frac{T}{p}\abs{\frac{\nl}{\al\pi}-p}
\vlp &=\kappa_{\la,\pi}^0+\mathfrak{v}_{\la,\pi},\label{vlp}\\%en vrai \frac{T}{p}\abs{\frac{\nl}{\al\pi}-p}
\elp &=\e_\la+\mathfrak{v}_{\la,\pi},\label{elp}%en vrai 2A\abs{\frac{\nl}{\al\pi}-p}
\end{align}
where $\mathfrak{v}_{\la,\pi} = \left(\frac{T}{p}\vee2A\right)\abs{\frac{\nl}{\al\pi}-p}$. Observe that $\klp/\nl$, $\vlp$ and $\elp$ tend to $0$ as $\la\to0$ and $\pi\to\infty$ in the regime $\cR(p)$.

For $x\in\intervalleoo{-A}{A}$, $\la\in\intervalleoo{0}{1}$ and $\pi\geq1$, we introduce
\begin{align}
(x)_\la &=\intervalleentier{\lfloor\nl x\rfloor-\ml}{\lfloor\nl x\rfloor+\ml}\subset\zz,\\
\langle x\rangle_{\la,\pi} &=\intervalleentier{\lfloor\nl x\rfloor-\klp}{\lfloor\nl x\rfloor+\klp}\subset\zz,\\
[x]_{\la,\pi} &=\intervalleentier{\lfloor\nl x\rfloor-\ml-2\klp}{\lfloor\nl x\rfloor+\ml+2\klp}\subset\zz.
\end{align}

\subsection{Occupation of vacant zone}
We start with some easy estimates.
\begin{lem}\label{speed}
Consider a family of i.i.d. Poisson processes $(N^S_t(i))_{t\geq0,i\in\zz}$. Let $a<b$.
\begin{enumerate}
        \item For $t<1$, $\lim_{\la\to 0} \proba{\forall i \in
\intervalleentier{\lfloor a\ml \rfloor}{\lfloor b\ml \rfloor},
N^S_{\al t}(i)>0}=0$;
        \item For $t\geq1$, $\lim_{\la\to 0} \proba{\forall i \in
\intervalleentier{\lfloor a\ml \rfloor}{\lfloor b\ml \rfloor},
N^S_{\al t}(i)>0}=1$;
        \item For $t<1$, $\lim_{\la\to 0} \proba{\forall i \in
\intervalleentier{\lfloor a\nl \rfloor}{\lfloor b\nl \rfloor}, N^S_{\al
t}(i)>0}=0$;
        \item For $t\geq1$, $\lim_{\la\to 0} \proba{\forall i \in
\intervalleentier{\lfloor a\nl \rfloor}{\lfloor b\nl \rfloor}, N^S_{\al
t}(i)>0}=1$;
        \item For $t>0$, $\lim_{\la\to 0} \proba{\exists i \in
\intervalleentier{\lfloor a\nl \rfloor}{\lfloor b\nl \rfloor}, N^S_{\al
t}(i)>0}=1$;
		\item For $t>0$ and $\delta>0$, $\lim_{\la\to0}\proba{\forall i\in\intervalleentier{-\lfloor\la^{-(t+\delta)}\rfloor}{\lfloor\la^{-(t+\delta)}\rfloor}, N^S_{\al t}(i)>0}=0$;
		\item For $t>0$ and $\delta>0$, $\lim_{\la\to0}\proba{\forall i\in\intervalleentier{-\lfloor\la^{-(t-\delta)}\rfloor}{\lfloor\la^{-(t-\delta)}\rfloor}, N^S_{\al t}(i)>0}=1$;
        \item For $t<1$, $\lim_{\substack{\la\to 0\\ \pi\to\infty}} \proba{\forall i \in
\intervalleentier{\lfloor a\klp \rfloor}{\lfloor b\klp \rfloor},
N^S_{\al t}(i)>0}=0$ (when $\la\to0$ and $\pi\to\infty$ in the regime $\cR(p)$);
        \item For $t\geq1$, $\lim_{\substack{\la\to 0\\ \pi\to\infty}} \proba{\forall i \in
\intervalleentier{\lfloor a\klp \rfloor}{\lfloor b\klp \rfloor},
N^S_{\al t}(i)>0}=1$ (when $\la\to0$ and $\pi\to\infty$ in the regime $\cR(p)$).
\end{enumerate}
\end{lem}

\begin{proof}
This lemma is closely related to Lemma \ref{speed infty}. For $r_\la\xrightarrow[\la\to0]{}\infty$, we have
\[\proba{\forall i \in
\intervalleentier{-\lfloor ar_\la\rfloor}{\lfloor br_\la\rfloor}, N^S_{\al
t}(i)>0}\simeq(1-e^ {-\al t})^{(b-a)r_\la}\simeq e^{-(b-a)r_\la \la^t}.\]
Observe now that
\[\ml\la^t\simeq\frac{\la^{t-1}}{\ba_\la^2}\xrightarrow[\la\to0]{}\begin{cases}
\infty &\text{if } t<1,\\
0 &\text{if } t\geq1,
\end{cases}\]
from which points 1 and 2 follow, that
\[\nl\la^t\simeq\frac{\la^{t-1}}{\ba_\la}\xrightarrow[\la\to0]{}\begin{cases}
\infty &\text{if } t<1,\\
0 &\text{if } t\geq1,
\end{cases}\]
which implies points 3 and 4.  For the point 5, it suffices to note that, for any $i\in\zz$,
\[\proba{N_{\al t}^S(i)=0}=e^{-\al t}.\]
Hence
\[\proba{\exists i \in
\intervalleentier{\lfloor a\nl \rfloor}{\lfloor b\nl \rfloor}, N^S_{\al
t}(i)>0}\simeq 1-e^{-\al \nl t(b-a)}\xrightarrow[\la\to0]{}1.\]
For $t>0$ and $\delta>0$, we have
\[\proba{\forall i\in\intervalleentier{-\lfloor\la^{-(t+\delta)}\rfloor}{\lfloor\la^{-(t+\delta)}\rfloor}, N^S_{\al t}(i)>0}\simeq e^{-2\la^{-\delta}}\xrightarrow[\la\to0]{}0,\]
which prove point 6, whence
\[\proba{\forall i\in\intervalleentier{-\lfloor\la^{-(t-\delta)}\rfloor}{\lfloor\la^{-(t-\delta)}\rfloor}, N^S_{\al t}(i)>0}\simeq e^{-2\la^{\delta}}\xrightarrow[\la\to0]{}1\]
which is Point 7.

For the two last statement, as $\la\to0$ and $\pi\to\infty$ in the regime $\cR(p)$, we have, observing that $\mathfrak{v}_{\la,\pi}\to0$,
\[\klp\la^t\simeq\al\pi\la^t(\e_\la+\mathfrak{v}_{\la,\pi})\simeq \frac{\nl\la^t}{p}\left(\e_\la+\mathfrak{v}_{\la,\pi}\right)\simeq\frac{\la^{t-1}}{\al p}\left(1/\al^3+\mathfrak{v}_{\la,\pi}\right) \xrightarrow[\la,\pi]{}\begin{cases}
\infty &\text{if } t<1,\\
0 &\text{if } t\geq1.\qedhere
\end{cases}\]
\end{proof}

\subsection{Height of the barrier}\label{heightp}
We describe here the time needed for a destroyed microscopic cluster to be
regenerated. Roughly, we assume that the zone $\intervalleentier{-\ml}{\ml}$ around $0$ has been made vacant at some time $\al t_0$. Then we consider the situation where a match falls on $0$ at some time $\al t_1\in\intervalleoo{\al t_0}{\al(t_0+1)}$ and we compute the delay needed for the destroyed cluster to be fully regenerated. We have to distinguish two cases.
\begin{enumerate}[label=\alph*)]
	\item We first consider the case where a match falls on $0$ at time $\al t_1\in\intervalleoo{0}{\al}$. This case is closely related to Lemma \ref{height infty}.
	\item We then consider the case where a fire propagates through $\intervalleentier{-\ml}{\ml}$ at time $\al t_0$ and a match falls on $0$ at time $\al t_1\in\intervalleoo{\al t_0}{\al (t_0+1)}$. This case is a little bit different but is proved in the same way as the previous case.
\end{enumerate}

\begin{lem}\label{micro fire p}
Consider two Poisson processes $(N_t^S(i))_{t\geq0,i\in\zz}$ and
$(N_t^P(i))_{t\geq0,i\in\zz}$ with respective rates $1$ and $\pi$, all this
processes being independent. Consider also $\cM\coloneqq(i_0;t_0,t_1)\in\zz\times(\rr_+)^2$ with $|i_0|\in\intervalleentier{\ml}{\ml+2\klp}$, $t_0\in\{0\}\cup\intervalleoo{1}{\infty}$ and $t_1\in\intervalleoo{t_0}{t_0+1}$. For $i\in\zz$ and $t\geq0$, we consider the process
\begin{align*}
\zeta^{\la,\pi,\cM}_t(i) =& \left(1+\indiq{t\geq\al
(t_0-\vlp),i=i_0}\right)\times\indiq{t_0>1}\\
&+\indiq{t\geq\al
t_1,i=0,\zeta^{\la,\pi,\cM}_{\al t_1-}(0)=1}+\int_0^t\indiq{\zeta^{\la,\pi,\cM}_{s-}(i)=0}\diff
N_s^S(i)\\
&+\int_0^t \indiq{\zeta^{\la,\pi,\cM}_{s-}(i+1)=2,\zeta^{\la,\pi,\cM}_{s-}(i)=1}\diff
N_s^P(i+1)\\
&+\int_0^t
\indiq{\zeta^{\la,\pi,\cM}_{s-}(i-1)=2,\zeta^{\la,\pi,\cM}_{s-}(i)=1}\diff N_s^P(i-1)\\
&- 2\int_0^t \indiq{\zeta^{\la,\pi,\cM}_{s-}(i)=2}\diff N_s^P(i).
\end{align*}
Using the propagation processes $(N^P_t(i))_{t\geq0,i\in\zz}$, consider the burning times $(T^1_i)_{i\in\zz}$ of the propagation process iginited at $(0,t_1)$, recall Definition \ref{definition1 application}, and define the destroyed
cluster due to the match falling in $0$ at time $\al t_1$, recall \eqref{destroyed comp},
\[C^P((\zeta^{\la,\pi,\cM}_t(i))_{t\geq0,i\in\zz},(0,t_1))\coloneqq\intervalleentier{i^g}{i^d}.\]

We finally define the time needed for $C^P((\zeta^{\la,\pi,\cM}_t(i))_{t\geq0,i\in\zz},(0,t_1))$ to
become again occupied
\[\Theta_{\cM}^{\la,\pi}\coloneqq\inf\left\{ t>t_1 : \forall i\in
C^P((\zeta^{\la,\pi,\cM}_t(i))_{t\geq0,i\in\zz},(0,t_1)),\zeta^{\la,\pi,\cM}_{\al t}(i)=1 \right\}.\]

For all $\delta>0$, there holds that, 
\[\lim_{\la ,\pi} \proba{\abs{\Theta_{\cM}^{\la,\pi} - (t_1-t_0)}\geq
\delta} = 0\]
when $\la\to0$ and $\pi\to\infty$ in the regime $\cR(p)$. 
\end{lem}
Let us explain the behaviour of the process $(\zeta^{\la,\pi,\cM}_t(i))_{t\geq 0,i\in\zz}$. If $t_0=0$, then the process starts from a vacant initial situation and a match falls on $0$ at time $\al t_1$. It does not depend on $i_0$ and since $0<t_1<1$, the zone $\intervalleentier{-\ml}{\ml}$ is not completely filled at time $\al (t_1+\kappa_{\la,\pi}^0)$, see Lemma \ref{speed}-1 (and because $\kappa_{\la,\pi}^0\to0$). The process is then governed by the propagation processes $(N_t^P(i))_{t\geq0,i\in\zz}$ and the seed processes $(N_t^S(i))_{t\geq0,i\in\zz}$ with the same rules as the $(\la,\pi)-$FFP. As seen in {\bf Micro$(p)$} in Subsection \ref{application ffp}, the fire is extinguished at time $\al(t_1+\kappa_{\la,\pi}^0)$.

If $t_0>1$, then the process starts at time $0$ from an occupied initial situation, nothing happens until a match falls on $i_0$ at time $\al(t_0-\vlp)$. Two fires start: one goes to the left and one goes to the right. Thus, on $\Omega_{\la,\pi}^{P,T}(i_0/\nl, t_0-\vlp)$, recall Definition \ref{definition2 application}, and since
\[\lfloor\al\pi(3\vlp-\e_\la)\rfloor\geq 2\ml+2\klp,\]
recall \eqref{klp} and \eqref{vlp}, each site of $\intervalleentier{-\ml}{\ml}$ burns and extinguishes before $\al(t_0+2\vlp)$, recall Lemma \ref{propagation lemma p}. Hence, the zone $\intervalleentier{-\ml}{\ml}$ is not completely filled when the match falls on $0$ at time $\al t_1$, see Lemma \ref{speed}-1 and because $\al(t_0+2\vlp)<\al t_1<\al(t_0+1)$ for all $(\la,\pi)$ sufficiently close to the regime $\cR(p)$.
\begin{proof}
The proof is in the same spirit as the proof of Lemma \ref{micro fire infty}. We first define the simplest process with an instantaneous propagation: if a match falls in a cluster, it destroys
instantaneously the entire connected component. Secondly, we flank the killed cluster
$C^P((\zeta^{\la,\pi,\cM}_t(i))_{t\geq0,i\in\zz},(0,t_1))$ to estimate the time needed to become again occupied, see Figure \ref{fig height barrierp}.

\md

\noindent{\bf Step 1.} Let $\tau_0<\tau_1<\tau_0+1$ be fixed. Put
$\vartheta_{\tau_0,t}^\la(i)=\min(N^{S}_{\al (\tau_0+t)}(i)-N^{S}_{\al \tau_0}(i),1)$ and
$\vartheta_{\tau_1,t}^\la(i)=\min(N^{S}_{\al(\tau_1+t)}(i)-N^{S}_{\al
\tau_1}(i),1)$ for all $t>0$ and all $i\in \zz$. We define the time needed for the destroyed cluster to be fully regenerated
\[\Xi_{\tau_0,\tau_1}^\la=\inf \left\{ t>0 : \forall i \in
C(\vartheta_{\tau_0,\tau_1-\tau_0}^\la,0),\; \vartheta_{\tau_1,t}^\la(i)=1 \right\}.\]
Then for all $\delta>0$, 
\[\lim_{\la \to 0} \proba{ |\Xi_{\tau_0,\tau_1}^\la -(\tau_1- \tau_0)|\geq \delta} =
0.\]

This has been checked in Step 1 of the proof of Lemma \ref{micro fire infty} when $\tau_0=0$. This of course extends without any difficulty, using time stationarity.

\md

\noindent{\bf Step 2.} Assume $t_0=0$. In that case, the process not depends on $i_0$. Consider the event $\Omega^{P,T}_{\la,\pi}(0,t_1)$, recall Definition \ref{definition2 application}. We  define
\begin{multline*}
\tOmega^{P,T,\cM}_{\la,\pi}=\Omega^{P,T}_{\la,\pi}(0,t_1)\cap\{\exists i_1\in\intervalleentier{-\ml}{0}, N^S_{\al (t_1+\kappa_{\la,\pi}^0)}(i_1)=0\}\\
\cap\{\exists i_2\in\intervalleentier{0}{\ml}, N^S_{\al (t_1+\kappa_{\la,\pi}^0)}(i_2)=0\}.
\end{multline*}
Lemma \ref{propagation lemma p} together with Lemma \ref{speed}-1 show that $\proba{\tOmega^{P,T,\cM}_{\la,\pi}}$ tends to $1$ when $\la\to0$ and $\pi\to\infty$ in the regime $\cR(p)$ (because $t_1+\kappa_{\la,\pi}^0<(t_1+1)/2<1$ for all $(\la,\pi)$ sufficiently close to the regime $\cR(p)$).

Next, on $\tOmega^{P,T}_{\la,\pi}(0,t_1)$, there holds that
\[C(\vartheta_{0,t_1+\kappa_{\la,\pi}^0}^\la,0)\coloneqq\intervalleentier{C^-}{C^+}\subset\intervalleentier{i_1}{i_2}\subset\intervalleentier{-\ml}{\ml}.\]
Since, by definition, no seed falls on $C^+$ and on $C^-$ until $\al(t_1+\kappa_{\la,\pi}^0)$ and since we start from a vacant initial situation, we deduce that
\[\zeta^{\la,\pi,\cM}_t(C^-)=\zeta^{\la,\pi,\cM}_t(C^+)=0\]
for all $t\in\intervalleff{0}{\al(t_1+\kappa_{\la,\pi}^0)}\supset\intervalleff{\al t_1}{\al(t_1+\kappa_{\la,\pi}^0)}$. As seen in {\bf Micro$(p)$} in Subsection \ref{application ffp}, the fire destroys exactly the zone $C^P((\zeta^{\la,\pi,\cM}_t(i))_{t\geq 0,i\in\zz},(0,t_1))$ and 
\[C^P((\zeta^{\la,\pi,\cM}_t(i))_{t\geq 0,i\in\zz},(0,t_1))\subset\intervalleentier{C^-}{C^+}\subset\intervalleentier{-\ml}{\ml}\]
with $\zeta^{\la,\pi,\cM}_{\al (t_1+\kappa_{\la,\pi}^0)}(i)\leq 1$ for all $i\in\zz$ (the fire is extinguished at time $\al(t_1+\kappa_{\la,\pi}^0)$). 

Since $C^P((\zeta^{\la,\pi,\cM}_t(i))_{t\geq 0,i\in\zz},(0,t_1))$ clearly contains $C(\vartheta_{0,t_1}^\la,0)$, we deduce that, on $\tOmega^{P,T,\cM}_{\la,\pi}$,
\[t_1+\Xi_{0,t_1}^\la\leq t_1+\Theta^{\la,\pi}_{\cM}\leq t_1+\kappa_{\la,\pi}^0+\Xi_{0,t_1+\kappa_{\la,\pi}^0}^\la.\]
Remark now that the function $\colon t\mapsto t+\Xi^\la_{0,t}$ is a.s. non decreasing and right-continuous. We thus deduce from Step 1 that
\[t_1+\Theta_{\cM}^{\la,\pi}\xrightarrow[\la,\pi]{\pp}2t_1\]
in probability, whence  for all $\delta>0$ and all $\e>0$, there holds  that $\proba{
\abs{\Theta^{\la,\pi}_{\cM}- t_1}\geq
\delta} <\e$ for all $(\la,\pi)$ sufficiently close to the regime $\cR(p)$.

\md

\noindent{\bf Step 3.} Assume now $t_0>1$. We may and will assume $i_0\in\intervalleentier{-\ml-2\klp}{-\ml}$, by symetry.

Consider the events $\Omega^{P,T}_{\la,\pi}(i_0/\nl,t_0-\vlp)$ and $\Omega^{P,T}_{\la,\pi}(0,t_1)$, recall Definition \ref{definition2 application}. We define
\begin{multline*}
\tOmega^{P,T,\cM}_{\la,\pi}\coloneqq \Omega^{P,T}_{\la,\pi}(0, t_1)\cap\Omega^{P,T}_{\la,\pi}(i_0/\nl,t_0-\vlp)\\
\cap\{\exists i_1 \in\intervalleentier{-\ml}{0},N^S_{\al (t_1+\kappa_{\la,\pi}^0)}(i_1)-N^S_{\al (t_0-\vlp)}(i_1)=0\}\\
\cap\{\exists i_2 \in\intervalleentier{0}{\ml},N^S_{\al (t_1+\kappa_{\la,\pi}^0)}(i_2)-N^S_{\al (t_0-\vlp)}(i_2)=0\}.
\end{multline*}
Lemma \ref{propagation lemma p} together with Lemma \ref{speed}-1 directly imply that $\proba{\tOmega^{P,T,\cM}_{\la,\pi}}$ tends to $1$ when $\la\to0$ and $\pi\to \infty$ in the regime $\cR(p)$ (because $t_1+\kappa_{\la,\pi}^0-(t_0-\vlp)=t_1-t_0+\kappa_{\la,\pi}^0+\vlp<(t_1-t_0+1)/2<1$ for all $(\la,\pi)$ sufficiently close to the regime $\cR(p)$).

Recall Lemma \ref{propagation lemma p}. Since all the sites are occupied at time $\al(t_0-\vlp)$ and since
\[i_0+\lfloor\al\pi(3\vlp-\e_\la)\rfloor\geq \ml,\]
on $\Omega^{P,T}_{\la,\pi}(i_0/\nl,t_0-\vlp)$, there is no more burning tree in $\intervalleentier{-\ml}{\ml}$ at time $\al(t_0+2\vlp)$ nor during the time interval $\intervallefo{\al(t_0+2\vlp)}{\al t_1}$. Thus, the match falling in $0$ at time $\al t_1$ destroys at least the zone $C(\vartheta_{t_0+2\vlp,t_1}^\la,0)$.

Next, on $\tOmega_{\la,\pi}^{P,T,\cM}$, we have
\[C(\vartheta_{t_0-\vlp,t_1+\kappa_{\la,\pi}^0}^\la,0)\coloneqq \intervalleentier{C^-}{C^+}\subset \intervalleentier{i_1}{i_2}\subset\intervalleentier{-\ml}{\ml}.\]

Since no seed falls on $C^-$ and on $C^+$ during $\intervalleff{\al(t_0-\vlp)}{\al(t_1+\kappa_{\la,\pi}^0)}$ and since $C^-$ and $C^+$ are made vacant during the time interval $\intervalleff{\al (t_0-\vlp)}{\al (t_0+2\vlp)}$, thanks to $\Omega^{P,T}_{\la,\pi}(i_0/\nl,t_0-\vlp)$, we deduce that there is no burning tree in $\intervalleentier{C^-}{C^+}$ at time $\al t_1-$ and
\[\zeta^{\la,\pi,\cM}_{\al t}(C^-)=\zeta^{\la,\pi,\cM}_{\al t}(C^+)=0 \text{ for all } t\in\intervalleff{t_1}{t_1+\kappa_{\la,\pi}^0}.\]
Hence, as seen in {\bf Micro$(p)$} in Subsection \ref{application ffp}, the match falling on $0$ at time $\al t_1$ destroys at most the zone $\intervalleentier{C^-}{C^+}\subset \intervalleentier{i_1}{i_2}$ and there is no more burning tree in $\intervalleentier{C^-}{C^+}$ at time $\al(t_1+\kappa_{\la,\pi}^0)$.

To summarize, on $\tOmega^{P,T,\cM}_{\la,\pi}$, see Figure \ref{fig height barrierp},  we have
\[C(\vartheta_{t_0+2\vlp,t_1}^\la,0) \subset
C^P((\zeta^{\la,\pi,\cM}_t(i))_{t\geq 0,i\in\zz},(0, t_1))\subset C(\vartheta_{t_0-\vlp,t_1+\kappa_{\la,\pi}^0}^\la,0)\subset \intervalleentier{i_1}{i_2}\]
with additionally $\zeta^{\la,\pi,\cM}_{\al (t_1+\kappa_{\la,\pi}^0)}(i)\leq 1$ for all $i\in\intervalleentier{-\ml}{\ml}$.

Since no fire affect the zone $\intervalleentier{-\ml}{\ml}$ during $\intervalleff{\al(t_1+\kappa_{\la,\pi}^0)}{\al T}$, thanks to $\Omega^{P,T}_{\la,\pi}(i_0/\nl,t_0-\vlp)$, we deduce that, on $\tOmega_{\la,\pi}^{P,T,\cM}$ and for all $(\la,\pi)$ sufficiently close to the regime $\cR(p)$,
\[t_1+\Xi_{t_0+2\vlp,t_1}^\la\leq t_1+\Theta^{\la,\pi}_{\cM}\leq t_1+\kappa_{\la,\pi}^0+\Xi_{t_0-\vlp,t_1+\kappa_{\la,\pi}^0}^\la.\]

Then, one easily concludes. The function $s\mapsto t_1+\Xi_{t_0+s,t_1}^\la$ is a.s. non increasing and right-continuous while the function $s\mapsto t_1+s+\Xi_{t_0-s,t_1+s}^\la$ is a.s. non decreasing and right-continuous. Since $\kappa_{\la,\pi}^0\to0$, we thus deduce from Step 1 that
\[t_1+\Theta^{\la,\pi}_{\cM}\xrightarrow[\la,\pi]{\pp}2t_1-t_0,\]
as desired.\qedhere
\begin{figure}[h!]
\fbox{
\begin{minipage}[c]{0.95\textwidth}
\centering
\begin{tikzpicture}
\draw[thick,decorate,decoration={random steps,segment length=0.5mm,amplitude=0.2mm},color=red]
	(-6,.4) -- (-4.5,0) -- (6,2.625);%pente 1/4

\draw[->] (-6,0) -- (-6,8) node[right] {$t$};
\draw (6,0)--(-6,0) node[left] {\footnotesize $\al(t_0-\vlp)$};

\draw (0,.1)--(0,-.1) node[below] {$0$};
\draw (-4.5,.1)--(-4.5,-.1) node[below] {\footnotesize $i_0$};
\draw (-5.6,.1)--(-5.6,-.1) node[below] {\footnotesize $-\ml-2\klp$};
\draw (5.6,.1)--(5.6,-.1) node[below] {\footnotesize $\ml+2\klp$};
\draw (-3.5,.1)--(-3.5,-.1) node[below] {\footnotesize $-\ml$};
\draw (3.5,.1)--(3.5,-.1) node[below] {\footnotesize $\ml$};

\draw[thick,decorate,decoration={random steps,segment length=0.5mm,amplitude=0.2mm},color=red]
	(-2,5.5) -- (0,5)--(1.5,5.4);
	
\draw[dashed] (1.5,5.4) -- (1.5,0) node[below] {\footnotesize $i^d$};
\draw[dashed] (-2,5.5) -- (-2,0) node[below] {\footnotesize $i^g$};
\draw[dashed] (2.3,6) -- (2.3,0) node[below] {\footnotesize $C^+$};
\draw[dashed] (-2.5,6) -- (-2.5,0) node[below] {\footnotesize $C^-$};

\draw[ultra thick] (-1.5,5)--(1,5);
\draw[decorate,decoration={brace,raise=0.2cm,mirror}] (-1.5,5)--(1,5) node[below=0.2cm,pos=0.5,sloped] {$C(\vartheta_{t_0+2\vlp,t_1}^\lambda,0)$};

\draw[ultra thick] (-2.5,6)--(2.3,6);
\draw[decorate,decoration={brace,raise=0.2cm}] (-2.5,6)--(2.3,6) node[above=0.2cm,pos=0.5,sloped] {$C(\vartheta_{t_0-\vlp,t_1+\kappa_{\la,\pi}^0}^\lambda,0)$};

\draw (-5.9,6) -- (-6.1,6) node[left] {\footnotesize $\al(t_1+\kappa_{\la,\pi}^0)$};
\draw (-5.9,5) -- (-6.1,5) node[left] {\footnotesize $\al t_1$};
\draw[dashed] (3.5,2.3) -- (-6.1,2.3) node[left] {\footnotesize $\al (t_0+2\vlp)$};
\draw[dashed] (3.5,2.3) -- (3.5,0);
\end{tikzpicture} \caption{Height of a barrier in the regime $\cR(p)$, for $p>0$.}\label{fig height barrierp}
\vspace{.5cm}
\parbox{13.3cm}{
\footnotesize{
At time $\al(t_0-\vlp)-$, all the sites are occupied. A match falls on $i_0$ at time $\al (t_0-\vlp)$.  Two fires start: one goes to the left and one goes to the right. Thus, on $\Omega_{\la,\pi}^{P,T}(i_0/\nl, t_0-\vlp)$, each site of $\intervalleentier{-\ml}{\ml}$ burns and extinguishes before $\al(t_0+2\vlp)$ (because $i_0+\lfloor\al\pi(3\vlp-\e_\la)\rfloor\geq \ml$).

Next, a match falls on $0$ at time $\al t_1$. Since no seed fall on $C^-\in\intervalleentier{-\ml}{0}$ and $C^+\in\intervalleentier{0}{\ml}$ during $\intervalleff{\al(t_0-\vlp)}{\al(t_1+\kappa_{\la,\pi}^0)}$, they remain vacant after burning. Thus, the true killed cluster $\intervalleentier{i^g}{i^d}$ contains $C(\vartheta_{t_0+2\vlp,t_1}^\lambda,0)$ but is included in $\intervalleentier{C^-}{C^+}=C^P((\zeta^{\la,\pi\cM}_t)_{t\geq0,i\in\zz},(0,t_1))$.
}}
\end{minipage}}
\end{figure}
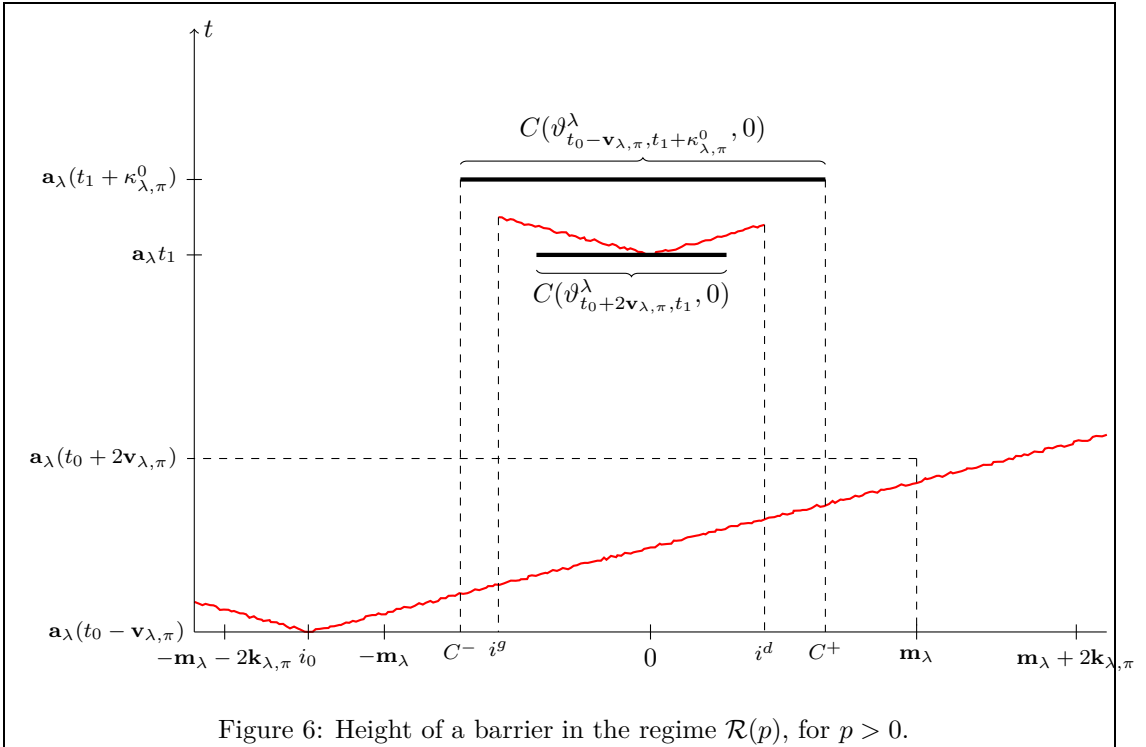
\end{proof}

\subsection{Persistent effect of microscopic fires}\label{persisp}
Here we study the effect of microscopic fires. First, they produce a barrier, and
then, if there are alternatively macroscopic fires on the left and right, they still
have an effect. This phenomenon is illustrated on Figure \ref{persiseffect} in
the case of the limit process.

We say that $\cP=(t_0,t_1,\dots,t_K)$
satisfies $(PP1)$ (like ping-pong) if 
\begin{enumerate}
        \item $K\geq2$;
        \item $t_0\in\{0\}\cup\intervalleoo{1}{\infty}$ and $t_0<t_1<t_2<\dots<t_K$;
        \item for all $k=0,\dots,K-1$, $t_{k+1}-t_k<1$;
        \item $t_2-t_0>1$ and for all $k=2,\dots,K-2, t_{k+2}-t_k>1$.
\end{enumerate}

We say that $\cI=(\e;i_0,i_2,\dots,i_K)$
satisfies $(PP2)$ if 
\begin{enumerate}
	\item $\e\in\{-1,1\}$;
	\item $|i_0|\in\intervalleentier{\ml}{\ml+2\klp}$;
	\item for all $k=2,\dots,K$, $\e_k i_k\in\intervalleentier{\ml}{\ml+2\klp}$, where we set $\e_k=(-1)^k\e$.
\end{enumerate}

Finally, we say that $\mathfrak{P}=(\cP,\cI)$ satisfies $(PP)$ if $\cP$ satisfies $(PP1)$ and $\cI$ satisfies $(PP2)$.

Let $\fP$ satisfy $(PP)$. Consider two Poisson processes $(N_t^S(i))_{t\geq0,i\in\zz}$ and
$(N_t^P(i))_{t\geq0,i\in\zz}$ with respective rates $1$ and $\pi$, all this
processes being independent. We define the process $(\zeta^{\la,\pi,\mathfrak{P}}_t(i))_{t\geq 0,i\in\zz}$ as follows
\begin{align*}
{\zeta}^{\lambda,\pi,\mathfrak{P}}_t(i)=&(1+\indiq{i=i_0,t\geq \al (t_0-\vlp)})\indiq{t_0\geq1}+\indiq{i=0,t\geq\al t_1,{\zeta}^{\lambda,\pi,\mathfrak{P}}_{\al t_1-}(0)=1}\\
&+\sum_{k=2}^K \indiq{i=i_k,t\geq\al (t_k-\vlp),{\zeta}^{\lambda,\pi,\mathfrak{P}}_{\al(t_k-\vlp)-}(i_k)=1}\\
&+\int_0^t\indiq{{\zeta}^{\lambda,\pi,\mathfrak{P}}_{s-}(i)=0}\diff
N_s^S(i)\\
&+\int_0^t
\indiq{{\zeta}^{\lambda,\pi,\mathfrak{P}}_{s-}(i-1)=2,{\zeta}^{\lambda,\pi,\mathfrak{P}}_{s-}(i)=1}\diff
N_s^P(i-1)+\int_0^t
\indiq{{\zeta}^{\lambda,\pi,\mathfrak{P}}_{s-}(i+1)=2,{\zeta}^{\lambda,\pi,\mathfrak{P}}_{s-}(i)=1}\diff
N_s^P(i+1)\\
&- 2\int_0^t \indiq{{\zeta}^{\lambda,\pi,\mathfrak{P}}_{s-}(i)=2}\diff
N_s^P(i).
\end{align*}
We now explain the behaviour of the process $(\zeta^{\la,\pi,\mathfrak{P}}_t(i))_{t\geq 0,i\in\zz}$.
\begin{itemize}
	\item If $t_0=0$, then the process starts from a vacant initial configuration. The match falling on $0$ at time $\al t_1\in\intervalleoo{0}{\al}$ creates a barrier, see Lemma \ref{micro fire p}, because $t_1\in\intervalleoo{0}{1}$. Then, fires start in $i_k$ alternately on the right and on the left of $0$ at times $\al(t_k-\vlp)$ for all $k=2,\dots,K$ and fires spread accross $\zz$ according to the same rules as the $(\la,\pi)-$FFP.
	\item If $t_0>1$, the process starts from an occupied initial situation. Nothing happens until a match falls on $i_0$ at time $\al (t_0-\vlp)$ and spreads across $\zz$. Next, a match falls on $0$ at time $\al t_1\in\intervalleoo{\al t_0}{\al(t_0+1)}$. It then creates a barrier, see Lemma \ref{micro fire p}. Afterwards, matches fall successively in $i_k$ at time $\al (t_k-\vlp)$ for each $k=2,\dots,K$ and fires spread accross $\zz$ according to the same rules as the $(\la,\pi)-$FFP.
\end{itemize}

Consider the event
\[\Omega_{\mathfrak{P}}^{S,P}(\la,\pi)=\{\forall k\in\{2,\dots,K\},\,\exists j\in\intervalleentier{-\ml}{\ml}, \,\forall t\in\intervallefo{t_k+2\vlp}{t_k+1-\vlp},\,\zeta^{\la,\pi,\fP}_{\al t}(j)=0\}.\]

\begin{lem}\label{persisplem}
Let $\cP=(t_0,\dots,t_K)$ satisfy $(PP1)$ and $\cI=(\e;i_0,i_2,\dots,i_K)$ satisify $(PP2)$. For each
$\lambda\in\intervalleoo{0}{1}$ and each $\pi\geq1$, consider the process $({\zeta}^{\lambda,\pi,\mathfrak{P}}_t(i))_{t\geq0,i\in\zz}$ defined above.

If $t_2-t_1<t_1-t_0$, when $\la\to0$ and $\pi\to\infty$ in the regime $\cR(p)$, there holds
\[\lim_{\lambda,\pi}\proba{\Omega_\mathfrak{P}^{S,P}(\lambda,\pi)}=1.\]
\end{lem}

\begin{proof}
We define, recall Definition \ref{definition2 application}, 
\[\Omega^{P,T,\mathfrak{P}}_{\la,\pi}=\Omega^{P,T}_{\la,\pi}(0,t_1)\cap\bigcap_{k=0,2,\dots,K}\Omega^{P,T}_{\la,\pi}\left(\frac{i_k}{\nl},t_k-\vlp\right).\]
There holds that $\proba{\Omega^{P,T,\fP}_{\la,\pi}}$ tends to $1$ as $\la\to0$ and $\pi\to\infty$ in the regime $\cR(p)$, by Lemma \ref{propagation lemma p}. In the whole proof, we work on $\Omega^{P,T,\fP}_{\la,\pi}$ and assume that $(\la,\pi)$ is sufficiently close to the regime $\cR(p)$ in such a way that $3\vlp<\min_{i=1,\dots,K}(t_{i+1}-t_i)<1-3\vlp$.

For simplicity, we assume that $\e=-1$, $t_0=0$ and that $K$ is even (see for example Step 3 in Lemma \ref{micro fire p}. The other cases are treated similarly. Fix $\alpha=1/K$. We define $\cM=(0;0,t_1)$, recall Lemma \ref{micro fire p}.

Observe that on $\Omega^{P,T,\fP}_{\la,\pi}$, a burning tree at time $\al t$ necessarily belongs to $\intervalleentier{i_k+\lfloor\al\pi(t-t_k-\e_\la)\rfloor}{i_k+\lfloor\al\pi(t-t_k+\e_\la)\rfloor}$ or to $\intervalleentier{i_k-\lfloor\al\pi(t-t_k+\e_\la)\rfloor}{i_k-\lfloor\al\pi(t-t_k-\e_\la)\rfloor}$, for some $k\in\{0,\dots,K\}$ and is either a front of a fire or has vacant neighbors.

Observe that for all $i\in\intervalleentier{-\ml-2\klp}{-\ml}$, we have, recall \eqref{klp} and \eqref{vlp},
\begin{equation}\label{estimpingmoins}
i+\lfloor\al\pi(3\vlp-\e_\la)\rfloor\geq\ml
\end{equation}
whence for all $i\in\intervalleentier{\ml}{\ml+2\klp}$, we have
\begin{equation}\label{estimpingplus}
i-\lfloor\al\pi(3\vlp-\e_\la)\rfloor\leq-\ml.
\end{equation}

\md

\noindent{\bf First fire.} We put $C^P=C^P((\zeta^{\la,\pi,\fP}_t(i))_{t\geq 0,i\in\zz},(0,t_1))$, the destroyed cluster due to the match falling on $0$ at time $\al t_1$, recall \eqref{destroyed comp}. Since $0<t_1<1$, there holds  $C^P\subset\intervalleentier{-\lfloor\alpha\ml\rfloor}{\lfloor\alpha\ml\rfloor}$ with probability tending to $1$ (use Lemma \ref{speed}-1, space/time stationarity and {\bf Micro$(p)$} in Subsection \ref{application ffp}). Thus the match falling at time $\al t_1$ destroys nothing outside $\intervalleentier{-\lfloor\alpha\ml\rfloor}{\lfloor\alpha\ml\rfloor}$ and there is no more burning tree in $\zz$ at time $\al(t_1+\kappa_{\la,\pi}^0)$.

\md

\noindent{\bf Second fire.} Since $t_2-\vlp>1$, at least one seed has fallen, during $\intervallefo{0}{\al(t_2-\vlp)}$, on each site of $\intervalleentier{-\ml-2\klp}{-\lfloor\alpha\ml\rfloor-1}$ with probability tending to $1$ (use Lemma \ref{speed}-2 and space/time stationarity). Since this zone has not been affected by a fire during the time interval $\intervallefo{0}{\al (t_2-\vlp)}$, this zone is completely occupied at time $\al(t_2-\vlp)-$.
     
Besides, with probability tending to $1$, there is (at least) an empty site in $C^P\subset\intervalleentier{-\lfloor\alpha\ml\rfloor}{\lfloor\alpha\ml\rfloor}$ during the time interval $\intervalleoo{\al(t_1+\kappa_{\la,\pi}^0)}{\al(t_2+2\vlp)}$ because $t_2+2\vlp<t_1+\Theta^{\la,\pi}_{\cM}$ with probability tending to $1$ (by Lemma \ref{micro fire p}, $\Theta^{\la,\pi}_\cM\simeq t_1-t_0=t_1$ and $t_2-t_1<t_1-t_0=t_1$ by assumption) and because by definition of $\Theta^{\la,\pi}_{\cM}$, there is an empty site in $C^P\subset\intervalleentier{-\lfloor\alpha\ml\rfloor}{\lfloor\alpha\ml\rfloor}$ during $\intervalleff{\al(t_1+\kappa_{\la,\pi}^0)}{\al(t_1+\Theta^{\la,\pi}_{\cM})}$.
        
Thus, the fire ignited on $i_2\in\intervalleentier{-\ml-2\klp}{-\ml}$ at time $\al (t_2-\vlp)$ burns each site of $\intervalleentier{-\ml-2\klp}{-\lfloor\alpha\ml\rfloor-1}$ before $\al(t_2+2\vlp)$ and does not affect the zone $\intervalleentier{\lfloor\alpha\ml\rfloor+1}{\ml+2\klp}$, thanks to \eqref{estimpingmoins} and $\Omega^{P,T}_{\la,\pi}(i_2/\nl,t_2-\vlp)$ (because the right front of the fire $2$ reach a vacant site and thus extinguish).

\md

\noindent{\bf Third fire.} All the sites of $\intervalleentier{\lfloor\alpha\ml\rfloor}{\ml+2\klp}$ are occupied at time
$\al(t_3-\vlp)-$ with probability tending to $1$
(because on $\Omega^{P,T}_{\la,\pi}(0,t_1)\cap\Omega^{P,T}_{\la,\pi}(i_2/\nl,t_2-\vlp)$, they have not been affected by a fire during
$\intervallefo{0}{\al (t_3-\vlp)}$
and because $t_3-\vlp>t_2-\vlp>1$, see Lemma \ref{speed}-2.).

Next, the probability that there is a site in
$\intervalleentier{-\lfloor2\alpha\ml\rfloor}{-\lfloor \alpha\ml\rfloor}$ where no seed falls during $\intervalleff{\al (t_2-\vlp)}{\al(t_2-\vlp+1)}$ tends to $1$ as $\la\to0$ and $\pi\to\infty$ in the regime $\cR(p)$ (use Lemma \ref{speed}-1 and space/time stationarity). Thus, since $t_3-t_2<1$, with probability tending to $1$, there exists a vacant site in $\intervalleentier{-\lfloor2\alpha\ml\rfloor}{-\lfloor \alpha\ml\rfloor}$ during 
\[\intervalleff{\al (t_2+2\vlp)}{\al(t_2-\vlp+1)}\supset\intervalleff{\al(t_3-\vlp)}{\al(t_3+2\vlp)}\]
(because each site of $\intervalleentier{-\lfloor2\alpha\ml\rfloor}{-\lfloor\alpha\ml\rfloor}$ has been made vacant by the second fire during $\intervalleff{\al(t_2-\vlp)}{\al(t_2+2\vlp)}$).

Thus, the fire ignited on $i_3\in\intervalleentier{\ml}{\ml+2\klp}$ at time $\al (t_3-\vlp)$ burns each site of $\intervalleentier{\lfloor\alpha\ml\rfloor+1}{\ml+2\klp}$ before $\al(t_3+2\vlp)$ and does not affect the zone $\intervalleentier{-\ml-2\klp}{-\lfloor\alpha\ml\rfloor-1}$ with probability tending to $1$, thanks to \eqref{estimpingplus} and $\Omega^{P,T}_{\la,\pi}(i_3/\nl,t_3-\vlp)$ (because the left front of the fire $3$ reach a vacant site and thus extinguish).

\md

\noindent{\bf Fourth fire.} All the sites of $\intervalleentier{-\ml-2\klp}{-\lfloor\alpha\ml\rfloor-1}$ are occupied at time $\al(t_4-\vlp)-$ with probability tending to $1$ (because on $\Omega^{P,T}_{\la,\pi}(0,t_1)\cap\Omega^{P,T}_{\la,\pi}(i_2/\nl,t_2-\vlp)\cap\Omega^{P,T}_{\la,\pi}(i_3/\nl,t_3-\vlp)$, they have not been affected by a fire during $\intervalleoo{\al(t_2+2\vlp)}{\al (t_4-\vlp)}$ and because $t_4-3\vlp-t_2>1$, see Lemma \ref{speed}-2 and spae/time stationarity).

The probability that there is a site in
$\intervalleentier{\lfloor \alpha\ml\rfloor+1}{\lfloor2\alpha\ml\rfloor}$ where no seed falls during $\intervalleff{\al(t_3-\vlp)}{\al(t_3-\vlp+1)}$ tends to $1$ as $\la\to0$ and $\pi\to\infty$ in the regime $\cR(p)$ (use Lemma \ref{speed}-1 and space/time stationarity). Hence, since $t_4-t_3<1$, there is at least one vacant site in $\intervalleentier{\lfloor \alpha\ml\rfloor+1}{\lfloor2\alpha\ml\rfloor}$
during
\[\intervalleff{\al(t_3+2\vlp)}{\al(t_3-\vlp+1)}\supset\intervalleff{\al (t_4-\vlp)}{\al(t_4+2\vlp)},\]
with probability tending to $1$.

Thus, the fire ignited on $i_4\in\intervalleentier{-\ml-2\klp}{-\ml}$ at time $\al (t_4-\vlp)$ burns each site of $\intervalleentier{-\ml-2\klp}{-\lfloor\alpha\ml\rfloor-1}$ before $\al(t_4+2\vlp)$ and does not affect the zone $\intervalleentier{\lfloor\alpha\ml\rfloor}{\ml+2\klp}$ with probability tending to $1$,  thanks to \eqref{estimpingmoins} and $\Omega^{P,T}_{\la,\pi}(i_4/\nl,t_4-\vlp)$.

\md

\noindent{\bf Last fire and conclusion.} Iterating the procedure, we see that
with a probability tending to $1$ as $\la\to 0$ and $\pi\to\infty$ in the regime $\cR(p)$, the zone $\intervalleentier{-\ml-2\klp}{-\lfloor(K\alpha/2)\ml\rfloor-1} =\intervalleentier{-\ml-2\klp}{- \lfloor \ml/2
\rfloor-1}$ is completely occupied at time $\al(t_K-\vlp)-$ and there is at least one vacant site in $\intervalleentier{\lfloor(K-1)\alpha/2\ml\rfloor}{\lfloor (K\alpha/2)\ml\rfloor}$ during the time interval $\intervalleoo{\al(t_{K-1}+2\vlp)}{\al(t_{K-1}-\vlp+1)}\supset\intervalleoo{\al(t_K-\vlp)}{\al(t_K+2\vlp)}$. Thus, the fire ignited on $i_K\in\intervalleentier{-\ml-2\klp}{-\ml}$ at time $\al(t_K-\vlp)$ destroys each site of the zone
$\intervalleentier{-\ml-2\klp}{- \lfloor \ml/2
\rfloor-1}$ before $\al(t_K+2\vlp)$ and does not affect the zone $\intervalleentier{\ml/2}{\ml}$, thanks to \eqref{estimpingmoins} and $\Omega^{P,T}_{\al,\pi}(i_K/\nl,t_K-\vlp)$. 

Finally, the probability that there is at least
one site in $\intervalleentier{-\ml}{- \ml/2}$ with
no seed falling during $\intervalleff{\al(t_K-\vlp)}{\al (t_K-\vlp+1)}$ tends to $1$  (by Lemma \ref{speed}-1.). Consequently, the probability that there is a vacant site in $\intervalleentier{-\ml}{- \ml/2}$ during $\intervalleff{\al(t_K+2\vlp)}{\al (t_K-\vlp+1)}$ tends to $1$ (because it has been made vacant by the fire $K$).

All this implies that for all $k\in\{2,\dots,K\},$, there is $j\in\intervalleentier{-\ml}{\ml}$ such that for all $t\in\intervallefo{t_k+2\vlp}{t_k+1-\vlp}$ there holds $\zeta^{\la,\pi,\fP}_{\al t}(j)=0$, as desired.
\qedhere

\begin{figure}[h!]
\fbox{
\begin{minipage}[c]{0.95\textwidth}
\centering
\begin{tikzpicture}

\fill[fill=gray!50!white] (-5,0) --(5,2.5)-- (5,0)--cycle;
\draw (-5,0)--(5,0) node at (0,-.3) {$0$};% node at (-5.1,-.3) {$-\alpha$} node at (5,-.3) {$\alpha$}
%\draw (-5,0)--(-5,10);
%\draw (5,0)--(5,10);

\draw[red] (-5,0) --(5,2.5) node at (-5,0) {$\bullet$};
\draw[dashed] (0,1.25)--(-5,1.25) node[left] {$t_0$};

\fill[fill=gray!50!white] (-5,2) --(0,3.25)-- (0,4.25)--(-5,3)--cycle;
\draw[red] (-5,3) --(0,4.25) node at (-5,3) {$\bullet$};
\draw[dashed] (0,4.25) -- (-5,4.25) node[left] {$t_2$};

\fill[fill=gray!50!white] (0,3.25)--(5,4.5)--(5,4.9)-- (0,6.15)--(0,5.75)--cycle;
\draw[red] (5,4.9) --(0,6.15) node at (5,4.9) {$\bullet$};
\draw[dashed] (0,6.15) -- (5,6.15) node[right] {$t_3$};

\fill[fill=gray!50!white] (-5,5) --(0,6.25)-- (0,7.25)--(-5,6)--cycle;
\draw[red] (-5,6) --(0,7.25) node at (-5,6) {$\bullet$};
\draw[dashed] (0,7.25) -- (-5,7.25) node[left] {$t_4$};

\draw[dashed] (0,2.9)--(5,2.9) node[right] {$t_1$};
\draw[thick] (0,2.9)--(0,4.55);
\draw[red] node at (0,2.9) {$\bullet$};

\fill[fill=gray!50!white]
(5,6.9)--(0,8.15)--(0,9.25)--(-5,8)--(-5,10)--(5,10)--(5,8.75)--cycle;
\end{tikzpicture} \caption{Persistent effect of microscopic fires.}\label{persiseffect}
\vspace{.5cm}
\parbox{13.3cm}{
\footnotesize{
Here we focus on the limit process with $t_0>1$. A first fire starts at time $\al(t_0-\vlp)$ and spread across $\zz$. Thus, the match falling in $0$ at time $\al t_1$ creates a barrier during $\al(t_1-t_0)$. If there are alternatively macroscopic fires on the left and right, there necessarily exists a vacant site around $0$ during $\intervalleoo{\al (t_0+2\vlp)}{\al (t_K+1-\vlp)}$.
}}
\end{minipage}}
\end{figure}
\end{proof}

\subsection{Heart of the proof}

\subsubsection{The coupling}\label{coupling}
We are going to construct a coupling between the $(\la,\pi,A)-$FFP
(on the time interval $\intervalleff{0}{\al T}$) and the $A-$LFFP$(p)$ (on $\intervalleff{0}{T}$). Let $\pi_M$ be a Poisson measure on $\rr\times \intervallefo{0}{\infty}$ with  intensity  measure $\diff x\diff t$.

First, we take for the matches of the discrete process the Poisson processes
\[N^M_t(i)=\pi_M(\intervallefo{i/\nl}{(i+1)/\nl}\times\intervalleff{0}{t/\al})\]
for all $i\in\zz$ and $t\in\intervalleff{0}{T}$.

We call $n:=\pi_M([0,T]\times\intervalleff{-A}{A})$
and we consider the marks $(T_q,X_q)_{q=1,\dots,n}$ of $\pi_M$ ordered 
in such a way that $0<T_1<\dots<T_n<T$.

Next, we introduce some i.i.d. families of i.i.d. Poisson processes $(N^{S,q}_t(i))_{t\geq 0,i\in \zz}$ and $(N^{P,q}_t(i))_{t\geq 0,i\in \zz}$ with respective parameter
$1$ and $\pi$, for $q=0,1,\dots$, independent of $\pi_M$.

Then we build two families of i.i.d. Poisson processes $(N^{S,\la,\pi}_t(i))_{t\geq 0, i \in \zz}$ and $(N^{P,\la,\pi}_t(i))_{t\geq 0, i \in \zz}$ as follows.
\begin{itemize}
	\item For $q \in \{1, \ldots, n\}$, for all $i \in [X_q]_{\la,\pi}$,
set $(N^{S,\la,\pi}_t(i))_{t\geq 0}=
(N^{S,q}_t(i-\lfloor \nl X_q \rfloor))_{t\geq 0}$ and $(N^{P,\la,\pi}_t(i))_{t\geq 0}=
(N^{P,q}_t(i-\lfloor \nl X_q \rfloor))_{t\geq 0}$ 
(if $i$ belongs to  $[X_q]_{\la,\pi}\cap [X_r]_{\la,\pi}$ for some $q<r$, set \eg 
$(N^{S,\la,\pi}_t(i))_{t\geq 0}=(N^{S,q}_t(i-\lfloor \bn_\la X_q \rfloor))_{t\geq 0}$ and $(N^{P,\la,\pi}_t(i))_{t\geq 0}=(N^{P,q}_t(i-\lfloor \bn_\la X_q \rfloor))_{t\geq 0}$.
This will occur with a very small probability, so that this choice is not
important).
	\item For all other $i \in \zz$ 
set $(N^{S,\la,\pi}_t(i))_{t\geq 0} = (N^{S,0}_t(i))_{t\geq 0}$ and $(N^{P,\la,\pi}_t(i))_{t\geq 0} = (N^{P,0}_t(i))_{t\geq 0}$.
\end{itemize}

The $(\la,\pi,A)-$FFP
$(\eta^{\la,\pi}_t(i))_{t\geq 0, i \in I_A^\la}$ is built from the seed processes
$(N^{S,\la,\pi}_t(i))_{t\geq 0, i \in \zz}$, the match processes $(N^{M}_t(i))_{t\geq 0, i \in \zz}$ and the propagation processes $(N^{P,\la,\pi}_t(i))_{t\geq 0, i \in \zz}$.

Finally, we build the $A-$LFFP$(p)$ $(Z_t(x),H_t(x),F_t(x))_{t\in[0,T],x\in\intervalleff{-A}{A}}$ from $\pi_M$ and observe that  it is
independent of $(N^{S,q}_t(i))_{t\in [0,\al T], i \in \zz, q\geq 0}$ and $(N^{P,q}_t(i))_{t\in [0,\al T], i \in \zz, q\geq 0}$.

Observe that if a match falls at some $X_q$ at time $T_q$ for the LFFP$(p)$, it will fall at $\lfloor\nl X_q \rfloor$ at time $\al T_q$ in the discrete process, and thus if the resulting fire is
microscopic in the limit process, it will involve with the same seed and propagation processes for all values of $\la$ and $\pi$ in discrete process.

\subsubsection{A favorable event}
We set $T_0=0$ and introduce 
\[\mathcal{T}_M=\{T_0,T_1,\dots,T_n\}\text{ and }\mathcal{B}_M=\{X_1,\dots,X_n\}.\]

For $q\in\{1,\dots,n\}$, $x\in\intervalleff{-A}{A}$ and $t\in\intervalleff{0}{T}$, we define
\begin{align}
T_q(x) &=T_q+p|x-X_q|\\
X_q^+(t) &=X_q+\frac{t-T_q}{p}\\
X_q^-(t) &=X_q-\frac{t-T_q}{p}
\end{align}
which are respectively the possible transit time in $x$ of the fire starting in $X_q$ at time $T_q$ and the possible location of the right and the left front at time $t$ of the fire starting in $X_q$ at time $T_q$. Observe that all $x\in\intervalleff{-A}{A}$ either equal to $X_k^+(T_k(x))$ or $X_k^-(T_k(x))$.

We next introduce, for $q\in\{1,\dots,n\}$,
\[\cS_{M,q}=\enstq{T_k(X_q)=T_k+p\abs{X_q-X_k}}{ k\neq q}\]
the set of all the possible transit times in $X_q$ of the other fire $k$ and 
\[\cS_M=\cup_{q=1,\dots,n}\,\cS_{M,q}.\]

We also introduce
\[\mathcal{S}^1_M=\{2T_q-s : q\in\{1,\dots,n\}, s\in\cS_{M,q},s< T_q\}\]
which has to be seen as the set of the possible end of the microscopic fires, recall Lemma \ref{micro fire p} and, for $q\in\{2,\dots,n\}$,
\[\cB^1_{M,q}=\enstq{X_k^+(T_q)=X_k+\frac{T_q-T_k}{p}}{1\leq k<q}\cup\enstq{X_k^-(T_q)=X_k+\frac{T_k-T_q}{p}}{1\leq k<q}\]
which has to be seen as the set of the possible locations of the fire $k$ at time $T_q$.

We finally introduce 
\[\cB^2_M=\enstq{\frac{T_q-T_k}{2p}+\frac{X_q+X_k}{2}}{X_k<X_q}\text{ and }\cS^2_M=\enstq{\frac{T_q+T_k}{2}+p\frac{X_q+X_k}{2}}{1\leq k< q\leq n}\]
which has to be seen as the set of the possible locations and the set of the possible times where two fires may meet as well as the set $\cC_M$ of connected component of $\intervalleff{-A}{A}\setminus(\cB_M\cup\cB_M^2)$ (sometimes refers as cells).

For $\alpha>0$, we consider the event
\begin{multline*}
\Omega_M(\alpha)=\left\{
\min_{\substack{s,t\in\mathcal{T}_M\cup\mathcal{S}_M\cup\mathcal{S}_M^1\cup\mathcal{S}_M^2,\\
s\neq t}}\abs{t-s}\geq3\alpha,
\min_{s,t\in\mathcal{T}_M\cup\mathcal{S}_M\cup\mathcal{S}_M^1\cup\mathcal{S}_M^2,}\abs{t-(s+1)}\geq3\alpha,\right.\\
\left.
\min_{\substack{x,y\in\mathcal{B}_M\cup\mathcal{B}^2_M\cup\{-A,A\},\\
x\neq y}}\abs{x-y}\geq\frac{3\alpha}{p}\right\}
\end{multline*}
which clearly satisfies $\lim_{\alpha\to0}\proba{\Omega_M(\alpha)}=1$. For any given $\alpha>0$, there exists $\la_\alpha\in\intervalleoo{0}{1}$ and $\e_\alpha>0$ such that for all $\la\in\intervalleoo{0}{\la_\alpha}$ and all $\pi\geq1$ in such a way that $|\nl/(\al\pi)-p|<\e_\alpha$, on $\Omega_M(\alpha)$, there
holds that for all $x,y\in\mathcal{B}_M\cup\mathcal{B}_M^2\cup\{-A,A\}$, with $x\neq
y$, $[x]_{\la,\pi}\cap[y]_{\la,\pi}=\emptyset$.

For $q\in\{1,\dots,n\}$, using the seed processes $(N^{S,\la,\pi}_t(i))_{t\geq0,i\in\zz}$ and the propagation processes $(N^{P,\la,\pi}_t(i))_{t\geq0,i\in\zz}$, we build, recall Definition \ref{definition1 application}, $(\check{\zeta}_t^{\la,\pi,q}(i))_{t\geq0,i\in\zz}$ (the propagation process ignited at $(X_q,T_q)$), $(i^{q,+}_t)_{t\geq0}$ and $(i^{q,-}_t)_{t\geq0}$ (the corresponding right and  left fronts) and $(T^q_i)_{i\in\zz}$ (the associated burning times). We also use $\Omega^{P,T}_{\la,\pi}(X_q,T_q)$, recall Definition \ref{definition2 application}. We set
\[\Omega^{P,T}(\la,\pi)=\bigcap_{q=1,\dots,n}\,\Omega^{P,T}_{\la,\pi}(X_q,T_q).\]
Since $\pi_M$ is independent of the processes $(N^{S,\la,\pi}_t(i))_{t\geq0,i\in\zz}$ and $(N^{P,\la,\pi}_t(i))_{t\geq0,i\in\zz}$, Lemma \ref{propagation lemma p} implies that $\proba{\Omega^{P,T}(\la,\pi)}$ tends to $1$ when $\la\to0$ and $\pi\to\infty$ in the regime $\cR(p)$.

Let $q\in\{1,\dots,n\}$. We define
\begin{align}
\cI^{q,+}&\coloneqq\enstq{\lfloor\nl X_k\rfloor+i^{k,+}_{\al (T_k(X_q)-\vlp-T_k)}-\lfloor\nl X_k^+(T_k(X_q))\rfloor}{k\neq q}\label{iplus}\\
\cI^{q,-}&\coloneqq\enstq{\lfloor\nl X_k\rfloor+i^{k,-}_{\al (T_k(X_q)-\vlp-T_k)}-\lfloor\nl X_k^-(T_k(X_q))\rfloor}{k\neq q}\label{imoins}.
\end{align}
Observe that, on $\Omega^{P,T}(\la,\pi)$,  $\cI^{q,-}\subset\intervalleentier{\ml}{\ml+2\klp}$ whence $\cI^{q,+}\subset\intervalleentier{-\ml-2\klp}{-\ml}$. We then call $\mathcal{U}_q$ the set of all possible
$\fP=(\cP,\cI)$ satisfying $(PP)$ where
\begin{itemize}
	\item $\cP=(t_0,T_q,t_2,\dots,t_K)$ satisfies $(PP1)$ with $\{t_0,t_2,\dots,t_K\}\subset\cS_{M,q}\cup\{0\}$ and with $T_q-t_0>t_2-T_q$;
	\item $\cI=(\e;i_0,i_2,\dots,i_K)$ satisfies $(PP2)$ with $\e\in\{-1,1\}$ and $\{i_0,i_2,\dots,i_K\}\subset \cI^{q,+}\cup\cI^{q,-}$.
\end{itemize}
For $\fP\in\mathcal{U}_q$, we introduce the
event $\Omega_\fP^{S,P,q}(\lambda,\pi)$, defined as in Subsection \ref{persisp}, with the Poisson
processes $(N^{S,q}_t(i))_{t\geq0,i\in\zz}$ and $(N^{P,q}_t(i))_{t\geq0,i\in\zz}$. Then we put
\[\Omega_1^{S,P}(\la,\pi)=\cap_{q=1}^n\left\{ \text{for all }\fP\in\mathcal{U}_q,\, \Omega_{\fP}^{S,P,q}(\la,\pi)\text{ holds}\right\},\]
which satisfies $\lim_{\la,\pi}\proba{\Omega_1^{S,P}(\la,\pi)}=1$ when $\la\to0$ and $\pi\to\infty$ in the regime $\cR(p)$. 
Indeed, by construction, $\pi_M$ is independent of $(N^{S,q}_t(i))_{t\geq0,i\in\zz}$ and $(N^{P,q}_t(i))_{t\geq0,i\in\zz}$. Observe that for $l\in\{1,\dots,n\}$, the location $i^{l,+}_{\al (T_l(X_q)-\vlp-T_l)}$ depends only on the propagation process $N^{P,\la,\pi}$ restricted to $\intervalleff{\al T_l}{\al (T_l(X_q)-\vlp)}\times\zz$ whereas the event $\Omega_{\fP}^{S,P,q}(\la,\pi)$ depends on the location only after $\al (T_l(X_q)-\vlp)$. Thus, it suffices to work with some
fixed $\{t_0, t_2,\dots, t_K\}\subset\cS_{M,q}$ and some fixed $(i_k)_{k=0,2,\dots,K}\subset\cI^{q,+}\cup\cI^{q,-}$. The result then follows from Lemma \ref{persisplem}.

We also consider the event $\Omega_2^S(\la,\pi)$ on which the following conditions hold:
for all $t_1,t_2\in\cT_M\cup\cS_M\cup\cS^1_M$ with $0<t_2-t_1<1$, for all $q=1,\dots,n$, there are
\[-\ml-2\klp<i_1<-\ml<i_2<0<i_3<\ml<i_4<\ml+2\klp\]
such that $N^{S,q}_{\al (t_2+4\vlp)}(i_j)-N^{S,q}_{\al (t_1-4\vlp)}(i_j)=0$ for $j=1,\dots,4$. There holds that $\proba{\Omega_2^S(\la,\pi)}$ tends to $1$ as $\la\to$ and $\pi\to\infty$ in the regime $\cR(p)$. Indeed, it suffices to prove
that almost surely, $\lim_{\substack{\la\to0\\ \pi\to\infty}}\probacond{\Omega_2^S(\la,\pi)}{\pi_M}=1$. Since
there are a.s. finitely many possibilities for $q,t_1,t_2$ and since $\pi_M$ is independent
of $(N^{S,q}_t(i))_{t\geq0,i\in\zz}$, it suffices to work with a fixed $q\in\{1,\dots,n\}$ and some
fixed $0<t_2-t_1<1$. The  result then follows from Lemma \ref{speed}-1,8 together with space/time stationarity and the fact that $\vlp\to0$.

Next we introduce the event $\Omega^S_3(\la,\pi)$ on which the following conditions hold: for all $q\in\{1,\dots,n\}$ and all $i\in I^\la_A$
\[N^{S,\la,\pi}_{\al (T_q(i/\nl)+1+\elp)}(i)-N^{S,\la,\pi}_{\al(T_q(i/\nl)+\elp)}(i)>0\]
and if $T_q(i/\nl)\geq 1$,
\[N^{S,\la,\pi}_{\al (T_q(i/\nl)-4\vlp)}(i)-N^{S,\la,\pi}_{\al(T_q(i/\nl)-1-4\vlp)}(i)>0.\]
There holds that $\proba{\Omega_3^S(\la,\pi)}$ tends to $1$ as $\la\to$ and $\pi\to\infty$ in the regime $\cR(p)$. Observing that $\abs{I_A^\la}\simeq 2A\nl$, Lemma \ref{speed} and space/time
stationarity shows the result. 

We also need $\Omega_4^{S,P}(\gamma,\la,\pi)$, defined for $\gamma>0$ as follows: for all
$q=1,\dots, n$, for all $\cM=(i_0;t_0,T_q)$ such that $t_0\in\cS_{M,q}\cup\{0\}$ with $t_0<T_q<t_0+1$ and  $i_0\in\cI^{q,+}\cup\cI^{q,-}$, there holds
that $\abs{\Theta^{\la,\pi,q}_\cM-(T_q-t_0)}<\gamma$. Here, $\Theta^{\la,\pi,q}_\cM$ is defined as in Lemma \ref{micro fire p} with
the seed processes family $(N^{S,q}_t(i))_{t\geq0,i\in\zz}$ and the propagation processes family $(N^{P,q}_t(i))_{t\geq0,i\in\zz}$. Lemma \ref{micro fire p} directly implies that for any $\gamma>0$, $\proba{\Omega_4^{S,P}(\gamma,\la,\pi)}$ tends to $1$ as $\la\to$ and $\pi\to\infty$ in the regime $\cR(p)$.

We finally introduce the event 
\[\Omega(\alpha,\gamma,\la,\pi)=\Omega_M(\alpha)\cap\Omega^{P,T}(\la,\pi)\cap\Omega^{S,P}_1(\la,\pi)\cap\Omega^S_2(\la,\pi)\cap\Omega^S_3(\la,\pi)\cap\Omega^{S,P}_4(\gamma,\la,\pi).\]
We have shown that for any $\delta>0$, there exists $\alpha\in\intervalleoo{0}{1}$ such that for any $\gamma>0$, there holds that $\proba{\Omega(\alpha,\gamma,\la,\pi)}>1-\delta$ for all $(\la,\pi)$ sufficiently close to the regime $\cR(p)$.

\subsubsection{Heart of the proof}
Consider the $A-$LFFP$(p)$ $( Z_t(x), H_t(x), F_t(x))_{t\geq0, x\in\intervalleff{-A}{A}}$.

For $x\in\intervalleoo{-A}{A}$, we put
\begin{gather*}
Z_{t-}(x)=\lim_{s\nearrow t}Z_s(x),\\
Z_t(x+)=\lim_{y\searrow x}Z_t(y)\text{ and }Z_t(x-)=\lim_{y\nearrow x}Z_t(y),\\
Z_{t-}(x+)=\lim_{y\searrow x}Z_{t-p(y-x)-}(y)\text{ and }Z_{t-}(x-)=\lim_{y\nearrow x}Z_{t+p(y-x)-}(y).
\end{gather*}

For $t\in\intervalleff{0}{T}$, we set
\begin{align*}
\chi_t^+ &=\enstq{x\in\intervalleff{-A}{A}}{F_t(x)>0\text{ and }Z_t(x+)=1},\\
\chi_t^- &=\enstq{x\in\intervalleff{-A}{A}}{F_t(x)>0\text{ and }Z_t(x-)=1},\\
\chi_t^0
&=\enstq{x\in\intervalleff{-A}{A}}{H_t(x)>0\text{ or }(F_t(x)=0\text{ and }
Z_t(x+)\neq Z_t(x-))}\cup\{-A,A\},\\
\chi_t &=\chi_t^+\cup\chi_t^-\cup\chi_t^0.
\end{align*}

For $x\in\mathcal{B}_M$ and $t\geq0$ we set
\begin{equation}
\tH_t(x)=\max(H_t(x),1-Z_t(x),1-Z_t(x+),1-Z_t(x-)).
\end{equation}
Actually, $Z_{t-}(x)$ always
equals either $Z_{t-}(x-)$ or $Z_{t-}(x+)$ and these can be distinct only at a point
where has occured a microscopic fire (that is if $x=X_q$ for some $q\in\{1\dots,n\}$
with $T_q<t$ and $Z_{T_q-}(X_q)<1$).

For all $x\in\intervalleoo{-A}{A}$ we define for all $t\in\intervalleff{0}{T}$
\begin{equation}\label{tau}
\tau_t(x)=\sup\enstq{s\leq t}{F_s(x)>0\text{ and }\tH_{s-}(x)=0}\vee0,
\end{equation}
which represents the last time before $t$ that a macroscopic fire has crossed $x$. Observe that 
\begin{align}
\text{for } x\not\in\mathcal{B}_M,\, Z_t(x) &=\min(t-\tau_t(x),1)\mtext{for all}
t\in\intervalleff{0}{T},\label{zpasbmp}\\
\text{for } q=1,\dots, n,\, Z_t(X_q) &=\min(t-\tau_t(X_q),1) \mtext{for all}
t\in\intervallefo{0}{T_q}.\label{zbmp}
\end{align}

We also define for all $i\in I^\la_A$ and all  $t\in\intervalleff{0}{T}$
\begin{equation}\label{crossfire}
\rho_t^{\la,\pi}(i) = \sup\enstq{s\leq t}{\eta^{\la,\pi}_{\al s-}(i)=2}
\end{equation}
where $\al \rho_t^{\la,\pi}(i)$ represents the last time before $\al t$ that the site $i$ has been burnt in the discrete process (with the convention $\eta^{\la,\pi}_{0-}(i)=2$ and $\eta^{\la,\pi}_0(i)=0$ for all $i\in I^\la_A$).

For $q\in\{1,\dots,n\}$, we define \emph{the death time of the right front of the $q$'s fire} as the time where the fire $q$ is stopped in the limit process, that is,
\begin{equation}\label{deathtime}
T_q^{D,+}=\inf\enstq{t\geq T_q}{F_t\left(X_q+\frac{t-T_q}{p}\right)=0}
\end{equation}
as well as \emph{the death position of the right front of the $q$'s fire} as the position where the fire $q$ is stopped in the limit process, that is,
\begin{equation}\label{deathpos}
X_q^{D,+}=X_q+\frac{T_q^{D,+}-T_q}{p}.
\end{equation}
Similarly, \emph{the death time and position of the left front of the $q$'s fire} are defined as
\[T_q^{D,-}=\inf\enstq{t\geq T_q}{F_t(X_q-\frac{t-T_q}{p})=0}\text{ and }X_q^{D,-}=X_q-\frac{T_q^{D,-}-T_q}{p}.\]
Observe that, if $Z_{T_q-}(X_q)<1$,  then $T_q^{D,-}=T_q=T_q^{D,+}$ and $X_q^{D,+}=X_q=X_q^{D,-}$.

We set
\begin{align}
\cB_M^D&\coloneqq\{X_1^{D,+},X_1^{D,-},\dots,X_n^{D,+},X_n^{D,-}\}\subset\cB_M\cup\cB_M^2,\label{deathposset}\\
\cT_M^D&\coloneqq\{T_1^{D,+},T_1^{D,-},\dots,T_n^{D,+},T_n^{D,-}\}\subset\cT_M\cup\cS_M\cup\cS_M^2.\label{deathtimeset}
\end{align}

Let $t\in\intervalleff{0}{T}$ and $q\in\{1,\dots,n\}$. If $t\in\intervallefo{0}{T_q^{D,+}+\vlp}$, we set
\[\Omega_{q,t}^{\la,\pi,+}=\{\forall s\in\intervalleff{T_q}{(T_q^{D,+}+\vlp)\wedge t}, \eta^{\la,\pi}_{\al s}(\lfloor\nl X_q\rfloor+i^{q,+}_{\al(s-T_q)})=2\}\]
and, if $t\in\intervalleff{T_q^{D,+}+\vlp}{T}$, we set
\[\Omega_{q,t}^{\la,\pi,+}=\Omega^{\la,\pi,+}_{q,T_q^{D,+}}\cap\{\exists s\in\intervalleff{T_q^{D,+}-\vlp}{T_q^{D,+}+\vlp}, \eta^{\la,\pi}_{\al s}(\lfloor\nl X_q\rfloor+i^{q,+}_{\al(s-T_q)})=0\}.\]

Similarly, we set, if $t\in\intervallefo{0}{T_q^{D,-}+\vlp}$, 
\[\Omega_{q,t}^{\la,\pi,-}=\{\forall s\in\intervalleff{T_q}{(T_q^{D,-}-\vlp)\wedge t}, \eta^{\la,\pi}_{\al s}(\lfloor\nl X_q\rfloor+i^{q,-}_{\al(s-T_q)})=2\}\]
and, if $t\in\intervalleff{T_q^{D,-}+\vlp}{T}$, we set
\[\Omega_{q,t}^{\la,\pi,-}=\Omega^{\la,\pi,-}_{q,T_q^{D,-}}\cap\{\exists s\in\intervalleff{T_q^{D,-}-\vlp}{T_q^{D,-}+\vlp}, \eta^{\la,\pi}_{\al s}(\lfloor\nl X_q\rfloor+i^{q,-}_{\al(s-T_q)})=0\}.\]

Finally, we set, for all $t\in\intervalleff{0}{T}$,
\[\Omega^{\la,\pi}_t=\bigcap_{q=1,\dots,n}\left(\Omega^{\la,\pi,+}_{q,t}\cap\Omega^{\la,\pi,-}_{q,t}\right).\]

The aim of this section is to prove the following Lemma.
\begin{lem}\label{heartlem}
Let $\alpha>\gamma>0$. For all $(\la,\pi)$ sufficiently close to the regime $\cR(p)$ in such a way that $4(\vlp+p(\ml+2\klp)/\nl)\leq\alpha$, $\Omega_T^{\la,\pi}$ a.s. holds on $\Omega(\alpha,\gamma,\la,\pi)$.
\end{lem}

We work on $\Omega(\alpha,\gamma,\la,\pi)$. We fix $\e_\alpha>0$ and $\la_\alpha\in\intervalleoo{0}{1}$ such that for all $\la\in\intervalleoo{0}{\la_\alpha}$ and all $\pi\geq1$ in such a way $|\nl/(\al\pi)-p|<\e_\alpha$, we have $\vlp+3p(\ml+2\klp)/\nl\leq\alpha$. Observe that for all $x,y\in\mathcal{B}_M\cup\mathcal{B}_M^2\cup\{-A,A\}$, with $x\neq
y$, we then have $[x]_{\la,\pi}\cap[y]_{\la,\pi}=\emptyset$. Clearly, $\Omega^{\la,\pi}_{T_1}$ a.s. holds, because no match falls in $I^\la_A$ before $\al T_1$. We will show that for $q=0,\dots,n-1$,
$\Omega^{\la,\pi}_{T_q}$ implies $\Omega^{\la,\pi}_{T_{q+1}}$. This will prove that
$\Omega^{\la,\pi}_{T_{n}}$ holds. The extension to $\Omega^{\la,\pi}_{T}$ will be straightforward and will be omitted.

We thus fix  $q\in\{0,\dots,n-1\}$ and assume $\Omega^{\la,\pi}_{T_q}$. Let $\cA_q$ be the set of points where a fire stops during the time interval $\intervalleoo{T_q}{T_{q+1}}$ that is, $(x,t)\in\cA_q$ if $(x,t)=(X_k^{D,+},T_k^{D,+})$ (or $(X_k^{D,-},T_k^{D,-})$) for some $k\leq q$ with $T_k^{D,+}$ (or $T_k^{D,+}$) in $\intervalleoo{T_q}{T_{q+1}}$. We then put 
\[\{(X_q^0,T_q^0),\dots,(X_q^{N_q+1},T_q^{N_q+1})\}=\cA_q\cup\{(X_q,T_q),(X_{q+1},T_{q+1})\}\]
ordered chronologically (thus $(X_q,T_q)=(X_q^0,T_q^0)$ and $(X_{q+1},T_{q+1})=(X_q^{N_q+1},T_q^{N_q+1})$).

We recall that if $Z_{T_l-}(X_l)=1$, for some $l\in\{1,\dots,n\}$, on $\Omega_M(\alpha)$, we have by construction,
\begin{itemize}
	\item $T^{D,+}_l\wedge T_l^{D,-}\geq T_l+3\alpha$;
	\item $Z_{T_l-}(y)=1$ for all $y\in\intervalleoo{X_l-3\alpha/p}{X_l+3\alpha/p}$;
	\item $F_{T_l(y)}(y)=1$ and $\tH_{T_l(y)-}(y)=0$ for all $y\in\intervalleoo{X_l^{D,-}}{X_l^{D,+}}$;
	\item for all $t\in\intervalleff{T_l}{T^{D,+}_l-3\alpha}$ and all $y\in\intervalleoo{X_l^+(t)}{X_l^+(t)+3\alpha/p}$, $\tH_{t}(y)=0$ (similar thing for $X_l^-(t)$);
	\item for all $t\in\intervallefo{T^{D,+}_l-3\alpha}{T^{D,+}_l}$ and all $y\in\intervalleoo{X_l^+(t)}{X_l^+(t)+(T_l^{D,+}-t)/p}$, $Z_t(y)=1$ (similar thing for $X_l^-(t)$).
%	\item $|x-y|>3\alpha/p$ for all $x,y\in\cB_M\cup\cB_{M,l}^1$.
\end{itemize}

Recall that on $\Omega_M(\alpha)$, for all $k\in\intervalleentier{0}{N_q}$,
\[T_q^{k+1}-T_q^k>3\alpha.\]

We decompose the proof in four stages. 
\begin{enumerate}[label=$-$]
	\item{\itshape Stage 0.} We deduce, on $\Omega(\alpha,\gamma,\la,\pi)$, the last time that a site has been burned.
	\item{\itshape Stage 1.} We prove that on $\Omega(\alpha,\gamma,\la,\pi)$, $\Omega_{T_q}^{\la,\pi}$ implies $\Omega_{T_q+4\vlp}^{\la,\pi}$.
	\item{\itshape Stage 2.} We prove that on $\Omega(\alpha,\gamma,\la,\pi)$, for $0\leq k< N_q$, $\Omega_{T_q^k+4\vlp}^{\la,\pi}$ implies $\Omega_{T_q^{k+1}+4\vlp}^{\la,\pi}$.
	\item{\itshape Stage 3.} We prove that on $\Omega(\alpha,\gamma,\la,\pi)$, $\Omega_{T_q^{N_q}+4\vlp}^{\la,\pi}$ implies $\Omega_{T_{q+1}}^{\la,\pi}$, which is the goal.
\end{enumerate}

In the whole proof, we repeatedly use the following estimates. For $k\in\{1,\dots,n\}$ and $t\geq  T_k$, there holds that, recall \eqref{klp}, \eqref{vlp} and \eqref{elp},
\begin{equation}\label{estimposplus}
\intervalleentier{\lfloor\nl X_k\rfloor+\lfloor\al\pi(t-T_k-\e_\la)\rfloor}{\lfloor\nl X_k\rfloor+\lfloor\al\pi(t-T_k+\e_\la)\rfloor}\subset\langle X_k^+(t)\rangle_{\la,\pi}
\end{equation}
which is the possible location of the right front of the fire $k$ at time $\al t$, recall Lemma \ref{propagation lemma p}, 
\begin{multline}\label{estimposplusavant}
\intervalleentier{\lfloor\nl X_k\rfloor+\lfloor\al\pi(t-\vlp-T_k-\e_\la)\rfloor}{\lfloor\nl X_k\rfloor+\lfloor\al\pi(t-\vlp-T_k+\e_\la)\rfloor} \\
\subset\intervalleentier{\lfloor\nl X_k^+(t)\rfloor-\ml-2\klp}{\lfloor\nl X_k^+(t)\rfloor-\ml}
\end{multline}
which is the possible location of the right front of the fire $k$ at time $\al (t-\vlp)$,
\begin{multline}\label{estimposplusapres}
\intervalleentier{\lfloor\nl X_k\rfloor+\lfloor\al\pi(t+\vlp-T_k-\e_\la)\rfloor}{\lfloor\nl X_k\rfloor+\lfloor\al\pi(t+\vlp-T_k+\e_\la)\rfloor} \\
\subset\intervalleentier{\lfloor\nl X_k^+(t)\rfloor+\ml}{\lfloor\nl X_k^+(t)\rfloor+\ml+2\klp}
\end{multline}
which is the possible location of the right front of the fire $k$ at time $\al (t+\vlp)$.

For $k\in\{1,\dots,n\}$ and $t\geq T_k$ there also holds true that
\begin{equation}\label{reach}
\lfloor\nl X_k\rfloor+\lfloor\al\pi(t-\elp-T_k+\e_\la)\rfloor \leq \lfloor\nl X_k^+(t)\rfloor
\end{equation}
and
\begin{align}
\lfloor\nl X_k\rfloor+\lfloor\al\pi(t-4\vlp-T_k+\e_\la)\rfloor &\leq \lfloor\nl X_k^+(t)\rfloor-\ml-3\klp,\label{entree}\\
\lfloor\nl X_k\rfloor+\lfloor\al\pi(t+4\vlp-T_k-\e_\la)\rfloor &\geq \lfloor\nl X_k^+(t)\rfloor+\ml+3\klp.\label{sortie}
\end{align}

Very similar estimations of course hold for $X_k^-(t)$. 

Finally, for all $i\in I^\la_A$ and all $k\in\{1,\dots,n\}$, there holds that
\begin{equation}\label{estimtps}
\left[T_k+\frac{|i-\lfloor\nl x\rfloor|}{\al\pi}-\e_\la,\,T_k+\frac{|i-\lfloor\nl x\rfloor|}{\al\pi}+\e_\la\right]\subset\left[T_k\left(\frac{i}{\nl}\right)-\elp,\,T_k\left(\frac{i}{\nl}\right)+\elp\right]
\end{equation}
which has to be seen as the time interval where a tree may be burn due to the fire $k$.
\begin{center}
{\bf \MakeUppercase{Stage 0.}}
\end{center}
In this Stage we fix some $s_0\in\intervalleff{0}{T}$ and work on $\Omega(\alpha,\gamma,\la,\pi)\cap\Omega^{\la,\pi}_{s_0}$. We deduce an estimate of the last time that a given site has been burned.

\begin{lem}\label{corestim}
Let $s_0\in\intervalleff{0}{T}$ and $q_0$ such that $s_0\in\intervallefo{T_{q_0}}{T_{q_0+1}}$. On $\Omega(\alpha,\gamma,\la,\pi)\cap\Omega_{s_0}^{\la,\pi}$, for all $(i,t)\in I^\la_A\times \intervalleff{0}{s_0}$ such that  $i\not\in\bigcup_{x\in\chi_t}\langle x\rangle_{\la,\pi}\cup\bigcup_{1\leq k\leq q_0}\left([X_k^{D,+}]_{\la,\pi}\cup[X_k^{D,-}]_{\la,\pi}\right)$,
\begin{enumerate}
	\item $\tau_t(i/\nl)=0$ if and only if $\rho^{\la,\pi}_t(i)=0$;
	\item if $\tau_t(i/\nl)=T_k(i/\nl)$, for some $k\in\{1,\dots,q_0\}$, then
\[\rho^{\la,\pi}_t(i)\in\left[T_k+\frac{|i-\lfloor\nl X_k\rfloor|}{\al\pi}-\e_\la,\,T_k+\frac{|i-\lfloor\nl X_k\rfloor|}{\al\pi}+\e_\la\right].\]
\end{enumerate}
\end{lem}

Observe that for $(i,t)$ be as in the statement, in the two cases, there holds that, using \eqref{estimtps},
\[\abs{\rho^{\la,\pi}_t(i)-\tau_t(i/\nl)}\leq \elp.\]

For $t\in\intervalleff{0}{s_0}$ and $x\in\intervalleoo{-A}{A}$ in such a way that $[x]_{\la,\pi}\cap[y]_{\la,\pi}=\emptyset$ for all $y\in\chi_t\cup\cB_M^D$, if $\tau_t(x)=T_l(x)$, for some $l\in\{1,\dots,n\}$, then by construction $\tau_t(i/\nl)=T_l(i/\nl)$ for all $i\in[x]_{\la,\pi}$.
Thus, using \eqref{estimposplusavant} and \eqref{estimposplusapres}, Lemma \ref{corestim} implies that for all $i\in(x)_\la$,
\[\abs{\rho^{\la,\pi}_t(i)-\tau_t(x)}\leq \vlp\]
whence, using \eqref{entree} and \eqref{sortie}, for all $i\in[x]_{\la,\pi}$, there holds that
\[\abs{\rho^{\la,\pi}_t(i)-\tau_t(x)}\leq 4\vlp.\]

\begin{proof}
Let $s_0\in\intervalleff{0}{T}$ and $q_0$ such that $s_0\in\intervallefo{T_{q_0}}{T_{q_0+1}}$.

\md

\noindent{\bf Step 1.} The key of the proof is the observation that if a site $i\in I^\la_A$ is burning at time $\al t\leq \al s_0$ then there exists $k\in\{1,\dots,q_0\}$ such that $\zeta^{\la,\pi,k}_{\al(t-T_k)}(i-\lfloor\nl X_k\rfloor)=2$ (a burning tree in the $(\la,\pi,A)-$FFP corresponds to a burning tree in some propagation process).

Indeed, assume that a match falls on $\lfloor\nl X_k\rfloor$ at time $\al T_k\leq \al t$. Recall that the propagation process ignited at $(X_k,T_k)$ is defined using the seed processes $(N^{S,\la,\pi}_t(i))_{t\geq0,i\in\zz}$ and the propagation processes $(N^{P,\la,\pi}_t(i))_{t\geq0,i\in\zz}$.  Thus, with our coupling, the right front of the fire in the propagation process $(\zeta^{\la,\pi,k}_t(i))_{t\geq0,i\in\zz}$ at some time $\al s$ is $i^{k,+}_{\al s}$ whence the (hypothetical) right front of the $(\la,\pi,A)-$FFP at time $\al (s+T_k)$ is $\lfloor\nl X_k\rfloor+i^{k,+}_{\al s}$. Recall that a spark in the propagation process $(\zeta^{\la,\pi,k}_t(i))_{t\geq 0,i\in\zz}$ corresponds to a site $i\in\zz$ where a seed has fallen between the instant at which $i$ propagates for the first time and the instant at which $i+1$ if $i\geq0$ or $i-1$ if $i\leq0$ propagates for the first time. On $\Omega^{P,T}_{\la,\pi}(X_k,T_k)$, such a spark has vacant neighbors. Thus, with our coupling, the site $\lfloor\nl X_k\rfloor+i$ is a spark in the $(\la,\pi)-$FFP (that is a burning tree which is not a front of a fire) if the site $i$ is a spark in the propagation process. Such a spark in the $(\la,\pi,A)-$FFP has inevitably vacant neighbors.

\md

\noindent{\bf Step 2.}  By Step 1, Lemma \ref{propagation lemma p} and \eqref{estimposplus}, we deduce that a burning tree at time $\al t$ in the $(\la,\pi,A)-$FFP necessarily belongs to
\[\intervalleentier{\lfloor\nl X_k\rfloor+\lfloor\al\pi(t-T_k-\e_\la)\rfloor}{\lfloor\nl X_k\rfloor+\lfloor\al\pi(t-T_k+\e_\la)\rfloor}\subset\langle X_k^+(t)\rangle_{\la,\pi}\]
or to 
\[\intervalleentier{\lfloor\nl X_k\rfloor-\lfloor\al\pi(t-T_k+\e_\la)\rfloor}{\lfloor\nl X_k\rfloor-\lfloor\al\pi(t-T_k-\e_\la)\rfloor}\subset\langle X_k^-(t)\rangle_{\la,\pi}\]
for some $k\in\{1,\dots,q_0\}$  such that $T_k\leq t$.

Conversely, if a site $i\in I^\la_A$ is burning at time $\al t\leq \al s_0$ then there is $k\in\{1,\dots,n\}$ such that, recalling \eqref{estimtps},
\[t\in\left[T_k+\frac{|i-\lfloor\nl X_k\rfloor|}{\al\pi}-\e_\la,\,T_k+\frac{|i-\lfloor\nl X_k\rfloor|}{\al\pi}+\e_\la\right]\subset\left(T_k\left(\frac{i}{\nl}\right)-\elp,\,T_k\left(\frac{i}{\nl}\right)+\elp\right).\]

\md

\noindent{\bf Step 3.} Next, we observe that if a site $j$ is burning at some time $\al u\leq \al s_0$, then there is $k\in\{1,\dots,q_0\}$ such that $u\in\intervalleff{T_k+(T^k_{j-\lfloor\nl X_k\rfloor}/\al)}{T_k+\frac{|j-\lfloor\nl X_k\rfloor|}{\al\pi}+\e_\la}$ and for all $s\in\intervalleff{T_k}{T_k+(T^k_{j-\lfloor\nl X_k\rfloor}/\al)}$ we have 
\[\eta^{\la,\pi}_{\al s}(\lfloor\nl X_ k\rfloor+i^{k,+}_{\al(s-T_k)})=2\]
if $j\geq\lfloor\nl X_k\rfloor$ while if $j\leq\lfloor\nl X_k\rfloor$, we have
\[\eta^{\la,\pi}_{\al s}(\lfloor\nl X_ k\rfloor+i^{k,-}_{\al(s-T_k)})=2.\]

Indeed, by construction, a fire starting on $\lfloor\nl X_k\rfloor$ at time $\al T_k$, for some $k\in\{1,\dots,q_0\}$, does not affect the site $j$ before $\al T_k+T^k_{j-\lfloor\nl X_k\rfloor}$ and by $\Omega^{P,T}_{\la,\pi}(X_k,T_k)$, as been checked on Step 1, does not affect the site $j$ after $\al T_k+\frac{|j-\lfloor\nl X_k\rfloor|}{\pi}+\al\e_\la$.

Assume \eg that $j\geq \lfloor\nl X_k\rfloor$ and that there is $s\in\intervallefo{T_k}{T_k+(T^k_{j-\lfloor\nl X_k\rfloor}/\al)}$ such that $\eta^{\la,\pi}_{\al s}(\lfloor\nl X_ k\rfloor+i^{k,+}_{\al(s-T_k)})=0$: the right front reaches a vacant site. Since sparks has vacant neighbors, the right front can not propagate more and is stopped (after a while, thanks to our coupling). Hence, the right front cannot reach $j$.

\md

\noindent{\bf Step 4.} Here we prove that for $i$ and $t$ be as in the statement and if $\tau_t(i/\nl)=T_k(i/\nl)>0$, for some $k\in\{1,\dots,q_0\}$, then $i$ is not affected (in the discrete process) by any fire during the time interval $\intervalleff{\al (T_k+\frac{|i-\lfloor\nl X_k\rfloor|}{\al\pi}+\e_\la)}{\al t}$.

Assume \eg that $i/\nl=X_k^+(T_k(i/\nl))\in\chi_{T_k(i/\nl)}^+$. We have $i/\nl\leq X_k^{D,+}$ and $T_k(i/\nl)\leq T_k^{D,+}$ whence $\lfloor\nl X_k\rfloor\leq i\leq\lfloor \nl X_k^{D,+}\rfloor-\ml-2\klp$ (because $i\not\in[X^{D,+}_k]_{\la,\pi}$) and $T_k(i/\nl)+\vlp\leq T^{D,+}_k$ (thanks to \eqref{estimposplusavant}).

So that there is $u_0\in\intervalleff{T_k+\frac{|i-\lfloor\nl X_k\rfloor|}{\al\pi}+\e_\la}{t}$ such that the site $i$ is burning at time $\al u_0$, it is necessary that there is $l\neq k$ such that $u_0\in\intervalleff{T_l+\frac{|i-\lfloor\nl X_l\rfloor|}{\al\pi}-\e_\la}{T_l+\frac{|i-\lfloor\nl X_l\rfloor|}{\al\pi}+\e_\la}$, recall Step 3, with
\[\eta^{\la,\pi}_{\al T_l+T^l_{j-\lfloor\nl X_l\rfloor}}(j)=2\text{ for all }j\in\intervalleentier{\lfloor\nl X_l\rfloor}{i}\]
if $i\geq\lfloor\nl X_k\rfloor$, or
\[\eta^{\la,\pi}_{\al T_l+T^l_{j-\lfloor\nl X_l\rfloor}}(j)=2\text{ for all }j\in\intervalleentier{i}{\lfloor\nl X_l\rfloor}\]
if $i\leq\lfloor\nl X_k\rfloor$.

If $i/\nl=X_l^+(T_l(i/\nl))$, then $i\geq \lfloor\nl X_l^{D,+}\rfloor+\ml+2\klp$ whence $T_l(i/\nl)\geq T_l^{D,+}+\vlp$, thanks to \eqref{estimposplusapres}. Indeed
\begin{enumerate}[label=(\alph*)]
	\item if $t\in\intervalleff{T_l+\frac{i-\lfloor\nl X_l\rfloor}{\al\pi}-\e_\la}{T_l+\frac{i-\lfloor\nl X_l\rfloor}{\al\pi}+\e_\la}$, then $i\in\langle X_l^+(t)\rangle_{\la,\pi}$. Since $i\not\in\bigcup_{x\in\chi_t}\langle x\rangle_{\la,\pi}$, we deduce that $X_l^+(t)\not\in\chi_t^+$ whence $T^{D,+}_l\leq t$. But $i\not\in[X_l^{D,+}]_{\la,\pi}$, thus $T^{D,+}_l<t-\vlp$, recall \eqref{estimposplusapres}, and $i\geq \lfloor\nl X_l^{D,+}\rfloor+\ml+2\klp$;
	\item if $t\geq T_l+\frac{i-\lfloor\nl X_l\rfloor}{\al\pi}+\e_\la\geq T_k+\frac{i-\lfloor\nl X_k\rfloor}{\al\pi}+\e_\la$ and $i\leq \lfloor\nl X_l^{D,+}\rfloor-\ml-2\klp$, using $\Omega^{\la,\pi}_t$, we deduce that $T_l(i/\nl)+\elp\leq t$ and $T_l(i/\nl)+\vlp\leq T_l^{D,+}$, recall \eqref{estimposplus} and  \eqref{estimposplusavant}. Thus, $F_{T_l(i/\nl)}(i/\nl)=1$. But by construction there holds that $|T_l(i/\nl)-T_k(i/\nl)|\geq 3\alpha$, thanks to $\Omega_M(\alpha)$, whence $T_l(i/\nl)\geq T_k(i/\nl)+3\alpha$, a contradiction since $\tau_t(i/\nl)=T_k(i/\nl)$. Thus, $i\geq \lfloor\nl X_l^{D,+}\rfloor+\ml+2\klp$, as desired.
\end{enumerate}

If $i/\nl=X_l^-(T_l(i/\nl))$, then $i\leq \lfloor\nl X_l^{D,-}\rfloor-\ml-2\klp$ whence $T_l(i/\nl)\geq T_l^{D,-}+\vlp$, thanks to \eqref{estimposplusapres}. Indeed
\begin{enumerate}[label=(\alph*')]
	\item if $t\in\intervalleff{T_l+\frac{\lfloor\nl X_l\rfloor-i}{\al\pi}-\e_\la}{T_l+\frac{\lfloor\nl X_l\rfloor-i}{\al\pi}+\e_\la}$, we conclude as in case (a) above that  $T^{D,-}_l\leq t-\vlp$ and $i\leq \lfloor\nl X_l^{D,-}\rfloor-\ml-2\klp$;
	\item if $t\geq T_l+\frac{\lfloor\nl X_l\rfloor-i}{\al\pi}+\e_\la\geq T_k+\frac{i-\lfloor\nl X_k\rfloor}{\al\pi}+\e_\la$ and $i\geq \lfloor\nl X_l^{D,-}\rfloor+\ml+2\klp$, using $\Omega^{\la,\pi}_t$, we deduce that $T_l(i/\nl)+\elp\leq t$ and $T_l(i/\nl)+\vlp\leq T_l^{D,-}$, thanks to \eqref{estimposplus} and \eqref{estimposplusavant}. Thus, $F_{T_l(i/\nl)}(i/\nl)=1$ and $Z_{T_l(i/\nl)-}(i/\nl)=1$ whence $T_l(i/\nl)\geq T_k(i/\nl)+1$, a contradiction since $\tau_t(i/\nl)=T_k(i/\nl)$. Thus $i\leq \lfloor\nl X_l^{D,-}\rfloor-\ml-2\klp$, as desired.
\end{enumerate}

Using $\Omega_t^{\la,\pi}$, we deduce that
\begin{itemize}
	\item if $i/\nl=X_l^+(T_l(i/\nl))$, there is $s\in\intervalleff{T_l^{D,+}-\vlp}{T_l^{D,+}+\vlp}$ such that $\eta^{\la,\pi}_{\al s}(\lfloor\nl X_l\rfloor+i^{l,+}_{\al(s-T_l)})=0$ whence $\eta^{\la,\pi}_{\al T_l+T^l_{j-\lfloor\nl X_l\rfloor}}(j)=0$ for some $j\in[X_l^{D,+}]_{\la,\pi}$, thanks to \eqref{estimposplusavant} and \eqref{estimposplusapres};
	\item if $i/\nl=X_l^-(T_l(i/\nl))$, there is $s\in\intervalleff{T_l^{D,-}-\vlp}{T_l^{D,-}+\vlp}$ such that $\eta^{\la,\pi}_{\al s}(\lfloor\nl X_l\rfloor+i^{l,-}_{\al(s-T_l)})=0$ whence $\eta^{\la,\pi}_{\al T_l+T^l_{j-\lfloor\nl X_l\rfloor}}(j)=0$ for some $j\in[X_l^{D,-}]_{\la,\pi}$, thanks to \eqref{estimposplusavant} and \eqref{estimposplusapres}.
\end{itemize}

Thus, the site $i$ can not be burned during the time interval $\intervalleff{T_k+\frac{|i-\lfloor\nl X_k\rfloor|}{\al\pi}+\e_\la}{t}$.

\md

\noindent{\bf Step 5.} Here we prove that for $i$ and $t$ be as in the statement, if $\tau_t(i/\nl)=T_k(i/\nl)>0$ for some $k\in\{1,\dots,n\}$, then $\eta^{\la,\pi}_{\al T_k+T^k_{i-\lfloor\nl X_k\rfloor}}(i)=2$.

Indeed, assume for example that $i/\nl=X^+_k(T_k(i/\nl))$, for some $k\in\{1,\dots,n\}$. By construction, there holds that $i/\nl\leq X^{D,+}_k$ and $i/\nl\leq X_k^+(s_0)$ whence $\lfloor\nl X_k\rfloor\leq i\leq \lfloor\nl X_k^{D,+}\rfloor-\ml-2\klp$ (because $i\not\in[X_k^{D,+}]_{\la,\pi}$) and $\lfloor\nl X_k\rfloor\leq i\leq \lfloor\nl X_k^+(s_0)\rfloor-\klp$ (because if $s_0\leq T_k^{D,+}$ then  $i\not\in\langle X_k^+(s_0)\rangle_{\la,\pi}$ and if $s_0> T_k^{D,+}$ then $\lfloor\nl X_k^+(s_0)\rfloor-\klp\geq \lfloor\nl X_k^{D,+}\rfloor-\ml-2\klp$). We distinguish two cases.
\begin{itemize}
	\item If $s_0\geq T_k^{D,+}-\vlp$, then by $\Omega^{\la,\pi}_{s_0}$, we deduce that $\eta^{\la,\pi}_{\al s}(\lfloor\nl X_k\rfloor+i^{k,+}_{\al (s-T_k)})=2$ for all $s\in\intervalleff{T_k}{T_k^{D,+}-\vlp}$. This also implies, thanks to \eqref{estimposplusavant}, that $\eta^{\la,\pi}_{\al T_k+T^k_{j-\lfloor\nl X_k\rfloor}}(j)=2$ for all $j\in\intervalleentier{\lfloor\nl X_k\rfloor}{\lfloor\nl X_k^{D,+}\rfloor-\ml-2\klp}$. It especially holds for $i$, thanks to the previous observation.
	\item If $s_0< T_k^{D,+}-\vlp$,  we deduce, by $\Omega^{P,T}(\la,\pi)$, \eqref{estimposplus} and the previous observation, that
\begin{equation}\label{blop}
\lfloor\nl X_k\rfloor\leq i\leq \lfloor\nl X_k^+(s_0)\rfloor-\klp\leq \lfloor\nl X_k\rfloor+\lfloor\al\pi(s_0-T_k-\e_\la)\rfloor\leq\lfloor\nl X_k\rfloor+i^{k,+}_{\al(s_0-T_k)}.
\end{equation}
Finally, by $\Omega^{\la,\pi}_{s_0}$, we have $\eta^{\la,\pi}_{\al u}(\lfloor\nl X_k\rfloor+i^{k,+}_{\al (u-T_k)})=2$ for all $u\in\intervalleff{T_k}{s_0}$ which implies the claim.
\end{itemize}

\md

\noindent{\bf Step 6.} We now conclude in the case $\tau_t(i/\nl)=T_k(i/\nl)>0$. By Step 4, we deduce that 
\[\rho^{\la,\pi}_t(i)\leq T_k+\frac{|i-\lfloor\nl X_k\rfloor|}{\al\pi}+\e_\la.\]
By Step 5, we deduce that $\rho^{\la,\pi}_t(i)\geq T_k+T^k_{i-\lfloor\nl X_k\rfloor|}/\al$
and conclude using $\Omega^{P,T}(\la,\pi)$ that
\[\rho^{\la,\pi}_t(i)\geq T_k+\frac{|i-\lfloor\nl X_k\rfloor|}{\al\pi}-\e_\la.\]

\md

\noindent{\bf Step 7.} Finally, if $\tau_t(i/\nl)=0$, we conclude, using similar argument as in Step 4 (recall that $i\not\in\bigcup_{1\leq k\leq q_0}\left([X_k^{D,+}]_{\la,\pi}\cup[X_k^{D,-}]_{\la,\pi}\right)$), that no fire can affect the site $i$ until $\al t$ and thus $\rho^{\la,\pi}_t(i)=0$.

Conversely, if $\rho^{\la,\pi}_t(i)=0$, then for all $l\in\{1,\dots,n\}$ such that $T_l(i/\nl)<t$, we necessarily have $F_{T_l(i/\nl)}(i/\nl)=0$ (else, applying $\Omega_t^{\la,\pi}$, one should have $\eta^{\la,\pi}_{\al T_l+T^l_{i-\lfloor\nl X_l\rfloor}}(i)=2$). This concludes the proof.
\end{proof}

\begin{center}
{\bf \MakeUppercase{Stage 1.}}
\end{center}

The aim of this stage is to prove that on $\Omega(\alpha,\gamma,\la,\pi)$, $\Omega_{T_q}^{\la,\pi}$ implies $\Omega_{T_q+4\vlp}^{\la,\pi}$. 

Observe that for all $i\in I^\la_A\setminus\{\lfloor\nl X_q\rfloor\}$, \[\eta^{\la,\pi}_{\al T_q}(i)=\eta^{\la,\pi}_{\al T_q-}(i)\]
while 
\[\eta^{\la,\pi}_{\al T_q}(\lfloor\nl X_q\rfloor)=2\indiq{\eta^{\la,\pi}_{\al T_q-}(\lfloor\nl X_q\rfloor)=1}.\]

First, we situate the burning trees at time $\al T_q$ for the $(\la,\pi,A)-$FFP.
\begin{lem}\label{macro match}
We work on $\Omega_{T_q}^{\la,\pi}\cap\Omega(\alpha,\gamma,\la,\pi)$.
\begin{enumerate}
	\item At time $\al T_q$, a burning tree which is not located at $\lfloor\nl X_q\rfloor$ necessarily belongs to $\langle x\rangle_{\la,\pi}$, for some $x\in\chi_{T_q}^+\cup \chi_{T_q}^-\subset\cB_{M,q}^1$, and is either at $\lfloor\nl X_k\rfloor+i_{\al(T_q-T_k)}^{k,+}$ or at
$\lfloor\nl X_k\rfloor+i_{\al(T_q-T_k)}^{k,-}$, for some $k< q$, or has vacant neighbors.
	\item If $X_k^+(T_q)=X_k+\frac{T_q-T_k}{p}\in\chi_{T_q}^+$ for some $k<q$, then $\eta^{\la,\pi}_{\al T_q}(\lfloor\nl X_k\rfloor+i^{k,+}_{\al (T_q-T_k)})=2$ and $\eta^{\la,\pi}_{\al T_q}(i)=1$ for all $i\in\intervalleentier{\lfloor\nl X_k\rfloor+i^{k,+}_{\al (T_q-T_k)}+1}{\lfloor\nl (X_k+2\alpha/p)\rfloor}$.
	\item If $X_k^-(T_q)=X_k-\frac{T_q-T_k}{p}\in\chi_{T_q}^-$ for some $k<q$, then $\eta^{\la,\pi}_{\al T_q}(\lfloor\nl X_k\rfloor+i^{k,-}_{\al (T_q-T_k)})=2$ and $\eta^{\la,\pi}_{\al T_q}(i)=1$ for all $i\in\intervalleentier{\lfloor\nl (X_k-2\alpha/p)\rfloor}{\lfloor\nl X_k\rfloor+i^{k,+}_{\al (T_q-T_k)}-1}$.
\end{enumerate}
\end{lem}

\begin{proof}
First, observe that, by $\Omega_M(\alpha)$, $|x-y|>3\alpha/p$ for all $x,y\in\cB_{M,q}^1\cup\cB_M^D$ with $x\neq y$. Hence, for all $x\in\cB_{M,q}^1$, there is a unique $k<q$ such that $x=X_k^+(T_q)$ or $x=X_k^-(T_q)$. 

In the whole proof, we work on $\Omega(\alpha,\gamma,\la,\pi)\cap\Omega^{\la,\pi}_{T_q}$.

\md

\noindent{\bf Step 1.} We first prove 1. As claimed in Step 2 in the proof of Lemma \ref{corestim}, due to $\Omega^{P,T}(\la,\pi)$, if a tree burns at time $\al T_q$ in the $(\la,\pi,A)-$FFP, it necessarily belongs to $\langle X_k^+(T_q)\rangle_{\la,\pi}$ or $\langle X_k^-(T_q)\rangle_{\la,\pi}$ for some $k<q$ and is either $\lfloor\nl X_k\rfloor+i_{\al(T_q-T_k)}^{k,+}$ or
$\lfloor\nl X_k\rfloor+i_{\al(T_q-T_k)}^{k,-}$, or has vacant neighbors. 

It remains to prove that if $x\in\cB^1_{M,q}\setminus(\chi_{T_q}^+\cup\chi_{T_q}^-)$, then there is no burning tree in $\langle x\rangle_{\la,\pi}$ at time $\al T_q$. We assume \eg that  $x=X_k^+(T_q)$ for some $k<q$. Since $x\not\in\chi_{T_q}^+$, there holds that $T_k^{D,+}\leq T_q$ whence $T_k^{D,+}\leq T_q-3\alpha$ and $x\geq X_k^{D,+}+3\alpha/p$, due to $\Omega_M(\alpha)$. We deduce, by $\Omega_{T_q}^{\la,\pi}$, that there is $s\in\intervalleff{T_k^{D,+}-\vlp}{T_k^{D,+}+\vlp}$ such that $\eta^{\la,\pi}_{\al s}(\lfloor\nl X_k\rfloor+i^{k,+}_{\al (s-T_k)})=0$ whence as usual (using \eqref{estimposplusavant} and \eqref{estimposplusapres}) that there is $j\in[X_k^{D,+}]_{\la,\pi}$ such that $\eta^{\la,\pi}_{\al T_k+T^k_{j-\lfloor\nl X_k\rfloor}}(j)=0$. Since $k$ is unique, we conclude, using same arguments as in Step 3 in the proof of Lemma \ref{corestim}, that there can not be burning tree in $\langle x\rangle_{\la,\pi}$ at time $\al T_q$ (because the right front has been stopped in $[X_k^{D,+}]_{\la,\pi}$ and  $\lfloor\nl x\rfloor-\klp\geq  \lfloor X_k^{D,+}\rfloor+\ml+2\klp$).

\md

\noindent{\bf Step 2.} We next prove 2. Let $k<q$. We set $x\coloneqq X_k^+(T_q)\in\cB_{M,q}^1$. Since $x\not\in\cB_m^D$, we have $T_k^{D,+}>T_q>T_k$ whence, by $\Omega_M(\alpha)$, $T_k^{D,+}>T_q+3\alpha>T_k+6\alpha$. Recall that, since $Z_{T_q-}(x)=1$, there holds that $T_q-\tau_{T_q-}(x)\geq 1$ whence $T_q-\tau_{T_q-}(x)\geq 1+3\alpha$, thanks to $\Omega_M(\alpha)$. We deduce that $Z_{T_q-}(y)=1$ and $T_q-\tau_{T_q-}(y)\geq 1+\alpha$ for all $y\in\intervalleff{x}{x+2\alpha/p}$. We set $\tau_{T_q-}(x)=T_l(x)$, for some $l\in\{0,\dots,q-1\}$.

Let us fix $i\in\intervalleentier{\lfloor\nl x\rfloor+\klp+1}{\lfloor\nl(x+2\alpha/p)\rfloor}$. Observing that $i\not\in\bigcup_{x\in\chi_{T_q}}\langle x\rangle_{\la,\pi}\cup\bigcup_{1\leq k\leq q}\left([X_k^{D,+}]_{\la,\pi}\cup[X_k^{D,-}]_{\la,\pi}\right)$, we deduce from Lemma \ref{corestim} and by \eqref{estimtps} that $\rho^{\la,\pi}_{T_q-}(i)\leq \tau_{T_q-}(i/\nl)+\elp$ whence 
\[\rho^{\la,\pi}_{T_q-}(i)\leq T_q-1-\alpha+\elp.\]
We conclude using $\Omega^S_3(\la,\pi)$ that $i$ is occupied at time $\al T_q$.

Let now $i\in\intervalleentier{\lfloor\nl X_k\rfloor+i^{k,+}_{\al (T_q-T_k)}+1}{\lfloor\nl x\rfloor+\klp}$. The site $i$ has not (yet) been affected by the fire $k$. Observe that if $\rho^{\la,\pi}_{T_q-}(i)=0$, since $T_q\geq 1$, we deduce by $\Omega_3^S(\la,\pi)$ that $i$ is occupied at time $\al T_q$. If $\rho^{\la,\pi}_{T_q}(i)>0$, by $\Omega^{P,T}(\la,\pi)$, we necessarily have $\rho^{\la,\pi}_{T_q}(i)\in\intervalleff{T_l+\frac{|i-\lfloor\nl X_l\rfloor|}{\al\pi}-\e_\la}{T_l+\frac{|i-\lfloor\nl X_l\rfloor|}{\al\pi}+\e_\la}$. We deduce as above that
\[\rho^{\la,\pi}_{T_q}(i)\leq T_l(i/\nl)+\elp\leq T_q-1-\alpha+\elp\]
and conclude using using $\Omega^S_3(\la,\pi)$ that $i$ is occupied at time $\al T_q$.

\md

\noindent{\bf Step 3.} Finally, point 3 is proved exactly as Point 2.
\end{proof}

We finally examine the $(\la,\pi,A)-$FFP around $\lfloor\nl X_q\rfloor$ at time $\al T_q$. 
\begin{lem}\label{micro match}
We work on $\Omega(\alpha,\gamma,\la,\pi)\cap\Omega_{T_q}^{\la,\pi}$.
\begin{enumerate}
	\item If $Z_{T_q-}(X_q)<1$ then there are $j_1,j_2\in(X_q)_\la$ such that
$j_1<\lfloor\nl X_q\rfloor<j_2$ and $\eta^{\la,\pi}_{\al s}(j_1)=\eta^{\la,\pi}_{\al
s}(j_2)=0$ for all $s\in\intervalleff{T_q}{T_q+\kappa_{\la,\pi}^0}$.
	\item If $Z_{T_q-}(X_q)=1$ then $\eta^{\la,\pi}_{\al T_q-}(i)=1$ for all $i\in\intervalleentier{\lfloor\nl(X_q-2\alpha/p)\rfloor}{\lfloor\nl(X_q+2\alpha/p)\rfloor}$.
\end{enumerate}
\end{lem}

\begin{proof}
First observe that $|x-X_q|>3\alpha/p$ for all $y\in\cB_{M,q}^1\cup\cB_m^D$ whence $F_{T_q-}(y)=0$ for all $y\in\intervalleoo{X_q-3\alpha/p}{X_q+3\alpha/p}$. We deduce, by Lemma \ref{macro match}, that there is no burning tree in $\intervalleentier{\lfloor\nl (X_q-2\alpha/p)\rfloor}{\lfloor\nl (X_q+2\alpha/p)\rfloor}$ at time $\al T_q-$ in the $(\la,\pi,A)-$FFP. We distinguish two cases.

\md

\noindent{\bf Step 1.} We first study the case $\tau_{T_q-}(X_q)>0$. By construction, recalling \eqref{zbmp} and since no match has fallen in $X_q$ during $\intervallefo{0}{T_q}$, there is a unique $k<q$ such that $\tau_{T_q-}(y)=T_k(y)$ for all $y\in\intervalleoo{X_q-3\alpha/p}{X_q+3\alpha/p}$.
\begin{description}
	\item[If $Z_{T_q-}(X_q)<1$,] then $Z_{T_q-}(X_q)=T_q-\tau_{T_q-}(X_q)<1$ whence $T_q-\tau_{T_q-}(X_q)<1-3\alpha$, thanks to $\Omega_M(\alpha)$. Recall that for $i\in(X_q)_\la$, seeds fall according to $(N^{S,q}_t(i-\lfloor\nl X_q\rfloor))_{t\geq0}$.
	
By Lemma \ref{corestim}, for all $i\in(X_q)_\la$, 
\begin{multline*}
\rho^{\la,\pi}_{T_q-}(i)\in\intervalleff{T_k+\frac{|i-\lfloor\nl X_k\rfloor|}{\al \pi}-\e_\la}{T_k+\frac{|i-\lfloor\nl X_k\rfloor|}{\al \pi}+\e_\la}\\
\subset\intervalleoo{\tau_{T_q-}(X_q)-\vlp}{\tau_{T_q-}(X_q)+\vlp}.
\end{multline*}
Since we work on $\Omega_2^S(\la,\pi)$ and since $T_q,\tau_{T_q-}(X_q)\in\cB_M\cup\cB_{M,q}^1$, there are some $-\ml<i_1<0<i_2<\ml$ such that	no seed has fallen on $i_1$ and on $i_2$ during $\intervalleff{\al(\tau_{T_q-}(X_q)-4\vlp)}{\al(T_q+4\vlp)}\supset\intervalleff{\al T_q}{\al(T_q+\kappa_{\la,\pi}^0)}$. All this implies that $i_1$ and $i_2$ remain vacant during (at least) the time interval $\intervalleff{\al T_q}{\al(T_q+\kappa_{\la,\pi}^0)}$.
	\item[If $Z_{T_q-}(X_q)=1$,] then $T_q-\tau_{T_q-}(X_q)\geq 1$ whence $T_q-\tau_{T_q-}(X_q)>1+3\alpha$ and $T_q-\tau_{T_q-}(y)>1+\alpha$ for all 
$y\in\intervalleoo{x-2\alpha/p}{x+2\alpha/p}$, thanks to $\Omega_M(\alpha)$.

By Lemma \ref{corestim}, for all $i\in\intervalleentier{\lfloor\nl(X_q-2\alpha/p)\rfloor}{\lfloor\nl(X_q+2\alpha/p)\rfloor}$, we deduce 
\[\rho^{\la,\pi}_{T_q-}(i)\in\intervalleff{T_k(i/\nl)-\elp}{T_k(i/\nl)+\elp}.\]
Since we work on $\Omega_3^S(\la,\pi)$, at least one seed has fallen on each site during $\intervalleff{\al(T_k(i/\nl)+\elp)}{\al(T_k(i/\nl)+1+\elp)}\subset \intervallefo{\al(T_k(i/\nl)+\elp)}{\al T_q}$.
Since, by definition, $i$ cannot been affected by a fire during $\intervalleoo{\rho^{\la,\pi}_{T_q-}(i)}{\al T_q}$, we deduce that the zone $\intervalleentier{\lfloor\nl(X_q-2\alpha/p)\rfloor}{\lfloor\nl(X_q+2\alpha/p)\rfloor}$ is completely filled at time $\al T_q-$.
\end{description}

\md

\noindent{\bf Step 2.} Here we study the case $\tau_{T_q-}(X_q)=0$. By $\Omega_M(\alpha)$, we have $\tau_{T_q-}(y)=0$ for all $y\in\intervalleoo{X_q-3\alpha/p}{X_q+3\alpha/p}$.
\begin{description}
	\item[If $Z_{T_q-}(X_q)<1$,] then $Z_{T_q}(X_q)=T_q<1$ whence $T_q<1-3\alpha$. Since we still work on $\Omega_2^S(\la,\pi)$, there are some $-\ml<i_1<0<i_2<\ml$ such that no seed has fallen on $i_1$ and on $i_2$ during $\intervalleff{0}{\al(T_q+4\vlp)}\supset\intervalleff{0}{\al(T_q+\kappa_{\la,\pi}^0)}$. Since we start with a vacant initial configuration, we deduce that $i_1$ and $i_2$ remain vacant during (at least) the time interval $\intervalleff{\al T_q}{\al(T_q+\kappa_{\la,\pi}^0)}$.
	\item[If $Z_{T_q-}(X_q)=1$,] then $T_q>1$ whence $T_q>1+3\alpha$. By Lemma \ref{corestim} we deduce that $\rho^{\la,\pi}_{T_q-}(i)=0$ for all $i\in\intervalleentier{\lfloor\nl(X_q-2\alpha/p)\rfloor}{\lfloor\nl(X_q+2\alpha/p)\rfloor}$ and thus 
\[\eta^{\la,\pi}_{\al T_q-}(i)=\min(N^{S,\la,\pi}_{\al T_q-}(i),1).\] Since we work on $\Omega_3^S(\la,\pi)$, at least one seed has fallen on each site during $\intervalleff{0}{\al}\subset \intervalleff{0}{\al T_q}$.
All this implies that the zone $\intervalleentier{\lfloor\nl(X_q-2\alpha/p)\rfloor}{\lfloor\nl(X_q+2\alpha/p)\rfloor}$ is completely filled at time $\al T_q-$.\qedhere
\end{description}
\end{proof}

The following corollary completes Stage 1.
\begin{cor}\label{propstep1}
On $\Omega(\alpha,\gamma,\la,\pi)$, $\Omega_{T_q}^{\la,\pi}$ implies $\Omega_{T_q+4\vlp}^{\la,\pi}$.
\end{cor}
\begin{proof}
Let $k<q$ such that $T_k^{D,+}\in\intervalleoo{T_q}{T_{q+1}}$. By $\Omega_M(\alpha)$, we have $T_q+3\alpha<T_k^{D,+}$ whence $T_q+4\vlp<T_k^{D,+}-\vlp$. Thus, no fire extinguishes during $\intervalleff{T_q}{T_q+4\vlp}$ (in the limit process). Hence, we have to prove that 
\begin{itemize}
	\item if $X_k^+(T_q)\in\chi_{T_q}^+$, for some $k\leq q$, then $\eta^{\la,\pi}_{\al t}(\lfloor\nl X_k\rfloor+i^{k,+}_{\al(t-T_k)})=2$ for all $t\in\intervalleff{T_q}{T_q+4\vlp}$;
	\item if $X_k^-(T_q)\in\chi_{T_q}^-$, for some $k\leq q$, then $\eta^{\la,\pi}_{\al t}(\lfloor\nl X_k\rfloor+i^{k,-}_{\al(t-T_k)})=2$ for all $t\in\intervalleff{T_q}{T_q+4\vlp}$;
	\item if $Z_{T_q-}(X_q)<1$, then the left and right fronts of the fire ignited at $(X_q,T_q)$ are stopped during the time interval $\intervalleff{\al T_q}{\al(T_q+\vlp)}$.
\end{itemize}

Observe that, on $\Omega^{P,T}(\la,\pi)$ there a.s. holds that, for all $k\leq q$,
\[0\leq i_{\al (T_q+4\vlp-T_k)}^{k,+}-i_{\al (T_q-T_k)}^{k,+}\leq4(\ml+2\klp)\leq \lfloor\nl\alpha/p\rfloor\]
and
\[-\lfloor\nl\alpha/p\rfloor\leq-4(\ml+2\klp)\leq i_{\al (T_q+4\vlp-T_k)}^{k,-}-i_{\al (T_q-T_k)}^{k,-}\leq 0.\]

All this implies that a front of a fire at time $\al T_q$, which belong to $\langle x\rangle_{\la,\pi}$ for some $x\in\cB_{M,q}^1\cup\{\nl X_q\}$, can not affect the zone outside $\intervalleentier{\lfloor\nl (x-\alpha/p)\rfloor}{\lfloor\nl (x+\alpha/p)\rfloor}$ during the time interval $\intervalleff{\al T_q}{\al(T_q+4\vlp)}$.

\md

\noindent{\bf Step 1.} Here we prove that for $k\leq q$ such that $x\coloneqq X_k^+(T_q)\in\chi_{T_q}^+$ then $\eta_{\al t}^{\la,\pi}(\lfloor\nl X_k\rfloor+ i_{\al (t-T_k)}^{k,+})=2$ for all $t\in\intervalleff{T_q}{T_q+4\vlp}$.

Indeed, by Lemma \ref{macro match}-2 if $k<q$ or by Lemma \ref{micro match}-2 if $k=q$, there holds that 
\[\eta^{\la,\pi}_{\al T_q}(\lfloor\nl X_k\rfloor+i^{k,+}_{\al (T_q-T_k)})=2\]
and 
\[\eta^{\la,\pi}_{\al T_q}(i)=1\text{ for all }i\in\intervalleentier{\lfloor\nl X_k\rfloor+i^{k,+}_{\al (T_q-T_k)}+1}{\lfloor\nl (x+2\alpha/p)\rfloor}.\]
But by the previous consideration, no fire, except this one, can affect the zone $\intervalleentier{\lfloor\nl X_k\rfloor+i^{k,+}_{\al (T_q-T_k)}+1}{\lfloor\nl(x+\alpha/p)\rfloor}$ during $\intervalleff{\al T_q}{\al (T_q+4\vlp)}$ and conversely, this fire can not affect the zone outside $\intervalleentier{\lfloor \nl (x-\alpha/p)\rfloor}{\lfloor\nl(x+\alpha/p)\rfloor}$. Hence, the right front of the fire $k$ is not stopped during the time interval $\intervalleff{\al T_q}{\al (T_q+4\vlp)}$, as desired.

\md

\noindent{\bf Step 2.} Let $k\leq q$, if $x\coloneqq X_k^-(T_q)\in\chi_{T_q}^-$ then $\eta_{\al t}^{\la,\pi}(\lfloor\nl X_k\rfloor+ i_{\al (t-T_k)}^{k,-})=2$ for all $t\in\intervalleff{T_q}{T_q+4\vlp}$. This can be shown using similar arguments as in Step 1 above.

\md

\noindent{\bf Step 3.} If $Z_{T_q-}(X_q)<1$, we have $T_q=T_q^{D,+}=T_q^{D,-}$. By Lemma \ref{micro match}-1, we deduce that there are $j_1,j_2\in(X_q)_\la$ such that
$j_1<\lfloor\nl X_q\rfloor<j_2$ and
\[\eta^{\la,\pi}_{\al s}(j_1)=\eta^{\la,\pi}_{\al
s}(j_2)=0\text{ for all }s\in\intervalleff{T_q}{T_q+\kappa_{\la,\pi}^0}.\]
Hence, on $\Omega^{P,T}_{\la,\pi}(X_q,T_q)$, $\eta^{\la,\pi}_{\al T_q+T^q_{j_1-\lfloor\nl X_q\rfloor}}(\lfloor\nl X_q\rfloor+i^{q,-}_{T^q_{j_1-\lfloor\nl X_q\rfloor}})=0$ because $T_q +T^q_{j_1-\lfloor\nl X_q\rfloor}/\al\leq T_q+\kappa_{\la,\pi}^0$ and $\eta^{\la,\pi}_{\al T_q+T^q_{j_2-\lfloor\nl X_q\rfloor}}(\lfloor\nl X_q\rfloor+i^{q,+}_{T^q_{j_1-\lfloor\nl X_q\rfloor}})=0$ because $T_q +T^q_{j_2-\lfloor\nl X_q\rfloor}/\al\leq T_q+\kappa_{\la,\pi}^0$, as desired.
\end{proof}

\begin{center}
{\bf \MakeUppercase{Stage 2.}}
\end{center}

In this Stage, we assume that $\cA_q\neq\emptyset$ and we fix $k\in\intervalleentier{0}{N_q-1}$. We work on $\Omega(\alpha,\gamma,\la,\pi)\cap\Omega^{\la,\pi}_{T_q^k+4\vlp}$ and prove that $\Omega_{T_q^{k+1}+4\vlp}$ a.s. holds. We repeatedly use the fact that no match falls in $\intervalleff{-A}{A}$ during the time interval $\intervalleff{T_q^k+4\vlp}{T_q^{k+1}+\alpha}$.
Observe that, for all $i\in I^\la_A$,
\[\eta^{\la,\pi}_{\al (T_q^k+4\vlp)-}(i)=\eta^{\la,\pi}_{\al (T_q^k+4\vlp)}(i).\]

We first examine the position of the burning trees of the $(\la,\pi,A)-$FFP at time $\al(T_q^k+4\vlp)$.

\begin{lem}\label{positionfire2}
We work on $\Omega(\alpha,\gamma,\la,\pi)\cap\Omega_{T_q^k+4\vlp}^{\la,\pi}$.
\begin{enumerate}
	\item At time $\al (T_q^k+4\vlp)$, a burning tree necessarily belongs to $\langle x\rangle_{\la,\pi}$, for some $x\in\chi_{T_q^k+4\vlp}^+\cup \chi_{T_q^k+4\vlp}^-$, and is either $\lfloor\nl X_l\rfloor+i_{\al(T_q^k+4\vlp-T_l)}^{l,+}$ or $\lfloor\nl X_l\rfloor+i_{\al(T_q^k+4\vlp-T_l)}^{l,-}$, for some $l\leq q$, or has vacant neighbors.
	\item If $X_l^+(T_q^k+4\vlp)\in\chi_{T_q^k+4\vlp}^+$ for some $l\leq q$, then $\eta^{\la,\pi}_{\al T_q^k+4\vlp}(\lfloor\nl X_l\rfloor+i^{l,+}_{\al (T_q^k+4\vlp-T_l)})=2$ and $\eta^{\la,\pi}_{\al T_q^k+4\vlp}(i)=1$ for all $i\in\intervalleentier{\lfloor\nl X_l\rfloor+i^{l,+}_{\al (T_q^k+4\vlp-T_l)}+1}{\lfloor\nl (X_l^+(T_q^k+4\vlp)+2\alpha/p)\rfloor}$.
	\item If $X_l^-(T_q^k+4\vlp)\in\chi_{T_q^k+4\vlp}^-$ for some $l\leq q$, then $\eta^{\la,\pi}_{\al T_q^k+4\vlp}(\lfloor\nl X_l\rfloor+i^{l,-}_{\al (T_q^k+4\vlp-T_l)})=2$ and $\eta^{\la,\pi}_{\al T_q^k+4\vlp}(i)=1$ for all $i\in\intervalleentier{\lfloor\nl (X_l^-(T_q^k+4\vlp)-2\alpha/p)\rfloor}{\lfloor\nl X_l\rfloor+i^{l,-}_{\al (T_q^k+4\vlp-T_l)}-1}$.
\end{enumerate}
\end{lem}
\begin{proof}
The proof is very similar to the proof of Lemma \ref{macro match}.

Indeed, we prove point 1 using $\Omega^{P,T}(\la,\pi)$ (as in the proof of Lemma \ref{corestim}) which implies that a burning tree necessarily belongs to $\langle X_l^+(T_q^k+4\vlp)\rangle_{\la,\pi}$ or $\langle X_l^-(T_q^k+4\vlp)\rangle_{\la,\pi}$ for some $l\leq q$ and is either $\lfloor\nl X_l\rfloor+i_{\al(T_q^k+4\vlp-T_l)}^{l,+}$ or $\lfloor\nl X_l\rfloor+i_{\al(T_q^k+4\vlp-T_l)}^{l,-}$ or has vacant neighbors. Furthermore, if $X_l^+(T_q^k+4\vlp)<X_{l'}^-(T_q^k+4\vlp)$, for some $l\neq l'$, we deduce, by $\Omega_M(\alpha)$, that
\[X_{l'}^-(T_q^k+4\vlp)-X_l^+(T_q^k+4\vlp)>(3\alpha-8\vlp)/p>\frac{5\alpha}{2p}.\]
Thus, as claimed in Step 3 in the proof of Lemma \ref{corestim}, for a site $i_0$ in $\langle X^+_l(T_q^k+4\vlp)\rangle_{\la,\pi}$ is burning at time $\al(T_q^k+4\vlp)$, since $l$ is unique, it is necessary that 
\[\eta_{\al T_l+T^l_{j-\lfloor\nl X_l\rfloor}}^{\la,\pi}(j)=2\text{ for all }j\in\intervalleentier{\lfloor\nl X_l\rfloor}{i_0}.\]
But, if $X^+_l(T_q^k+4\vlp)\not\in\chi_{T_q^k+4\vlp}^+$ then $T_l^{D,+}\leq T_q^k$. By $\Omega_{T_q^k+4\vlp}^{\la,\pi}$, we deduce that there is $j\in[X_l^{D,+}]_{\la,\pi}$ such that $\eta_{\al T_l+T^l_{j-\lfloor\nl X_l\rfloor}}^{\la,\pi}(j)=0$ (because there is $s\in\intervalleff{T_l^{D,+}-\vlp}{T_l^{D,+}+\vlp}$ such that $\eta^{\la,\pi}_{\al s}(\lfloor\nl X_l\rfloor+i^{l,+}_{\al (s-T_l)})=0$, recall \eqref{estimposplusavant} and \eqref{estimposplusapres}). Since $\langle X^+_l(T_q^k+4\vlp)\rangle_{\la,\pi}\cap [X^{D,+}_l]_{\la,\pi}=\emptyset$, thanks \eqref{sortie} (recall that $X_l^{D,+}=X_l^+(T_l^{D,+})$), there is no burning tree in $\langle X^+_l(T_q^k+4\vlp)\rangle_{\la,\pi}$ at time $\al(T_q^k+4\vlp)$.

Point 2 (or point 3) is proved as in Lemma \ref{macro match}. Indeed if $X_l^+(T_q^k+4\vlp)\in\chi_{T_q^k+4\vlp}^+$, then $T^{D,+}_l\geq T_q^{k+1}\geq T_q^k+3\alpha$ and $|X_l^+(T_q^k+4\vlp)-y|>2\alpha$ for all $y\in\cB_M^D$. Furthermore, on $\Omega_M(\alpha)$, by construction, we have 
\[\tH_{T_q^k+4\vlp}(y)=0\text{ for all }y\in\intervalleoo{X_l^+(T_q^k+4\vlp)}{X_l^+(T_q^k+4\vlp)+(3\alpha-4\vlp)/p)}\]
Thus, we prove that $\eta^{\la,\pi}_{\al(T_q^k+4\vlp)}(j)=1$ for all $j\in\intervalleentier{\lfloor\nl X_l\rfloor+i^{l,+}_{\al (T_q^k+4\vlp-T_l)}+1}{\lfloor\nl (X_l^+(T_q^k+4\vlp)+2\alpha/p)\rfloor}$ by distinguishing the cases $j\in\intervalleentier{\lfloor\nl X_l\rfloor+i^{l,+}_{\al (T_q^k+4\vlp-T_l)}+1}{\lfloor\nl X_l^+(T_q^k+4\vlp)\rfloor+\klp}$ and $j\in\intervalleentier{\lfloor\nl X_l^+(T_q^k+4\vlp)\rfloor+\klp}{\lfloor\nl (X_l^+(T_q^k+4\vlp)+2\alpha/p)\rfloor}$ (recalling that $X_l^+(T_q^k+4\vlp)\not\in\cB_M^D$).
\end{proof}

We then compute the cluster destroyed by a microscopic fire. We use the notation introduced in Lemma \ref{micro fire p}.
\begin{lem}\label{destroyed component}
Let $m\leq q$, if $Z_{T_m-}(X_m)<1$, we define $t_0=T_m-Z_{T_m-}(X_m)$,  which is nothing but $\tau_{T_m-}(X_m)$, recall \eqref{zbmp}. We then define, recall \eqref{iplus} and \eqref{imoins},
\begin{enumerate}[label=(\roman*)]
	\item if $t_0=T_l(X_m)>0$ for some $l<m$ and if $X_m=X_l^+(t_0)$, 
\[\cM\coloneqq (\lfloor\nl X_l\rfloor+i^{l,+}_{\al (t_0-\vlp-T_l)}-\lfloor\nl X_m\rfloor; t_0,T_m);\]
	\item if $t_0=T_l(X_m)>0$ for some $l<m$ and if $X_m=X_l^-(t_0)$, 
\[\cM\coloneqq (\lfloor\nl X_l\rfloor+i^{l,-}_{\al (t_0-\vlp-T_l)}-\lfloor\nl X_m\rfloor; t_0,T_m);\]
	\item if $t_0=0$,
\[\cM\coloneqq (0; 0,T_m),\]
\end{enumerate}
Then, working on $\Omega(\alpha,\gamma,\la,\pi)\cap\Omega^{\la,\pi}_{T_q^k+4\vlp}$, in each case, there holds that
\[(\eta^{\la,\pi}_{\al t}(i))_{t\in\intervalleff{t_0-\vlp}{T_m+\kappa_{\la,\pi}^0},i\in (X_m)_\la}=(\zeta^{\la,\pi,\mathcal{M},m}_{\al t}(i-\lfloor\nl X_m\rfloor))_{t\in\intervalleff{t_0-\vlp}{T_m+\kappa_{\la,\pi}^0},i\in  (X_m)_\la}\]
where the last process is defined as in Lemma \ref{micro fire p} using the seed processes family $(N^{S,m}_t(i))_{t\geq 0,i\in\zz}$ and the propagation processes family $(N^{P,m}_t(i))_{t\geq 0,i\in\zz}$.

This in particular implies that, still on $\Omega(\alpha,\gamma,\la,\pi)\cap\Omega^{\la,\pi}_{T_q^k+4\vlp}$,
\[C^P((\eta^{\la,\pi}_{t}(i))_{t\geq 0,i\in\zz},(X_m,T_m))=\intervalleentier{\lfloor\nl X_m\rfloor+i^g}{\lfloor\nl X_m\rfloor+i^d}\subset(X_m)_\la\]
where $\intervalleentier{i^g}{i^d}= C^P((\zeta^{\la,\pi,\mathcal{M},m}_t(i))_{t\geq 0,i\in \zz},(0,T_m))\subset\intervalleentier{-\ml}{\ml}$, recall Lemma \ref{micro fire p}.
\end{lem}

\begin{proof}
We only treat the case $(i)$. The case (ii) is of course similar and the case (iii) is easier.

We thus fix $1\leq l<m\leq q$ in such a way that 
\[\tau_{T_m-}(X_m)=t_0=T_l(X_m)\text{ and }X_m=X_l^+(t_0).\]
By $\Omega_M(\alpha)$, we deduce that $T_l^{D,+}>t_0+3\alpha$ and $T_m>t_0+3\alpha>T_l+6\alpha$. Hence, by construction, there holds that $Z_{t_0-\vlp}(y)=1$ for all $y\in\intervalleoo{X_m-\vlp/p}{X_m+2\alpha/p}$. Observe that $T_q^k+4\vlp\geq T_m+\kappa_{\la,\pi}^0$.

By $\Omega^{\la,\pi}_{T_q^k+4\vlp}$, we deduce that at time $\al(t_0-\vlp)$ the site 
\[\lfloor\nl X_l\rfloor+i^{l,+}_{\al (t_0-\vlp-T_l)}\in\intervalleentier{\lfloor\nl X_m\rfloor-\ml-2\klp}{\lfloor\nl X_m\rfloor-\ml}\]
is burning whereas the zone $\intervalleentier{\lfloor\nl X_l\rfloor+i^{l,+}_{\al (t_0-\vlp-T_l)}+1}{\lfloor\nl (X_m+2\alpha/p)\rfloor}$ is completely occupied (use very similar arguments as in Lemma \ref{positionfire2}-2, recalling that no match falls on $X_m$ during $\intervallefo{0}{T_m}\supset\intervallefo{0}{t_0}$). Comparing $(\eta^{\la,\pi}_{ t}(i))_{t\geq0,i\in \zz}$ and $(\zeta^{\la,\pi,\mathcal{M},m}_{t}(i-\lfloor\nl X_m\rfloor))_{t\geq0,i\in \zz}$, we deduce that they are equal on $\intervalleentier{\lfloor\nl X_l\rfloor+i^{k,+}_{\al (t_0-\vlp-T_l)}}{\lfloor\nl X_m\rfloor+\ml+2\klp}\supset(X_m)_\la$ at time $\al(t_0-\vlp)$.

Since, with our coupling, seeds fall according to the same processes and fires spread according to the same processes on $[X_m]_{\la,\pi}$, we deduce that the fire preads in the same way through $\intervalleentier{\lfloor\nl X_l\rfloor+i^{k,+}_{\al (t_0-\vlp-T_l)}}{\lfloor\nl X_m\rfloor+\ml+2\klp}$. Thus, $(\eta^{\la,\pi}_{ t}(i))_{t\geq0,i\in \zz}$ and $(\zeta^{\la,\pi,\mathcal{M},m}_{t}(i-\lfloor\nl X_m\rfloor))_{t\geq0,i\in \zz}$ remain equal on $\intervalleentier{\lfloor\nl X_l\rfloor+i^{k,+}_{\al (t_0-\vlp-T_l)}+1}{\lfloor\nl X_m\rfloor+\ml+2\klp}\supset(X_m)_\la$ during the time interval $\intervalleff{\al(t_0-\vlp)}{\al(t_0+4\vlp)}$, recall \eqref{sortie}. No other fire affect the zone $(X_m)_\la$ until a match falls on $\lfloor\nl X_m\rfloor$ at time $\al T_m$ because the zone $(X_m)_\la$ is protected by vacant site during the time interval $\intervalleff{\al(t_0+4\vlp)}{\al(T_m+\kappa_{\la,\pi}^0)}$ (by construction for $\zeta^{\la,\pi,\mathcal{M},m}$ and because in the $(\la,\pi,A)-$FFP, on $\Omega_2^S(\la,\pi)$, there are 
\[-\ml-2\klp<i_1<-\ml<\ml<i_2<\ml+2\klp\]
where no seed fall during the time interval $\intervalleoo{\al(t_0-4\vlp)}{\al(T_m+\kappa_{\la,\pi}^0)}$ and because the sites $\lfloor\nl X_m\rfloor+i_1$ and $\lfloor\nl X_m\rfloor+i_2$ has been made vacant by the fire $l$ during $\intervalleoo{\al(t_0-4\vlp)}{\al(t_0+4\vlp)}$, recall \eqref{entree} and \eqref{sortie}). Thus, since seeds fall on $[X_m]_{\la,\pi}$ according to the same processes, $(\eta^{\la,\pi}_{ t}(i))_{t\geq0,i\in \zz}$ and $(\zeta^{\la,\pi,\mathcal{M},m}_{t}(i-\lfloor\nl X_m\rfloor))_{t\geq0,i\in \zz}$ remain equal on $(X_m)_\la$ during $\intervallefo{\al(t_0+4\vlp)}{\al T_m}$. Finally, by $\Omega_2^S(\la,\pi)$, we deduce that there are some sites 
\[-\ml<i_3<0<i_4<\ml\]
where no seed fall during the time interval $\intervalleff{\al(t_0-\vlp)}{\al(T_m+\kappa_{\la,\pi}^0)}$ whence, as usual, in both cases, the sites  $\lfloor\nl X_m\rfloor+i_3$ and $\lfloor\nl X_m\rfloor+i_4$ are vacant during $\intervalleff{\al(t_0+\vlp)}{\al(T_m+\kappa_{\la,\pi}^0)}$, recall \eqref{estimposplusapres} (because they are made vacant by the fire $l$). Since the two processes evolve according to the same rules, the match falling on $\lfloor\nl X_m\rfloor$ at time $\al T_m$ destroys the same zone. Thus, $(\eta^{\la,\pi}_{ t}(i))_{t\geq0,i\in \zz}$ and $(\zeta^{\la,\pi,\mathcal{M},m}_{t}(i-\lfloor\nl X_m\rfloor))_{t\geq0,i\in \zz}$ are also equal on $(X_m)_\la$ during $\intervalleff{\al T_m}{\al (T_m+\kappa_{\la,\pi}^0)}$.

We deduce, on $\Omega_2^S(\la,\pi)$, as seen in {\bf Micro$(p)$} in Subsection \ref{application ffp}, that
\[C^P((\zeta^{\la,\pi,\mathcal{M},m}_t(i))_{t\geq 0,i\in \zz},(0,T_m))\coloneqq\intervalleentier{i^g}{i^d}\subset\intervalleentier{-\ml}{\ml}\]
and that there is no more burning tree in $(X_m)_\la$ at time $\al(T_m+\kappa_{\la,\pi}^0)$, whence
\[C^P((\eta^{\la,\pi}_{t}(i))_{t\geq 0,i\in\zz},(X_m,T_m))=\intervalleentier{\lfloor\nl X_m\rfloor+i^g}{\lfloor\nl X_m\rfloor+i^d}\subset(X_m)_\la.\qedhere\]
\end{proof}

We will need the following lemma.

\begin{lem}\label{noaffect}
Let $s_0\in\intervalleff{T_q^k+\alpha}{T_q^{k+1}+\alpha}$. We work on $\Omega(\alpha,\gamma,\la,\pi)\cap\Omega^{\la,\pi}_{T_q^k+4\vlp}$.
\begin{enumerate}
	\item In the limit process, if, for some $l\leq q$, $X_l^+(T_q^k+4\vlp)\in\chi_{T_q^k+4\vlp}^+$ in such a way that $s_0\leq T_l^{D,+}$ and 
\begin{equation}\label{condition}
F_{T_q^k+4\vlp}(y)=0\text{ for all }y\in\intervalleoo{X_l^+(T_q^k+4\vlp)}{X_l^+(s_0+\alpha)},
\end{equation}
then, in the discrete process, the site $\lfloor\nl X_l^+(s_0)\rfloor$ is not affected by a fire during the time interval $\intervalleff{\al (T_q^k+4\vlp)}{\al  (s_0-\elp)}$.
	\item In the limit process, if, for some $l\leq q$, $X_l^-(T_q^k+4\vlp)\in\chi_{T_q^k+4\vlp}^-$ in such a way that $s_0\leq T_l^{D,-}$ and $F_{T_q^k+4\vlp}(y)=0$ for all $y\in\intervalleoo{X_l^-(s_0+\alpha)}{X_l^-(T_q^k+4\vlp)}$, then, in the discrete process, the site $\lfloor\nl X_l^-(s_0)\rfloor$ is not affected by a fire during the time interval $\intervalleff{\al (T_q^k+4\vlp)}{\al  (s_0-\elp)}$.
\end{enumerate}
\end{lem}

\begin{proof}
It of course suffices to prove 1.

First, using \eqref{condition}, we deduce that
\[\intervalleoo{X_l^+(T_q^k+4\vlp)}{X_l^+(s_0+\alpha)}\cap\left(\chi^+_{T_q^k+4\vlp}\cup\chi^-_{T_q^k+4\vlp}\right)=\emptyset.\]
Hence, by Lemma \ref{positionfire2}-1 and by \eqref{estimposplus}, we deduce that there is no burning tree in $\intervalleentier{\lfloor\nl X_l\rfloor+i^{l,+}_{\al (T_q^k+4\vlp-T_l)}+1}{\lfloor\nl X_l^+(s_0+\alpha)\rfloor-\klp}$ at time $\al (T_q^k+4\vlp)$.

On the one hand, on $\Omega(\alpha,\gamma,\la,\pi)$, recall \eqref{reach} and Lemma \ref{propagation lemma p}, there holds that 
\[\lfloor\nl X_l\rfloor+i^{l,+}_{\al (s_0-\elp-T_l)}< \lfloor\nl X_l^+(s_0)\rfloor.\]
Thus the right front of the fire $l$ does not reach $\lfloor\nl X_l^+(s_0)\rfloor$ before $\al (s_0-\elp)$. Hence, no fire coming from the left can affect the site $\lfloor\nl X_l^+(s_0)\rfloor$ during the considered time interval.

On the other hand, no fire coming from the right can affect $\lfloor\nl X_l^+(s_0)\rfloor$ before $\al (s_0-\elp)$. Indeed, since there is no fire in $\intervalleentier{\lfloor\nl X_l^+(s_0)\rfloor}{\lfloor\nl X_l^+(s_0+\alpha)\rfloor-\klp}$ at time $\al (T_q^k+4\vlp)$, we deduce, by $\Omega(\alpha,\gamma,\la,\pi)$, that if a fire affect the site $\lfloor\nl X_l^+(s_0)\rfloor$ during the time interval $\intervalleff{\al(T_q^k+4\vlp)}{\al(s_0-\elp)}$, it is necessarily a left front. But, by construction, if $X_{l'}^-(T_q^k+4\vlp)\in\chi^-_{T_q^k+4\vlp}$, for some $l'\leq q$, then $X_l^+(s_0)\leq X_{l'}^-(s_0)$ (because $s_0\leq T_l^{D,+}$). By \eqref{reach} and Lemma \ref{propagation lemma p}, we then have
\[\lfloor\nl X_{l'}\rfloor+i^{l',-}_{\al (s_0-\elp-T_{l'})}>\lfloor\nl X_{l'}^-(s_0)\rfloor\geq \lfloor\nl X_{l}^-(s_0)\rfloor.\]
Hence, no fire coming from the right can affect $\lfloor\nl X_l^+(s_0)\rfloor$ during the considered time interval.
\end{proof}

The two following lemmas are the keys of this Stage. The first of them insure that a fire indeed propagates. The second insure that a fire is stopped when it meet a microscopic zone.

\begin{lem}\label{propastep2}
Let $s_0\in\intervalleff{T_q^k+\alpha}{T_q^{k+1}+\alpha}$. We work on $\Omega(\alpha,\gamma,\la,\pi)\cap\Omega^{\la,\pi}_{T_q^k+4\vlp}$. 
\begin{enumerate}
	\item In the limit process, if $X_l^+(T_q^k+4\vlp)\in\chi_{T_q^k+4\vlp}^+$ for some $l\leq q$ in such a way that $s_0\leq T_l^{D,+}$ and $F_{T_q^k+4\vlp}(y)=0$ for all $y\in\intervalleoo{X_l^+(T_q^k+4\vlp)}{X_l^+(s_0+\alpha)}$,
then 
\[\eta^{\la,\pi}_{\al T_l+T^l_{i-\lfloor\nl X_l\rfloor}}(i)=2\]
for all $i\in\intervalleentier{\lfloor\nl X_l\rfloor+i^{l,+}_{\al(T_q^k+4\vlp-T_l)}}{\lfloor\nl X_l^+(s_0)\rfloor-\ml-2\klp}$.
	\item In the limit process, if $X_l^-(T_q^k+4\vlp)\in\chi_{T_q^k+4\vlp}^-$ for all $y\in\intervalleoo{X_l^-(s_0+\alpha)}{X_l^-(T_q^k+4\vlp)}$,
then $\eta^{\la,\pi}_{\al T_l+T^l_{i-\lfloor\nl X_l\rfloor}}(i)=2$ for all $i\in\intervalleentier{\lfloor\nl X_l^-(s_0)\rfloor+\ml+2\klp}{\lfloor\nl X_l\rfloor+i^{l,-}_{\al(T_q^k+4\vlp-T_l)}}$.
\end{enumerate}
\end{lem}

We have the propagation of the fire $l$ only to $\lfloor\nl X_l^+(s_0)\rfloor-\ml-2\klp$. Unfortunately, in the case where $s_0=T_q^{k+1}=T^{D,+}_l$ and $X_l^+(T_q^{k+1})=X_q^{k+1}=X_l^{D,+}$ (that is if the right front of the fire $l$ is stopped at time $T_q^{k+1}$ in the limit process), we can not say anything more on the discrete process, due to \eqref{estimposplusavant}. We will show below (see Lemma \ref{stop}) that, in this special case, the zone $\intervalleentier{\lfloor\nl X_q^{k+1}\rfloor-\ml-2\klp}{\lfloor\nl X_q^{k+1}\rfloor-\ml}$ is actually completely occupied at time $\al (T_q^{k+1}-4\vlp)$. This will imply that the fire propagates indeed until $\al(T_q^{k+1}-\vlp)$, thanks to \eqref{estimposplusavant}.

\begin{proof}
Lemma \ref{noaffect} shows that the site $\lfloor\nl X_l^+(s_0)\rfloor$ is not affected by a fire during $\intervalleff{\al (T_q^k+4\vlp)}{\al (s_0-\elp)}$. Hence, no fire coming from the right affect the zone $\intervalleentier{\lfloor\nl X_l\rfloor+i^{l,+}_{\al(T_q^k+4\vlp-T_l)}+1}{\lfloor\nl X_l^+(s_0)\rfloor}$ during the time interval $\intervalleff{\al(T_q^k+4\vlp)}{\al (s_0-\vlp)}$ and, conversely, the right front of the fire $l$ does not affect the zone on the right of $\lfloor\nl X_l^+(s_0)\rfloor$. Since $\eta^{\la,\pi}_{\al (T_q^k+4\vlp)}(\lfloor\nl X_l\rfloor+i^{l,+}_{\al(T_q^k+4\vlp-T_l)})=2$, thanks to Lemma \ref{positionfire2}-2, it then suffices to show that for all $i\in\intervalleentier{\lfloor\nl X_l\rfloor+i^{l,+}_{\al(T_q^k+4\vlp-T_l)}+1}{\lfloor\nl X_l^+(s_0)\rfloor-\ml-2\klp}$, 
\[\eta^{\la,\pi}_{\al T_l+T^l_{i-\lfloor\nl X_l\rfloor}-}(i)=1\]
\ie the site $i$ is occupied just before that the right front of the fire $l$ reaches $i$.

Observe that by construction, in the limit process, no fire affect the site $i/\nl\in\intervalleoo{X_l^+(T_q^k+4\vlp)}{X_l^+(s_0)}$ during $\intervalleoo{T_q^k+4\vlp}{T_l(i/\nl)}$ whence in the discrete process, no fire can affect the site $i\in\intervalleentier{\lfloor\nl X_l\rfloor+i^{l,+}_{\al(T_q^k+4\vlp-T_l)}+1}{\lfloor\nl X_l^+(s_0)\rfloor-\ml-2\klp}$ during $\intervallefo{\al(T_q^k+4\vlp)}{\al T_l+T^l_{i-\lfloor\nl X_l\rfloor}}$. All this implies that for all  $i/\nl\in\intervalleoo{X_l^+(T_q^k+4\vlp)}{X_l^+(s_0)}$, we  have 
\[\tau_{T_l(i/\nl)-}(i/\nl) = \tau_{T_q^k+4\vlp}(i/\nl)\]
while for all $i\in\intervalleentier{\lfloor\nl X_l\rfloor+i^{l,+}_{\al(T_q^k+4\vlp-T_l)}+1}{\lfloor\nl X_l^+(s_0)\rfloor-\ml-2\klp}$ we have 
\[\rho^{\la,\pi}_{T_l+T^l_{i-\lfloor\nl X_l\rfloor}/\al-}(i)=\rho^{\la,\pi}_{T_q^k+4\vlp}(i).\]
 
\md

\noindent{\bf Step 1.} Here we show that for all $j\in\intervalleentier{\lfloor\nl X_l\rfloor+i_{\al(T_q^k+4\vlp-T_l)}^{l,+}+1}{\lfloor\nl X_l^+(T_q^k+4\vlp)\rfloor+\klp}$, we have $\eta^{\la,\pi}_{\al T_l+T^l_{j-\lfloor\nl X_l\rfloor}-}(j)=1$.

In Lemma \ref{positionfire2}-2 we have proved that $\eta^{\la,\pi}_{\al (T_q^k+4\vlp)}(j)=1$ for all $j\in\intervalleentier{\lfloor\nl X_l\rfloor+i^{l,+}_{\al (T_q^k+4\vlp-T_l)}+1}{\lfloor\nl X_l^+(T_q^k+4\vlp)\rfloor+\klp}$. The result follows from the previous observation.

\md

\noindent{\bf Step 2.} Here we show that for all
$j\in\intervalleentier{\lfloor\nl X_l^+(T_q^k+4\vlp)\rfloor+\klp+1}{\lfloor\nl X_l^+(s_0)\rfloor-\ml-2\klp}\setminus\cup_{y\in\mathcal{B}_M^D}[y]_{\la,\pi}$, we have 
$\eta_{\al T_l+T^l_{j-\lfloor\nl X_l\rfloor}-}^ {\la,\pi}(j)=1$.

Indeed, on the one hand, $Z_{T_l(j/\nl)-}(j/\nl)=1$, then $T_l(j/\nl)-\tau_{T_l(j/\nl)-}(j/\nl)>1$ whence 
\[\tau_{T_l(j/\nl)-}(j/\nl)<T_l(j/\nl)-1-3\alpha,\]
thanks to $\Omega_M(\alpha)$. On the other hand, recalling that there is no burning tree in $\intervalleentier{\lfloor\nl X_l^+(T_q^k+4\vlp)\rfloor+\klp+1}{\lfloor\nl X_l^+(s_0)\rfloor}$ at time $\al(T_q^k+4\vlp)$ (thus $j\not\in\bigcup_{x\in\chi_{T_q^k+4\vlp}}
\langle x\rangle_{\la,\pi}$) and since $j\not\in\bigcup_{x\in\cB_M^D}[x]_{\la,\pi}$, we deduce from Lemma \ref{corestim} and by \eqref{estimtps} that 
\[\rho^{\la,\pi}_{T_q^k+4\vlp}(j)\leq \tau_{T_q^k+4\vlp}(j/\nl)+\elp.\]

All this implies that
\[\rho^{\la,\pi}_{T_l+T^l_{j-\lfloor\nl X_l\rfloor}/\al-}(j)\leq T_l(j/\nl)-1-3\alpha+\elp.\]
Recalling that $T_l+T^l_{j-\lfloor\nl X_l\rfloor}/\al\geq T_l(j/\nl)-\elp$, thanks to \eqref{estimtps}, and $\elp<\alpha$, we conclude using $\Omega_3^S(\la,\pi)$ that the site $j$ is occupied at time $\al T_l+T^l_{j-\lfloor\nl X_l\rfloor}-$.

\md

\noindent{\bf Step 3.} Here we show that for all $y\in\cB_M^D\cap\intervalleoo{X_l^+(T_q^k+4\vlp)}{X_l^+(s_0)}$, for all $j\in[y]_{\la,\pi}$, there holds $\eta^{\la,\pi}_{\al(T_l(y)-4\vlp)}(j)=1$. This will conclude Lemma \ref{propastep2} since $\al T_l+T^l_{j-\lfloor\nl X_l\rfloor}\geq \al(T_l(y)-4\vlp)$ for all $j\in[y]_{\la,\pi}$, thanks to \eqref{entree}.
        
{\it Preliminary considerations.}
Let $y\in\cB_M^D\cap\intervalleoo{X_l^+(T_q^k+4\vlp)}{X_l^+(s_0)}$. Since $X_l^+(s_0)\leq X_l^{D,+}$, we have $y\leq X^{D,+}_l-3\alpha/p$. We may assume $X_l^+(s_0)\geq y+\alpha/p$, by $\Omega_M(\alpha)$. We know that
$\tH_{T_l(y)-}(y)=0$, whence $H_{T_l(y)-}(y)=0$ and
$Z_{T_l(y)-}(y)=Z_{T_l(y)-}(y_+)=Z_{T_l(y)-}(y_-)=1$. This implies that
$T_l(y)\geq 1$ (because $Z_t(y)=t$ for all $t<1$ and all $y\in
[-A,A]$).
      
As pointed out in Step 2, we have, setting $j_g=\lfloor\nl y\rfloor-\ml-2\klp-1$ and observing that $T_l+T^l_{j_g-\lfloor\nl X_l\rfloor}/\al\geq T_l(y)-4\vlp\geq T_q^k+4\vlp$, 
\begin{multline*}
\rho^{\la,\pi}_{T_l(y)-4\vlp}(j_g)\leq T_l(j_g/\nl)-1-3\alpha+\elp=T_l(y)-1-3\alpha+\elp-p\frac{\ml+2\klp+1}{\nl}.
\end{multline*}
Using a similar argument for $j_d=\lfloor\nl y\rfloor+\ml+2\klp+1$, we conclude that no match falling outside $[y]_{\la,\pi}=\intervalleentier{j_g+1}{j_d-1}$ can affect $[y]_{\la,\pi}$ during $\intervalleoo{\al(T_l(y)-1-\alpha)}{\al (T_l(y)-4\vlp)}$, because 
\[\rho^{\la,\pi}_{T_l(y)-4\vlp}(j_g)+2\e_\la+2\frac{\ml+2\klp}{\al\pi}\leq T_l(y)-1-\alpha\]
and because to affect a site $i\in [y]_{\la,\pi}$, a match falling outside $[y]_{\la,\pi}$ needs to cross $j_d$ or $j_g$ and thus must verify, recall Lemma \ref{corestim},
\[\rho^{\la,\pi}_{T_l(y)-4\vlp}(i)\leq (\rho^{\la,\pi}_{T_l(y)-4\vlp}(j_g/\nl)\vee \rho^{\la,\pi}_{T_l(y)-4\vlp}(j_d/\nl))+2(\kappa_{\la,\pi}^0+\elp).\]

\md

{\it Case 1.} First assume that $y\in\cB_M^2$. Then we know that no match has fallen on
$[y]_{\la,\pi}$ during $\intervallefo{0}{\al T_l(y)}$. Due to the preliminary considerations, we deduce that no fire at all has concerned $[y]_{\la,\pi}$ during
$\intervalleoo{\al(T_l(y)-1-\alpha)}{\al (T_l(y)-4\vlp)}$. Using
$\Omega_3^ S(\la,\pi)$, we conclude that $[y]_{\la,\pi}$ is completely occupied at time
$\al(T_l(y)-4\vlp)$.

\md

{\it Case 2.} Assume that $y=X_m\in\cB_M$ with $m\geq q+1$. Then we know that no match has fallen on $[X_m]_{\la,\pi}$ during $\intervallefo{0}{\al T_l(X_m)}\subset\intervallefo{0}{\al T_m}$. We conclude as in Case 1 using $\Omega^S_3(\la,\pi)$ that the zone $[X_m]_{\la,\pi}$ is completely occupied at time $\al(T_l(y)-4\vlp)$.

\md

{\it Case 3.} Assume that $y=X_m\in\cB_M$ with $m\leq q$ and $Z_{T_m-}(X_m)=1$, so that there
already has been a macroscopic fire in $[X_m]_{\la,\pi}$ (at time $\al T_m$). There is no more burning tree in $[X_m]_{\la,\pi}$ at time $\al(T_m+4\vlp)$, thanks to $\Omega^{P,T}_{\la,\pi}(X_m,T_m)$ and \eqref{sortie}. Since
$Z_{T_m}(X_m)=0$ and $Z_{T_l(X_m)-}(X_m)=1$, we deduce that
$T_l(X_m)-T_m\geq1$, whence $T_l(X_m)-T_m\geq1+3\alpha$ as usual. We
conclude as in case 1 that no fire at all has concerned $[X_m]_{\la,\pi}$ during
$\intervalleoo{\al(T_l(X_m)-1-\alpha)}{\al (T_l(X_m)-4\vlp)}$, which implies
the claim by $\Omega_3^ S(\la,\pi)$.

\md

{\it Case 4.} Assume that $y=X_m\in\cB_M$ with $m\leq q$ and $Z_{T_m-}(X_m)<1$ and
$T_l(X_m)-T_m\geq 1$, whence $T_l(X_m)-T_m\geq 1+3\alpha$ due to
$\Omega_M(\alpha)$. Then there already has been a microscopic fire in $[X_m]_{\la,\pi}$ (at
time $\al T_m$). There is no more burning tree in $[X_m]_{\la,\pi}$ at time $\al(T_m+4\vlp)$, thanks to $\Omega^{P,T}_{\la,\pi}(X_m,T_m)$ and \eqref{sortie}. No match falls on $[X_m]_{\la,\pi}$ during $\intervalleoo{\al
(T_m+4\vlp)}{\al (T_l(X_m)-4\vlp)}\supset
\intervalleoo{\al(T_l(X_m)-1-\alpha)}{\al (T_l(X_m)-4\vlp)}$ and we conclude
as in case 1.

\md

{\it Case 5.} Assume that $y=X_m\in\cB_M$ with $m\leq q$ and $Z_{T_m-}(X_m)<1$ and
$T_l(X_m)-T_m< 1$, whence $T_l(X_m)-T_m\leq 1-3\alpha$ due to
$\Omega_M(\alpha)$. There has been a microscopic fire in $[X_m]_{\la,\pi}$ (at time $\al
T_m$). Since $H_{T_l(X_m)}(X_m)=0$, we deduce that
$T_m+Z_{T_m-}(X_m)\leq T_l(X_m)$, whence $T_m+Z_{T_m-}(X_m)\leq
T_l(X_m)-3\alpha$ by $\Omega_M(\alpha)$. We define $\cM=(i_0;t_0,T_m)$ as in Lemma \ref{destroyed component}.

Consider the zone $C^P\coloneqq C^P((\eta^{\la,\pi}_t(i))_{t\geq0,i\in\zz},(X_m,T_m))\subset(X_m)_\la$ destroyed by
the match falling on $\lfloor\nl X_m\rfloor$ at time $\al T_m$. This zone is completely occupied at time
$\al (T_m+\Theta^{\la,\pi,m}_\cM)$: this follows from the definition
of $\Theta^{\la,\pi,m}_{\cM}$ (see Lemma \ref{micro fire p}), from Lemma \ref{destroyed component} and from the preliminary considerations (because $T_m\geq T_l(X_m)-1-\alpha$). Using $\Omega^S_4(\gamma,\la,\pi)$, we deduce that
$T_m+\Theta^{\la,\pi,m}_\cM \leq T_m+Z_{T_m-}(X_m)+\gamma <
T_l(X_m)-2\alpha$, since $\gamma<\alpha$. Hence $C^P$ is completely occupied at time $\al(T_l(X_m)-4\vlp)$.

Consider now $i\in [X_m]_{\la,\pi} \setminus C^P$. Then $i$ has not been killed by the fire
starting at $\lfloor \nl X_m \rfloor$. Thus $i$ cannot have been killed during
$\intervalleoo{\al(T_l(X_m)-1-\alpha)}{\al (T_l(X_m)-4\vlp)}$ (due to the
preliminary considerations) and we conclude, using $\Omega^S_3(\la,\pi)$, that $i$ is
occupied at time $\al (T_l(X_m)-4\vlp)$. This implies the claim.
\end{proof}

We now examine the process at time $\al T^{k+1}_q$ around $\lfloor\nl X_q^{k+1}\rfloor$ in the case where the fire is stopped by a microscopic zone (in the limit process).

\begin{lem}\label{stop}
On $\Omega(\alpha,\gamma,\la,\pi)\cap\Omega^{\la,\pi}_{T_q^k+4\vlp}$, if $F_{T_q^{k+1}}(X_q^{k+1})\leq 1$, there exists $i\in(X_q^{k+1})_\la$ such that 
\[\eta^{\la,\pi}_{\al s}(i)=0\text{ for all }s\in\intervalleff{T_q^{k+1}-\vlp}{T_q^{k+1}+\vlp}.\]

Furthermore,
\begin{enumerate}[label=(\roman*)]
	\item if $X_q^{k+1}=X_l^+(T_q^{k+1})$ for some $l\leq q$, then the zone $\intervalleentier{\lfloor\nl X_q^{k+1}\rfloor-\ml-2\klp}{\lfloor\nl X_q^{k+1}\rfloor-\ml}$ is completely occupied at time $\al(T_q^{k+1}-4\vlp)$;
	\item if $X_q^{k+1}=X_l^-(T_q^{k+1})$ for some $l\leq q$, then the zone $\intervalleentier{\lfloor\nl X_q^{k+1}\rfloor+\ml}{\lfloor\nl X_q^{k+1}\rfloor+\ml+2\klp}$ is completely occupied at time $\al(T_q^{k+1}-4\vlp)$.
\end{enumerate}
\end{lem}

\begin{proof}
We have $\tH_{T_q^{k+1}}(X_q^{k+1})>0$: in the limit process, a fire is stopped in $X_q^{k+1}$ at time $T_q^{k+1}$ by a microscopic zone. Without loss of generality, we assume that $Z_{T_q^{k+1}-}(X_q^{k+1}-)=1$. We have either $H_{T_q^{k+1}-}(X_q^{k+1})>0$ or  $Z_{T_q^{k+1}-}(X_q^{k+1}+)<1$. Clearly, $X_q^{k+1}=X_m\in\cB_M$ for some $m\leq q$,
with $Z_{T_m-}(X_m)<1$ (else, we would have  $H_{T_q^{k+1}}(X_q^{k+1})=0$ and
$Z_{T_q^{k+1}-}(X_q^{k+1}-)=Z_{T_q^{k+1}-}(X_q^{k+1}+)$). We define $\cM=(i_0;t_0,T_m)$ as in Lemma \ref{destroyed component}.

By construction, there is $l\in\{1,\dots,q\}$ such that $X_m=X^+_l(T_q^{k+1})$. Hence, $T_q^{k+1}=T_l^{D,+}$ and $X_l^+(T_q^k+4\vlp)\in\chi_{T_q^k+4\vlp}^+$ with $F_{T_q^k+4\vlp}(y)=0$ for all $y\in\intervalleoo{X_l^+(T_q^k+4\vlp)}{X_q^{k+1}+\alpha/p}$. By Lemma \ref{positionfire2}, we deduce that there is no burning tree in $\intervalleentier{\lfloor\nl X_l\rfloor+i^{l,+}_{\al(T_q^k+4\vlp-T_l)}+1}{\lfloor\nl X_q^{k+1}\rfloor}$ at time $\al(T_q^k+4\vlp)$ whence by Lemma \ref{noaffect}, that the site $\lfloor\nl X_q^{k+1}\rfloor$ is not affected by a fire during $\intervalleff{\al(T_q^k+4\vlp)}{\al(T_q^{k+1}-4\vlp)}$. The site $\lfloor\nl X_q^{k+1}\rfloor-\ml-2\klp-1$ is not been affected by any fire during the time interval $\intervalleoo{\al(T_q^{k+1}-1-2\alpha)}{\al(T_q^{k+1}-4\vlp)}$, recall Step 2 in the proof of Lemma \ref{propastep2}.

\md

\noindent{\bf Case 1.} Assume first that $H_{T_q^{k+1}-}(X_q^{k+1})>0$. Then by construction, there holds $T_m+Z_{T_m-}(X_m) >
T_q^{k+1}>T_m$, whence by $\Omega_M(\alpha)$, $T_m+Z_{T_m-}(X_m) > T_q^{k+1} + 2\alpha>T_m+4\alpha$. 

We deduce from Lemma \ref{micro fire p} that there is a vacant site in 
\[C^P=  C^P((\zeta^{\la,\pi,\cM,m}_t(i))_{t\geq0,i\in\zz},(0,T_m))=\intervalleentier{i^g}{i^d}\subset\intervalleentier{-\ml}{\ml}\]
during the time interval $\intervalleff{\al(T_m+\kappa_{\la,\pi}^0)}{\al(T_m+\Theta^{\la,\pi,m}_\mathcal{M})}$ (by definition of $\Theta^{\la,\pi,m}_\mathcal{M}$). By Lemma \ref{destroyed component} and with our coupling (recall that seeds fall on $(X_m)_\la$ according to the processes $(N^{S,m}_t(i-\lfloor\nl X_m\rfloor))_{t\geq0,i\in(X_m)}$), we deduce that there is also a vacant site in $\intervalleentier{\lfloor\nl X_m\rfloor+i^g}{\lfloor\nl X_m\rfloor+i^d}\subset(X_m)_\la$ during $\intervalleff{\al(T_m+\kappa_{\la,\pi}^0)}{\al(T_m+\Theta^{\la,\pi,m}_\mathcal{M})}$. But by $\Omega_4^{S,P}(\gamma,\la,\pi)$, we see that $\Theta^{\la,\pi,m}_\mathcal{M}\geq Z_{T_m-}(X_m)- \gamma$ whence $T_m+\Theta^{\la,\pi,m}_\mathcal{M}\geq T_m+Z_{T_m-}(X_m)-\gamma>T_q^{k+1}+2\alpha-\gamma>T_q^{k+1}+\vlp$ since $\gamma<\alpha$ and $\vlp<\alpha$. All this implies that there is a vacant site in $C^P\subset(X_m)_\la$ during $\intervalleff{\al (T_q^{k+1}-\vlp)}{\al (T_q^{k+1}+\vlp)}$.

Since the match falling on $\lfloor\nl X_m\rfloor$ does not affect the zone outside $(X_m)_\la$, we deduce from the preliminary considerations that the zone $\intervalleentier{\lfloor\nl X_q^{k+1}\rfloor-\ml-2\klp}{\lfloor\nl X_q^{k+1}\rfloor-\ml}$ is not affected by any fire during $\intervalleff{\al(T_q^{k+1}-1-\alpha)}{\al(T_q^{k+1}-4\vlp)}$, which implies the claim by $\Omega_3^S(\la,\pi)$.

\md

\noindent{\bf Case 2.} Assume that $H_{T_q^{k+1}-}(X_m)=0$. Then by construction, there holds $T_q^{k+1}-
[T_m-Z_{T_m-}(X_m)] \geq 1$, whence $T_q^{k+1}-
[T_m-Z_{T_m-}(X_m)] \geq 1+3\alpha$. Since $H_{T_q^{k+1}-}(X_m)=0$, we have $Z_{T_q^{k+1}-}(X_m+)< 1 =
Z_{T_q^{k+1}-}(X_m-)$ and $T_m+ Z_{T_m-}(X_m) \leq T_q^{k+1}$, so that $T_m+ Z_{T_m-}(X_m)
\leq T_q^{k+1}-3\alpha$.

We aim to use the event $\Omega^{S,P}_1(\la,\pi)$. We recall that
$t_0=T_m-Z_{T_m-}(X_m)=\tau_{T_m-}(X_m)$. Observe that
$Z_{t_0-}(X_m)=Z_{t_0-}(X_m-)=Z_{t_0-}(X_m+)=1$ because there is no match falling on $x$ during $\intervallefo{0}{T_m}$.

Set now $t_1=T_m$. Observe that $0<t_1-t_0<1$ (because $Z_ {T_m}(X_m)<1$). Necessarily,
$Z_{t-}(x_+)$ has jumped to $0$ at least one time between $t_0$ and $T_q^{k+1}-$ (else,
one would have  $Z_{T_q^{k+1}-} (x_+)=1$, since $T_{q}^{k+1}-t_0\geq 1$ by assumption) and
this jump occurs after $t_0+1>t_1$ (since a jump of  $Z_{t-}(x_+)$ requires that
$Z_{t-}(x_+)=1$, and since for all $t\in (t_0,t_0+1)$, $Z_{t-}(x_+)=t-t_0<1$).

We thus may denote by $t_2<t_3<\dots<t_K$, for some $K\geq 2$, the successive times
of jumps of the process $(Z_{t-}(x_-),Z_{t-}(x_+))$  during $(t_0+1,T_q^{k+1})$. Then we observe that $Z_{t-}(x+)$ and $Z_{t-}(x-)$ do never jump to
0 at the same time during $\intervalleoo{t_0}{T_q^{k+1}}$ (else it would mean that $x$ is crossed by a fire at some time $u$, whence necessarily
$H_r(x)=0$ and $Z_{r-}(x+)=Z_{r-}(x-)$ for all $r\in\intervalleff{u}{T_q^{k+1}}$). 

Furthermore there is always at least one jump of $(Z_{t-}(x_-),Z_{t-}(x_+))$ of any time
interval of length $1$ (during $\intervalleoo{t_0}{T_q^{k+1}}$), because else, $Z_{t-}(x_-)$
and $Z_{t-}(x_+)$ would both become to be equal to $1$ and thus would remain equal
forever.

Finally, observe that two jumps of $Z_{t-}(x_+)$  cannot occur in a time of length $1$
(since a jump of $Z_{t-}(x_+)$ requires that $Z_{t-}(x_+)=1$) and the same thing holds for
$Z_{t-}(x_-)$.

Consequently the family $\cP=\{t_0,\dots,t_K\}$ necessarily
satisfies the condition $(PP1)$ of Subsection \ref{persisp}.

For each $l\in\{0,2,\dots,K\}$, there is a unique (thanks to $\Omega_M(\alpha)$) $k_l\in\intervalleentier{0}{q}$ such that $t_l=T_{k_l}(X_m)$. We set, for all $l\in\{0,2,\dots,K\}$,
\[i_l=\lfloor\nl X_{k_l}\rfloor + i^{k_l,+}_{\al (t_l-\vlp-T_{k_l})}-\lfloor\nl X_m\rfloor\]
if the jump at time $t_l$ is a jump of $Z_{t-}(X_m-)$ (that is if $x=X_{k_l}^+(t_l)$) and
\[i_l=\lfloor\nl X_{k_l}\rfloor + i^{k_l,-}_{\al (t_l-\vlp-T_{k_l})}-\lfloor\nl X_m\rfloor\]
if the jump at time $t_l$ is a jump of $Z_{t-}(X_m+)$ (that is if $x=X_{k_l}^-(t_l)$). Set for example $i_0=0$ if $t_0=0$. We also put $\e=-1$ if $x=X_{l_2}^+(t_2)$ and $\e=1$ else. We thus may denote $\cI=(\e;i_{k_0},i_{k_2},\dots,i_{k_K})$. Clearly, $\cI$ satisfies $(PP2)$, thanks to \eqref{estimposplusavant}.

All this implies that $\fP=(\cP,\cI)$ satifies $(PP)$.

Next, there holds that $t_2-t_1 <Z_{T_m-}(X_m)=t_1-t_0$, because else, we would have
$H_{t_2-}(X_m)=0$ and thus the fire $k_2$ would cross $X_m$, so
that $Z_{t-}(x_+)$ and $Z_{t-}(x_-)$ would remain equal forever.  Furthermore, we have $0<T_q^{k+1}-t_K<1$ because else, we would have $Z_{T_q^{k+1}}(X_m-)=Z_{T_q^{k+1}}(X_m+)=1$.

Finally, we check that 
\begin{equation}\label{equalproc}
(\eta^{\la,\pi}_{\al t}(i))_{t\in\intervalleff{t_0-\vlp}{t_K+4\vlp}, i \in
(X_m)_\la}=(\zeta^{\la,\pi,\fP,m}_{\al t}(i-\lfloor\nl x\rfloor))_{t\in\intervalleff{t_0-\vlp}{t_K+4\vlp}, i \in (X_m)_\la}
\end{equation}
 this
last process being built with the family of seed processes $(N^{S,m}_{
t}(i))_{t\geq 0, i \in \zz}$ and the family of propagation processes $(N^{P,m}_{ t}(i))_{t\geq 0, i \in \zz}$ as in
Subsection \ref{persisp}. We do \eg it in the case where $\e=-1$ and $t_0>1$, the other cases being treated similarly. 

Observe that for all $l\in\{0,2,\dots,K\}$ there holds $t_l=T_{k_l}(X_m)=T^{D,+}_{k_l}$ (if $X_m=X_{k_l}^+(t_l)$) or $T_{k_l}^{D,-}$ (if $X_m=X_{k_l}^-(t_l)$). Hence, since $T_q^k+4\vlp\geq T_l+\vlp$, we have 
\begin{equation}\label{sitefire}
\eta^{\la,\pi}_{\al(t_l-\vlp)}(\lfloor\nl X_m\rfloor +i_l)=2
\end{equation}
for all $l\in\{0,2,\dots,K\}$, thanks to $\Omega^{\la,\pi}_{T_q^k+4\vlp}$. 

We already have checked in Lemma \ref{destroyed component} that $(\eta^{\la,\pi}_{t}(i))_{t\geq0, i \in
\zz}$ and $(\zeta^{\la,\pi,\fP,m}_{t}(i-\lfloor\nl x\rfloor))_{t\geq0, i \in \zz}$ are equal on $(X_m)_\la$ during the time interval $\intervalleff{\al (t_0-\vlp)}{\al(T_m+\kappa_{\la,\pi}^0)}$. Using similar argument, observing that seeds fall on $[X_m]_{\la,\pi}$ and fires spreads through $[X_m]_{\la,\pi}$ according to the same processes and using \eqref{sitefire}, we easily deduce that \eqref{equalproc} holds on $\Omega(\alpha,\gamma,\la,\pi)$.

We thus can use $\Omega^{S,P}_1(\la,\pi)$ and conclude that 
\begin{itemize}
	\item there is $i\in(X_m)_\la$ with $\eta^{\la,\pi}_{\al t}(i)=0$ for all $t\in\intervalleff{T_q^{k+1}-\vlp}{T_q^{k+1}+\vlp}\subset\intervalleff{t_K+2\vlp}{t_K+1-\vlp}$;
	\item no fire coming from the right can affect the zone on the left of $\lfloor\nl X_q^{k+1}\rfloor-\ml$ during the time interval $\intervalleff{\al T_m}{\al (T_q^{k+1}-4\vlp)}$ (because the fire are stopped by vacant site in $(X_m)_\la$). Hence, to affect the zone $\intervalleentier{\lfloor\nl X_q^{k+1}\rfloor-\ml-2\klp}{\lfloor\nl X_q^{k+1}\rfloor-\ml}$ during this time interval, a fire must come from the left and thus must affect the site $\lfloor\nl X_q^{k+1}\rfloor-\ml-2\klp-1$. We deduce from the preliminary considerations that the zone $\intervalleentier{\lfloor\nl X_q^{k+1}\rfloor-\ml-2\klp}{\lfloor\nl X_q^{k+1}\rfloor-\ml}$ is not affected by any fire during $\intervalleff{\al(T_q^{k+1}-1-\alpha)}{\al(T_q^{k+1}-4\vlp)}$ which implies the claim by $\Omega_3^S(\la,\pi)$.\qedhere
\end{itemize}
\end{proof}

We deduce the following corollary, which is the goal of Stage 2.
\begin{cor}\label{corstage2}
On $\Omega(\alpha,\gamma,\la,\pi)$, $\Omega_{T_q^k+4\vlp}^{\la,\pi}$ implies $\Omega_{T_q^{k+1}+4\vlp}^{\la,\pi}$.
\end{cor}

\begin{proof}
We have to prove that for $l\leq q$,
\begin{enumerate}[label=(\alph*)]
	\item if $X_l^+(T_q^k+4\vlp)\in\chi_{T_q^k+4\vlp}^+$ and if $T_l^{D,+}\neq T_q^{k+1}$, then $\eta^{\la,\pi}_{\al s}(\lfloor\nl X_l\rfloor+i^{l,+}_{\al (s-T_l)})=2$ for all $s\in\intervalleff{T_q^k+4\vlp}{T_q^{k+1}+4\vlp}$;
	\item if $X_l^-(T_q^k+4\vlp)\in\chi_{T_q^k+4\vlp}^-$ and if $T_l^{D,-}\neq T_q^{k+1}$, then $\eta^{\la,\pi}_{\al s}(\lfloor\nl X_l\rfloor+i^{l,-}_{\al (s-T_l)})=2$ for all $s\in\intervalleff{T_q^k+4\vlp}{T_q^{k+1}+4\vlp}$;
	\item if $X_l^+(T_q^k+4\vlp)\in\chi_{T_q^k+4\vlp}^+$ and if $T_l^{D,+}= T_q^{k+1}$, then $\eta^{\la,\pi}_{\al s}(\lfloor\nl X_l\rfloor+i^{l,+}_{\al (s-T_l)})=2$ for all $s\in\intervalleff{T_q^k+4\vlp}{T_q^{k+1}-\vlp}$ and there is $s\in\intervalleff{T_q^{k+1}-\vlp}{T_q^{k+1}+\vlp}$ such that $\eta^{\la,\pi}_{\al s}(\lfloor\nl X_l\rfloor+i^{l,+}_{\al (s-T_l)})=0$;
	\item if $X_l^-(T_q^k+4\vlp)\in\chi_{T_q^k+4\vlp}^-$ and if $T_l^{D,-}= T_q^{k+1}$, then $\eta^{\la,\pi}_{\al s}(\lfloor\nl X_l\rfloor+i^{l,-}_{\al (s-T_l)})=2$ for all $s\in\intervalleff{T_q^k+4\vlp}{T_q^{k+1}-\vlp}$ and there is $s\in\intervalleff{T_q^{k+1}-\vlp}{T_q^{k+1}+\vlp}$ such that $\eta^{\la,\pi}_{\al s}(\lfloor\nl X_l\rfloor+i^{l,-}_{\al (s-T_l)})=0$.
\end{enumerate}
All this will imply the result (observe that only these four cases may occur).

Observe that either $F_{T_q^{k+1}}(X_q^{k+1})=2$ (\ie two fires meet at time $T_q^{k+1}$) or $F_{T_q^{k+1}}(X_q^{k+1})\leq1$ (\ie a fire is stopped by a microscopic zone).

\md

\noindent{\bf Step 1.} We start by studying the case where $F_{T_q^{k+1}}(X_q^{k+1})=2$. There are $l_1$ and $l_2$ such that $X_{l_1}^+(T_q^{k+1})=X_q^{k+1}=X_{l_2}^-(T_q^{k+1})$. In this Step, we prove (c) for the fire $l_1$ and (d) for the fire $l_2$.

By construction, we have $X_{l_1}^+(T_q^{k}+4\vlp)\in\chi_{T_q^k+4\vlp}^+$ and $X_{l_2}^-(T_q^k+4\vlp)\in\chi_{T_q^k+4\vlp}^-$ with $F_{T_q^k+4\vlp}(y)=0$ for all $y\in\intervalleoo{X_{l_1}^+(T_q^{k}+4\vlp)}{X_{l_2}^-(T_q^{k}+4\vlp)}$ and $X_{l_2}^-(T_q^{k}+4\vlp)-X_{l_1}^+(T_q^{k}+4\vlp)=2(T_q^{k+1}-T_q^k-4\vlp)/p\geq 5\alpha/p$.

We first prove that $\eta^{\la,\pi}_{\al s}(\lfloor\nl X_{l_1}\rfloor+i^{l_1,+}_{\al(s-T_{l_1})})=2$ for all $s\in\intervalleff{\al(T_q^k+4\vlp)}{\al(T_q^{k+1}-\vlp)}$. Equivalently, we prove that
\[\eta^{\la,\pi}_{\al T_{l_1}+T^{l_1}_{j-\lfloor\nl X_{l_1}\rfloor}}(j)=2\]
for all $j\in\intervalleentier{\lfloor\nl X_{l_1}\rfloor+i^{l_1,+}_{\al(T_q^k+4\vlp-T_{l_1})}}{\lfloor\nl X_{l_1}\rfloor+i^{l_1,+}_{\al(T_q^{k+1}-\vlp-T_{l_1})}}$.

Firstly, Lemma \ref{propastep2} with $s_0=T_q^{k+1}$ directly implies that 
$\eta^{\la,\pi}_{\al T_{l_1}+T^{l_1}_{j-\lfloor\nl X_{l_1}\rfloor}}(j)=2$ for all $j\in\intervalleentier{\lfloor\nl X_{l_1}\rfloor+i^{l_1,+}_{\al(T_q^k+4\vlp-T_{l_1})}}{\lfloor\nl X_q^{k+1}\rfloor-\ml-2\klp}$.

Secondly, we prove that
\[\eta^{\la,\pi}_{\al (T_q^{k+1}-4\vlp)}(i)=1\text{ for all }i\in[X_q^{k+1}]_{\la,\pi}.\]
This will completes the claim, using similar arguments as in Lemma \ref{propastep2} since there is no burning tree in $\intervalleentier{\lfloor\nl X_{l_1}\rfloor+i^{l_1,+}_{\al(T_q^k+4\vlp-T_{l_1})}+1}{\lfloor\nl X_{l_2}\rfloor+i^{l_2,-}_{\al(T_q^k+4\vlp-T_{l_2})}+1}$ at time $\al(T_q^k+4\vlp)$, thanks to Lemma \ref{positionfire2} and since $\lfloor\nl X_{l_1}\rfloor+i^{l_1,+}_{\al(T_q^{k+1}-\vlp-T_{l_1})}\leq \lfloor\nl X_q^{k+1}\rfloor-\ml$ and $\lfloor\nl X_{l_2}\rfloor+i^{l_2,-}_{\al(T_q^{k+1}-\vlp-T_{l_2})}\geq \lfloor\nl X_q^{k+1}\rfloor+\ml$, thanks to $\Omega^{P,T}(\la,\pi)$ and \eqref{estimposplusavant}.

No fire can affect the zone $[X_q^{k+1}]_{\la,\pi}$ during $\intervalleff{\al(T_q^k+4\vlp)}{\al(T_q^{k+1}-4\vlp)}$, thanks to \eqref{entree} and to Lemma \ref{positionfire2},
 (which implies that there is no burning tree in $\intervalleentier{\lfloor\nl X_{l_1}\rfloor+i^{l_1,+}_{\al(T_q^k+4\vlp-T_{l_1})}+1}{\lfloor\nl X_{l_2}\rfloor+i^{l_2,-}_{\al(T_q^k+4\vlp-T_{l_2})}-1}$).
By construction, we have $Z_{T_q^{k+1}-}(X_q^{k+1})=Z_{T_q^{k+1}-}(X_q^{k+1}+)=Z_{T_q^{k+1}-}(X_q^{k+1}-)=1$, whence $T_q^{k+1}-\tau_{T_q^{k+1}}(X_q^{k+1})\geq 1$ and $T_q^{k+1}-\tau_{T_q^{k+1}}(X_q^{k+1})\geq 1+3\alpha$ by $\Omega_M(\alpha)$. Since no match has fallen on $X_q^{k+1}\in\cB_M^2$ during $\intervalleff{0}{T_q^{k+1}}$, using similar argument as in Case 1 Step 3 in the proof of Lemma \ref{propastep2}, we then deduce that for all $j\in[X_q^{k+1}]_{\la,\pi}$,
\[\rho^{\la,\pi}_{\al(T_q^{k+1}-4\vlp)}(j)\leq T_q^{k+1}-1-\alpha,\]
which implies the claim by $\Omega_3^{S}(\la,\pi)$. Same thing of course holds for $l_2$.

Furthermore, we have shown that at time $\al(T_q^{k+1}-\vlp)$, the sites $\lfloor\nl X_{l_1}\rfloor+i^{l_1,+}_{\al(T_q^{k+1}-\vlp-T_{l_1})}$ and $\lfloor\nl X_{l_2}\rfloor+i^{l_2,-}_{\al(T_q^{k+1}-\vlp-T_{l_2})}$ are burning and
\begin{equation}\label{fireb2}
\eta^{\la,\pi}_{\al (T_q^{k+1}-\vlp)}(i)=1
\end{equation}
for all $i\in\intervalleentier{\lfloor\nl X_{l_1}\rfloor+i^{l_1,+}_{\al(T_q^{k+1}-\vlp-T_{l_1})}+1}{\lfloor\nl X_{l_2}\rfloor+i^{l_2,-}_{\al(T_q^{k+1}-\vlp-T_{l_2})}-1}$.

We next show that the fires are stopped during $\intervalleff{\al (T_q^{k+1}-\vlp)}{\al (T_q^{k+1}+\vlp)}$. Observe that, on $\Omega^{P,T}(\la,\pi)$, thanks to \eqref{estimposplusapres}, there is $i_0\in [X_q^{k+1}]_{\la,\pi}$ such that
\[i_0=\lfloor\nl X_{l_1}\rfloor+i^{l_1,+}_{T^{l_1}_{i_0+1-\lfloor\nl X_{l_1}\rfloor-}}=\lfloor\nl X_{l_2}\rfloor+i^{l_2,-}_{T^{l_2}_{i_0-1-\lfloor\nl X_{l_2}\rfloor-}}.\]
We deduce from \eqref{fireb2}, that 
\[\eta^{\la,\pi}_{\al T_{l_1}+T^{l_1}_{j-\lfloor\nl X_{l_1}\rfloor}}(j)=2\text{ for all }j\in\intervalleentier{\lfloor\nl X_{l_1}\rfloor+i^{l_1,+}_{\al(T_q^{k+1}-\vlp-T_{l_1})}}{i_0}\]
and
\[\eta^{\la,\pi}_{\al T_{l_2}+T^{l_2}_{j-\lfloor\nl X_{l_2}\rfloor}}(j)=2\text{ for all }j\in\intervalleentier{i_0}{\lfloor\nl X_{l_2}\rfloor+i^{l_2,-}_{\al(T_q^{k+1}-\vlp-T_{l_2})}}.\]

We know that the fire in $i_0$ propagates at time 
\[\al T_{l_1}+T^{l_1}_{i_0+1-\lfloor\nl X_{l_1}\rfloor}=\al T_{l_2}+T^{l_2}_{i_0-1-\lfloor\nl X_{l_2}\rfloor}.\]
Thus, with our coupling and on $\Omega^{P,T}(\la,\pi)$, at time $\al T_{l_1}+T^{l_1}_{i_0+1-\lfloor\nl X_{l_1}\rfloor}$, either the site $i_0+1$ is vacant (because it has been burnt by the fire $l_2$) or the site $i_0+1$ is occupied but has vacant neighbors until it propagates, that is until $\al T_{l_1}+T^{l_1}_{i_0+2-\lfloor\nl X_{l_1}\rfloor}$ (because it is a spark for the fire $l_2$). In any case, since 
\[\al T_{l_1}+T^{l_1}_{i_0+2-\lfloor\nl X_{l_1}\rfloor}\in\intervalleff{\al(T_q^{k+1}-\vlp)}{\al(T_q^{k+1}+\vlp)},\]
recall \eqref{estimtps}, there is $s_1\in\intervalleff{T_q^{k+1}-\vlp}{T_q^{k+1}+\vlp}$ such that $\eta^{\la,\pi}_{\al s}(\lfloor\nl X_{l_1}\rfloor+i^{l_1,+}_{\al (s_1-T_{l_1})})=0$. Similarly, we can find $s_2\in\intervalleff{T_q^{k+1}-\vlp}{T_q^{k+1}+\vlp}$ such that $\eta^{\la,\pi}_{\al s_2}(\lfloor\nl X_{l_2}\rfloor+i^{l_2,+}_{\al (s_2-T_{l_2}})=0$, which completes this Step.

\md

\noindent{\bf Step 2.} Here, we study the case where $F_{T_q^{k+1}}(X_q^{k+1})\leq 1$ and $X_q^{k+1}\not\in\{-A,A\}$. Assume for example that $X_q^{k+1}=X_{l_0}^+(T_q^{k+1})$ for some $l_0\leq q$. In this Step, we prove (c) for the fire $l_0$.

By construction, $X_{l_0}^+(T_q^{k}+4\vlp)\in\chi_{T_q^k+4\vlp}^+$ and $F_{T_q^k+4\vlp}(y)=0$ for all $y\in\intervalleoo{X_{l_0}^+(T_q^{k}+4\vlp)}{X_q^{k+1}+\alpha/p}$.

We first prove that $\eta^{\la,\pi}_{\al s}(\lfloor\nl X_{l_0}\rfloor+i^{l_0,+}_{\al(s-T_{l_0})})=2$ for all $s\in\intervalleff{\al(T_q^k+4\vlp)}{\al(T_q^{k+1}-\vlp)}$. Equivalently, we prove that
\[\eta^{\la,\pi}_{\al T_{l_0}+T^{l_0}_{j-\lfloor\nl X_{l_0}\rfloor}}(j)=2\]
for all $j\in\intervalleentier{\lfloor\nl X_{l_0}\rfloor+i^{l_0,+}_{\al(T_q^k+4\vlp-T_{l_0})}}{\lfloor\nl X_{l_0}\rfloor+i^{l_0,+}_{\al(T_q^{k+1}-\vlp-T_{l_0})}}$.

Firstly, using Lemma \ref{propastep2} with $s_0=T_q^{k+1}$, we deduce that $\eta^{\la,\pi}_{\al T_{l_0}+T^{l_0}_{j-\lfloor\nl X_{l_0}\rfloor}}(j)=2$ for all $j\in\intervalleentier{\lfloor\nl X_{l_0}\rfloor+i^{l_0,+}_{\al(T_q^k+4\vlp-T_{l_0})}}{\lfloor\nl X_q^ {k+1}\rfloor-\ml-2\klp}$. 

Secondly, Lemma \ref{stop}-1 shows that the zone $\intervalleentier{\lfloor\nl X_q^{k+1}\rfloor-\ml-2\klp}{\lfloor\nl X_q^{k+1}\rfloor-\ml}$ is completely occupied at time $\al(T_q^{k+1}-4\vlp)$. Since no fire coming from the right can affect the zone on the left of $\lfloor\nl X_q^{k+1}\rfloor$  until $\al(T_q^{k+1}-\vlp)$, we deduce the claim using similar argument as in Lemma \ref{propastep2}.

Finally, Lemma \ref{stop} directly imply that there is $j\in(X_q^{k+1})_\la$ such that $\eta^{\la,\pi}_{\al s}(j)=0$ for all $s\in\intervalleff{T_q^{k+1}-\vlp}{T_q^{k+1}+\vlp}$. Since
\[\lfloor\nl X_{l_0}\rfloor+i^{l_0,+}_{\al (T_q^{k+1}+\vlp-T_{l_0})}\geq \lfloor\nl X_q^{k+1}\rfloor+\ml,\]
recall \eqref{estimposplusapres}, there is $s\in\intervalleff{T_q^{k+1}-\vlp}{T_q^{k+1}+\vlp}$ such that $\eta^{\la,\pi}_{\al s}(\lfloor\nl X_{l_0}\rfloor+i^{l_0,+}_{\al (s-T_{l_0})})=0$, as desired.

\md

\noindent{\bf Step 3.} Here we study the case where $X_q^{k+1}\in\{-A,A\}$. Assume for example that $X_q^{k+1}=X_{l_0}^+(T_q^{k+1})=A$ for some $l_0\leq q$. In this Step, we prove (c) for the fire $l_0$.

This case is very simple: by construction, $X_{l_0}^+(T_q^{k}+4\vlp)\in\chi_{T_q^k+4\vlp}^+$ and $F_{T_q^k+4\vlp}(y)=0$ for all $y\in\intervalleoo{X_{l_0}^+(T_q^{k}+4\vlp)}{A}$. 

Since there is no burning tree in $\intervalleentier{\lfloor\nl X_{l_0}\rfloor+i^{l_0,+}_{\al(T_q^k+4\vlp-T_{l_0})}+1}{\lfloor\nl A\rfloor}$ at time $\al(T_q^k+4\vlp)$ (thanks to Lemma \ref{positionfire2}), we deduce, using similar argument as in the proof of Lemma \ref{propastep2}, that $\eta^{\la,\pi}_{\al T_{l_0}+T^l_{j-\lfloor\nl X_l\rfloor}}(j)=2$ for all $j\in\intervalleentier{\lfloor\nl X_{l_0}\rfloor+i^{l_0,+}_{\al(T_q^k+4\vlp-T_{l_0})}}{\lfloor\nl A\rfloor-\ml-2\klp}$. The zone $\intervalleentier{\lfloor\nl A\rfloor-\ml-2\klp}{\lfloor\nl A\rfloor}$ is not affected by any fire during $\intervalleff{\al(T_q^{k+1}-1-\alpha)}{\al(T_q^{k+1}-4\vlp)}$ (recall Step 3 in the proof of Lemma \ref{propastep2}) and no match falls in this zone during $\intervalleff{0}{\al T}$. We deduce as usual, using $\Omega_3^S(\la,\pi)$, that this zone is completely occupied at time $\al(T_q^{k+1}-4\vlp)$. Thus, we have
\[\eta^{\la,\pi}_{\al T_{l_0}+T^l_{j-\lfloor\nl X_{l_0}\rfloor}}(j)=2\]
for all $j\in\intervalleentier{\lfloor\nl A\rfloor-\ml-2\klp}{\lfloor\nl A\rfloor}$, which implies the claim since $\lfloor\nl X_{l_0}\rfloor +i^{l_0,+}_{\al(T_q^{k+1}-\vlp-T_{l_0})} \leq \lfloor\nl A\rfloor-\ml$.

We immediately deduce the claim since $\eta^{\la,\pi}_{s}(\lfloor\nl A\rfloor+1)=0$ for all $s\in\intervallefo{0}{\infty}$.

\md

\noindent{\bf Step 4.} Here we study the case where $x_0\coloneqq X_{l_0}^+(T_q^k+4\vlp)\in\chi_{T_q^k+4\vlp}^+$ with $T_{l_0}^{D,+}\neq T_q^{k+1}$, for some $l_0\leq q$. We prove (a) for the fire $l_0$. By $\Omega_M(\alpha)$, there holds $T_{l_0}^{D,+}\geq T_q^{k+1}+3\alpha$.

By $\Omega_M(\alpha)$, we have $T_{l_0}^{D,+}\geq T_q^{k+1}+3\alpha$. If $F_{T_q^k+4\vlp}(y)=0$ for all $y>x_0$, necessarily $F_{T_q^k+4\vlp}(y)=0$ for all $y\in\intervalleoo{x_0}{X_{l_0}^+(T_q^{k+1}+3\alpha)}$. Lemma \ref{propastep2} with $s_0=T_q^{k+1}+2\alpha$ directly implies the result, since on $\Omega^{P,T}(\la,\pi)$, recall \eqref{estimposplus}, there holds
\begin{multline*}
\lfloor\nl X_{l_0}\rfloor+i^{l_0,+}_{\al (T_q^{k+1}+4\vlp-T_{l_0})}\leq \lfloor\nl X_{l_0}^+(T_q^{k+1}+4\vlp)\rfloor+\klp\\
\leq  \lfloor\nl X_{l_0}^+(T_q^{k+1}+2\alpha)\rfloor-\ml-2\klp.
\end{multline*}

Else, we define 
\[x_1\coloneqq \inf\enstq{y>x_0}{F_{T_q^k+4\vlp}(y)=1}\]
and distinguish several cases.

\md

{\it Case 1.} Assume that $x_1-x_0>(T_q^{k+1}-T_q^k+2\alpha)/p$. Using Lemma \ref{propastep2} with $s_0=T_q^{k+1}+\alpha$, we immediately deduce that 
\[\eta^{\la,\pi}_{\al T_{l_0}+T^l_{i-\lfloor\nl X_{l_0}\rfloor}}(i)=2\]
for all $i\in\intervalleentier{\lfloor\nl X_l\rfloor+i^{{l_0},+}_{\al(T_q^k+4\vlp-T_{l_0})}}{\lfloor\nl X_{l_0}^+(T_q^{k+1}+\alpha)\rfloor-\ml-2\klp}$ whence 
\[\eta^{\la,\pi}_{\al s}(\lfloor\nl X_{l_0}\rfloor+i^{{l_0},+}_{\al (s-T_{l_0})})=2\text{ for all }s\in\intervalleff{T_q^k+4\vlp}{T_q^{k+1}+4\vlp}\]
because on $\Omega^{P,T}(\la,\pi)$, there holds $\lfloor\nl X_{l_0}\rfloor+i^{l_0,+}_{\al(T_q^{k+1}+4\vlp-T_{l_0})}\leq \lfloor\nl X_l^+(T_q^{k+1}+\alpha)\rfloor-\ml-2\klp$.

\md

{\it Case 2.} Assume that $x_1-x_0\leq(T_q^{k+1}-T_q^k+2\alpha)/p$ but $F_{T_q^k+4\vlp}(y)=0$ for all $y\in\intervalleoo{x_1}{x_1+(T_q^{k+1}-T_q^k+2\alpha)/p}$. Necessarily $x_1=X_{l_1}^+(T_q^k+4\vlp)\in\chi_{T_q^k+4\vlp}^+$ for some $l_1\leq q$. 

Using Lemma \ref{propastep2} with $s_0=T_q^{k+1}\leq T_{l_1}^{D,+}$, we deduce that $\eta^{\la,\pi}_{\al T_{l_1}+T^{l_1}_{i-\lfloor\nl X_{l_1}\rfloor}}(i)=2$ for all $i\in\intervalleentier{\lfloor\nl X_ {l_1}\rfloor+i^{l_1,+}_{\al (T_q^k+4\vlp-T_{l_1})}}{\lfloor\nl X_{l_1}^+(T_q^{k+1})\rfloor-\ml-2\klp}$. Thus, using \eqref{entree}, we deduce
\[\eta^{\la,\pi}_{\al s}\left(\lfloor\nl X_{l_1}\rfloor+i^{l_1,+}_{\al (s-T_{l_1})}\right)=2\text{ for all }s\in\intervalleff{T_{l_1}}{T_q^{k+1}-4\vlp}.\]

We now prove that for all $i\in\intervalleentier{\lfloor\nl X_ {l_0}\rfloor+i^{l_0,+}_{\al (T_q^k+4\vlp-T_{l_0})}}{\lfloor\nl X_ {l_0}\rfloor+i^{l_0,+}_{\al (T_q^{k+1}+4\vlp-T_{l_0})}}$, we have 
\[\eta^{\la,\pi}_{\al T_{l_0}+T^{l_0}_{i-\lfloor\nl X_0\rfloor}}(i)=2.\]
This will concludes this case.

Firstly, by construction, we have $x_1>x_0+1/p$ whence by $\Omega_M(\alpha)$, $x_1>x_0+(1+3\alpha)/p$. Thus, using again Lemma \ref{propastep2} with $s_0=T_{l_0}(x_1)-\alpha$, we deduce that 
\[\eta^{\la,\pi}_{\al T_{l_0}+T^{l_0}_{j-\lfloor\nl X_{l_0}\rfloor}}(j)=2\]
for all $j\in\intervalleentier{\lfloor\nl X_ {l_0}\rfloor+i^{l_0,+}_{\al (T_q^k+4\vlp-T_{l_0})}}{\lfloor\nl (x_1-\alpha/p)\rfloor-\ml-2\klp}$ (recall that $X_{l_0}^+(T_{l_0}(x_1))=x_1$). 

Secondly, oberve that $T_{l_1}<T_q^k$ (because else $T_{l_1}=T_q^k$ and $X_{l_1}^-(T_q^k+4\vlp)\in\chi_{T_q^k+4\vlp}^-$ with $x_0<X_{l_1}^-(T_q^k+4\vlp)<X_{l_1}^+(T_q^k+4\vlp)$) whence by $\Omega_M(\alpha)$, $T_{l_1}<T_q^k-3\alpha$.
This especially imply that $T_{l_0}(y)\geq T_{l_1}(y)+1+3\alpha$ for all $y\in\intervalleff{x_1-3\alpha/p}{X_{l_0}^+(T_q^{k+1}+\alpha)}$. Recall that no match falls on any site $y\in\intervalleoo{x_1-3\alpha/p}{X_{l_0}^+(T_q^{k+1}+\alpha)}$ during the time interval $\intervalleoo{T_q^k-3\alpha}{T_q^{k+1}+\alpha}$. Thus, in the limit process, for all $y\in\intervalleoo{x_1-3\alpha/p}{X_{l_0}^+(T_q^{k+1}+\alpha)}$, we have $\tau_{T_{l_0}(y)-}(y)=T_{l_1}(y)$.

Let now $i\in\intervalleentier{\lfloor\nl (x_1-2\alpha/p)\rfloor}{\lfloor\nl X_{l_0}^+(T_q^{k+1}+\alpha)\rfloor}$. Observe that there is no burning tree in $\intervalleentier{\lfloor\nl X_{l_0}\rfloor+i^{l_0,+}_{\al(T_q^k+4\vlp-T_{l_0})}+1}{\lfloor \nl x_1\rfloor-\klp}$ at time $\al(T_q^k+4\vlp)$, thanks to Lemma \ref{positionfire2}. Since no match falls on $i$ during $\intervalleff{\al(T_{l_1}(i/\nl)+\elp)}{\al(T_q^{k+1}+\alpha)}$, we deduce that no fire at all can affect the site $i$ during the time interval $\intervallefo{\al(T_{l_1}(i/\nl)+\elp)}{\al T_{l_0}+T^{l_0}_{j-\lfloor\nl X_{l_0}\rfloor}}$
whence
\[\rho^{\la,\pi}_{T_{l_0}+T^{l_0}_{j-\lfloor\nl X_{l_0}\rfloor}/\al-}(i)\leq T_{l_1}(i/\nl)+\elp.\]

Thus, for all $i\in\intervalleentier{\lfloor\nl (x_1-2\alpha/p)\rfloor}{\lfloor\nl X_{l_0}^+(T_q^{k+1}+\alpha)\rfloor}$, we have 
\[\rho^{\la,\pi}_{T_{l_0}+T^{l_0}_{j-\lfloor\nl X_{l_0}\rfloor}/\al-}(i)\leq T_{l_0}(i/\nl)-1-3\alpha+\elp\]
and conclude using $\Omega_3^S(\la,\pi)$ that $\eta^{\la,\pi}_{\al T_{l_0}+T^{l_0}_{i-\lfloor\nl X_{l_0}\rfloor}-}(i)=1$
whence
\[\eta^{\la,\pi}_{\al T_{l_0}+T^{l_0}_{i-\lfloor\nl X_{l_0}\rfloor}}(i)=2\]
because $\eta^{\la,\pi}_{\al (T_q^k+4\vlp)}(\lfloor\nl X_{l_0}\rfloor+i^{l_0,+}_{\al (T_q^k+4\vlp-T_{l_1})})=2$.

All this implies that $\eta^{\la,\pi}_{\al T_{l_0}+T^{l_0}_{i-\lfloor\nl X_{l_0}\rfloor}}(i)=2$ for all $i\in\intervalleentier{\lfloor\nl X_{l_0}\rfloor+i^{l_0,+}_{\al(T_q^k+4\vlp-T_{l_0})}}{\lfloor\nl X_{l_0}^+(T_q^{k+1}+\alpha)\rfloor}$ whence
\[\eta^{\la,\pi}_{\al s}(\lfloor\nl X_{l_0}\rfloor+i^{{l_0},+}_{\al (s-T_{l_0})})=2\text{ for all }s\in\intervalleff{T_q^k+4\vlp}{T_q^{k+1}+4\vlp}\]
since $\lfloor\nl X_{l_0}\rfloor+i^{{l_0},+}_{\al (T_q^{k+1}+4\vlp-T_{l_0})}\leq \lfloor\nl X_{l_0}^+(T_q^{k+1}+\alpha)\rfloor$. This completes this case.

\md

{\it Case 3.} In the general case, by construction, there are $x_0<x_1<x_2<\dots<x_m$ such that, for all $j\in\{0,\dots,m-1\}$,
\[x_j-x_{j+1}\leq (T_q^{k+1}-T_q^k+2\alpha)/p\]
and 
\[F_{T_q^k+4\vlp}(y)=0\text{ for all }y\in\intervalleoo{x_j}{x_{j+1}}\]
and finally 
\[F_{T_q^k+4\vlp}(y)=0\text{ for all }y\in\intervalleoo{x_{m}}{x_m+(T_q^{k+1}-T_q^k-2\alpha)/p}.\]

Clearly, for all $j\in\{1,\dots,m\}$, we have $x_j\in\chi_{T_q^k+4\vlp}^+$ whence there exists $l_j\in\{1,\dots,q\}$ such that $x_j= X_{l_j}^+(T_q^k+\vlp)$.

We first prove, exactly as in case 2, that
\[\eta^{\la,\pi}_{\al s}(\lfloor\nl X_{l_m}\rfloor+i^{{l_m},+}_{\al (s-T_{l_m})})=2\text{ for all }s\in\intervalleff{T_q^k+4\vlp}{T_q^{k+1}-4\vlp}.\]

Next, exactly as in Case 2, we can prove that
\[\eta^{\la,\pi}_{\al s}(\lfloor\nl X_{l_{m-1}}\rfloor+i^{{l_{m-1}},+}_{\al (s-T_{l_{m-1}})})=2\text{ for all }s\in\intervalleff{T_q^k+4\vlp}{T_q^{k+1}+4\vlp}\]
and so on.

\md

\noindent{\bf Step 5.} Finally, if $x_0\coloneqq X_{l_0}^-(T_q^k+4\vlp)\in\chi_{T_q^k+4\vlp}^-$ with $T_{l_0}^{D,+}\neq T_q^{k+1}$, for some $l_0\leq q$, we deduce (b) for the fire $l_0$ using similar argument as in Step 4.

This completes the proof.
\end{proof}

\begin{center}
{\bf \MakeUppercase{Stage 3.}}
\end{center}

In this Stage, we treat the time interval $\intervalleff{T_q^{N_q}+4\vlp}{T_{q+1}}$. On this time interval, no fire is stopped in the limit process. A match falls in $X_{q+1}$ at time $T_{q+1}$. The proof of the following lemma is very similar to the proof of the previous Stage.
\begin{lem}\label{finallemma}
On $\Omega(\alpha,\la,\gamma,\pi)$, $\Omega^{\la,\pi}_{T_q^{N_q}+4\vlp}$ implies
$\Omega^{\la,\pi}_{T_{q+1}}$.
\end{lem}

\begin{proof}[Sketch of the proof]
Observe that $\cT_M^D\cap\intervalleoo{T_q^{N_q}}{T_{q+1}}=\emptyset$. Hence, we have to prove that if $x\coloneqq X_l^+(T_q^{N_q}+4\vlp)\in\chi_{T_q^{N_q}+4\vlp}^+$ (or $X_l^-(T_q^{N_q}+4\vlp)\in\chi_{T_q^{N_q}+4\vlp}^-$) for some $l\leq q$, then $\eta^{\la,\pi}_{\al s}(\lfloor\nl X_l\rfloor+i^{l,+}_{\al (s-T_l)})=2$ (or $\eta^{\la,\pi}_{\al s}(\lfloor\nl X_l\rfloor+i^{l,-}_{\al (s-T_l)})=2$) for all $s\in\intervalleff{T_q^{N_q}+4\vlp}{T_{q+1}}$ (because $T_l^{D,+}>T_{q+1}+3\alpha$).

We can prove similar lemmas as Lemmas \ref{noaffect} and \ref{propastep2} replacing $T_q^k$ by $T_q^{N_q}$ and $T_q^{k+1}$ by $T_{q+1}$. Thus, Lemma \ref{finallemma} follows exactly as in Step 4 and Step 5 in the proof of Corollary \ref{corstage2}.
\end{proof}

The proof of Lemma \ref{heartlem} is completed.

\subsection{Proof of Theorem \ref{converge restriction} for $p>0$}\label{proofp>0}
We finally give the proof of the Theorem \ref{converge restriction} in the case $p>0$.
\begin{proof}
Let us fix $x_0\in\intervalleoo{-A}{A}$, $t_0\in\intervalleoo{0}{T}$ and $\varepsilon>0$. We will prove that with our coupling (see Subsection \ref{coupling}), when $\la\to0$ and $\pi\to \infty$ in the regime $\cR(p)$, there holds that
\begin{enumerate}[label=(\alph*)]
	\item $\lim_{\la,\pi}\proba{\bdelta(D_{t_0}^{\la,\pi}(x_0),D_{t_0}(x_0))>\e}=0$;
	\item $\lim_{\la,\pi}\proba{\bdelta_T(D^{\la,\pi}(x_0),D(x_0))>\e}=0$;
	\item $\lim_{\la,\pi}\proba{ |Z_t^{\la,\pi}(x_0)-Z_t(x_0)|>\e}=0$;	
	\item $\lim_{\la,\pi}\proba{\int_0^T |Z_t^{\la,\pi}(x_0)-Z_t(x_0)|\diff t>\e}=0$;
	\item $\lim_{\la,\pi}\proba{|W^{\la,\pi}_{t_0}(x_0)-Z_{t_0}(x_0)|>\e}=0$, where
\[W^{\la,\pi}_{t_0}(x_0)=\left(\frac{\log(|C(\eta^{\la,\pi}_{\al t_0},\lfloor\nl x_0\rfloor)|)}{\log(1/\la)}\indiq{|C(\eta^{\la,\pi}_{\al t_0},\lfloor\nl x_0\rfloor)|\geq1}\right)\wedge1.\]
\end{enumerate}

These points will clearly imply the result.

First, we introduce the event $\Omega^{x_0,t_0}_{A,T}(\alpha,\la,\pi)$ on which 
\begin{enumerate}[label=(\roman*)]
	\item $x_0\not\in\cup_{y\in\cB_M^D\cup\chi_{t_0}}\intervalleoo{y-3\alpha/p}{y+3\alpha/p}$;
	\item for all $s\in\enstq{T_k(x_0)}{k=1,\dots,n}\cup\cT_M\cup\cS_M\cup\cS_M^1\cup\cS_M^2$ with $s\leq t_0$, there holds that $t_0-s>3\alpha$;
	\item if $t_0\neq 1$, for all $s\in\enstq{T_k(x_0)}{k=1,\dots,n}\cup\cT_M\cup\cS_M\cup\cS_M^1\cup\cS_M^2$ with $s\leq t_0$, there holds that $|t_0-(s+1)|>3\alpha$;
	\item if $t_0\geq 1$, for all $i\in I^\la_A$, $N^{S,\la,\pi}_{\al t_0}(i)-N^{S,\la,\pi}_{\al (t_0-1)}(i)>0$;
	\item if $t_c=t_0-\tau_{t_0-}(x_0)<1$, there are 
\[-\lfloor\la^{-(t_c+\alpha)}\rfloor<i_1<-\lfloor\la^{-(t_c-\alpha)}\rfloor
<0<\lfloor\la^{-(t_c-\alpha)}\rfloor<i_2<\lfloor\la^{-(t_c+\alpha)}\rfloor\]
such that 
\begin{itemize}
	\item $N^{S,0}_{\al t_0}(\lfloor\nl x_0\rfloor+i_1)-N^{S,0}_{\al (\tau_{t_0-}(x_0)-\vlp)}(\lfloor\nl x_0\rfloor+i_1)=0$ and $N^{S,0}_{\al t_0}(\lfloor\nl x_0\rfloor+i_2)-N^{S,0}_{\al (\tau_{t_0-}(x_0)-\vlp)}(\lfloor\nl x_0\rfloor+i_2)=0$;
	\item $N^{S,0}_{\al t_0}(\lfloor\nl x_0\rfloor+j)-N^{S,0}_{\al (\tau_{t_0-}(x_0)+\vlp)}(\lfloor\nl x_0\rfloor+j)>0$ for all $j\in\intervalleentier{-\lfloor\la^{-(t_c-\alpha)}\rfloor}{\lfloor\la^{-(t_c-\alpha)}\rfloor}$.
\end{itemize}
\end{enumerate}
Since $t_0-\tau_{t_0-}(x_0)=1$ occurs with positive probability only if $t_0=1$ (and $\tau_{t_0}(x_0)=0$) the probability of the three first points clearly tend to $1$ when $\alpha$ tends to $0$. Since $(\tau_{t}(x_0))_{t\geq0}$ is independent of $(N^{S,\la,\pi}_t(i))_{t\geq0,i\in\zz}$ and since $(\tau_{t}(x_0))_{t\geq0}\subset\enstq{T_k(x_0)}{k=1,\dots,n}$, the probability of the two last points tend to $1$ as $\alpha\to0$ and $\la\to0$ and $\pi\to\infty$ in the regime $\cR(p)$, thanks to Lemma \ref{speed}-4,6,7. All this implies that for all $\delta>0$, there is $\alpha>0$ such that $\proba{\Omega^{x_0,t_0}_{A,T}(\alpha,\la,\pi)}>1-\delta$ for all $(\la,\pi)$ sufficiently close to the regime $\cR(p)$. 

Let us now fix $\delta>0$. We consider $\alpha_0\in\intervalleoo{0}{\e}$, $\gamma_0\in\intervalleoo{0}{\alpha_0}$, $\la_0\in\intervalleoo{0}{1}$ and $\epsilon_0\in\intervalleoo{0}{1}$ such that for all $\la\in\intervalleoo{0}{\la_0}$ and all $\pi\geq1$ in such a way that $|\nl/(\al\pi)-p|<\epsilon_0$, we have 
\[\proba{\Omega(\alpha_0,\gamma_0,\la,\pi)\cap\Omega^{x_0,t_0}_{A,T}(\alpha_0,\la,\pi)}>1-\delta.\]

We then consider $\la_1\in\intervalleoo{0}{\la_0}$ and $\epsilon_1\in\intervalleoo{0}{\epsilon_0}$ such that for all all $\la\in\intervalleoo{0}{\la_1}$ and all $\pi\geq1$ in such a way that $|\nl/(\al\pi)-p|<\epsilon_1$, we have
\begin{itemize}
	\item $4(\vlp+p(\ml+2\klp)/\nl)\leq \alpha_0$;
	\item $\alpha_0+\log(\al)/\log(1/\la)<\e$;
	\item $4(\ml+\klp)/\nl< \e$;
	\item $1/(2\ml\la^{t_c-\e})< \delta$ and $1/(2\ml\la^{t_c+\vlp})< \delta$ if $t_c<1$.
\end{itemize}
All this can be done properly by using the fact that $\vlp\to0$ and $(\ml+\klp)/\nl\to0$.

In the rest of the proof, we consider $\la\in\intervalleoo{0}{\la_1}$ and $\pi\geq1$ in such a way that $|\nl/(\al\pi)-p|\leq \epsilon_1$. Observe that, on $\Omega(\alpha_0,\gamma_0,\la,\pi)$, there holds that $\tau_{t_0-}(x_0)=\tau_{t_0}(x_0)$ and $[x_0]_{\la,\pi}\cap\left(\bigcup_{x\in\cB_M^D\cup\chi_{t_0}}[x]_{\la,\pi}\right)=\emptyset$.

\md

\noindent{\bf Step 1.} We first show that (a) (which holds for an arbitrary value of $t_0\in\intervalleoo{0}{T}$) implies (b). Indeed, we have by construction, for any $t\in\intervalleff{0}{T}$, $\bdelta(D_t^{\la,\pi}(x_0),D_t(x_0))<4A$. Hence, by dominated convergence, (a) implies that $\lim_{\la,\pi}\mathbb{E}\left[\bdelta(D_t^{\la,\pi}(x_0),D_t(x_0))\right]=0$, whence again by dominated convergence, $\lim_{\la,\pi}\mathbb{E}\left[\bdelta_T(D^{\la,\pi}(x_0),D(x_0))\right]=0$.

\md

\noindent{\bf Step 2.} Next, (c) implies (d), exactly as in Step 1.

\md

\noindent{\bf Step 3.} Due to Lemma \ref{corestim}, we know that, on $\Omega(\alpha_0,\gamma_0,\la,\pi)\cap\Omega_{A,T}^{x_0,t_0}(\alpha_0,\la,\pi)$, since $t_0>\tau_{t_0}(x_0)+3\alpha_0$, for all $i\in(x_0)_\la$,
\[\abs{\rho_{t_0}^{\la,\pi}(i)-\tau_{t_0}(x_0)}\leq \vlp.\]

For all $i\in (x_0)_\la$, since $\eta^{\la,\pi}_{\al t_0}(i)\leq 1$, there holds	
\[\eta_ {\al t_0}^{\la,\pi}(i)=\min(N^{S,\la,\pi}_{\al t_0}(i)-N^{S,\la,\pi}_{\al \rho_{t_0}^{\la,\pi}(i)}(i),1).\]

Thus, for all $i\in(x_0)_\la$,	
\[\underline{\eta}^{\la,\pi}_{\al t_0}(i)\leq \eta_ {\al t_0}^{\la,\pi}(i)\leq \overline{\eta}^{\la,\pi}_{\al t_0}(i)\]
where
\begin{align*}
\underline{\eta}^{\la,\pi}_{\al t_0}(i) &\coloneqq \min(N^{S,0}_{\al t_0}(i)-N^{S,0}_{\al (\tau_{t_0}(x_0)+\vlp)}(i),1),\\
\overline{\eta}^{\la,\pi}_{\al t_0}(i) &\coloneqq \min(N^{S,0}_{\al t_0}(i)-N^{S,0}_{\al (\tau_{t_0}(x_0)-\vlp)\vee 0}(i),1).
\end{align*}
We also recall that by construction, $(\tau_t(x_0))_{t\geq 0}$ is independent of $(N^{S,0}_ t(i))_{t\geq0,i\in\zz}$.

\md

\noindent{\bf Step 4.} Here we prove (e). We work on $\Omega(\alpha_0,\gamma_0,\la,\pi)\cap\Omega_{A,T}^{x_0,t_0}(\alpha_0,\la,\pi)$. By Step 3 and point (v) of the event $\Omega_{A,T}^{x_0,t_0}(\alpha_0,\la,\pi)$, we observe that if $0<t_c=t_0-\tau_{t_0}(x_0)<1$, then 
\begin{multline*}
\intervalleentier{\lfloor\nl x_0\rfloor-\lfloor\la^{-(t_c-\alpha_0)}\rfloor}{\lfloor\nl x_0\rfloor+\lfloor\la^{-(t_c-\alpha_0)}\rfloor}\subset C(\underline{\eta}^{\la,\pi}_{\al t_0}, \lfloor\nl  x_0\rfloor)\\
\subset C(\eta^{\la,\pi}_{\al t_0}, \lfloor\nl  x_0\rfloor)\subset C(\overline{\eta}^{\la,\pi}_{\al t_0}, \lfloor\nl  x_0\rfloor)
\subset\intervalleentier{\lfloor\nl x_0\rfloor+i_1}{\lfloor\nl x_0\rfloor+i_2}\\
\subset\intervalleentier{\lfloor\nl x_0\rfloor-\lfloor\la^{-(t_c+\alpha_0)}\rfloor}{\lfloor\nl x_0\rfloor+\lfloor\la^{-(t_c+\alpha_0)}\rfloor}.
\end{multline*}
Thus, this implies that,
\[|W_{t_0}^{\la,\pi}(x_0)-(t_0-\tau_{t_0}(x_0))|\leq\alpha_0+\frac{\log(2)}{\log(1/\la)}<\e.\]

If now $t_0-\tau_{t_0}(x_0)>1$, then $t_0-\tau_{t_0}(x_0)>1+3\alpha_0$ thanks to  $\Omega_{A,T}^{x_0,t_0}(\alpha_0,\la,\pi)$. Then Step 3 and point (iv) of $\Omega_{A,T}^{x_0,t_0}(\alpha_0,\la,\pi)$  imply that $(x_0)_\la\subset C(\eta^{\la,\pi}_{\al t_0}, \lfloor\nl  x_0\rfloor)$ whence $|C(\eta^{\la,\pi}_{\al t_0}, \lfloor\nl  x_0\rfloor)|\geq  2\ml$. Consequently, 
\[W^{\la,\pi}_{t_0}(x_0)\geq 1-\frac{\log(\al)}{\log(1/\la)}>1-\e.\]

It only remains to study what happens when $t_0=1$. By construction, we have $\tau_{t_0}(x_0)=0$ and by Lemma \ref{corestim}, we have $\rho_{t_0}^{\la,\pi}(i)=0$ for all $i\in(x_0)_\la$. By Step 3 and point (iv) of the event $\Omega^{x_0,t_0}_{A,T}(\alpha_0,\la,\pi)$, we deduce as above that $(x_0)_\la\subset C(\eta^{\la,\pi}_{\al t_0}, \lfloor\nl  x_0\rfloor)$ and conclude  $|C(\eta^{\la,\pi}_{\al t_0}, \lfloor\nl  x_0\rfloor)|\geq  2\ml$ whence 
\[W^{\la,\pi}_{t_0}(x_0)\geq 1-\frac{\log(\al)}{\log(1/\la)}\geq 1-\e.\]

Recalling that $Z_{t_0}(x_0)=(t_0-\tau_{t_0}(x_0))\wedge1$, we have proved that
\[\proba{|W_{t_0}^{\la,\pi}(x_0)-Z_{t_0}(x_0))|<\e}\geq \proba{\Omega(\alpha_0,\gamma_0,\la,\pi)\cap\Omega^{x_0,t_0}_{A,T}(\alpha_0,\la,\pi)}\geq 1-\delta,\]
as desired.

\md

\noindent{\bf Step 5.} Here we prove (c). Recall that $Z^{\la,\pi}_{t_0}(x_0)=\left(-\frac{\log(1-K^{\la,\pi}_{t_0}(x_0))}{\log(1/\la)}\right)\wedge1$ where $K^{\la,\pi}_{t_0}(x_0)=(2\ml+1)^{-1}\abs{\enstq{i\in\intervalleentier{\lfloor\nl X_0\rfloor-\ml}{\lfloor\nl X_0\rfloor+\ml}}{\eta^{\la,\pi}_{\al t_0}(i)=1}}$. We work on $\Omega(\alpha_0,\gamma_0,\la,\pi)\cap\Omega^{x_0,t_0}_{A,T}(\alpha_0,\la,\pi)$ and set $t_c= t_0-\tau_{t_0}(x_0)$.

\md

{\it Case 1.} If $t_c\geq 1$, we have checked in Step 4 that $\eta^{\la,\pi}_{\al t_0}(i)=1$ for all $i\in(x_0)_\la$, whence $K^{\la,\pi}_{t_0}(x_0)=1$ and $Z^{\la,\pi}_{t_0}(x_0)=1$. 

\md

{\it Case 2.} If now $0<t_c< 1$, we deduce from Step 3 that
\[\underline{K}^{\la,\pi}_{t_0}(x_0)\leq K^{\la,\pi}_{t_0}(x_0)\leq \overline{K}^{\la,\pi}_{t_0}(x_0)\]
where
\begin{align*}
\underline{K}^{\la,\pi}_{t_0}(x_0) &=(2\ml+1)^{-1}\abs{\enstq{i\in\intervalleentier{\lfloor\nl X_0\rfloor-\ml}{\lfloor\nl x_0\rfloor+\ml}}{\underline{\eta}^{\la,\pi}_{\al t_0}(i)=1}},\\
\overline{K}^{\la,\pi}_{t_0}(x_0) &=(2\ml+1)^{-1}\abs{\enstq{i\in\intervalleentier{\lfloor\nl X_0\rfloor-\ml}{\lfloor\nl x_0\rfloor+\ml}}{\overline{\eta}^{\la,\pi}_{\al t_0}(i)=1}}.
\end{align*}

The random variable $\underline{X}^{\la,\pi}_{t_0}(x_0)=(2\ml+1)\underline{K}^{\la,\pi}_{t_0}(x_0)$ has a binomial distribution with parameters $2\ml+1$ and $1-\la^{t_c-\vlp}$. Then, using Bienaym\'e-Chebyshev's inequality,
\begin{align*}
\pp&\left[\underline{K}^{\la,\pi}_{t_0}(x_0)\leq  1-\la^{t_c-\e}\right]=\proba{\underline{X}^{\la,\pi}_{t_0}(x_0)\leq (2\ml+1)(1-\la^{t_c-\e})}\\
	&\leq \proba{\abs{\underline{X}^{\la,\pi}_{t_0}(x_0)-(2\ml+1)(1-\la^{t_c-\vlp})}\geq (2\ml+1)\left(\la^{t_c-\e}-\la^{t_c-\vlp}\right)}\\
	&\leq \frac{(2\ml+1)\left(1-\la^{t_c-\vlp}\right)\la^{t_c-\vlp}}{(2\ml+1)^2(\la^{t_c-\e}-\la^{t_c-\vlp})^2}\\
	&= \frac{1-\la^{t_c-\vlp}}{(2\ml+1)\la^{t_c-\vlp}(\la^{\vlp-\e}-1)^2}\simeq \frac{1}{2\ml\la^{t_c-2\e+\vlp}}\\
	&\leq \frac{1}{2\ml\la^{t_c-\e}}\,(\text{because }0<\vlp<\alpha_0<\e)\\
	&\leq \delta.
\end{align*}

By the same way, since $\overline{X}^{\la,\pi}_{t_0}(x_0)=(2\ml+1)\underline{K}^{\la,\pi}_{t_0}(x_0)$ has a binomial distribution with parameters $2\ml+1$ and $1-\la^{t_c+\vlp}$,
\begin{align*}
\pp&\left[\overline{K}^{\la,\pi}_{t_0}(x_0)\geq  1-\la^{t_c+\e}\right]=\proba{\overline{X}^{\la,\pi}_{t_0}(x_0)\geq (2\ml+1)(1-\la^{t_c+\e})}\\
	&\leq \proba{\abs{\overline{X}^{\la,\pi}_{t_0}(x_0)-(2\ml+1)(1-\la^{t_c+\vlp})}\geq (2\ml+1)\left(\la^{t_c+\vlp}-\la^{t_c+\e}\right)}\\
	&\leq \frac{(2\ml+1)\left(1-\la^{t_c+\vlp)}\right)\la^{t_c+\vlp}}{(2\ml+1)^2(\la^{t_c+\vlp}-\la^{t_c+\e})^2}	\simeq\frac{1}{2\ml\la^{t_c+\vlp}}\leq \delta.
\end{align*}

All this implies that, %on $\Omega(\alpha_0,\gamma_0,\la,\pi)\cap\Omega^{x_0,t_0}_{A,T}(\alpha_0,\la,\pi)$, there holds that 
\[\proba{ K^{\la,\pi}_{t_0}(x_0)\in\intervalleoo{1-\la^{t_c-\e}}{1-\la^{t_c+\e}}}\geq1-c\delta,\]
for some constante $c>0$, whence 
\[\proba{Z^{\la,\pi}_{t_0}(x_0)\in\intervalleoo{t_c-\e}{t_c+\e}}\geq1-c\delta.\]
This is nothing but the goal, since $Z_{t_0}(x_0)=t_0-\tau_{t_0}(x_0)=t_c$ as soon as $Z_{t_0}(x_0)<1$.

\md

\noindent{\bf Step 6.} It remains to prove (a). On $\Omega(\alpha_0,\gamma_0,\la,\pi)\cap\Omega^{x_0,t_0}_{A,T}(\alpha_0,\la,\pi)$, we check that
\begin{enumerate}[label=(\roman*)]
	\item If $Z_{t_0}(x_0)<1$, then $D_{t_0}(x_0)=\{x_0\}$ and $C(\eta^{\la,\pi}_{\al t_0},\lfloor\nl x_0\rfloor)\subset (x_0)_\la$ (see Step 4 above), whence 
\[D^{\la,\pi}_{t_0}(x_0)\subset\intervalleff{x_0-\ml/\nl}{x_0+\ml/\nl}.\]
We deduce that 
\[\bdelta(D^{\la,\pi}_{t_0}(x_0),D_{t_0}(x_0))\leq 2\ml/\nl.\]
	\item If $Z_{t_0}(x_0)=1$ and $D_{t_0}(x_0)=\intervalleff{a}{b}$, for some $a,b\in\chi_{t_0}$, then 
\begin{itemize}
	\item for all $i\in\intervalleentier{\lfloor\nl a\rfloor+\ml+2\klp}{\lfloor\nl b\rfloor-\ml-2\klp}\setminus\left(\cup_{x\in\cB_M^D}[x]_{\la,\pi}\right)$, $\eta^{\la,\pi}_{\al t_0}(i)=1$. Indeed, there is no burning tree in $\intervalleentier{\lfloor\nl a\rfloor+\klp}{\lfloor\nl b\rfloor-\klp}$ at time $\al t_0$ (use a very similar result as in Lemma \ref{macro match}). Next, by construction, $Z_{t_0}(y)=1$ for all $y\in\intervalleoo{a}{b}$ whence  $\tau_{t_0}(y)\leq t_0-1$. Using $\Omega^{x_0,t_0}_{A,T}(\alpha_0,\la,\pi)$, we deduce that $\tau_{t_0}(y)\leq t_0-1-3\alpha_0$. Using finally Lemma \ref{corestim} and $\Omega_3^S(\la,\pi)$, we deduce the claim;
	\item for all $x\in\cB_M^D\cap\intervalleoo{a}{b}$, and all $i\in[x]_{\la,\pi}$, $\eta^{\la,\pi}_{\al t_0}(i)=1$. Indeed, on $\Omega^{x_0,t_0}_{A,T}(\alpha_0,\la,\pi)$, we have $\tH_{t_0-}(x)=0$ whence $\tau_{t_0}(x_0)\leq t_0-1-3\alpha_0$. We deduce that no match falling outside $[x]_{\la,\pi}$ affect this zone during the time interval $\intervalleff{\al(t_0-1-\alpha_0)}{\al t_0}$ and conclude by distinguishing several cases, as in Step 3 in the proof of Lemma \ref{propastep2};
	\item if $a\in\chi_{t_0}^+\cup\chi_{t_0}^-$ there is $i\in \langle a\rangle_{\la,\pi}$ such that $\eta^{\la,\pi}_{\al t_0}(i)=2$ (thanks to $\Omega_{T}^{\la,\pi}$, since on $\Omega^{x_0,t_0}_{A,T}(\alpha_0,\la,\pi)$, we have $\abs{t_0-s}\geq3\alpha$ for all $s\in\cT_M^D$) whereas if $a\in\chi_{t_0}^0$, there is $i\in (a)_\la$ such that $\eta^{\la,\pi}_{\al t_0}(i)=0$ (use similar argument as in Lemma \ref{stop}, observing that $\abs{t_0-s}\geq3\alpha$ for all $s\in\cT_M^D$). Similar observation of course holds for $b$;
\end{itemize}
so that
\begin{multline*}
\intervalleentier{\lfloor\nl a\rfloor+\ml+2\klp}{\lfloor\nl b\rfloor-\ml-2\klp}\subset C(\eta^{\la,\pi}_{\al t_0},\lfloor\nl x_0\rfloor)\\
\subset \intervalleentier{\lfloor\nl a\rfloor-\ml-\klp}{\lfloor\nl b\rfloor+\ml+\klp}
\end{multline*}
and thus 
\[\intervalleff{a+\frac{\ml+2\klp}{\nl}}{b-\frac{\ml+2\klp}{\nl}}\subset D^{\la,\pi}_{t_0}(x_0)\subset \intervalleff{a-\frac{\ml+2\klp}{\nl}}{b+\frac{\ml+2\klp}{\nl}},\]
whence $\bdelta(D^{\la,\pi}_{t_0}(x_0),D_{t_0}(x_0))\leq 4(\ml+\klp)/\nl$.
\end{enumerate}

Thus, on $\Omega(\alpha_0,\gamma_0,\la,\pi)\cap\Omega^{x_0,t_0}_{A,T}(\alpha_0,\la,\pi)$, we always have $\bdelta(D^{\la,\pi}_{t_0}(x_0),D_{t_0}(x_0))\leq 4(\ml+\klp)/\nl$. We conclude that
\[\proba{\bdelta(D^{\la,\pi}_{t_0}(x_0),D_{t_0}(x_0))\leq \e}\geq \proba{\Omega(\alpha_0,\gamma_0,\la,\pi)\cap\Omega^{x_0,t_0}_{A,T}(\alpha_0,\la,\pi)}\geq 1-\delta.\]
This concludes the proof of Theorem \ref{converge restriction} for $p>0$.
\end{proof}

\subsection{Cluster size distribution when $p>0$}

The aim of this section is to prove Corollary \ref{cor1} when $p>0$.

\subsubsection{Study of the LFFP$(p)$}
Recall Subsection \ref{notations} and Definition \ref{dfplffp}.
\begin{defin}
Let $(Z_t(x),H_t(x),F_t(x))_{t\geq 0, x\in \rr}$ be a LFFP$(p)$. For all $x\in\rr$ and all $t\geq0$, we define
\[\mathscr{D}_t(x)=\intervalleff{\mathscr{L}_t(x)}{\mathscr{R}_t(x)}\]
where
\begin{align*}
\mathscr{L}_t(x) &= \inf\enstq{y\leq x}{\forall (r,v)\in\Lambda^
p_{(x,t)}(y,t-p(x-y))\, ,\,
Z_{v-}(r)=1\text{ and }H_{v-}(r)=0},\\
\mathscr{R}_t(x) &= \sup\enstq{y\geq x}{\forall (r,v)\in\Lambda^
p_{(x,t)}(y,t+p(x-y))\, ,\,
Z_{v-}(r)=1\text{ and }H_{v-}(r)=0}.
\end{align*}
\end{defin}

Observe that for all $t\in\intervalleff{0}{T}$ and all $x\in\rr$,
\begin{itemize}
	\item $Z_t(x)=0$ if and only if $\pi_M\left((\mathscr{D}_{t-}(x)\times\rr)\cap\Lambda^p_{(x,t)}\right)>0$;
	\item $\mathscr{D}_{t}(x)=\{x\}$ if $t\in\intervallefo{0}{1}$;
	\item $|\mathscr{D}_{t}(x)|\leq 2(t-1)/p$.
\end{itemize}
\begin{lem}\label{zunif}
Let $(Z_t(x),H_t(x),F_t(x))_{t\geq 0, x\in \rr}$ be a LFFP$(p)$ and consider $(D_t(x))_{t\geq 0, x\in \rr}$ and $(\mathscr{D}_t(x))_{t\geq 0, x\in \rr}$ the associated processes.  There
are some constants $0<c_1<c_2$ and $0 < \kappa_1<\kappa_2$, depending  only on $p$, such that the
following estimates hold.
\begin{enumerate}[label = (\roman*)]
        \item For any $t\in (1,\infty)$, any $x\in \rr$, any $z\in [0,1)$, $\proba{Z_t(x)=z}=0$.
        \item For any $t\in[0,\infty)$, any $B>0$, any $x\in\rr$, $\proba{|D_t(x)|=B}=0$.
        \item For all $t\in [0,\infty)$, all $x\in \rr$, all $B>0$, $\proba{|D_t(x)|\geq B}
\leq c_2 e^{-\kappa_1 B}$.
        \item For all $t\in [\frac{11}{8},\infty)$, all $x\in\rr$, all $B>0$, $\proba{|D_t(x)|\geq B}
\geq c_1 e^{-\kappa_2 B}$.
		\item For all $t\in [0,\infty)$, all $x\in \rr$, all $B>0$, $\proba{|\mathscr{D}_t(x)|\geq B}
\leq c_2 e^{-\kappa_1 B}$.
		\item For all $t\in [\frac{3}{2},\infty)$, all $x\in\rr$, all $B\in\intervalleoo{0}{(2t-3)/p}$, $\proba{|\mathscr{D}_t(x)|\geq B}
\geq c_1 e^{-\kappa_2 (B+B^2)}$.
        \item For all $t\in[(5+p)/2,\infty)$, all $0\leq a < b < 1$, all $x\in \rr$, 
\[c_1(b-a)\leq \proba{Z_t(x)\in \intervalleff{a}{b}}   \leq c_2 (b-a).\]	
\end{enumerate}
\end{lem}

\begin{proof} By invariance by translation, it suffices to treat the case $x=0$.

\md 

\noindent{\bf Point (i).} For $t\in [0,1]$, we have a.s. $Z_t(0)=t$. But for $t> 1$ and $z\in
[0,1)$, $Z_t(0)=z$ implies that a fire has crossed $0$ at time
$t-z$, so that necessarily 
$\pi_M(\Lambda^p_{(0,t)})>0$, recall Subsection \ref{notations}. This happens with
probability $0$. 

\md 

\noindent{\bf Point (ii).} For any $t>0$, $|D_t(0)|$ is either $0$ or of the form
$\abs{x-y}$, for some $x,y\in\chi_t$. We easily conclude as previously
that for $B>0$, $\Pr(|D_t(0)|=B)=0$.

\md 

\noindent{\bf Point (iii).} First if $t \in [0,1)$, we have a.s. $|D_t(0)|=0$ and the result
is obvious. Recall that for $(X,\tau)$ a mark of $\pi_M$, we have $H_t(X)>0$ or
$Z_t(X)<1$ for all $t\in [\tau,\tau+1/2)$ (see the proof of Proposition
\ref{restriction limite}-Step 1).
This implies that for $t\geq 1$,
\begin{multline*}
\{|D_t(0)|\geq B\} \subset \{[0,B/2] \hbox{ is connected at time $t$ or $[-B/2,0]$ is
connected at time $t$} \}\\
\subset \left\{\pi_M([0,B/2]\times[t-1/4,t])=0 \right\}\cup \left\{\pi_M([-B/2,0]\times[t-1/4,t])=0\right\}.
\end{multline*}
Consequently, $\Pr[|D_t(0)|\geq B] \leq 2 e^{-B/8}$ as desired.

\md 

\noindent{\bf Point (iv).}
Fix $B>0$ and $t\geq 11/8$. Set $\Delta=\frac{3}{16p}$ and $K=\left\lfloor\frac{1}{\Delta}\left(B+\frac{11}{4p}\right)\right\rfloor+1$. Consider the event
$\Omega_{t,B}=\Omega^0_{t,B} \cap\bigcap_{k=0}^{K-1}\Omega_{t,B,k}$, illustrated by Figure \ref{omegatb}, where 
\begin{itemize}
	\item $\Omega^{0}_{t,B}=\{\pi_M([-5/(4p),B+5/(4p)]\times\intervalleff{t-5/4}{t})=0\}$;
	\item for all $k\in\intervalleentier{0}{K-1}$, $\Omega_{t,B,k}=\left\{\pi_M(D_k)=1\right\}\cap \left\{\pi_M\left(C_k\setminus D_k\right)=0\right\}$ where 
\begin{align*}
C_k &=\left[-\frac{11}{8p}+k\Delta,-\frac{11}{8p}+(k+1)\Delta\right]\times\intervalleff{t-11/8}{t-5/4}\\
D_k &=\left[-\frac{11}{8p}+(k+\frac{1}{3})\Delta,-\frac{11}{8p}+(k+\frac{2}{3})\Delta\right]\times\intervalleff{t-11/8}{t-5/4},
\end{align*}
see Figure \ref{omegatbk}. Observe that $\bigcup_{k=0}^{K-1} C_k\supset\intervalleff{-11/(8p)}{B+11/(8p)}$.
\end{itemize}

We have $\proba{\Omega_{t,B}^0}=\exp\left(-\frac{5}{4}(B+\frac{5}{2p})\right)$ whence for all $k\in\intervalleentier{0}{K-1}$,  $\proba{\Omega_{t,B,k}}=\frac{\Delta}{24}\times e^{-\frac{\Delta}{24}}\times e^{-\frac{\Delta}{12}}$. All these events being independent, we conclude that 
\[\proba{\Omega_{t,B}}=\exp\left(-\frac{5}{4}(B+\frac{5}{2p})\right)\times \left(\frac{\Delta}{24} e^{-\frac{\Delta}{8}}\right)^K\geq c_1 e^{-\kappa_2 B}\]
for some constant $c_1$ and $\kappa_2$ not depending on $B$. To conclude the proof of (iv), it thus suffices to check that $\Omega_{t,B}\subset \{[0,B]\subset D_t(0)\}$. But on $\Omega_{t,B}$, using the same arguments as in Point (iii),
we observe that:
\begin{itemize}
	\item for $(X,\tau)$ a mark of $\pi_M$, $H^A_s(X)>0$ or $Z_s^A(X)<1$ for all $s\in\intervalleff{\tau}{\tau+3/8}$. Thus, for all $k\in\intervalleentier{0}{K-1}$, there is $x\in D_k$ such that $H^A_s(x)>0$ or $Z_s^A(x)<1$ for all $s\in\intervalleff{t-5/4}{t-1}$;
	\item calling $(X_k,\tau_k)$ the mark of $\pi_M$ in $D_k$, we have $\tau_k+p(X_{k+1}-X_k)\in\intervalleff{t-5/4}{t-1}$ and $\tau_k+p(X_k-X_{k-1})\in\intervalleff{t-5/4}{t-1}$, see Figure \ref{omegatbk}. Thus, if the fire starting on $X_k$ at time $\tau_k$ is macroscopic, it is (at least) stopped by the marks $(X_{k-1},\tau_{k-1})$ and $(X_{k+1},\tau_{k+1})$ and does not affect the zone $\intervalleff{0}{B}$ after $t-1$;
	\item for $(Y,S)$ a mark of $\pi_M$ such that  $(Y,S)\not\in\intervalleff{-11/(8p)}{B+11/(8p)}\times\intervalleff{t-11/8}{t}$ and $Y+(t-S)/p\in\intervalleff{0}{B}$, then there exists $k\in\intervalleentier{0}{K-1}$ such that 
\[Y+\frac{t-11/8-S}{p}\in\left[-\frac{11}{8p}+(k-\frac{1}{3})\Delta,\,-\frac{11}{8p}+(k+\frac{2}{3})\Delta\right].\]
We immediately conclude that $S+p(X_{k+1}-Y)\in\intervalleff{t-5/4}{t-1}$. Thus, the right front of $(Y,S)$ is stopped by the match $(X_{k+1},\tau_{k+1})$ and does not affect the zone $\intervalleff{0}{B}$ after $t-1$;
	\item for $(Y,S)$ a mark of $\pi_M$ such that $(Y,S)\not\in\intervalleff{-11/(8p)}{B+11/(8p)}\times\intervalleff{t-11/8}{t}$ and $Y-(t-S)/p\in\intervalleff{0}{B}$, we prove as above that the left front of $(Y,S)$ is stopped by such a match $(X_{k-1},\tau_{k-1})$ and does not  affect the zone $\intervalleff{0}{B}$ after $t-1$;
	\item by construction, the other fires may not affect the zone $\intervalleff{-11/(8p)}{B+11/(8p)}$ during the time interval $\intervalleff{t-1}{t}$.
\end{itemize}

As a conclusion, the zone $\intervalleff{0}{B}$ is not affected by any fire
during $\intervalleff{t-1}{t}$. Since the length of this time interval is 
greater than $1$, we deduce that for all $x \in [0,B]$,
$Z_t(x)=\min(Z_{t-1}(x) + 1,1)=1$ 
and $H_t(x)=\max(H_{t-1}(x) - 1,0)=0$, whence $[0,B]\subset D_t(0)$.

\md 

\noindent{\bf Point (v)} First if $t \in [0,1)$, we have a.s. $|\mathscr{D}_t(0)|=0$ and the result
is obvious. If $t\geq1$ and $B>2(t-1)/p$, 
\[\proba{|\mathscr{D}_t(0)|\geq B}=0.\]
Recall that for $(X,\tau)$ a mark of $\pi_M$, we have $H_t(X)>0$ or
$Z_t(X)<1$ for all $t\in [\tau,\tau+1/2)$ (see the proof of Proposition
\ref{restriction limite}-Step 1).
This implies that for $t\geq 1$ and $B\in\intervalleoo{0}{2(t-1)/p}$,
\begin{multline*}
\{|\mathscr{D}_t(0)|\geq B\} \subset \{[0,B/2]\subset\intervalleff{0}{\mathscr{R}_t(x)} \text{ or }[-B/2,0]\subset\intervalleff{\mathscr{L}_t(x)}{0}\}\\
\subset \left\{\pi_M\left(\enstq{(r,v)\in\Lambda^p_{(0,s)}(B/2,s-pB/2)}{s\in\intervalleff{t-1/4}{t}}\right)=0 \right\}\\
\cup \left\{\pi_M\left(\enstq{(r,v)\in\Lambda^p_{(0,s)}(-B/2,s-pB/2)}{s\in\intervalleff{t-1/4}{t}}\right)=0 \right\}.
\end{multline*}
Consequently, $\proba{|\mathscr{D}_t(0)|\geq B} \leq 2 e^{-B/8}$, as desired.

\md

\noindent{\bf Point (vi)} Let $t\geq3/2$ and $B\in\intervalleoo{0}{(2t-3)/p}$. From Point (iv), using space/time stationarity, we define an event $\tOmega_{t,B}$, depending on the Poisson measure $\pi_M(\diff x,\diff s)$ restricted to $\intervalleff{-B/2-11/(8p)}{B/2+11/(8p)}\times\intervalleff{t-pB/2-3/2}{t-pB/2}$, on which $D_{t-pB/2}(0)\supset\intervalleff{-B/2}{B/2}$. Next consider the event
\[\tOmega_{t,B}^0=\left\{\pi_M\left(\intervalleff{-B/2}{B/2}\times \intervalleff{t-pB/2}{t}\right)=0\right\}.\]
We have $\proba{\tOmega_{t,B}^0}=e^{-pB^2/2}$.

The events $\tOmega_{t,B}$ and $\tOmega_{t,B}^0$ are independent, thus we have, recalling point (iv)
\[\proba{\tOmega_{t,B}\cap\tOmega_{t,B}^0}= \proba{\tOmega_{t,B}}\times\proba{\tOmega_{t,B}^0}\geq c_1 e^{-\kappa_2 (B+ B^2)}.\]

Finally, we observe that for $(X,t-pB/2)$ a fire a time $t-pB/2$ with, for example, $X<-B/2$, we have, by construction, $X+(t-(t-pB/2))/p<0$. Thus,
\[\tOmega_{t,B}\cap\tOmega_{t,B}^0\subset\{|\mathscr{D}_t(0)|\geq B\}.\]
This concludes the point.

\md

\noindent{\bf Point (vii)} For $0\leq a\leq b<1$ and $t\geq 1$, we have $Z_t(0)\in\intervalleff{a}{b}$ if and only if there is $\tau\in\intervalleff{t-b}{t-a}$ such that $Z_\tau(0)=0$. And this happens if and only if 
\[X_{t,a,b}\coloneqq \int_{t-b}^{t-a}\int_\rr\indiq{(y,s-p|x-y|)\in \mathscr{D}_{s-}(0)\times\intervalleff{0}{s}}\pi_M(\diff y,\diff s)\geq 1.\]
We deduce that
\[\proba{Z_t(0)\in\intervalleff{a}{b}}=\proba{X_{t,a,b}\geq 1}\leq \mathbb{E}\left[X_{t,a,b}\right]=\int_{t-b}^{t-a}\mathbb{E}\left[\abs{\mathscr{D}_s(0)}\right]\diff s\leq C(b-a),\]
where we used Point (v) for the last inequality.

Next, we have $\{\pi_M(\mathscr{D}_{t-b}(0)\times\intervalleff{t-b}{t-a})\geq 1\}\subset\{X_{t,a,b}\geq 1\}$: it suffices to note that a.s., 
\begin{multline*}
\{X_{t,a,b}=0\}\subset\{X_{t,a,b}=0,\mathscr{D}_{t-b}(0)\subset \mathscr{D}_s(0)\text{ for all }s\in\intervalleff{t-b}{t-a}\}\\
\subset\{\pi_M(\mathscr{D}_{t-b}(0)\times\intervalleff{t-b}{t-a})=0\}.
\end{multline*} Since now $\mathscr{D}_{t-b}(0)$ is independent of $\pi_M(\diff x,\diff s)$ restricted to $\rr\times\intervalleoo{t-b}{\infty}$, we deduce that for $t\geq (5+p)/2$
\begin{align*}
\proba{Z_t(0)\in\intervalleff{a}{b}}&\geq \proba{\pi_M(\mathscr{D}_{t-b}(0)\times\intervalleff{t-b}{t-a}) \geq1}\\
	&\geq \proba{\abs{\mathscr{D}_{t-b}(0)}\geq 1}(1-e^{-(b-a)})\\
	&\geq c(1-e^{-(b-a)}),
\end{align*}
where we used Point (vi) (here $t-b\geq 3/2$ and $(2t-3)/p\geq1$) to get the last inequality. This concludes the proof, since $1-e^{-x}\geq x/2$ for all $x\in\intervalleff{0}{1}$.\qedhere

\begin{figure}[h!]
\fbox{
\begin{minipage}[c]{0.95\textwidth}
\centering
\begin{tikzpicture}[scale=.83]%pente 2
\draw (0,0) -- (4.5,9) -- (10,9) -- (14.5,0) -- cycle node[below] at (0,0) {\footnotesize $-\frac{11}{8p}$} node[below] at (14.5,0) {\footnotesize $B+\frac{11}{8p}$};
\draw[dashed] (0,0) -- (-.5,0) node[left] {\footnotesize $t-11/8$};
\draw[dashed] (4.5,9) -- (-.5,9) node[left] {\footnotesize $t$};
\draw[dashed] (10,9) -- (14.5,9);
\draw (4.5,9) -- (4.5,0) node[below] {\footnotesize $0$} node[above] at (4.5,9) {$0$};
\draw (10,9) -- (10,0) node[below] {\footnotesize $B$} node[above] at (10,9) {$B$};
\draw (.5,1) -- (14,1);
\draw[dashed] (.5,1) -- (-.5,1) node[left] {\footnotesize $t-5/4$};
\draw[dashed] (14,1) -- (14.5,1);
\draw (1.5,3) -- (13,3);
\draw[dashed] (1.5,3) -- (-.5,3) node[left] {\footnotesize $t-1$};
\draw[dashed] (13,3) -- (14.6,3);

\foreach \x in {0,...,58} {
\draw (\x/4,0.1cm) -- (\x/4,-0.1cm);
}
\foreach \x in {0,...,57} {
\draw[dashed] (\x/4+.125,1cm) -- (\x/4+.125,3cm);
}
\foreach \x in {0,...,12} {
\draw node at (\x/2+.125,.5+\x/30) {\footnotesize $\bullet$};
}
\foreach \x in {13,...,21} {
\draw  node at (\x/2+.125,.75-\x/70) {\footnotesize $\bullet$};
}
\foreach \x in {22,...,28} {
\draw  node at (\x/2+.125,.5+\x/100) {\footnotesize $\bullet$};
}
\foreach \x in {0,...,14} {
\draw node at (\x+.375,.9-\x/20) {\footnotesize $\bullet$};
}%(\x+.375,0.25cm) 
\foreach \x in {0,...,13} {
\draw node at (\x+.875,.3+\x/20) {\footnotesize $\bullet$};
}
\end{tikzpicture} \caption{The event $\Omega_{t,B}$.}\label{omegatb}
\vspace{.5cm}
\parbox{13.3cm}{
\footnotesize{
The marks of $\pi_M$ are represented by \textcolor{red}{$\bullet$}. A match falls on each zone $D_k\subset C_k$.
}}
\end{minipage}}
\end{figure}

\begin{figure}[h!]
\fbox{
\begin{minipage}[c]{0.95\textwidth}
\centering
\begin{tikzpicture}%pente 1
\draw (0,0) rectangle (9,6) node[below] at (0,0) {\footnotesize $(k-1)\Delta$} node[below] at (9,0) {\footnotesize $(k+2)\Delta$};
\draw (0,2) -- (9,2);
\draw[dashed] (0,0) -- (-1,0) node[left] {$t-11/8$};
\draw[dashed] (0,2) -- (-1,2) node[left] {$t-5/4$};
\draw[dashed] (0,6) -- (-1,6) node[left] {$t-1$};
\draw[dashed] (9,0) -- (10,0);
\draw[dashed] (9,2) -- (10,2);
\draw[dashed] (9,6) -- (10,6);

\draw[dashed] (1,6) -- (1,0);
\draw[dashed] (2,0) -- (2,6);
\draw (3,6) -- (3,0) node[below] {\footnotesize $k\Delta$};

\draw[dashed] (4,6) -- (4,0) node at (4,-.5) {\footnotesize $\frac{3k+1}{3}\Delta$};
\draw[ultra thick] (4.8,2) --(4.8,6);
\draw[ultra thick,dashed] (4.8,0.5) --(4.8,2);
\draw[red, dashed, very thick] (4.8,0.5) -- (9,4.7) node at (4.8,0.5) {$\bullet$};
\draw[red, dashed, very thick] (4.8,0.5) -- (0,5.3);
\draw[dashed] (5,6) -- (5,0) node at (5,-.7) {\footnotesize $\frac{3k+2}{3}\Delta$};
\draw (6,6) -- (6,0) node[below] {\footnotesize $(k+1)\Delta$};

\draw (5,0) -- (9,4);
\draw (4,0) -- (0,4);
\draw (5,3) -- (8,6);
\draw (4,3) -- (1,6);

\draw[dashed] (7,0) -- (7,6);
\draw[dashed] (8,0) -- (8,6);
\end{tikzpicture} \caption{The event $\Omega_{t,B,k}$.}\label{omegatbk}
\vspace{.5cm}
\parbox{13.3cm}{
\footnotesize{
A match falls on $D_k =\left[-\frac{11}{8p}+(k+\frac{1}{3})\Delta,-\frac{11}{8p}+(k+\frac{2}{3})\Delta\right]\times\intervalleff{t-11/8}{t-5/4}$ and is represented by \textcolor{red}{$\bullet$}. The dashed slope lines stand for the hypothetical fronts of the fire. The plain slope lines stand for the upper and lower possible positions of the fronts. The plain vertical thick line is the possible microscopic zone due to the fire in $D_k$. Thus, if the match falling on $D_k$ is macroscopic, it is necessarily stopped by a microscopic zone in $D_{k+1}$ or in $D_{k-1}$, since $H_s(X_{k+1})>0$ or $Z_s(X_{k+1})<1$ for all $s\in\intervalleff{t-5/4}{t-1}$ and $H_s(X_{k-1})>0$ or $Z_s(X_{k-1})<1$ for all $s\in\intervalleff{t-5/4}{t-1}$.
}}
\end{minipage}}
\end{figure}
\end{proof}

\subsubsection{Proof of Corollary \ref{cor1} when $p>0$}
We finally give the
\begin{proof}[Proof of Corollary \ref{cor1} when $p>0$.]
For each $\la\in\intervalleoo{0}{1}$ and each $\pi\geq 1$, consider a $(\la,\pi)-$FFP $(\eta^{\la,\pi}_t(i))_{t\geq0,i\in\zz}$. Let also $(Z_t(x),H_t(x),F_t(x))_{t\geq0,x\in\rr}$ be a LFFP$(p)$ and consider the corresponding process $(D_t(x))_{t\geq0,x\in\rr}$.

\md

\noindent{\bf Point (b).} Using Lemma \ref{zunif}-(iii)-(iv) and recalling that $|C(\eta^{\la,\pi}_{\al t},0)|/\nl=|D_t^{\la,\pi}(0)|$, it suffices to check that for all $t\geq 3/2$ and all $B>0$, when $\la\to0$ and $\pi\to\infty$ in the regime $\cR(p)$,
\[\lim_{\la,\pi}\proba{|D_t^{\la,\pi}(0)|\geq B}=\proba{|D_t(0)|\geq B}.\]
This follows from Theorem \ref{converge}-2, which implies that $|D_t^{\la,\pi}(0)|$ goes in law to $\abs{D_t(0)}$ and from Lemma \ref{zunif}-(ii).

\md

\noindent{\bf Point (a).} Due to Lemma \ref{zunif}-(v) we only need that for all $0<a<b<1$, all $t\geq (5+p)/2$, when $\la\to0$ and $\pi\to\infty$ in the regime $\cR(p)$,
\[\lim_{\la,\pi}\proba{|C(\eta^{\la,\pi}_{\al t},0)|\in\intervalleff{\la^{-a}}{\la^{-b}}}=\proba{Z_t(0)\in\intervalleff{a}{b}}.\]
But using Theorem \ref{converge}-3 and Lemma \ref{zunif}-(i), we know that
\[\lim_{\la,\pi}\proba{\frac{\log(|C(\eta^{\la,\pi}_{\al t},0)|)}{\log(1/\la)}\indiq{|C(\eta^{\la,\pi}_{\al t},0)|\geq1}\in\intervalleff{ a}{b}}=\proba{Z_t(0)\in\intervalleff{a}{b}}\]
 when $\la\to0$ and $\pi\to\infty$ in the regime $\cR(p)$. One immediately concludes.
\end{proof}

\section{Convergence in the regime $\cR(0)$}\label{convergence in the regime 0}
The aim of this section is to prove Theorem \ref{converge restriction} when $p=0$ and this will conclude the proof of Theorem \ref{converge}.

In the whole section, we fix the parameters $A>0$ and $T>2$. We omit the subscript/superscript $A$ in the whole proof. The proof follows the ideas of the Section \ref{convergence in the regime p}.

We recall that $\al=\log(1/\la)$, $\nl=\lfloor 1/(\la\al)\rfloor$, $\ml=\lfloor 1/(\la\ba_\la^2)\rfloor$, $\e_\la=1/\ba_\la^ {3}$. We set as usual $A_\la=\lfloor\nl A\rfloor$ and $I_A^\la=\intervalleentier{-A_\la}{A_\la}$. For $i\in\zz$, we set $i_\la=\intervallefo{i/\nl}{(i+1)/\nl}$. For $\intervalleff{a}{b}$ an interval of $\intervalleff{-A}{A}$ and $\la\in\intervalleoo{0}{1}$, we recall, assuming that $-A<a<b<A$, that
\begin{align*}
\intervalleff{a}{b}_\la &= \intervalleentier{\lfloor\nl a\rfloor+\ml+1}{\lfloor\nl b\rfloor-\ml-1}\subset\zz,\\
\intervalleff{-A}{b}_\la &= \intervalleentier{-A_\la}{\lfloor\nl b\rfloor-\ml-1}\subset\zz,\\
\intervalleff{a}{A}_\la &= \intervalleentier{\lfloor\nl a\rfloor+\ml+1}{A_\la}\subset\zz.
\end{align*}
For $\la\in\intervalleoo{0}{1}$ and $\pi\geq1$, we recall that
\[\varkappa_{\la,\pi} = \frac{2\nl A}{\al\pi}+\e_\la.\]

For $x\in\intervalleoo{-A}{A}$, $\la\in\intervalleoo{0}{1}$ and $\pi\geq1$, we also recall that
\[(x)_\la =\intervalleentier{\lfloor\nl x\rfloor-\ml}{\lfloor\nl x\rfloor+\ml}\subset\zz.\]

\subsection{Occupation of vacant zone}
For simplicity, we recall Lemma \ref{speed}.
\begin{lem}\label{speed0}
Consider a family of i.i.d. Poisson processes $(N^S_t(i))_{t\geq0,i\in\zz}$. Let $a<b$.
\begin{enumerate}
        \item For $t<1$, $\lim_{\la\to 0} \proba{\forall i \in
\intervalleentier{\lfloor a\ml \rfloor}{\lfloor b\ml \rfloor},
N^S_{\al t}(i)>0}=0$;
        \item For $t\geq1$, $\lim_{\la\to 0} \proba{\forall i \in
\intervalleentier{\lfloor a\ml \rfloor}{\lfloor b\ml \rfloor},
N^S_{\al t}(i)>0}=1$;
        \item For $t<1$, $\lim_{\la\to 0} \proba{\forall i \in
\intervalleentier{\lfloor a\nl \rfloor}{\lfloor b\nl \rfloor}, N^S_{\al
t}(i)>0}=0$;
        \item For $t\geq1$, $\lim_{\la\to 0} \proba{\forall i \in
\intervalleentier{\lfloor a\nl \rfloor}{\lfloor b\nl \rfloor}, N^S_{\al
t}(i)>0}=1$;
        \item For $t>0$, $\lim_{\la\to 0} \proba{\exists i \in
\intervalleentier{\lfloor a\nl \rfloor}{\lfloor b\nl \rfloor}, N^S_{\al
t}(i)>0}=1$;
		\item For $t>0$ and $\delta>0$, $\lim_{\la\to0}\proba{\forall i\in\intervalleentier{-\lfloor\la^{-(t+\delta)}\rfloor}{\lfloor\la^{-(t+\delta)}\rfloor}, N^S_{\al t}(i)>0}=0$;
		\item For $t>0$ and $\delta>0$, $\lim_{\la\to0}\proba{\forall i\in\intervalleentier{-\lfloor\la^{-(t-\delta)}\rfloor}{\lfloor\la^{-(t-\delta)}\rfloor}, N^S_{\al t}(i)>0}=1$.
\end{enumerate}
\end{lem}

\subsection{Height of the barrier}\label{height0}
We describe here the time needed for a destroyed microscopic cluster to be
regenerated. Roughly, we assume that the zone $(x_1)_\la$ around $\lfloor\nl x_1\rfloor$, for some $x_1\in\intervalleff{-A}{A}$, has been made vacant at some time $\al t_0$. Then we consider the situation where a match falls on $\lfloor\nl x_1\rfloor$ at some time $\al t_1\in\intervalleoo{\al t_0}{\al(t_0+1)}$ and we compute the delay needed for the destroyed cluster to be fully regenerated. As in Subsection \ref{heightp}, we have to distinguish the cases $t_0=0$ and $t_0>1$.

\begin{lem}\label{micro fire 0}
Consider two Poisson processes $(N_t^S(i))_{t\geq0,i\in\zz}$ and
$(N_t^P(i))_{t\geq0,i\in\zz}$ with respective rates $1$ and $\pi$, all these
processes being independent. Consider also $\cM\coloneqq ((x_0,t_0),(x_1,t_1))$ with $x_0,x_1\in\intervalleoo{-A}{A}$, $t_0\in\{0\}\cup\intervalleoo{1}{\infty}$ and $t_1\in\intervalleoo{t_0}{t_0+1}$. For $i\in I^\la_A$ and $t\geq0$, we consider the process
\begin{align*}
\zeta^{\la,\pi,\cM}_t(i) =& \left(1+\indiq{t\geq\al
t_0,i=\lfloor\nl x_0\rfloor}\right)\times\indiq{t_0>1}\\
&+\indiq{t\geq\al
t_1,i=\lfloor\nl x_1\rfloor,\zeta^{\la,\pi,\cM}_{\al t_1-}(\lfloor\nl x_1\rfloor)=1}+\int_0^t\indiq{\zeta^{\la,\pi,\cM}_{s-}(i)=0}\diff
N_s^S(i)\\
&+\int_0^t \indiq{\zeta^{\la,\pi,\cM}_{s-}(i+1)=2,\zeta^{\la,\pi,\cM}_{s-}(i)=1}\diff
N_s^P(i+1)\\
&+\int_0^t
\indiq{\zeta^{\la,\pi,\cM}_{s-}(i-1)=2,\zeta^{\la,\pi,\cM}_{s-}(i)=1}\diff N_s^P(i-1)\\
&- 2\int_0^t \indiq{\zeta^{\la,\pi,\cM}_{s-}(i)=2}\diff N_s^P(i)
\end{align*}
with the convention $\zeta^{\la,\pi,\cM}_t(\lfloor\nl A\rfloor+1)=\zeta^{\la,\pi,\cM}_t(-\lfloor\nl A\rfloor-1)=0$ for all $t\in\intervallefo{0}{\infty}$.

Using the Poisson processes $(N^P(i))_{t\geq 0,i\in\zz}$, consider the burning times $(T^1_i)_{i\in\zz}$ of the propagation processes iginited at $(x_1,t_1)$, recall Definition \ref{definition1 application}, and define the destroyed
cluster due to the match falling in $\lfloor\nl x_1\rfloor$ at time $\al t_1$, recall Definition \ref{def destroyed comp},
\[C^P((\zeta^{\la,\pi}_t(i))_{t\geq0,i\in\zz},(x_1,t_1))\coloneqq\intervalleentier{\lfloor\nl x_1\rfloor+i^g}{\lfloor\nl x_1\rfloor+i^d}.\]

We finally define the time needed for $C^P((\zeta^{\la,\pi,\cM}_t(i))_{t\geq0,i\in\zz},(x_1,t_1))$ to
become again occupied
\[\Theta_{\cM}^{\la,\pi}\coloneqq\inf\left\{ t>t_1 : \forall i\in
C^P((\zeta^{\la,\pi,\cM}_t(i))_{t\geq0,i\in\zz},(x_1,t_1)),\zeta^{\la,\pi,\cM}_{\al t}(i)=1 \right\}.\]

For all $\delta>0$, there holds that, 
\[\lim_{\la ,\pi} \proba{\abs{\Theta_{\cM}^{\la,\pi} - (t_1-t_0)}\geq
\delta} = 0\]
when $\la\to0$ and $\pi\to\infty$ in the regime $\cR(0)$. 
\end{lem}
The process $(\zeta^{\la,\pi,\cM}_t(i))_{t\geq 0,i\in\zz}$ defined in Lemma \ref{micro fire 0} is closely related to the process defined in Lemma \ref{micro fire p}. If $t_0=0$, then the process starts from a vacant initial situation and a match falls on $\lfloor\nl x_1\rfloor$ at time $\al t_1$. It does not depend on $x_0\in\rr$. Since $0<t_1<1$, the zone $(x_1)_\la$ is not completely filled at time $\al (t_1+\varkappa_{\la,\pi})$, see Lemma \ref{speed0}-1 (using space stationarity). The process is then governed by the propagation processes $(N_t^P(i))_{t\geq0,i\in\zz}$ and the seed processes $(N_t^S(i))_{t\geq0,i\in\zz}$ with the same rules as the $(\la,\pi)-$FFP. As seen in {\bf Micro$(0)$} in Subsection \ref{application ffp}, the fire is extinguish at time $\al(t_1+\varkappa_{\la,\pi})$.

If $t_0>1$, then the process starts at time $0$ from an occupied initial situation, nothing happens until a match falls in $\lfloor\nl x_0\rfloor\in I_{A}^\la$ at time $\al t_0$. Two fires start: one goes to the left and one goes to the right. Thus, on $\Omega_{\la,\pi}^{P,2A,2A}(x_0, t_0)$, recall Definition \ref{definition2 application}, each site of $I_{A}^\la$ burns and extinguishes before $\al(t_0+\varkappa_{\la,\pi})$, recall {\bf Macro$(0)$} in Subsection \ref{application ffp}. Hence, the zone $(x_1)_\la$ is not completely filled when the match falls on $\lfloor\nl x_1\rfloor$ at time $\al t_1$, see Lemma \ref{speed0}-1, because $\al(t_0+\varkappa_{\la,\pi})<\al t_1<\al(t_0+1)$ for all $(\la,\pi)$ sufficiently close to the regime $\cR(0)$.
\begin{proof}
The proof is very similar to the proof of Lemma \ref{micro fire p}. We first define the simplest process with an instantaneous propagation: if a match falls in a cluster, it destroys
instantaneously the entire connected component. Secondly, we flank the killed cluster
$C^P((\zeta^{\la,\pi,\cM}_t(i))_{t\geq0,i\in\zz},(x_1,t_1))$ to estimate the time needed to become again occupied.

Without loss of generality, we assume that $x_1=0$ and $x_0\in\intervalleff{-A}{A}$ (using space stationarity).

\md

\noindent{\bf Step 1.} Let $\tau_0<\tau_1<\tau_0+1$ be fixed. Put
$\vartheta_{\tau_0,t}^\la(i)=\min(N^{S}_{\al (\tau_0+t)}(i)-N^{S}_{\al \tau_0}(i),1)$ and
$\vartheta_{\tau_1,t}^\la(i)=\min(N^{S}_{\al(\tau_1+t)}(i)-N^{S}_{\al
\tau_1}(i),1)$ for all $t>0$ and all $i\in \zz$. We define the time needed for the destroyed cluster to be fully regenerated
\[\Xi_{\tau_0,\tau_1}^\la=\inf \left\{ t>0 : \forall i \in
C(\vartheta_{\tau_0,\tau_1-\tau_0}^\la,0),\; \vartheta_{\tau_1,t}^\la(i)=1 \right\}.\]
Then for all $\delta>0$, 
\[\lim_{\la \to 0} \proba{ |\Xi_{\tau_0,\tau_1}^\la -(\tau_1- \tau_0)|\geq \delta} =
0.\]
This has been checked in Step 1 in the proof of Lemma \ref{micro fire p}.

\md

\noindent{\bf Step 2.} Assume $t_0=0$. In that case, the process does not depend on $x_0$. Consider the event $\Omega^{P,2A,2A}_{\la,\pi}(0,t_1)$, recall Definition \ref{definition2 application}. We  define
\begin{multline*}
\tOmega^{P,A,\cM}_{\la,\pi}=\Omega^{P,2A,2A}_{\la,\pi}(0,t_1)\cap\{\exists i_1\in\intervalleentier{-\ml}{0}, N^S_{\al (t_1+\varkappa_{\la,\pi})}(i_1)=0\}\\
\cap\{\exists i_2\in\intervalleentier{0}{\ml}, N^S_{\al (t_1+\varkappa_{\la,\pi})}(i_2)=0\}.
\end{multline*}
Lemma \ref{propagation lemma 0} together with Lemma \ref{speed0}-1 show that $\proba{\tOmega^{P,A,\cM}_{\la,\pi}}$ tends to $1$ when $\la\to0$ and $\pi\to\infty$ in the regime $\cR(0)$ (because $t_1+\varkappa_{\la,\pi}<(t_1+1)/2<1$ for all $(\la,\pi)$ sufficiently close to the regime $\cR(0)$).

Next, on $\tOmega^{P,A,\cM}_{\la,\pi}(0,t_1)$, there holds that
\[C(\vartheta_{0,t_1+\kappa_{\la,\pi}}^\la,0)\coloneqq\intervalleentier{C^-}{C^+}\subset\intervalleentier{i_1}{i_2}\subset\intervalleentier{-\ml}{\ml}.\]
Since, by definition, no seed falls on $C^+$ and on $C^-$ until $\al(t_1+\varkappa_{\la,\pi})$ and since we start from a vacant initial situation, we also deduce that
\[\zeta^{\la,\pi,\cM}_t(C^-)=\zeta^{\la,\pi,\cM}_t(C^+)=0\]
for all $t\in\intervalleff{0}{\al(t_1+\varkappa_{\la,\pi})}\supset\intervalleff{\al t_1}{\al(t_1+\varkappa_{\la,\pi})}$. As seen in {\bf Micro$(0)$} in Subsection \ref{application ffp}, the match falling on $0$ at time $\al t_1$ destroys exactly the zone $C^P((\zeta^{\la,\pi,\cM}_t(i))_{t\geq 0,i\in\zz},(0,t_1))$ and 
\[C^P((\zeta^{\la,\pi,\cM}_t(i))_{t\geq 0,i\in\zz},(0,t_1))\subset\intervalleentier{C^-}{C^+}\subset\intervalleentier{-\ml}{\ml}\]
with $\zeta^{\la,\pi,\cM}_{\al (t_1+\varkappa_{\la,\pi})}(i)\leq 1$ for all $i\in\zz$ (the fire is extinguished at time $\al(t_1+\varkappa_{\la,\pi})$). 

Since $C^P((\zeta^{\la,\pi,\cM}_t(i))_{t\geq 0,i\in\zz},(0,t_1))$ clearly contains $C(\vartheta_{0,t_1}^\la,0)$, we deduce that, on $\tOmega^{P,A,\cM}_{\la,\pi}$,
\[t_1+\Xi_{0,t_1}^\la\leq t_1+\Theta^{\la,\pi}_{\cM}\leq t_1+\varkappa_{\la,\pi}+\Xi_{0,t_1+\varkappa_{\la,\pi}}^\la.\]
Remark now that the function $\colon t\mapsto t+\Xi^\la_{0,t}$ is a.s. non decreasing and right-continuous. We thus deduce from Step 1 that
\[t_1+\Theta_{\cM}^{\la,\pi}\xrightarrow[\la,\pi]{\pp}2t_1\]
in probability, whence  for all $\delta>0$ and all $\e>0$, there holds  that $\proba{
\abs{\Theta^{\la,\pi}_{\cM}- t_1}\geq
\delta} <\e$ for all $(\la,\pi)$ sufficiently close to the regime $\cR(0)$.

\md

\noindent{\bf Step 3.} Assume now $t_0>1$. We may and will assume $x_0\in\intervalleoo{-A}{0}$, by symmetry.

Consider the events $\Omega^{P,2A,2A}_{\la,\pi}(x_0,t_0)$ and $\Omega^{P,2A,2A}_{\la,\pi}(0,t_1)$, recall Definition \ref{definition2 application}. We define
\begin{multline*}
\tOmega^{P,A,\cM}_{\la,\pi}\coloneqq \Omega^{P,2A,2A}_{\la,\pi}(0, t_1)\cap\Omega^{P,2A,2A}_{\la,\pi}(x_0,t_0)\\
\cap\{\exists i_1 \in\intervalleentier{-\ml}{0},N^S_{\al (t_1+\varkappa_{\la,\pi})}(i_1)-N^S_{\al t_0}(i_1)=0\}\\
\cap\{\exists i_2 \in\intervalleentier{0}{\ml},N^S_{\al (t_1+\varkappa_{\la,\pi})}(i_2)-N^S_{\al t_0}(i_2)=0\}.
\end{multline*}
Lemma \ref{propagation lemma 0} together with Lemma \ref{speed0}-1 directly imply that $\proba{\tOmega^{P,A,\cM}_{\la,\pi}}$ tends to $1$ when $\la\to0$ and $\pi\to \infty$ in the regime $\cR(0)$ (because $t_1+\varkappa_{\la,\pi}-t_0<(t_1-t_0+1)/2<1$ for all $(\la,\pi)$ sufficiently close to the regime $\cR(0)$).

First, since the sites $\lfloor\nl A\rfloor+1$ and $-\lfloor\nl A\rfloor-1$ remain vacant all the time and since $I^\la_A$ is completely occupied at time $\al t_0$, on $\Omega^{P,2A,2A}_{\la,\pi}(x_0,t_0)$, as seen in {\bf Macro$(0)$} in Subsection \ref{application ffp}, the match falling on $\lfloor\nl x_0\rfloor$ at time $\al t_0$ destroys each site of $I^\la_{A}$ during the time interval $\intervalleff{\al t_0}{\al(t_0+\varkappa_{\la,\pi})}$. Furthermore, there is no more burning tree in $I^\la_A$ at time $\al(t_0+\varkappa_{\la,\pi})$.

Next, on $\tOmega_{\la,\pi}^{P,A,\cM}$, since no seed falls on $i_1$ and $i_2$ during the time interval $\intervalleff{\al t_0}{\al(t_1+\varkappa_{\la,\pi})}$, we clearly have
\[C(\vartheta_{t_0,t_1+\varkappa_{\la,\pi}}^\la,0)\coloneqq \intervalleentier{C^-}{C^+}\subset \intervalleentier{i_1}{i_2}\subset\intervalleentier{-\ml}{\ml}.\]

Since, by definition, no seed falls on $C^-$ and on $C^+$ during $\intervalleff{\al t_0}{\al(t_1+\varkappa_{\la,\pi})}$ and since $C^-$ and $C^+$ are made vacant during the time interval $\intervalleff{\al t_0}{\al (t_0+\varkappa_{\la,\pi})}$, we deduce that
\[\zeta^{\la,\pi,\cM}_{\al t}(C^-)=\zeta^{\la,\pi,\cM}_{\al t}(C^+)=0 \text{ for all } t\in\intervalleff{t_1}{t_1+\varkappa_{\la,\pi}}.\]
Hence, as seen in {\bf Micro$(0)$} in Subsection \ref{application ffp}, the match falling on $0$ at time $\al t_1$ destroys exactly the zone $C^P((\zeta^{\la,\pi,\cM}_t(i))_{t\geq 0,i\in\zz},(0,t_1))\subset\intervalleentier{C^-}{C^+}\subset \intervalleentier{i_1}{i_2}$.

To summarize, since $C^P((\zeta^{\la,\pi,\cM}_t(i))_{t\geq 0,i\in\zz},(0, t_1))$ clearly contains $C(\vartheta_{t_0+\varkappa_{\la,\pi},t_1}^\la,0)$, on $\tOmega^{P,A,\cM}_{\la,\pi}$, we have
\[C(\vartheta_{t_0+\varkappa_{\la,\pi},t_1}^\la,0) \subset
C^P((\zeta^{\la,\pi,\cM}_t(i))_{t\geq 0,i\in\zz},(0, t_1))\subset C(\vartheta_{t_0,t_1+\varkappa_{\la,\pi}}^\la,0)\subset \intervalleentier{i_1}{i_2}\]
with additionally $\zeta^{\la,\pi,\cM}_{\al (t_1+\varkappa_{\la,\pi})}(i)\leq 1$ for all $i\in I^\la_{A}$.

We deduce that, on $\tOmega_{\la,\pi}^{P,A,\cM}$ and for all $(\la,\pi)$ sufficiently close to the regime $\cR(0)$,
\[t_1+\Xi_{t_0+\varkappa_{\la,\pi},t_1}^\la\leq t_1+\Theta^{\la,\pi}_{\cM}\leq t_1+\varkappa_{\la,\pi}+\Xi_{t_0,t_1+\varkappa_{\la,\pi}}^\la.\]

Then, one easily concludes. The function $s\mapsto t_1+\Xi_{t_0+s,t_1}^\la$ is a.s. non increasing and right-continuous, while the function  $s\mapsto t_1+s+\Xi_{t_0,t_1+s}^\la$ is a.s. non decreasing and right-continuous. We thus deduce from Step 1 that
\[t_1+\Theta^{\la,\pi}_{\cM}\xrightarrow[\la,\pi]{\pp}2t_1-t_0,\]
as desired.\qedhere
\end{proof}

\subsection{Persistent effect of microscopic fires}\label{persis0}
Here we study the effect of microscopic fires. First, they produce a barrier, and
then, if there are alternatively macroscopic fires on the left and right, they still
have an effect. This phenomenon is illustrated on Figure \ref{persiseffect0} in the case of the limit process.

We say that $\cP=(\e;(x_0,t_0),(x_1,t_1),\dots,(x_K,t_K))$
satisfies $(PP)$ if 
\begin{enumerate}
        \item $K\geq2$ and $\e\in\{-1,1\}$;
        \item $t_0\in\{0\}\cup\intervalleoo{1}{\infty}$ and $t_0<t_1<t_2<\dots<t_K$;
        \item for all $k=0,\dots,K-1$, $t_{k+1}-t_k<1$;
        \item $t_2-t_0>1$ and for all $k=2,\dots,K-2, t_{k+2}-t_k>1$;
        \item for all $k=0,\dots,K$, $x_k\in\intervalleoo{-A}{A}$ and for all $k=2,\dots,K$, $\e_k (x_k-x_1)>0$, where we set $\e_k=(-1)^k\e$.
\end{enumerate}

Let $\cP$ satisfy $(PP)$. Consider two Poisson processes $(N_t^S(i))_{t\geq0,i\in\zz}$ and
$(N_t^P(i))_{t\geq0,i\in\zz}$ with respective rates $1$ and $\pi$, all these
processes being independent. We define the process $(\zeta^{\la,\pi,\cP}_t(i))_{t\geq 0,i\in I^\la_A}$ as follows
\begin{align*}
{\zeta}^{\lambda,\pi,\cP}_t(i)=&(1+\indiq{i=\lfloor\nl x_0\rfloor,t\geq \al t_0})\indiq{t_0\geq1}+\indiq{i=\lfloor\nl x_1\rfloor,t\geq\al t_1,{\zeta}^{\lambda,\pi,\cP}_{\al t_1-}(\lfloor\nl x_1\rfloor)=1}\\
&+\sum_{k=2}^K \indiq{i=\lfloor\nl x_k\rfloor,t\geq\al t_k,{\zeta}^{\lambda,\pi,\cP}_{\al t_k-}(\lfloor\nl x_k\rfloor)=1}\\
&+\int_0^t\indiq{{\zeta}^{\lambda,\pi,\cP}_{s-}(i)=0}\diff
N_s^S(i)\\
&+\int_0^t
\indiq{{\zeta}^{\lambda,\pi,\cP}_{s-}(i-1)=2,{\zeta}^{\lambda,\pi,\cP}_{s-}(i)=1}\diff
N_s^P(i-1)+\int_0^t
\indiq{{\zeta}^{\lambda,\pi,\cP}_{s-}(i+1)=2,{\zeta}^{\lambda,\pi,\cP}_{s-}(i)=1}\diff
N_s^P(i+1)\\
&- 2\int_0^t \indiq{{\zeta}^{\lambda,\pi,\cP}_{s-}(i)=2}\diff
N_s^P(i)
\end{align*}
with the convention ${\zeta}^{\lambda,\pi,\cP}_t(\lfloor\nl A\rfloor+1)={\zeta}^{\lambda,\pi,\cP}_t(-\lfloor\nl A\rfloor-1)=0$ for all $t\in\intervallefo{0}{\infty}$.

We now explain the behaviour of the process $(\zeta^{\la,\pi,\cP}_t(i))_{t\geq 0,i\in I_A^\la}$.
\begin{itemize}
	\item If $t_0=0$, then the process starts from a vacant initial configuration. The match falling on $\lfloor\nl x_1\rfloor$ at time $\al t_1\in\intervalleoo{0}{\al}$ creates a barrier, see Lemma \ref{micro fire 0}, because $t_1\in\intervalleoo{0}{1}$. Then, fires start in $\lfloor\nl x_k\rfloor$ alternately on the right and on the left of $\lfloor\nl x_1\rfloor$ at times $\al t_k$ for all $k=2,\dots,K$ and fires spread accross $\zz$ according to the same rules as the $(\la,\pi,A)-$FFP.
	\item If $t_0>1$, the process starts from an occupied initial situation. Nothing happens until a match falls in $\lfloor\nl x_0\rfloor$ at time $\al t_0$ and spreads across $I^\la_A$ (because all the sites are occupied at time $\al t_0-$ and $\lfloor\nl A\rfloor+1$ and $-\lfloor\nl A\rfloor-1$ are vacants). Next, a match falls on $\lfloor\nl x_1\rfloor$ at time $\al t_1\in\intervalleoo{\al t_0}{\al(t_0+1)}$. It then creates a barrier, see Lemma \ref{micro fire 0}. Afterwards, matches fall successively in $\lfloor\nl x_k\rfloor$ at times $\al  t_k$ for each $k=2,\dots,K$ and fires spread accross $I^\la_A$ according to the same rules as the $(\la,\pi,A)-$FFP.
\end{itemize}

Consider the event
\[\Omega_{\cP}^{S,P}(\la,\pi)=\{\forall k\in\{2,\dots,K\},\,\exists j\in(x_1)_\la, \,\forall t\in\intervalleff{t_k+\varkappa_{\la,\pi}}{t_k+1},\,\zeta^{\la,\pi,\cP}_{\al t}(j)=0\}.\]

\begin{lem}\label{persisplem0}
Let $\cP=(\e;(x_0,t_0),(x_1,t_1),\dots,(x_K,t_K))$ satisfy $(PP)$. For each
$\lambda\in\intervalleoo{0}{1}$ and each $\pi\geq1$, consider the process $({\zeta}^{\lambda,\pi,\cP}_t(i))_{t\geq0,i\in\zz}$ defined above.

If $t_2-t_1<t_1-t_0$, when $\la\to0$ and $\pi\to\infty$ in the regime $\cR(0)$, there holds
\[\lim_{\lambda,\pi}\proba{\Omega_{\cP}^{S,P}(\lambda,\pi)}=1.\]
\end{lem}

\begin{proof}
Without loss of generality, we assume $x_1=0$ and $(x_k)_{k=0,2,\dots,K}\subset\intervalleff{-A}{A}$.

We define, recall Definition \ref{definition2 application}, 
\[\Omega^{P,A,\cP}_{\la,\pi}=\Omega^{P,2A,2A}_{\la,\pi}(0,t_1)\cap\bigcap_{k=0,2,\dots,K}\Omega^{P,2A,2A}_{\la,\pi}(x_k,t_k).\]
There holds that $\proba{\Omega^{P,A,\cP}_{\la,\pi}}$ tends to $1$ as $\la\to0$ and $\pi\to\infty$ in the regime $\cR(0)$ by Lemma \ref{propagation lemma 0}. In the whole proof, we work on $\Omega^{P,A,\cP}_{\la,\pi}$ and assume that $(\la,\pi)$ is sufficiently close to the regime $\cR(0)$ in such a way that $\varkappa_{\la,\pi}<\min_{i\neq j}\abs{t_i-t_j}$ and $\min_{k=0,2,\dots,K}|\lfloor\nl x_k\rfloor|\geq \ml$.

For simplicity, we assume that $\e=-1$, $t_0=0$ and that $K$ is even. The other cases are treated similarly (see for example Lemma \ref{micro fire 0}). Fix $\alpha=1/K$. We define $\cM\coloneqq((0,0),(0,t_1))$, recall Lemma \ref{micro fire 0}.

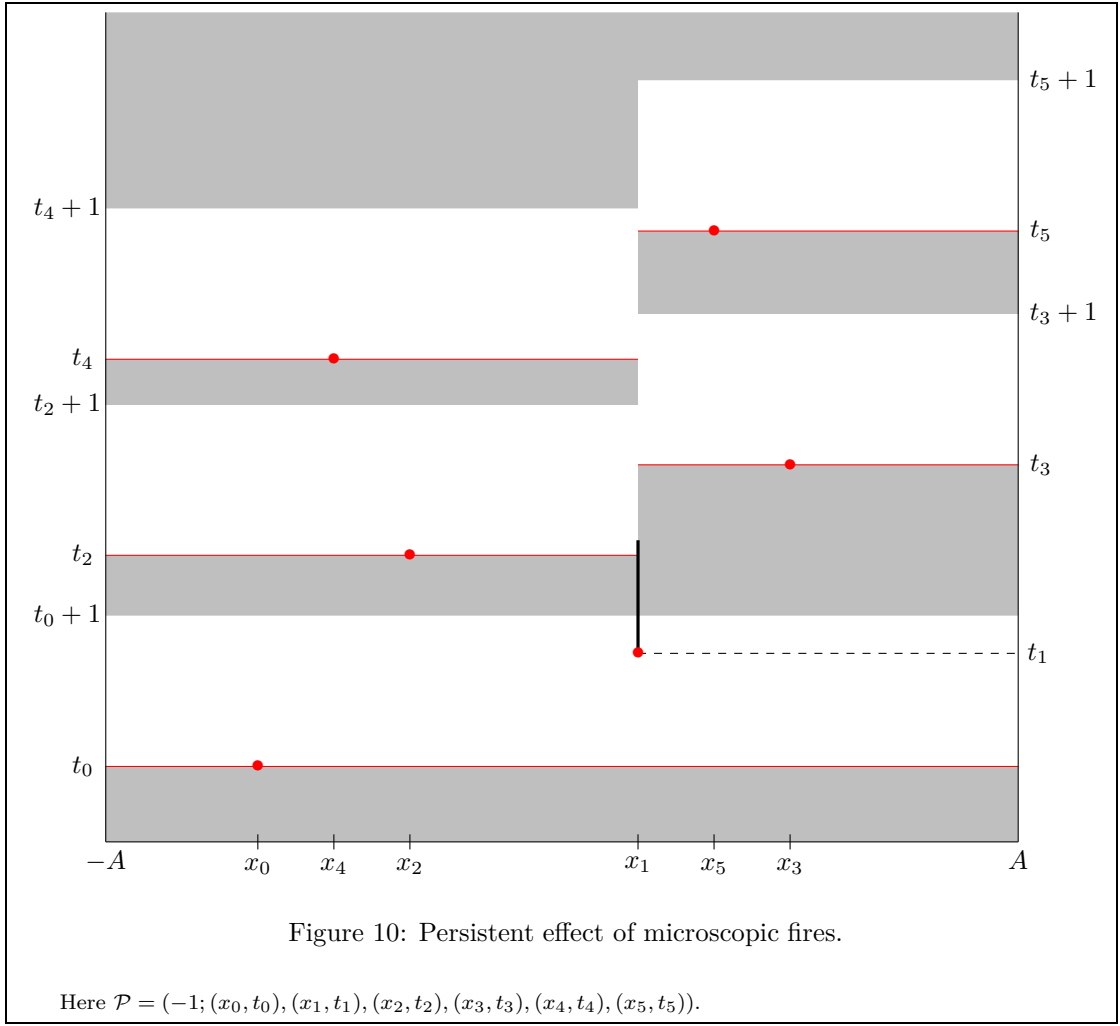
\begin{figure}[h!]
\fbox{
\begin{minipage}[c]{0.95\textwidth}
\centering
\begin{tikzpicture}
\fill[fill=gray!50!white] (-7,0) rectangle (5,1);

\draw[red] (5,1) --(-7,1) node at (-5,1) {$\bullet$};
\draw node at (-7.3,1) {$t_0$};
\draw (-5,.1) -- (-5,-.1) node[below] {$x_0$};

\fill[fill=gray!50!white] (-7,8.4)--(0,8.4)--(0,10.1)--(5,10.1)--(5,11)--(-7,11)--cycle;

\fill[fill=gray!50!white] (-7,3) -- (5,3) -- (5,5)--(0,5)--(0,3.8) -- (-7,3.8)--cycle;
\draw[red] (-7,3.8) --(0,3.8) node at (-3,3.8) {$\bullet$};
\draw node at (-7.3,3.8) {$t_2$};
\draw (-3,.1) -- (-3,-.1) node[below] {$x_2$};
\draw[red] (0,5) --(5,5) node at (2,5) {$\bullet$};
\draw node at (5.3,5) {$t_3$};
\draw (2,.1) -- (2,-.1) node[below] {$x_3$};

\draw[dashed] (0,2.5)--(5,2.5) node[right] {$t_1$};
\draw[very thick] (0,2.5)--(0,4);
\draw[red] node at (0,2.5) {$\bullet$};

\fill[fill=gray!50!white] (-7,5.8) --(0,5.8)-- (0,6.4)--(-7,6.4)--cycle;
\draw[red] (-7,6.4) --(0,6.4) node at (-4,6.4) {$\bullet$};
\draw[dashed] node at (-7.3,6.4) {$t_4$};
\draw (-4,.1) -- (-4,-.1) node[below] {$x_4$};

\fill[fill=gray!50!white] (0,7) --(5,7)-- (5,8.1)--(0,8.1)--cycle;
\draw[red] (0,8.1) --(5,8.1) node at (1,8.1) {$\bullet$};
\draw[dashed] node at (5.3,8.1) {$t_5$};
\draw (1,.1) -- (1,-.1) node[below] {$x_5$};

\draw (-7,0)--(5,0) node at (0,-.3) {$x_1$};
\draw (0,.1)--(0,-.1);
\draw (-7,11)--(-7,0) node[below] {$-A$};
\draw (5,11)--(5,0) node[below] {$A$};

\draw node at (-7.5,3) {$t_0+1$} node at (-7.5,5.8) {$t_2+1$} node at (-7.5,8.4) {$t_4+1$};

\draw node at (5.6,7) {$t_3+1$} node at (5.6,10.1) {$t_5+1$};

\end{tikzpicture} \caption{Persistent effect of microscopic fires.}\label{persiseffect0}
\vspace{.5cm}
\parbox{13.3cm}{
\footnotesize{
Here $\cP=(-1;(x_0,t_0),(x_1,t_1),(x_2,t_2),(x_3,t_3),(x_4,t_4),(x_5,t_5))$.
}}
\end{minipage}}
\end{figure}

Since $\lfloor\nl A\rfloor+1$ and $-\lfloor\nl A\rfloor-1$ remain vacant all the time, on $\Omega^{P,A,\cP}_{\la,\pi}$, a burning tree at time $\al t$ is either a front of a fire or has vacant neighbors. Thus, there is no burning tree outside $\cup_{k=1,\dots,K}\intervalleff{\al t_k}{\al(t_k+\varkappa_{\la,\pi})}$.

\md

\noindent{\bf First fire.} We put $C^P\coloneqq C^P((\zeta^{\la,\pi,\cP}_t(i))_{t\geq 0,i\in\zz},(0,t_1))$, the destroyed cluster, recall \eqref{destroyed comp}. Since $t_1+\varkappa_{\la,\pi}<1$, $C^P\subset\intervalleentier{-\lfloor\alpha\ml\rfloor}{\lfloor\alpha\ml\rfloor}$ with probability tending to $1$ (use Lemma \ref{speed0}-1, space/time stationarity and {\bf Micro$(0)$} in Subsection \ref{application ffp}). Thus the match falling at time $\al t_1$ destroys nothing outside $\intervalleentier{-\lfloor\alpha\ml\rfloor}{\lfloor\alpha\ml\rfloor}$ and there is no more burning tree in $I^\la_A$ at time $\al(t_1+\varkappa_{\la,\pi})$.

\md

\noindent{\bf Second fire.} Since $t_2>1$, at least one seed has fallen, during $\intervallefo{0}{\al t_2}$, on each site of $\intervalleentier{-\lfloor\nl A\rfloor}{-\lfloor\alpha\ml\rfloor-1}$ with probability tending to $1$ (use Lemma \ref{speed0}-4 and space/time stationarity). Since this zone has not been affected by a fire during the time interval $\intervallefo{0}{\al t_2}$, this zone is completely occupied at time $\al t_2-$.
     
Besides, with probability tending to $1$, there is (at least) an empty site in $C^P\subset\intervalleentier{-\lfloor\alpha\ml\rfloor}{\lfloor\alpha\ml\rfloor}$ during the time interval $\intervalleoo{\al(t_1+\varkappa_{\la,\pi})}{\al(t_2+\varkappa_{\la,\pi})}$ because $t_1+\varkappa_{\la,\pi}<t_2<t_2+\varkappa_{\la,\pi}<t_1+\Theta^{\la,\pi}_{\cM}$ with probability tending to $1$ (by Lemma \ref{micro fire 0}, $\Theta^{\la,\pi}_\cM\simeq t_1-t_0=t_1$ and $t_2-t_1<t_1-t_0=t_1$ by assumption) and because by definition of $\Theta^{\la,\pi}_{\cM}$, there is an empty site in $C^P\subset\intervalleentier{-\lfloor\alpha\ml\rfloor}{\lfloor\alpha\ml\rfloor}$ during $\intervalleff{\al(t_1+\varkappa_{\la,\pi})}{\al(t_1+\Theta^{\la,\pi}_{\cM})}$.
        
Thus, the fire ignited on $\lfloor\nl x_2\rfloor\in \intervalleentier{-\lfloor\nl A\rfloor}{-\ml}$ at time $\al t_2$ burns each site of the zone $\intervalleentier{-\lfloor\nl A\rfloor}{-\lfloor\alpha\ml\rfloor-1}$ before $\al(t_2+\varkappa_{\la,\pi})$ and does not affect the zone $\intervalleentier{\lfloor\alpha\ml\rfloor+1}{\lfloor\nl A\rfloor}$, thanks to $\Omega^{P,2A,2A}_{\la,\pi}(x_2,t_2)$, as seen in {\bf Macro$(0)$} in Subsection \ref{application ffp}.

\md

\noindent{\bf Third fire.} All the sites of $\intervalleentier{\lfloor\alpha\ml\rfloor+1}{\lfloor\nl A\rfloor}$ are occupied at time
$\al t_3-$ with probability tending to $1$
(because on $\Omega^{P,2A,2A}_{\la,\pi}(0,t_1)\cap\Omega^{P,2A,2A}_{\la,\pi}(x_2,t_2)$, they have not been affected by a fire during
$\intervallefo{0}{\al t_3}$
and because $t_3>t_2>1$, see Lemma \ref{speed0}-4).

Next, since $t_3-t_2<1$, the probability that there is a site in
$\intervalleentier{-\lfloor2\alpha\ml\rfloor}{-\lfloor \alpha\ml\rfloor-1}$ where no seed falls during $\intervallefo{\al t_2}{\al(t_2+1)}$ tends to $1$ as $\la\to0$ and $\pi\to\infty$ in the regime $\cR(0)$ (use Lemma \ref{speed0}-1 and space/time stationarity). Thus, with probability tending to $1$, there exists a vacant site in $\intervalleentier{-\lfloor2\alpha\ml\rfloor}{-\lfloor \alpha\ml\rfloor}$ during $\intervallefo{\al (t_2+\varkappa_{\la,\pi})}{\al(t_2+1)}\supset\intervalleff{\al t_3}{\al(t_3+\varkappa_{\la,\pi})}$ (because all the sites of $\intervalleentier{-\lfloor\nl A\rfloor}{-\lfloor\alpha\ml\rfloor-1}$ have been made vacant by the fire $2$).

Thus, the fire ignited on $\lfloor\nl x_3\rfloor\in\intervalleentier{\ml}{\lfloor\nl A\rfloor}$ at time $\al  t_3$ burns each site of $\intervalleentier{\lfloor\alpha\ml\rfloor+1}{\lfloor\nl A\rfloor}$ before $\al(t_3+\varkappa_{\la,\pi})$ and does not affect the zone $\intervalleentier{-\lfloor\nl A\rfloor}{-\lfloor2\alpha\ml\rfloor}$ with probability tending to $1$, thanks to $\Omega^{P,2A,2A}_{\la,\pi}(x_3,t_3)$, as seen in {\bf Macro$(0)$} in Subsection \ref{application ffp}.

\md

\noindent{\bf Fourth fire.} All the sites of $\intervalleentier{-\lfloor\nl A\rfloor}{-\lfloor2\alpha\ml\rfloor-1}$ are occupied at time $\al t_4-$ with probability tending to $1$ (because on $\Omega^{P,2A,2A}_{\la,\pi}(0,t_1)\cap\Omega^{P,2A,2A}_{\la,\pi}(x_2,t_2)\cap\Omega^{P,2A,2A}_{\la,\pi}(x_3,t_3)$, they have not been affected by a fire during $\intervalleoo{\al(t_2+\varkappa_{\la,\pi})}{\al t_4}$ and because $t_4-t_2-\varkappa_{\la,\pi}>1$, see Lemma \ref{speed0}-4 and space/time stationarity).

Since $t_4-t_3<1$, the probability that there is a site in
$\intervalleentier{\lfloor \alpha\ml\rfloor+1}{\lfloor2\alpha\ml\rfloor}$ where no seed falls during $\intervallefo{\al t_3}{\al(t_3+1)}$ tends to $1$ as $\la\to0$ and $\pi\to\infty$ in the regime $\cR(0)$ (use Lemma \ref{speed0}-1 and space/time stationarity). Hence there is at least one vacant site in $\intervalleentier{\lfloor \alpha\ml\rfloor+1}{\lfloor2\alpha\ml\rfloor}$
during
$\intervallefo{\al(t_3+\varkappa_{\la,\pi})}{\al(t_3+1)}\supset\intervalleff{\al t_4}{\al(t_4+\varkappa_{\la,\pi})}$, with probability tending to $1$ (because all the sites of $\intervalleentier{\lfloor\alpha\ml\rfloor+1}{\lfloor\nl A\rfloor}$ have been made vacant by the fire $3$).

Thus, the fire ignited on $\lfloor\nl x_4\rfloor\in\intervalleentier{-\lfloor\nl A\rfloor}{-\ml}$ at time $\al t_4$ burns each site of the zone $\intervalleentier{-\lfloor\nl A\rfloor}{-\lfloor2\alpha\ml\rfloor-1}$ before $\al(t_4+\varkappa_{\la,\pi})$ and does not affect the zone $\intervalleentier{\lfloor2\alpha\ml\rfloor+1}{\lfloor\nl A\rfloor}$ with probability tending to $1$,  thanks to $\Omega^{P,2A,2A}_{\la,\pi}(x_4,t_4)$, as seen in {\bf Macro$(0)$} in Subsection \ref{application ffp}.

\md

\noindent{\bf Last fire and conclusion.} Iterating the procedure, we see that
with probability tending to $1$ as $\la\to 0$ and $\pi\to\infty$ in the regime $\cR(0)$, the zone $\intervalleentier{-\lfloor\nl A\rfloor}{-\lfloor(K\alpha/2)\ml\rfloor-1} =\intervalleentier{-\lfloor\nl A\rfloor}{- \lfloor \ml/2
\rfloor-1}$ is completely occupied at time $\al t_K-$ and there is at least one vacant site in $\intervalleentier{\lfloor(K-1)\alpha/2\ml\rfloor+1}{\lfloor\ml/2\rfloor}$ during the time interval $\intervallefo{\al(t_{K-1}+\varkappa_{\la,\pi})}{\al(t_{K-1}+1)}\supset\intervalleff{\al t_K}{\al(t_K+\varkappa_{\la,\pi})}$. Thus, the fire ignited on $\lfloor\nl x_K\rfloor\in\intervalleentier{-\lfloor\nl A\rfloor}{-\ml}$ at time $\al t_K$ destroys each site of the zone
$\intervalleentier{-\lfloor\nl A\rfloor}{- \lfloor \ml/2
\rfloor-1}$ before $\al(t_K+\varkappa_{\la,\pi})$ and does not affect the zone $\intervalleentier{\lfloor\ml/2\rfloor}{\lfloor\nl A\rfloor}$. 

Finally, the probability that there is at least
one site in $\intervalleentier{-\ml}{- \ml/2}$ with
no seed falling during $\intervalleff{\al t_K}{\al (t_K+1)}$ tends to $1$  (by Lemma \ref{speed0}-1). Consequently, the probability that there is a vacant site in $\intervalleentier{-\ml}{-\lfloor \ml/2\rfloor}$ during $\intervalleff{\al(t_K+\varkappa_{\la,\pi})}{\al (t_K+1)}$ tends to $1$. All this  implies the claim.
\end{proof}

\subsection{Heart of the proof}

\subsubsection{The coupling}\label{coupling0}
We are going to construct a coupling between the $(\la,\pi,A)-$FFP
(on the time interval $\intervalleff{0}{\al T}$) and the $A-$LFFP$(0)$ (on $\intervalleff{0}{T}$). Let $\pi_M$ be a Poisson measure on $\rr\times \intervallefo{0}{\infty}$ with  intensity  measure $\diff x\diff t$.

First, we take for the matches of the discrete process the Poisson processes
\[N^M_t(i)=\pi_M(\intervallefo{i/\nl}{(i+1)/\nl}\times\intervalleff{0}{t/\al})\]
for all $i\in\zz$ and $t\in\intervalleff{0}{T}$.

We call $n:=\pi_M([0,T]\times\intervalleff{-A}{A})$
and we consider the marks $(T_q,X_q)_{q=1,\dots,n}$ of $\pi_M$ ordered 
in such a way that $0<T_1<\dots<T_n<T$.

Next, we introduce two families of i.i.d. Poisson processes $(N^{S}_t(i))_{t\geq 0,i\in \zz}$ and $(N^{P}_t(i))_{t\geq 0,i\in \zz}$ with respective parameters
$1$ and $\pi$, independent of $\pi_M$.

The $(\la,\pi,A)-$FFP
$(\eta^{\la,\pi}_t(i))_{t\geq 0, i \in I_A^\la}$ is built from the seed processes
$(N^{S}_t(i))_{t\geq 0, i \in \zz}$, the match processes $(N^{M}_t(i))_{t\geq 0, i \in \zz}$ and the propagation processes $(N^{P}_t(i))_{t\geq 0, i \in \zz}$.

Finally, we build the $A-$LFFP$(0)$ $(Z_t(x),H_t(x),F_t(x))_{t\in[0,T],x\in\intervalleff{-A}{A}}$ from $\pi_M$ and observe that  it is
independent of $(N^{S}_t(i))_{t\in [0,\al T], i \in \zz}$ and $(N^{P}_t(i))_{t\in [0,\al T], i \in \zz}$.

Observe that if a match falls on some $X_q$ at time $T_q$ for the $A-$LFFP$(0)$, it also falls on $\lfloor\nl X_q \rfloor$ at time $\al T_q$ in the discrete process.

\subsubsection{A favorable event}
We set $T_0=0$ and introduce 
\[\mathcal{T}_M=\{T_0,T_1,\dots,T_n\}\text{ and }\mathcal{B}_M=\{X_1,\dots,X_n\}\]
as well as the set $\cC_M$ of connected components of $\intervalleff{-A}{A}\setminus\cB_M$ (sometimes referred to as cells). We also introduce
\[\mathcal{S}_M=\{2t-s : s,t\in\cT_M,s<t\}\]
which has to be seen as the set of the possible extinction times of the microscopic fires, recall Lemma \ref{height0}.

For $\alpha>0$, we consider the event
\[\Omega_M(\alpha)=\left\{
\min_{\substack{s,t\in\mathcal{T}_M\cup\mathcal{S}_M\\
s\neq t}}\abs{t-s}\geq2\alpha,
\min_{s,t\in\mathcal{T}_M\cup\mathcal{S}_M}\abs{t-(s+1)}\geq2\alpha,
\min_{\substack{x,y\in\mathcal{B}_M\cup\{-A,A\},\\
x\neq y}}\abs{x-y}\geq2\alpha\right\}\]
which clearly satisfies $\lim_{\alpha\to0}\proba{\Omega_M(\alpha)}=1$. For any given $\alpha>0$, there exists $\la_\alpha>0$ 
such that for all $\la\in\intervalleoo{0}{\la_\alpha}$, on $\Omega_M(\alpha)$, there
holds that 
\begin{itemize}
	\item for all $x,y\in\mathcal{B}_M\cup\{-A,A\}$, with $x\neq
y$, $(x)_{\la}\cap(y)_{\la}=\emptyset$;
	\item the family $\{c_\la,c\in\cC_M\}\cup\{(x)_\la,x\in\cB_M\}$ is a partition of $I^\la_A$.
\end{itemize}

For $q\in\{1,\dots,n\}$, using the seed processes $(N^{S}_t(i))_{t\geq0,i\in\zz}$ and the propagation processes $(N^{P}_t(i))_{t\geq0,i\in\zz}$, we build, recall Definition \ref{definition1 application}, $(\check{\zeta}_t^{\la,\pi,q}(i))_{t\geq0,i\in\zz}$ (the propagation process ignited at $(X_q,T_q)$), $(i^{q,+}_t)_{t\geq0}$ and $(i^{q,-}_t)_{t\geq0}$ (the corresponding right and  left fronts) and $(T^q_i)_{i\in\zz}$ (the associated burning times). We also use $\Omega^{P,2A,2A}_{\la,\pi}(X_q,T_q)$, recall Definition \ref{definition2 application}. We set
\[\Omega^{S,P}_A(\la,\pi)=\bigcap_{q=1,\dots,n}\,\Omega^{P,2A,2A}_{\la,\pi}(X_q,T_q).\]
Since $\pi_M$ is independent of the processes $(N^{S}_t(i))_{t\geq0,i\in\zz}$ and $(N^{P}_t(i))_{t\geq0,i\in\zz}$, Lemma \ref{propagation lemma 0} implies that $\proba{\Omega^{S,P}_A(\la,\pi)}$ tends to $1$ when $\la\to0$ and $\pi\to\infty$ in the regime $\cR(0)$.

Let $q\in\{1,\dots,n\}$. We call $\mathcal{U}_q$ the set of all possible
$\cP=(\e;(x_0,t_0),(X_q,T_q),\dots,(x_K,t_K))$ satisfying $(PP)$ where $\{t_0,t_2,\dots,t_K\}\subset\cT_M$, $\{x_0,x_2,\dots,x_K\}\subset\cB_M$ with $T_q-t_0>t_2-T_q$ and with $\e\in\{-1,1\}$. For $\cP\in\mathcal{U}_q$, we introduce the
event $\Omega_\cP^{S,P}(\lambda,\pi)$, defined as in Subsection \ref{persis0}, with the Poisson
processes $(N^{S}_t(i))_{t\geq0,i\in\zz}$ and $(N^{P}_t(i))_{t\geq0,i\in\zz}$. Then we put
\[\Omega_1^{S,P}(\la,\pi)=\cap_{q=1}^n\cap_{\cP\in\cU_q}\Omega_{\cP}^{S,P}(\la,\pi),\]
which satisfies $\lim_{\la,\pi}\proba{\Omega_1^{S,P}(\la,\pi)}=1$ when $\la\to0$ and $\pi\to\infty$ in the regime $\cR(p)$, thanks to Lemma \ref{persisplem0}. 

We also consider the event $\Omega_2^S(\la,\pi)$ on which the following conditions hold:
for all $t_1,t_2\in\cT_M$ with $0<t_2-t_1<1$, for all $q=1,\dots,n$, there are
\[-\ml<i_1<0<i_2<\ml\]
such that $N^{S}_{\al (t_2+\varkappa_{\la,\pi})}(\lfloor\nl X_q\rfloor+i_j)-N^{S}_{\al t_1}(\lfloor\nl X_q\rfloor+i_j)=0$ for $j=1,2$. There holds that $\proba{\Omega_2^S(\la,\pi)}$ tends to $1$ as $\la\to$ and $\pi\to\infty$ in the regime $\cR(0)$. Indeed, it suffices to prove
that almost surely, $\lim_{\substack{\la\to0\\ \pi\to\infty}}\probacond{\Omega_2^S(\la,\pi)}{\pi_M}=1$. Since
there are a.s. finitely many possibilities for $q,t_1,t_2$ and since $\pi_M$ is independent
of $(N^{S}_t(i))_{t\geq0,i\in\zz}$, it suffices to work with a fixed $q\in\{1,\dots,n\}$ and some
fixed $0<t_2-t_1<1$. The  result then follows from Lemma \ref{speed0}-1 together with space/time stationarity.

Next we introduce the event $\Omega^S_3(\la,\pi)$ on which the following conditions hold: for all $t_1,t_2\in\cT_M\cup\cS_M$,
\begin{itemize}
	\item if $t_2-t_1>1$, for all $c\in\cC_M$, for all $i\in c_\la$ with $N^{S}_{\al t_2}(i)-N^{S}_{\al (t_1+\varkappa_{\la,\pi})}(i)>0$;
	\item if $t_2-t_1>1$, for all $x\in\cB_M$, for all $i\in (x)_\la$ with $N^{S}_{\al t_2}(i)-N^{S}_{\al (t_1+\varkappa_{\la,\pi})}(i)>0$.
\end{itemize}
There holds that $\proba{\Omega_3^S(\la,\pi)}$ tends to $1$ as $\la\to$ and $\pi\to\infty$ in the regime $\cR(0)$. As previously, it suffices to work with some fixed $t_1,t_2\in\cT_M$, $x\in\cB_M$ and $c=\intervalleoo{a}{b}\subset\intervalleoo{-A}{A}$. Observing that $|c_\la|\simeq(b-a)\nl$ and that $|(x)_\la|\simeq2\ml$, Lemma \ref{speed0} and space/time
stationarity shows the result. 

We also need $\Omega_4^{S,P}(\gamma,\la,\pi)$, defined for $\gamma>0$ as follows: for all
$q=1,\dots, n$, for all $\cM=((x_0,t_0),(X_q,T_q))$ such that $t_0\in\cT_M$ with $t_0<T_q<t_0+1$ and $x_0\in\cB_M\setminus\{X_q\}$, there holds
that $\abs{\Theta^{\la,\pi}_\cM-(T_q-t_0)}<\gamma$. Here, $\Theta^{\la,\pi}_\cM$ is defined as in Lemma \ref{micro fire 0} with
the seed processes family $(N^{S}_t(i))_{t\geq0,i\in\zz}$ and the propagation processes family $(N^{P}_t(i))_{t\geq0,i\in\zz}$. Lemma \ref{micro fire 0} directly implies that for any $\gamma>0$, $\proba{\Omega_4^{S,P}(\gamma,\la,\pi)}$ tends to $1$ as $\la\to$ and $\pi\to\infty$ in the regime $\cR(0)$.

We finally introduce the event 
\[\Omega(\alpha,\gamma,\la,\pi)=\Omega_M(\alpha)\cap\Omega^{S,P}_A(\la,\pi)\cap\Omega^{S,P}_1(\la,\pi)\cap\Omega^S_2(\la,\pi)\cap\Omega^S_3(\la,\pi)\cap\Omega^{S,P}_4(\gamma,\la,\pi).\]
We have shown that for any $\delta>0$, there exists $\alpha\in\intervalleoo{0}{1}$ such that for any $\gamma>0$, there holds that $\proba{\Omega(\alpha,\gamma,\la,\pi)}>1-\delta$ for all $(\la,\pi)$ sufficiently close to the regime $\cR(0)$.

\subsubsection{Heart of the proof}
We now handle the main part of the proof.

Consider the $A-$LFFP$(0)$. Observe that by construction, we have, for $c\in\cC_M$ and $x,y\in c$, $Z_t(x)=Z_t(y)$ for all $t\in\intervalleff{0}{T}$, thus we can introduce $Z_t(c)$.

If $x\in\cB_M$, it is at the boundary of two cells $c_-, c_+ \in\cC_M$  and then we set $Z_t(x_-)=Z_t(c_-)$ and $Z_t(x_+)=Z_t(c_+)$ for all $t\in\intervalleff{0}{T}$.

If $x\in\intervalleoo{-A}{A}\setminus\cB_M$, we put $Z_t(x_-)=Z_t(x_+)=Z_t(x)$ for all $t\in\intervalleff{0}{T}$.

For $x\in\cB_M$ and $t\geq0$ we set $\tH(x)=\min(H_t(x),1-Z_t(x),1-Z_t(x_-),1-Z_t(x_+))$.

Actually $Z_t(x)$ always equals either $Z_t(x_-)$ or $Z_t(x_+)$ and these can be distinct only at
a point where has occurred a microscopic fire (that is if $x = X_q$ for some $q\in\{1, \dots , n\}$ with $T_q< t$ and $Z_{T_q-} (X_q) < 1$).

For all $x\in \intervalleoo{-A}{A}$ and $t\in\intervalleff{0}{T}$, we put
\[\tau_t(x)=\sup\enstq{s\leq t}{Z_s(x_+)=Z_s(x_-)=Z_s(x)=0}\in\cT_M.\]
For $c\in\cC_M$ and $t\in\intervalleff{0}{T}$, we can define $\tau_t(c)$ as usual with the convention $Z_{0-}(x)=1$ for all $x\in\intervalleff{-A}{A}$.

Observe that
\begin{gather}
\text{for }x\not\in \cB_M,\, Z_t(x)=\min(t-\tau_t(x),1)\text{ for all }t\in\intervalleff{0}{T},\label{zbm}\\
\text{for } q = 1, \dots, n,\, Z_t(X_q) = \min (t - \tau_t (X_q), 1) \text{ for all } t\in\intervallefo{0}{T_q}.
\end{gather}
We also define, for all $t\in\intervalleff{0}{T}$, all $i\in I_\la^A$,
\[\rho^{\la,\pi}_t(i) =\sup\enstq{s\leq t}{\eta^{\la,\pi}_{\al s-}(i)=2}\]
with the convention $\eta^{\la,\pi}_{0-}(i)=2$ and $\eta^{\la,\pi}_{0}(i)=0$.

For $t\in\intervalleff{0}{T}$, consider the event
\[\Omega^{\la,\pi}_t=\left\{\forall s\in\intervalleff{0}{t}\setminus\bigcup_{q=1}^n\intervallefo{T_q}{T_q+\varkappa_{\la,\pi}},\,\forall c\in\cC_M,\,\forall i\in c_\la,\, \abs{\rho^{\la,\pi}_s(i)-\tau_s(c)}\leq\varkappa_{\la,\pi}\right\}.\]

\begin{lem}\label{lemconvergence}
Let $\alpha>\gamma>0$. For all $\la\in\intervalleoo{0}{\la_\alpha}$ and $\pi\geq1$ such that $\varkappa_{\la,\pi}\leq \alpha$, $\Omega_T^{\la,\pi}$ a.s. holds on $\Omega(\alpha,\gamma,\la,\pi)$.
\end{lem}

\begin{proof}
We work on $\Omega(\alpha,\gamma,\la,\pi)$ and assume that $\la\in\intervalleoo{0}{\la_\alpha}$ and $\pi\geq1$ are such that $\varkappa_{\la,\pi}\leq \alpha$. Clearly, $\tau_0(c) = 0$ and $\rho_0^{\la,\pi}(i) = 0$ for all $c\in\cC_M$ and all $i\in I^\la_A$, so that $\Omega^{\la,\pi}_0$ a.s. holds. We will show that for $q = 0,\dots, n - 1$, $\Omega^{\la,\pi}_{T_q}$ implies $\Omega_{T_{q+1}}^{\la,\pi}$. The extension to $\Omega_T^{\la,\pi}$ will be straightforward and will be omitted.

We thus fix $q\in\{0, \dots, n-1 \}$ and assume $\Omega_{T_q}^{\la,\pi}$. We repeatedly use below that for all $k\leq q$, on the
time interval $\intervalleoo{T_k}{T_{k+1}}$, there are no fires at all (in $\intervalleff{-A}{A}$) for the $A-$LFFP$(0)$ and, on $\Omega^{S,P}_A(\la,\pi)$, no burning tree at all (in $I^\la_A$) during $\intervalleoo{\al (T_k+\varkappa_{\la,\pi})}{\al T_{k+1}}$ for the $(\la,\pi,A)-$FFP.

Besides, $\eta^{\la,\pi}_{\al T_q-}(i)=\eta^{\la,\pi}_{\al T_q}(i)$ for all $i\in I^\la_A\setminus\{\lfloor\nl X_q\rfloor\}$ while 
\[\eta^{\la,\pi}_{\al T_q}(\lfloor\nl X_q\rfloor)=2\indiq{\eta^{\la,\pi}_{\al T_q-}(\lfloor\nl X_q\rfloor)=1}.\]

\md

\noindent{\bf Step 1.} Here we prove that, on $\Omega^{\la,\pi}_{T_q}$, for all $1\leq k<q$, if $D_{T_k-}(X_k)=\intervalleff{a}{b}$, for some $a<b$, $a,b\in\cB_M\cup\{-A,A\}$, then
\[\eta^{\la,\pi}_{\al T_k+T^k_{i-\lfloor\nl X_k\rfloor}}(i)=2\]
for all $i\in\intervalleff{a}{b}_\la$.

On the one hand, by construction, for all $c\in \cC_M$, $c\subset\intervalleoo{a}{b}$, we have $\tau_{T_k}(c)=T_k$. By $\Omega_{T_q}^{\la,\pi}\subset\Omega_{T_k+\varkappa_{\la,\pi}}^{\la,\pi}$, we deduce that $T_k\leq \rho^{\la,\pi}_{T_k+\varkappa_{\la,\pi}}(\lfloor\nl b\rfloor-\ml-1)\leq T_k+\varkappa_{\la,\pi}$.

On the other hand, recall Lemma \ref{propagation lemma 0}: on $\Omega^{P,2A,2A}_{\la,\pi}(X_k,T_k)$, a burning tree is either a front or has vacant neighbors. Recall that there is no burning tree at all in $I^\la_A$ at time $\al T_k-$. Assume for example that there is a site $i\in\intervalleentier{\lfloor\nl X_k\rfloor}{\lfloor\nl b\rfloor-\ml-1}$ such that $\eta^{\la,\pi}_{\al T_k+T^k_{i-\lfloor\nl X_k\rfloor}}(i)=0$. Then the fire starting at $\lfloor\nl X_k\rfloor$ at time $\al T_k$ does not affect thet zone $\intervalleentier{i}{\lfloor\nl A\rfloor}$, as seen in {\bf Macro$(0)$} in Subsection \ref{application ffp}. This especially implies that $\eta^{\la,\pi}_{\al t}(\lfloor\nl b\rfloor-\ml-1)\leq1$ for all $t\in\intervalleff{T_k}{T_k+\varkappa_{\la,\pi}}$ (because no other match falls on $I^\la_A$ during $\intervalleff{\al T_k}{\al(T_k+\varkappa_{\la,\pi})}$) whence $\rho^{\la,\pi}_{T_k+\varkappa_{\la,\pi}}(\lfloor\nl b\rfloor-\ml-1)<T_k$, a contradiction.

\md

\noindent{\bf Step 2.} We show that on $\Omega_{T_q}^{\la,\pi}$, for all $c\in\cC_M$, all $i\in c_\la$, 
\begin{equation}\label{ineq}
\underline{\eta}^{\la,\pi}_{\al T_q-}(i)\leq \eta^{\la,\pi}_{\al T_q-}(i)\leq \overline{\eta}^{\la,\pi}_{\al T_q-}(i)
\end{equation}
where
\begin{align*}
\underline{\eta}^{\la,\pi}_{\al T_q-}(i) &=\min(N^{S}_{\al T_q-}(i)-N^{S}_{\al \tau_{T_q-}(c)+\varkappa_{\la,\pi}}(i),1),\\
\overline{\eta}^{\la,\pi}_{\al T_q-}(i) &=\min(N^{S}_{\al T_q-}(i)-N^{S}_{\al \tau_{T_q-}(c)}(i),1).
\end{align*}

Indeed, thanks to $\Omega^{S,P}_A(\la,\pi)\cap\Omega_M(\alpha)$, there is no burning tree in $I^\la_A$ at time $\al T_q-$. Furthermore, for $c\in\cC_M$, by $\Omega^{\la,\pi}_{T_q}$, we have 
\[\tau_{T_q-}(c)\leq \rho^{\la,\pi}_{T_q-}(i)\leq \tau_{T_q-}(c)+\varkappa_{\la,\pi}\,\text{ for all }i\in c_\la.\]
By definition, no fire can affect the site $i$ during $\intervalleoo{\al \rho^{\la,\pi}_{T_q-}(i)}{\al T_q}$ whence  \eqref{ineq}.

\md

\noindent{\bf Step 3.} We show here that if $Z_{T_q-}(X_q)<1$, there exist $j_1,j_2\in(X_q)_\la$ such that
\begin{gather*}
j_1<\lfloor\nl X_q\rfloor<j_2\\
\eta^{\la,\pi}_{\al t}(j_1)=\eta^{\la,\pi}_{\al t}(j_2)=0\text{ for all } t\in\intervalleff{T_q}{T_q+\varkappa_{\la,\pi}}.
\end{gather*}

Indeed, since no match falls on $X_q$ during the time interval $\intervallefo{0}{T_q}$, we have $\tau_{T_q-}(X_q)=T_q-Z_{T_q-}(X_q)=T_k$, for some $0\leq k<q$. Observe that $Z_{T_q-}(X_q)<1$ implies that $T_q-\tau_{T_q-}(X_q)<1$.

\begin{itemize}
	\item If $1\leq k<q$, then, by construction, we have $X_q\in \mathring{D}_{T_k-}(X_k)=\intervalleoo{a}{b}$, for some $a,b\in\cB_M\cup\{-A,A\}$. By $\Omega_M(\alpha)$, we have $|a-X_k|\wedge|b-X_k|>2\alpha$ whence $(X_q)_\la\subset\intervalleff{a}{b}_\la$. We deduce from Step 1 that $\eta^{\la,\pi}_{\al T_k+T^k_{i-\lfloor\nl X_k\rfloor}}(i)=2$ for all $i\in(X_q)_\la$. Since we work on $\Omega_2^S(\la,\pi)$ and $T_k,T_q\in\cT_M$, we know that there are some sites
\[\lfloor\nl X_k\rfloor-\ml<j_1<\lfloor\nl X_k\rfloor<j_2<\lfloor\nl X_k\rfloor+\ml\]
such that no seed has fallen on $j_1$ and $j_2$ during $\intervalleff{\al \tau_{T_q-}(X_q)}{\al (T_q+\varkappa_{\la,\pi})}$. Since they are made vacant by the fire $k$ during the time interval $\intervallefo{\al T_k}{\al(T_k+\varkappa_{\la,\pi})}$, we deduce that they remain vacant during $\intervalleff{\al( T_k+\varkappa_{\la,\pi})}{\al(T_q+\varkappa_{\la,\pi})}\supset\intervalleff{\al T_q}{\al (T_q+\varkappa_{\la,\pi})}$.
	\item If $k=0$, that is if $\tau_{T_q-}(X_q)=0$ we deduce that $T_q<1$. We conclude using $\Omega_2^S(\la,\pi)$ that there are $j_1<\lfloor\nl X_q\rfloor<j_2$ with $j_1,j_2\in(X_q)_\la$ where no seed fall during $\intervalleff{0}{\al (T_q+\varkappa_{\la,\pi})}$. Since all the sites are vacant at time $0$, we deduce that $j_1$ and $j_2$ remain vacant until $\al(T_q+\varkappa_{\la,\pi})$.
\end{itemize}

\md

\noindent{\bf Step 4.} Next we check that if $Z_{T_q-}(c)=1$ for some $c\in\cC_M$, then 
\[\eta^{\la,\pi}_{\al T_q-}(i)=1\text{ for all }i\in c_\la.\]
Recalling \eqref{zbm}, we see that $Z_{T_q-}(c)=1$ implies that $T_q-\tau_{T_q-}(c)\geq 1$ and $T_q-\tau_{T_q-}(c)\geq 1+2\alpha$ by $\Omega_M(\alpha)$. Using Step 2, we see that for all $i\in c_\la$,
\[\eta^{\la,\pi}_{\al T_q-}(i)\geq \underline{\eta}^{\la,\pi}_{\al T_q}(i)=\min(N^{S}_{\al T_q-}(i)-N^{S}_{\al \tau_{T_q-}(c)+\varkappa_{\la,\pi}}(i),1).\]
We conclude using $\Omega_3^S(\la,\pi)$ that for all $i\in c_\la$, $\underline{\eta}^{\la,\pi}_{\al T_q}(i)=1$ whence $\eta^{\la,\pi}_{\al T_q-}(i)=1$, as desired.

\md

\noindent{\bf Step 5.} We now prove that if $\tH_{T_q-}(x)=0$ for some $x\in\cB_M$, then
\[\eta^{\la,\pi}_{\al T_q-}(i)=1 \text{ for all }i\in(x)_\la.\]

{\it Preliminary considerations.} Let $k\in\{1,\dots , n\}$ such that $x = X_k$, which is at the boundary of two cells $c_-,\,c_+\in\cC_M$. We know that $\tH_{T_q-}(x)=0$, whence $H_{T_{q}-} (x) = 0$ and $Z_{T_{q}-}(x) = Z_{T_{q}-} (c_+)= Z_{T_{q}-} (c_- ) = 1$. This implies that $T_{q}\geq1$ (because $Z_t(x) = t$ for all $t < 1$ and all $x\in\intervalleff{-A}{A}$) and thus $T_{q}\geq 1 +2\alpha$ due to $\Omega_M (\alpha)$.

No fire has concerned $j_g=\lfloor\nl X_k\rfloor-\ml-1\in(c_-)_\la$ during $\intervalleoo{\al\rho_{T_{q} -}^{\la,\pi} (j_g)}{\al T_q }$. By $\Omega_{T_q}^{\la,\pi}$, we deduce that $\tau_{T_q-}(c_-)\leq \rho_{T_{q} -}^{\la,\pi} (j_g)\leq \tau_{T_q-}(c_-)+\varkappa_{\la,\pi}$. Recalling \eqref{zbm}, $Z_{T_q-}(c_-)= 1$ implies that $\tau_{T_q-} (c_-) \leq  T_q-1$ whence, by $\Omega_M(\alpha)$, there holds that $\tau_{T_q-} (c_-)< T_q- 1 -2\alpha$. Using a similar
argument for $j_d=\lfloor\nl X_k\rfloor+\ml+1\in(c_+)_\la$, we conclude that no match falling outside $(X_k)_\la$ can affect $(X_k)_\la$ during $\intervalleoo{\al (T_q-1-\alpha)}{\al T_q}$ (because to affect $(X_k)_\la$, a match falling outside $(X_k)_\la$ needs to cross $j_d$ or $j_g$).

\md

{\it Case 1.} First assume that $k\geq q$. Then we know that no fire has fallen on $(X_k)_\la$ during $\intervallefo{0}{\al T_q}$. Due to the preliminary considerations, we deduce that no fire at all has
concerned $(X_k)_\la$ during $\intervalleoo{\al (T_q-1-\alpha)}{\al T_q}$. Using $\Omega_3^S(\la,\pi)$, we conclude that $(X_k)_\la$ is completely occupied at time $\al T_q-$.

\md

{\it Case 2.} Assume that $k< q$ and $Z_{T_k-}(X_k) = 1$, so that there already has been a macroscopic fire in $(X_k)_\la$ (at time $\al T_k$). Since $Z_{T_k}(X_k) = 0$ and $Z_{T_q-}(X_k) = 1$, we deduce that
$T_q-T_k\geq 1$, whence $T_q-T_k\geq 1+2\alpha$ as usual. Since there is no more burning tree in $(X_k)_\la$ at time $\al(T_k+\varkappa_{\la,\pi})$, thanks to $\Omega^{P,A}_{\la,\pi}(X_k,T_k)$, we conclude as in Case 1 that no fire at all has concerned $(X_k)_\la$ during $\intervalleoo{\al(T_q-1-\alpha)}{\al T_q}$, which implies the claim by $\Omega_3^S(\la,\pi)$.

\md

{\it Case 3.} Assume that $k< q$ and $Z_{T_k-}(X_k) < 1$ and $T_q-T_k\geq 1$, whence $T_q-T_k\geq 1+2\alpha$ due to $\Omega_M(\alpha)$. Then there already has been a microscopic fire in $(X_k)_\la$ (at time $\al T_k$). But there are no fire in $(X_k)_\la$ during $\intervalleoo{\al (T_k+\varkappa_{\la,\pi})}{\al T_q}\supset\intervalleoo{\al(T_q-1-\alpha)}{\al T_q}$ and we conclude as in Case 2.

\md

{\it Case 4.} Assume finally that $k < q$ and $Z_{T_k-}(X_k) < 1$ and $T_q- T_k < 1$, whence $T_q- T_k < 1-2\alpha$ due to $\Omega_M(\alpha)$. There has been a microscopic fire in $(X_k)_\la$ (at time $\al T_k$). Since $H_{T_q-}(X_k) = 0$, we deduce that $T_k + Z _{T_k} (X_k)\leq T_q$, whence $T_k + Z _{T_k} (X_k)\leq T_q-2\alpha$ by
$\Omega_M(\alpha)$. There is $l<k$ such that $\tau_{T_k-}(X_k)=T_l$. We set $\cM\coloneqq ((X_l,T_l),(X_k,T_k))$, recall Subsection \ref{height0} (if $l=0$ \ie $\tau_{T_k-}(X_k)=0$, set for example $X_0=0$).

We first show that
\begin{equation}\label{procstep5}
(\eta^{\la,\pi}_t(i))_{t\in\intervalleff{\al T_l}{\al (T_k+\varkappa_{\la,\pi})},i\in(X_k)_\la}=(\zeta^{\la,\pi,\cM}_t(i))_{t\in\intervalleff{\al T_l}{\al (T_k+\varkappa_{\la,\pi})},i\in(X_k)_\la}.
\end{equation}
Here, the process $(\zeta^{\la,\pi,\cM}_t(i))_{t\in\intervalleff{\al T_l}{\al (T_k+\varkappa_{\la,\pi})},i\in(X_k)_\la}$ is built as in Subsection \ref{height0} using the seed processes $(N^S_t(i))_{t\geq0,i\in\zz}$ and the propagation processes $(N^P_t(i))_{t\geq0,i\in\zz}$.

\begin{itemize}
	\item We first assume that $T_l\geq 1$, whence $T_l\geq1+2\alpha$ by $\Omega_M(\alpha)$. Since no match has fallen on $(X_k)_\la$ during $\intervalleff{0}{\al T_l}$ and since $Z_{T_l-}(X_k)=1$, the zone $(X_k)_\la$ is completely occupied at time $\al T_l-$, recall Case 1. Thus, $(\eta^{\la,\pi}_t(i))_{t\geq0,i\in\zz}$ and $(\zeta^{\la,\pi,\cM}_t(i))_{t\geq0,i\in\zz}$ are equal on $(X_k)_\la$ at time $\al T_l$. By Step 1, we deduce, that 
\[\eta^{\la,\pi}_{\al T_l+T^l_{i-\lfloor\nl X_l\rfloor}}(i)=2\text{ for all }i\in D_{T_l-}(X_l)_\la.\]
Since $(X_k)_\la\subset D_{T_l-}(X_l)_\la$, we deduce that $\eta^{\la,\pi}_{\al T_l+T^l_{i-\lfloor\nl X_l\rfloor}}(i)=2$ for all $i\in (X_k)_\la$. Observe that, with our coupling, the fire $l$ propagates according to the same processes in both cases. Since seeds fall on $(X_k)_\la$ according to the same processes and since $(\eta^{\la,\pi}_t(i))_{t\geq0,i\in\zz}$ and $(\zeta^{\la,\pi,\cM}_t(i))_{t\geq0,i\in\zz}$ evolve according to the same rules, we deduce that they remain equals on $(X_k)_\la$ during $\intervalleff{\al T_l}{\al(T_l+\varkappa_{\la,\pi})}$. Next, no fire affects the zone $(X_k)_\la$ during $\intervallefo{\al (T_l+\varkappa_{\la,\pi})}{\al T_k}$ (because to affect the zone $(X_k)_\la$, we need $Z_{s-}(c_-)=1$ or $Z_{s-}(c_+)=1$ for some $s\in\intervalleoo{T_l}{T_k}$ whereas $Z_{s}(c_-)=Z_{s}(c_+)=s-T_l$ for all $s\in\intervalleff{T_l}{T_k}$) and since seeds fall on $(X_k)_\la$ according to the same processes, they are again equal during this time interval. Finally, $C^P((\zeta^{\la,\pi,\cM}_t(i))_{t\geq0,i\in\zz},(X_k,T_k))\subset(X_k)_\la$, recall Lemma \ref{micro fire 0}. We deduce \eqref{procstep5} because the match falling on $\lfloor\nl X_k\rfloor$ at time $\al T_k$ destroys the same zone, since the two processes evolve with the same rules on $(X_k)_\la$.
	\item If $T_l<1$, then by construction $l=0$ and $\tau_{T_k-}(X_k)=0$. We also deduce \eqref{procstep5} using similar arguments as above (this case is easier).
\end{itemize}

Consider now the zone $C^P = C^P((\eta^{\la,\pi}_t(i))_{t\geq 0,i\in\zz} , (X_k,T_k))$ destroyed by the match falling on $\lfloor\nl X_k\rfloor$ at time $\al T_k$. This zone is completely occupied at time $\al (T_k + \Theta^{\la,\pi}_\cM)$: this follows from the definition of $\Theta^{\la,\pi}_\cM$, see Lemma \ref{micro fire 0},  from \eqref{procstep5} and from the preliminary considerations. Using $\Omega_4^S(\gamma,\la,\pi)$, we deduce that $T_k +\Theta^{\la,\pi}_\cM\leq T_k+Z_{T_k-}(X_k)+\gamma< T_q$, since $\gamma<\alpha$. Hence $C^P$ is completely occupied at time $\al T_q-$.

Consider now $i\in (X_k)_\la\setminus C^P$. Then $i$ has not been killed by the fire starting at $\lfloor\nl X_k\rfloor$. Thus $i$ cannot have been killed during $\intervalleoo{\al (T_q-1-\alpha)}{\al T_q}$ (due to the preliminary considerations) and we conclude, using $\Omega_3^S(\la,\pi)$, that $i$ is occupied at time $\al T_q-$. This implies the claim.

\md

\noindent{\bf Step 6.} Let us now prove that if $\tH_{T_q-}(x)>0$ and $Z_{T_q-}(x_+)=1$ for some $x\in\cB_M$, there is $i_1\in(x)_\la$ such that $\eta^{\la,\pi}_{\al t}(i_1)=0$ for all $t\in\intervalleff{T_q}{T_q+\varkappa_{\la,\pi}}$. Recall that $x$ is at the boundary of two cells $c_-,c_+$.

We have either $H_{T_q-}(x) > 0$ or $Z_{T_q-}(c_-) < 1$ (because $Z_{T_q-} (c_+) = 1$ by assumption). Clearly, $x = X_k$ for some $k<q$, with $Z_{T_k-}(X_k) < 1$ (else, we would have $H_t(x) = 0$ and
$Z_t(c_-) = Z_t (c_+)$ for all $t\in\intervallefo{0}{T_q}$). Thus, recalling \eqref{zbm}, $T_k - Z_{T_k-}(X_k) = \tau_{T_k-}(X_k)=T_l$, for some $l<k$.

As checked in case 4 in the previous Step, on $\Omega(\alpha,\gamma,\la,\pi)$, setting $\cM=((X_l,T_l),(X_k,T_k))$ (if $l=0$, set for example $X_0=0$)
\[(\eta^{\la,\pi}_t(i))_{t\in\intervalleff{\al T_l}{\al (T_k+\varkappa_{\la,\pi})},i\in(X_k)_\la}=(\zeta^{\la,\pi,\cM}_t(i))_{t\in\intervalleff{\al T_l}{\al (T_k+\varkappa_{\la,\pi})},i\in(X_k)_\la}\]
where the process $(\zeta^{\la,\pi,\cM}_t(i))_{t\in\intervalleff{\al T_l}{\al (T_k+\varkappa_{\la,\pi})},i\in(X_k)_\la}$ is built as in Subsection \ref{height0} using the seed processes $(N^S_t(i))_{t\geq0,i\in\zz}$ and the propagation processes $(N^P_t(i))_{t\geq0,i\in\zz}$. Hence, either $l=0$ whence $\eta^{\la,\pi}_0(i)=0$ for all $i\in(X_k)_\la$ or all the sites in $(X_k)_\la$ burn at least on time during $\intervallefo{\al T_l}{\al(T_l+\varkappa_{\la,\pi})}$.

\md

{\it Case 1.} Assume first that $H_{T_q-}(x) > 0$. Then by construction, there holds $T_k + Z_{T_k-}(X_k) > T_q > T_k$, whence by
$\Omega_M(\alpha)$, $T_k + Z_{T_k-}(X_k) > T_q+2\alpha > T_k + 4\alpha$.

Consider $C^P = C^P((\eta^{\la,\pi}_{t}(i))_{t\geq 0,i\in\zz} , (X_k,T_k))$ the zone destroyed by the match falling on $\lfloor\nl X_k\rfloor$ at time $\al T_k$. By $\Omega_2^S(\la,\pi)$ and \eqref{procstep5}, we have $C^P\subset (X_k)_\la$ (because $T_k-Z_{T_k-}(X_k)$ and $T_k$ belong to $\cT_M$, because $0<Z_{T_k-}(X_k) < 1$ and because all the sites in $(X_k)_\la$ have been made vacant during $\intervallefo{\al T_l}{\al(T_l+\varkappa_{\la,\pi})}$).

By Definition of $\Theta^{\la,\pi}_\cM$, see Lemma \ref{micro fire 0} and by \eqref{procstep5}, we deduce that $C^P$ is not completely occupied at time $\al(T_k + \Theta^{\la,\pi}_\cM)$ (because in both cases, seeds fall on $(X_k)_\la$ according to the same processes). 
But by $\Omega_4^{S,P}(\gamma,\la,\pi)$, we see that $\Theta^{\la,\pi}_\cM \geq Z_{T_k-}(X_k)-\gamma$, whence $T_k +\Theta^{\la,\pi}_\cM\geq T_k + Z_{T_k-}(X_k)-\gamma+2\alpha > T_q+\varkappa_{\la,\pi}$ since $\gamma<\alpha$ and $\varkappa_{\la,\pi}<\alpha$. All this implies that there is a vacant site in $C^P$ during $\intervalleff{\al T_q}{\al(T_q+\varkappa_{\la,\pi})}$.

\md

{\it Case 2.} Assume next that $H_{T_q-}(x) = 0$ and that $T_q- T_l < 1$ (whence $T_q- T_l < 1-2\alpha$).
\begin{itemize}
	\item If $l\geq1$, recall that a match has fallen (in the limit process) on $X_l\in\cB_M$ at time $T_l\in\cT_M$ with $X_k\in \mathring{D}_{T_l-}(X_l)$. Since $T_l$ and $T_q$
belong to $\cT_M$ and since their difference is smaller than $1$ by assumption, $\Omega_2^S(\la,\pi)$ guarantees us the existence of $i_1\in(X_k)_\la$, such that no seed fall on $i_1$ during $\intervalleff{\al T_l}{\al (T_q+\varkappa_{\la,\pi})}$. Since all the sites in $(X_k)_\la$ have been made vacant during the time interval $\intervalleff{\al T_l}{\al (T_l+\varkappa_{\la,\pi})}$ (see Step 1), one easily concludes that $i_1$ is vacant during $\intervalleff{\al T_q}{\al(T_q+\varkappa_{\la,\pi})}$.
	\item If $l=0$ that is if $0<T_q<1$, there holds $0<T_q<1-2\alpha$ by $\Omega_M(\alpha)$. We conclude using $\Omega_2^S(\la,\pi)$ that there is a site $i_1\in(X_k)_\la$ where no seed has fallen during $\intervalleff{0}{\al(T_q+\varkappa_{\la,\pi})}$ whence $\eta^{\la,\pi}_{\al s}(i_1)=0$ for all $s\in\intervalleff{\al T_q}{\al(T_q+\varkappa_{\la,\pi})}$, as desired.
\end{itemize}

\md

{\it Case 3.} Assume finally that $H_{T_q-}(x) = 0$ and that $T_q-[T_k- Z_{T_k-}(X_k)] \geq 1$, whence $T_q-[T_k-Z_{T_k-}(X_k)] \geq 1 +2\alpha$ by $\Omega_M(\alpha)$. Since $H_{T_q-}(x) = 0$, there holds
$Z_{T_q-}(c_-) < 1 = Z_{T_q-}(c_+)$ and $T_k + Z_{T_k-}(X_k) \leq T_q$, so that $T_k + Z_{T_k-}(X_k) \leq  T_q-2\alpha$.

We aim to use the event $\Omega^{S,P}_1(\la,\pi)$. We introduce 
\[t_0= T_k-Z_{T_k-}(X_k) = \tau_{T_k-}(X_k)=T_l.\]
Observe that $\tau_{T_k-}(c_-) = \tau_{T_k-}(c_+) = \tau_{T_k-}(x)$ because there has been no fire (exactly) at $x$ during $\intervallefo{0}{T_k}$. Thus $Z_{t_0-}(x) = Z_{t_0-}(x_-) = Z_{t_0-}(x_+) = 1$ and $Z_{t_0}(x) =
Z_{t_0}(c_-) = Z_{t_0}(c_+) = 0$ (using the convention $Z_{0-}(y)=1$ for all $y\in\intervalleff{-A}{A}$).

Set now $t_1= T_k$. Observe that $0 < t_1 - t_0 < 1$. Necessarily, $Z_t (c_-)$ has jumped to $0$ at least one time between $t_0$ and $T_q-$ (else, one would have $Z_{T_q-}(c_-) = 1$,
since $T_q- t_0\geq 1$ by assumption) and this jump occurs after $t_0 + 1 > t_1$ (since a jump of $Z_t (c_-)$ requires that $Z_t(c_-) = 1$, and since for all $t\in\intervallefo{t_0}{t_0 + 1}$, $Z_t(c_-) = t - t_0 < 1$).

We thus may denote by $t_2 < t_3 < \dots < t_K$, for some $K\geq2$, the successive times of jumps of the process $(Z_t (c_-), Z_t(c_+)$) during $\intervalleoo{t_0 + 1}{T_q}$ and say $x_2,\dots,x_K$ the corresponding locations of the fires. We also put $\e = 1$ if $t_2$ is a jump of $Z_t(c_+)$ and $\e = -1$ else. 

Then we observe that $Z_t (c_-)$ and $Z_t (c_+)$ do never jump to $0$ at the same time during $\intervalleoo{t_0}{T_q}$ (else, it would mean that they are killed by the same fire at some time $u$,
whence necessarily, $H_r(u) = 0$ and $Z_r(c_-)=Z_r(c_+)$ for all $r\in\intervalleoo{u}{T_q}$). Furthermore, there is always at least one jump of $(Z_t(c_-), Z_t(c_+))$ in any time interval of
length 1 (during $\intervallefo{t_0 + 1}{T_q}$), because else, $Z_t (c_+ )$ and $Z_t (c_-)$ would both become equal to $1$
and thus would remain equal forever.
Finally, observe that two jumps of $Z_t(c_-)$ cannot occur in a time interval of length $1$ (since a jump of $Z_t(c_-)$ requires that $Z_t(c_- ) = 1$) and the same thing holds for $Z_t(c_+)$.

Consequently, the family $\cP= \{\e; (x_0,t_0),(X_k,T_k),\dots , (x_K,t_K)\}$ necessarily satisfies the condition $(PP)$ of Subsection \ref{persis0}.

Next, there holds that $t_2 - t_1 < Z_{T_k-}(X_k) = t_1-t_0$, because else, we would have $H_{t_2-}(X_k) = 0$ and thus the fire destroying $c_+$ (or $c_-$) at time $t_2$ would also destroy $c_-$ (or $c_+$), we thus would have $Z_{t_2}(c_+) = Z{t_2}(c_-) = 0$, so that $Z_t(c_+)$ and $Z_t(c_-)$ would remain equal forever. Furthermore, we have $t_K<T_q<t_K+1$ because else, we would have $Z_{T_q-}(c_+)=Z_{T_q-}(c_-)=1$.

Finally, we check that 
\[(\eta^{\la,\pi}_{t}(i))_{ t\in\intervalleff{\al t_0}{\al (t_K+\varkappa_{\la,\pi})} ,i\in(X_k)_\la} = (\zeta^{\la,\pi,\cP}_{t} (i))_{ t\in\intervalleff{\al t_0}{\al(t_K+\varkappa_{\la,\pi})} ,i\in (X_k)_\la},\]
this last process being built upon the families $(N_t^{S} (i))_{t\geq 0,i\in\zz}$ and $(N_t^{P} (i))_{t\geq 0,i\in\zz}$ as in Subsection \ref{micro fire 0}. Indeed, seeds fall according to the same processes and fires propagate according to the same processes on $(X_k)_\la$. We already have checked that $(\eta^{\la,\pi}_{t}(i))_{ t\geq0 ,i\in\zz}$ and $(\zeta^{\la,\pi,\cP}_{t} (i))_{ t\geq0 ,i\in \zz}$ are equal on $(X_k)_\la$ during the time interval $\intervalleff{\al t_0}{\al(T_k+\varkappa_{\la,\pi})}$. Nothing happens on $(X_k)_\la$ during $\intervallefo{\al(T_k+\varkappa_{\la,\pi})}{\al t_2}$. In both cases (say $\e=-1$), a match falls on $\lfloor\nl x_2\rfloor\in\intervalleentier{-\lfloor\nl A\rfloor}{\lfloor\nl X_k\rfloor-\ml}$ at time $\al t_2$. This fire destroys destroys the zone containing $\lfloor\nl X_k\rfloor-\ml$ (by definition of $\zeta^{\la,\pi,\cP}$ and because, by construction, $D_{t_2-}(x_2)=\intervalleff{a}{X_k}$, for some $a\in\cB_M\cup\{-A\}$, whence  $\eta^{\la,\pi}_{\al t_2-}(j)=1$ for all $j\in\intervalleentier{\lfloor\nl x_2\rfloor}{\lfloor\nl X_m\rfloor-\ml}$, see Steps 4 and 5 above) at the same time, since with our coupling, the second fire spreads according to the same rules and to the same processes in both cases. This implies that $(\eta^{\la,\pi}_{t}(i))_{ t\geq0 ,i\in\zz}$ and $(\zeta^{\la,\pi,\cP}_{t} (i))_{ t\geq0 ,i\in \zz}$ are also equal on $(X_k)_\la$ during the time interval $\intervalleff{\al(T_k+\varkappa_{\la,\pi})}{\al(t_2+\varkappa_{\la,\pi})}$. And so on.

We thus can use $\Omega_1^{S,P}(\la,\pi)$ and conclude that there is a site $i_1$ in $(X_k)_\la$ which is vacant during $\intervalleff{\al (t_K+\varkappa_{\la,\pi})}{\al (t_K+1)}$ for $(\zeta^{\la,\pi,\cP}_{t} (i))_{ t\geq0 ,i\in \zz}$. Since seeds fall on $(X_k)_\la$ according to the same processes, we deduce that there is also a vacant site in $(X_k)_\la$ during $\intervalleff{\al (t_K+\varkappa_{\la,\pi})}{\al (t_K+1)} \subset\intervalleff{\al T_q}{\al(T_q+\varkappa_{\la,\pi})}$ for the $(\la,\pi,A)-$FFP, as desired.

\md

\noindent{\bf Step 7.} We now conclude. We put $z\coloneqq Z_{T_q-}(X_q)$ and consider separately the cases $z\in\intervalleoo{0}{1}$ and $z=1$. Observe that $z=0$ do never happens, since by construction, $Z_{T_q-}(X_q)=\min(Z_{T_{q-1}}(X_q)+T_q-T_{q-1},1)>0$ and since $T_q>T_{q-1}$.

\md

{\it Case $z\in\intervalleoo{0}{1}$.} Then in the $A-$LFFP$(0)$, we have $Z_{T_q-}(X_q)=Z_{T_q}(X_q)$ for all $x\in\intervalleoo{-A}{A}$ whence $\tau_{T_q-}(c)=\tau_{T_q}(c)=\tau_{T_q+\varkappa_{\la,\pi}}(c)$ for all $c\in \cC_M$. Using Step 3, as seen in {\bf Micro$(0)$} in Subsection \ref{application ffp}, we see that the match falling on $\lfloor\nl X_q\rfloor$ at time $\al T_q$ destroys nothing outside $\intervalleentier{j_1}{j_2}\subset(X_q)_\la$ and there is no more burning tree in $I^\la_A$ at time $\al(T_q+\varkappa_{\la,\pi})$. We deduce that $\rho^{\la,\pi}_{s}(i)=\rho^{\la,\pi}_{T_q}(i)$ for all $s\in\intervalleff{T_q}{T_q+\varkappa_{\la,\pi}}$ and all $i\not\in(X_q)_\la$. Thus, applying $\Omega_{T_q}^{\la,\pi}$, we deduce that for all $c\in\cC_M$ and all $i\in c_\la$,
\[\tau_{T_q+\varkappa_{\la,\pi}}(c)=\tau_{T_q}(c)\leq \rho^{\la,\pi}_{T_q}(i)=\rho^{\la,\pi}_{T_q+\varkappa_{\la,\pi}}(i)\leq \tau_{T_q}(c)+\varkappa_{\la,\pi}=\tau_{T_q+\varkappa_{\la,\pi}}(c)+\varkappa_{\la,\pi}.\]
Thus, on $\Omega(\alpha,\gamma,\la,\pi)$, $\Omega^{\la,\pi}_{T_q}$ implies $\Omega^{\la,\pi}_{T_q+\varkappa_{\la,\pi}}$. Since no match falls on $I^\la_A$ during $\intervalleoo{\al(T_q+\varkappa_{\la,\pi})}{\al T_{q+1}}$ and since $\eta^{\la,\pi}_{\al T_{q+1}-}(i)=\eta^{\la,\pi}_{\al T_{q+1}}(i)$ for all $i\neq\lfloor\nl X_{q+1}\rfloor$, we deduce that on $\Omega(\alpha,\gamma,\la,\pi)$, for all $c\in\cC_M$ and all $i\in c_\la$,
\[\rho^{\la,\pi}_{T_q+\varkappa_{\la,\pi}}(i)=\rho^{\la,\pi}_{T_{q+1}}(i)\text{ and }\tau_{T_q+\varkappa_{\la,\pi}}(c)=\tau_{T_{q+1}}(c).\]

All this implies that on $\Omega(\alpha,\gamma,\la,\pi)$, $\Omega^{\la,\pi}_{T_q}$ implies $\Omega_{T_{q+1}}^{\la,\pi}$ when $z\in\intervalleoo{0}{1}$.

\md

{\it Case $z=1$.} Then there are $a,b\in\cB_M\cup\{-A,A\}$ such that $D_{T_q-}(X_q)=\intervalleff{a}{b}$. We assume that $a,b\in\cB_M$, the other cases being treated similarly. By construction, we know that for all $c\in\cC_M$ with $c\subset\intervalleoo{a}{b}$, $Z_{T_q-}(c)=1$, for all $x\in\cB_M\cap\intervalleoo{a}{b}$, $\tH_{T_q-}(x)=0$ while finally $\tH_{T_q-}(a)>0$ and $\tH_{T_q-}(b)>0$.

For the $A-$LFFP$(0)$, we have
\begin{enumerate}[label=(\roman*)]
	\item $\tau_{T_q}(c)= T_q$ for all $c\in\cC_M$ with $c\subset\intervalleoo{a}{b}$,
	\item $\tau_{T_q}(c)=\tau_{T_q-}(c)$ for all $c\in\cC_M$ with $c\cap\intervalleoo{a}{b}=\emptyset$.
\end{enumerate}
Next, using Steps 4, 5, using Step 6 for $a$ (and a very similar result for $b$), we immediately check that the fire occurring on $\lfloor\nl X_q\rfloor$ at time $\al T_q$, as seen in {\bf Macro$(0)$} in Subsection \ref{application ffp},
\begin{itemize}
	\item destroys completely all the cells $c\in\cC_M$ with $c\subset\intervalleoo{a}{b}$,
	\item destroys completely all the zones $(x)_\la$ with $x\in\cB_M\cap\intervalleoo{a}{b}$,
	\item does not destroy completely $(a)_\la$ nor $(b)_\la$,
	\item does not destroy at all the sites $i\in I^\la_A$ with $i\not\in\intervalleentier{\lfloor\nl a\rfloor-\ml}{\lfloor\nl b\rfloor+\ml}$.
\end{itemize}
Consequently, we have, for all $c\in\cC_M$ with $c\subset\intervalleoo{a}{b}$ and all $i\in (c)_\la$,
\[\tau_{T_q+\varkappa_{\la,\pi}}(c)=\tau_{T_q}(c)=T_q\leq \rho^{\la,\pi}_{T_q+\varkappa_{\la,\pi}}(i)\leq T_q+\varkappa_{\la,\pi}=\tau_{T_q}(c)+\varkappa_{\la,\pi}=\tau_{T_q+\varkappa_{\la,\pi}}(c)+\varkappa_{\la,\pi},\]
while if $c\cap\intervalleoo{a}{b}=\emptyset$, for all $i\in (c)_\la$,
\begin{multline*}
\tau_{T_q+\varkappa_{\la,\pi}}(c)=\tau_{T_q}(c)=\tau_{T_q-}(c)\leq \rho^{\la,\pi}_{T_q-}(i)=\rho^{\la,\pi}_{T_q+\varkappa_{\la,\pi}}(i)\\
\leq \tau_{T_q-}(c)+\varkappa_{\la,\pi}=\tau_{T_q}(c)+\varkappa_{\la,\pi}=\tau_{T_q+\varkappa_{\la,\pi}}(c)+\varkappa_{\la,\pi}.
\end{multline*}
We conclude that when $z = 1$, $\Omega_{T_q}^{\la,\pi}$ implies $\Omega_{T_q+\varkappa_{\la,\pi}}^{\la,\pi}$. Since no match falls on $I^\la_A$ during $\intervallefo{\al(T_q+\varkappa_{\la,\pi})}{\al T_{q+1}}$ and since $\eta^{\la,\pi}_{\al T_{q+1}-}(i)=\eta^{\la,\pi}_{\al T_{q+1}}(i)$ for all $i\neq\lfloor\nl X_{q+1}\rfloor$, we deduce that on $\Omega(\alpha,\gamma,\la,\pi)$, $\Omega^{\la,\pi}_{T_q+\varkappa_{\la,\pi}}$ implies $\Omega_{T_{q+1}}^{\la,\pi}$. 

All this implies that on $\Omega(\alpha,\gamma,\la,\pi)$, $\Omega^{\la,\pi}_{T_q}$ implies $\Omega_{T_{q+1}}^{\la,\pi}$ when $z=1$. This completes the proof.
\end{proof}

\subsection{Proof of Theorem \ref{converge restriction} for $p=0$}
We finally give the proof of the Theorem \ref{converge restriction} in the case $p=0$. The proof is closely related to the proof in the case $p>0$, recall Subsection \ref{proofp>0}.
\begin{proof}
Let us fix $x_0\in\intervalleoo{-A}{A}$, $t_0\in\intervalleoo{0}{T}$ and $\varepsilon>0$. We will prove that with our coupling (see Subsection \ref{coupling0}), when $\la\to0$ and $\pi\to \infty$ in the regime $\cR(0)$, there holds that
\begin{enumerate}[label=(\alph*)]
	\item $\lim_{\la,\pi}\proba{\bdelta(D_{t_0}^{\la,\pi}(x_0),D_{t_0}(x_0))>\e}=0$;
	\item $\lim_{\la,\pi}\proba{\bdelta_T(D^{\la,\pi}(x_0),D(x_0))>\e}=0$;
	\item $\lim_{\la,\pi}\proba{\abs{Z_t^{\la,\pi}(x_0)-Z_t(x_0)}>\e}=0$;	
	\item $\lim_{\la,\pi}\proba{\int_0^T\abs{Z_t^{\la,\pi}(x_0)-Z_t(x_0)}\diff t>\e}=0$;
	\item $\lim_{\la,\pi}\proba{\abs{W^{\la,\pi}_{t_0}(x_0)-Z_{t_0}(x_0)}>\e}=0$, where
\[W^{\la,\pi}_{t_0}(x_0)=\left(\frac{\log(|C(\eta^{\la,\pi}_{\al t_0},\lfloor\nl x_0\rfloor)|)}{\log(1/\la)}\indiq{|C(\eta^{\la,\pi}_{\al t_0},\lfloor\nl x_0\rfloor)|\geq1}\right)\wedge1.\]
\end{enumerate}

These points will clearly imply the result.

First, we introduce the event $\Omega^{x_0,t_0}_{A,T}(\alpha,\la,\pi)$ on which 
\begin{enumerate}[label=(\roman*)]
	\item $x_0\not\in\cup_{y\in\cB_M}\intervalleoo{y-2\alpha}{y+2\alpha}$;
	\item for all $s\in\cT_M\cup\cS_M$ with $s\leq t_0$, there holds that $t_0-s>2\alpha$;
	\item if $t_0\neq 1$, for all $s\in\cT_M\cup\cS_M$ with $s\leq t_0$, there holds that $|t_0-(s+1)|>2\alpha$;
	\item if $t_0> 1$, for all $i\in I^\la_A$, $N^{S}_{\al t_0}(i)-N^{S}_{\al (t_0-1)}(i)>0$;
	\item if $t_c=t_0-\tau_{t_0-}(x_0)<1$, there are $i_1$ and $i_2$ such that
\[-\lfloor\la^{-(t_c+\alpha)}\rfloor<i_1<-\lfloor\la^{-(t_c-\alpha)}\rfloor
<0<\lfloor\la^{-(t_c-\alpha)}\rfloor<i_2<\lfloor\la^{-(t_c+\alpha)}\rfloor\]
and such that 
\begin{itemize}
	\item $N^{S}_{\al t_0}(\lfloor\nl x_0\rfloor+i_1)-N^{S}_{\al \tau_{t_0-}(x_0)}(\lfloor\nl x_0\rfloor+i_1)=0$ whereas $N^{S}_{\al t_0}(\lfloor\nl x_0\rfloor+i_2)-N^{S}_{\al \tau_{t_0-}(x_0)}(\lfloor\nl x_0\rfloor+i_2)=0$;
	\item for all $j\in\intervalleentier{-\lfloor\la^{-(t_c-\alpha)}\rfloor}{\lfloor\la^{-(t_c-\alpha)}\rfloor}$, $N^{S}_{\al t_0}(\lfloor\nl x_0\rfloor+j)-N^{S}_{\al (\tau_{t_0-}(x_0)+\varkappa_{\la,\pi})}(\lfloor\nl x_0\rfloor+j)>0$.
\end{itemize}
\end{enumerate}
Since $t_0-\tau_{t_0-}(x_0)=1$ occurs with positive probability only if $t_0=1$ (and $\tau_{t_0-}(x_0)=0$), the probability of the three first points clearly tend to $1$ when $\alpha$ tends to $0$. Since $(\tau_{t}(x_0))_{t\geq0}$ is independent of $(N^{S}_t(i))_{t\geq0,i\in\zz}$ and since $(\tau_{t}(x_0))_{t\geq0}\subset\cT_M\cup\cS_M$, the probability of the two last points also tend to $1$ as $\alpha\to0$ and $\la\to0$ and $\pi\to\infty$ in the regime $\cR(0)$, thanks to Lemma \ref{speed0}-4,6,7 and space/time stationarity (recall that $\varkappa_{\la,\pi}\to0$). All this implies that for all $\delta>0$, there is $\alpha>0$ such that $\proba{\Omega^{x_0,t_0}_{A,T}(\alpha,\la,\pi)}>1-\delta$ for all $(\la,\pi)$ sufficiently close to the regime $\cR(0)$.

Let us now fix $\delta>0$. In the rest of the proof, we consider $\alpha_0\in\intervalleoo{0}{\e}$, $\gamma_0\in\intervalleoo{0}{\alpha_0}$, $\la_0\in\intervalleoo{0}{1}$ and $\epsilon_0\in\intervalleoo{0}{1}$ such that for all $\la\in\intervalleoo{0}{\la_0}$ and all $\pi\geq1$ in such a way that $\nl/(\al\pi)<\epsilon_0$, we have
\[\proba{\Omega(\alpha_0,\gamma_0,\la,\pi)\cap\Omega^{x_0,t_0}_{A,T}(\alpha_0,\la,\pi)}>1-\delta.\]

We then consider $\la_1\in\intervalleoo{0}{\la_0}$ and $\epsilon_1\in\intervalleoo{0}{\epsilon_0}$ such that for all all $\la\in\intervalleoo{0}{\la_1}$ and all $\pi\geq1$ in such a way that $\nl/(\al\pi)<\epsilon_1$, we have
\begin{itemize}
	\item $\varkappa_{\la,\pi}\leq \alpha_0$;
	\item $\alpha_0+\log(\al)/\log(1/\la)<\e$;
	\item $4\ml/\nl\leq \e$;
	\item $1/(2\ml\la^{t_c-2\e})\leq \delta$ and $1/(2\ml\la^{t_c+\varkappa_{\la,\pi}})\leq \delta$ if $t_c<1$.
\end{itemize}
All this can be done properly by using the fact that $\varkappa_{\la,\pi}\to0$ and $\ml/\nl\to0$.

In the rest of the proof, we consider $\la\in\intervalleoo{0}{\la_1}$ and $\pi\geq1$ in such a way that $\nl/(\al\pi)\leq\epsilon_1$. Observe that, on $\Omega^{x_0,t_0}_{A,T}(\alpha_0,\la,\pi)$, we have $\tau_{t_0-}(x_0)=\tau_{t_0}(x_0)$ and $(x_0)_\la\cap\left(\bigcup_{x\in\cB_M}(x)_\la\right)=\emptyset$. We call $c_0\in\cC_M$ the cell containing $x_0$.

\md

\noindent{\bf Step 1.} As in Subsection \ref{proofp>0}, Steps 1 and 2, (a) (which holds for an arbitrary value of $t_0\in\intervalleoo{0}{T}$) implies (b) and (c) implies (d).

\md

\noindent{\bf Step 2.} Due to Lemma \ref{lemconvergence}, we know that, on $\Omega(\alpha_0,\gamma_0,\la,\pi)\cap\Omega_{A,T}^{x_0,t_0}(\alpha_0,\la,\pi)$, since $t_0>\tau_{t_0}(x_0)+3\alpha_0$, for all $i\in(x_0)_\la$,
\[\tau_{t_0}(c_0)\leq \rho^{\la,\pi}_{t_0}(i)\leq \tau_{t_0}(c_0)+\varkappa_{\la,\pi}.\]

For all $i\in (x_0)_\la$, since $\eta^{\la,\pi}_{\al t_0}(i)\leq 1$, there holds
\[\eta_ {\al t_0}^{\la,\pi}(i)=\min(N^{S,\la,\pi}_{\al t_0}(i)-N^{S,\la,\pi}_{\al \rho_{t_0}^{\la,\pi}(i)}(i),1).\]
Thus, for all $i\in(x_0)_\la$,	
\[\underline{\eta}^{\la,\pi}_{\al t_0}(i)\leq \eta_ {\al t_0}^{\la,\pi}(i)\leq \overline{\eta}^{\la,\pi}_{\al t_0}(i)\]
where
\begin{align*}
\underline{\eta}^{\la,\pi}_{\al t_0}(i) &\coloneqq \min(N^{S}_{\al t_0}(i)-N^{S}_{\al (\tau_{t_0}(x_0)+\varkappa_{\la,\pi})}(i),1),\\
\overline{\eta}^{\la,\pi}_{\al t_0}(i) &\coloneqq \min(N^{S}_{\al t_0}(i)-N^{S}_{\al \tau_{t_0}(x_0)}(i),1).
\end{align*}
We also recall that by construction, $(\tau_t(x_0))_{t\geq 0}$ is independent of $(N^{S}_ t(i))_{t\geq0,i\in\zz}$.

\md

\noindent{\bf Step 3.} Here we prove (e). We work on $\Omega(\alpha_0,\gamma_0,\la,\pi)\cap\Omega_{A,T}^{x_0,t_0}(\alpha_0,\la,\pi)$. By Step 2 and point (v) of the event $\Omega_{A,T}^{x_0,t_0}(\alpha_0,\la,\pi)$, we observe that if $0<t_c=t_0-\tau_{t_0}(x_0)<1$, then 
\begin{multline*}
\intervalleentier{\lfloor\nl x_0\rfloor-\lfloor\la^{-(t_c-\alpha_0)}\rfloor}{\lfloor\nl x_0\rfloor+\lfloor\la^{-(t_c-\alpha_0)}\rfloor}\\
\subset C(\underline{\eta}^{\la,\pi}_{\al t_0}, \lfloor\nl  x_0\rfloor)
\subset C(\eta^{\la,\pi}_{\al t_0}, \lfloor\nl  x_0\rfloor)\subset C(\overline{\eta}^{\la,\pi}_{\al t_0}, \lfloor\nl  x_0\rfloor)\\
\subset\intervalleentier{\lfloor\nl x_0\rfloor-\lfloor\la^{-(t_c+\alpha)}\rfloor}{\lfloor\nl x_0\rfloor+\lfloor\la^{-(t_c+\alpha)}\rfloor}.
\end{multline*}

Thus, this implies that
\[|W_{t_0}^{\la,\pi}(x_0)-(t_0-\tau_{t_0}(x_0))|\leq\alpha_0+\frac{\log(2)}{\log(1/\la)}<\e.\]

If now $t_0-\tau_{t_0}(x_0)>1$, then $t_0-\tau_{t_0}(x_0)>1+2\alpha_0$ thanks to  $\Omega_{A,T}^{x_0,t_0}(\alpha_0,\la,\pi)$. Then Step 2 and point (iv) of $\Omega_{A,T}^{x_0,t_0}(\alpha_0,\la,\pi)$  imply that $(x_0)_\la\subset C(\eta^{\la,\pi}_{\al t_0}, \lfloor\nl  x_0\rfloor)$ whence $|C(\eta^{\la,\pi}_{\al t_0}, \lfloor\nl  x_0\rfloor)|\geq  2\ml$. Consequently, 
\[W^{\la,\pi}_{t_0}(x_0)\geq 1-\frac{\log(\al)}{\log(1/\la)}>1-\e.\]

It only remains to study what happens when $t_0=1$. By construction, we have $\tau_{t_0}(x_0)=0$. Observe that on $\Omega(\alpha,\gamma,\la,\pi)$, a match falling on $\lfloor\nl X_k\rfloor$ at time $\al T_k\leq 1$, for some $k\in\{1,\dots,n\}$, does not affect the zone outside $(X_k)_\la$. Thus, for all $i\in(x_0)_\la$,
\[\eta^{\la,\pi}_{\al}(i)=\min(N^S_{\al}(i),1).\]
Using point (iv) of the event $\Omega^{x_0,t_0}_{A,T}(\alpha_0,\la,\pi)$, we deduce that
\[(x_0)_\la\subset C(\eta^{\la,\pi}_{\al t_0}, \lfloor\nl  x_0\rfloor)\]
and conclude that $|C(\eta^{\la,\pi}_{\al t_0}, \lfloor\nl  x_0\rfloor)|\geq  2\ml$, whence 
\[W^{\la,\pi}_{t_0}(x_0)\geq 1-\frac{\log(\al)}{\log(1/\la)}\geq 1-\e.\]

Recalling that $Z_{t_0}(x_0)=(t_0-\tau_{t_0}(x_0))\wedge1$, we have proved that
\[\proba{|W_{t_0}^{\la,\pi}(x_0)-Z_{t_0}(x_0))|<\e}\geq \proba{\Omega(\alpha_0,\gamma_0,\la,\pi)\cap\Omega^{x_0,t_0}_{A,T}(\alpha_0,\la,\pi)}\geq 1-\delta,\]
as desired.

\md

\noindent{\bf Step 4.} Here we prove (c). Recall that $Z^{\la,\pi}_{t_0}(x_0)=\left(-\frac{\log(1-K^{\la,\pi}_{t_0}(x_0))}{\log(1/\la)}\right)\wedge1$ where $K^{\la,\pi}_{t_0}(x_0)=(2\ml+1)^{-1}\abs{\enstq{i\in\intervalleentier{\lfloor\nl X_0\rfloor-\ml}{\lfloor\nl X_0\rfloor+\ml}}{\eta^{\la,\pi}_{\al t_0}(i)=1}}$. We work on $\Omega(\alpha_0,\gamma_0,\la,\pi)\cap\Omega^{x_0,t_0}_{A,T}(\alpha_0,\la,\pi)$ and set $t_c= t_0-\tau_{t_0}(x_0)$.

\md

{\it Case 1.} If $t_c\geq 1$, we have checked in Step 3 that $\eta^{\la,\pi}_{\al t_0}(i)=1$ for all $i\in(x_0)_\la$, whence $K^{\la,\pi}_{t_0}(x_0)=1$ and $Z^{\la,\pi}_{t_0}(x_0)=1$. 

\md

{\it Case 2.} If now $0<t_c< 1$, we deduce from Step 3 that
\[\underline{K}^{\la,\pi}_{t_0}(x_0)\leq K^{\la,\pi}_{t_0}(x_0)\leq \overline{K}^{\la,\pi}_{t_0}(x_0)\]
where
\begin{align*}
\underline{K}^{\la,\pi}_{t_0}(x_0) &=(2\ml+1)^{-1}\abs{\enstq{i\in\intervalleentier{\lfloor\nl X_0\rfloor-\ml}{\lfloor\nl x_0\rfloor+\ml}}{\underline{\eta}^{\la,\pi}_{\al t_0}(i)=1}},\\
\overline{K}^{\la,\pi}_{t_0}(x_0) &=(2\ml+1)^{-1}\abs{\enstq{i\in\intervalleentier{\lfloor\nl X_0\rfloor-\ml}{\lfloor\nl x_0\rfloor+\ml}}{\overline{\eta}^{\la,\pi}_{\al t_0}(i)=1}}.
\end{align*}

Recalling Step 5 in Subsection \ref{proofp>0}, we deduce that
\[\proba{ K^{\la,\pi}_{t_0}(x_0)\in\intervalleoo{1-\la^{t_c-\e}}{1-\la^{t_c+\e}}}\geq1-c\delta,\]
for some constant $c>0$, whence 
\[\proba{Z^{\la,\pi}_{t_0}(x_0)\in\intervalleoo{t_c-\e}{t_c+\e}}\geq1-c\delta.\]
This is nothing but the goal, since $Z_{t_0}(x_0)=t_0-\tau_{t_0}(x_0)=t_c$ as soon as $Z_{t_0}(x_0)<1$.

\md

\noindent{\bf Step 5.} It remains to prove (a). On $\Omega(\alpha_0,\gamma_0,\la,\pi)\cap\Omega^{x_0,t_0}_{A,T}(\alpha_0,\la,\pi)$, we check that
\begin{enumerate}[label=(\roman*)]
	\item If $Z_{t_0}(x_0)<1$, then $D_{t_0}(x_0)=\{x_0\}$ and $C(\eta^{\la,\pi}_{\al t_0},\lfloor\nl x_0\rfloor)\subset (x_0)_\la$ (see Step 3 above), whence $D^{\la,\pi}_{t_0}(x_0)\subset\intervalleff{x_0-\ml/\nl}{x_0+\ml/\nl}$. We deduce that $\bdelta(D^{\la,\pi}_{t_0}(x_0),D_{t_0}(x_0))\leq 2\ml/\nl$.
	\item If $Z_{t_0}(x_0)=1$ and $D_{t_0}(x_0)=\intervalleff{a}{b}$, for some $a,b\in\cB_M\cup\{-A,A\}$, then 
\begin{itemize}
	\item for $c\in\cC_M$ with $c\subset\intervalleoo{a}{b}$, $\eta^{\la,\pi}_{\al t_0}(i)=1$ for all $i\in c_\la$ (see Step 4 of the preceeding proof);
	\item for $x\in\cB_M\cap\intervalleoo{a}{b}$, $\eta^{\la,\pi}_{\al t_0}(i)=1$ for all $i\in (x)_\la$ (see Step 5 of the preceeding proof);
	\item there are $i\in(a)_\la$ and $j\in(b)_\la$ such that $\eta^{\la,\pi}_{\al t_0}(i)=\eta^{\la,\pi}_{\al t_0}(j)=0$ (see Step 6 of the preceeding proof);
\end{itemize}
so that
\[\intervalleentier{\lfloor\nl a\rfloor+\ml}{\lfloor\nl b\rfloor-\ml}\subset C(\eta^{\la,\pi}_{\al t_0},\lfloor\nl x_0\rfloor)\subset \intervalleentier{\lfloor\nl a\rfloor-\ml}{\lfloor\nl b\rfloor+\ml}\]
and thus 
\[\intervalleff{a+\ml/\nl}{b-\ml/\nl}\subset D^{\la,\pi}_{t_0}(x_0)\subset \intervalleff{a-\ml/\nl}{b+\ml/\nl},\]
whence $\bdelta(D^{\la,\pi}_{t_0}(x_0),D_{t_0}(x_0))\leq 4\ml/\nl$.
\end{enumerate}

Thus, on $\Omega(\alpha_0,\gamma_0,\la,\pi)\cap\Omega^{x_0,t_0}_{A,T}(\alpha_0,\la,\pi)$, we always have $\bdelta(D^{\la,\pi}_{t_0}(x_0),D_{t_0}(x_0))\leq 4\ml/\nl$. We conclude that
\[\proba{\bdelta(D^{\la,\pi}_{t_0}(x_0),D_{t_0}(x_0))\leq \e}\geq \proba{\Omega(\alpha_0,\gamma_0,\la,\pi)\cap\Omega^{x_0,t_0}_{A,T}(\alpha_0,\la,\pi)}\geq 1-\delta.\]
This concludes the proof.\qedhere
\end{proof}

\subsection{Cluster size distribution when $p=0$}

The aim of this section is to prove Corollary \ref{cor1} when $p=0$. We first recall a result of [\cite{bfnew}, Lemma 3.11.1].

\begin{lem}\label{zunif0}
Let $(Z_t(x),H_t(x),F_t(x))_{t\geq 0, x\in \rr}$ be a LFFP$(0)$ and consider $(D_t(x))_{t\geq 0, x\in \rr}$ the associated process.  There
are some constants $0<c_1<c_2$ and $0 < \kappa_1<\kappa_2$ such that the
following estimates hold.
\begin{enumerate}[label = (\roman*)]
        \item For any $t\in (1,\infty)$, any $x\in \rr$, any $z\in [0,1)$, $\proba{Z_t(x)=z}=0$.
        \item For any $t\in[0,\infty)$, any $B>0$, any $x\in\rr$, $\proba{|D_t(x)|=B}=0$.
        \item For all $t\in [0,\infty)$, all $x\in \rr$, all $B>0$, $\proba{|D_t(x)|\geq B}
\leq c_2 e^{-\kappa_1 B}$.
        \item For all $t\in [\frac{3}{2},\infty)$, all $x\in\rr$, all $B>0$, $\proba{|D_t(x)|\geq B}
\geq c_1 e^{-\kappa_2 B}$.
        \item For all $t\in[5/2,\infty)$, all $0\leq a < b < 1$, all $x\in \rr$,  \[c_1(b-a)\leq \proba{Z_t(x)\in \intervalleff{a}{b}}   \leq c_2 (b-a).\]
\end{enumerate}
\end{lem}

We now handle the
\begin{proof}[Proof of Corollary \ref{cor1} when $p=0$.]
For each $\la\in\intervalleoo{0}{1}$ and each $\pi\geq 1$, consider a $(\la,\pi)-$FFP $(\eta^{\la,\pi}_t(i))_{t\geq0,i\in\zz}$. Let also $(Z_t(x),H_t(x),F_t(x))_{t\geq0,x\in\rr}$ be a LFFP$(0)$ and consider the corresponding process $(D_t(x))_{t\geq0,x\in\rr}$.

\md

\noindent{\bf Point (b).} Using Lemma \ref{zunif0}-(iii)-(iv) and recalling that $|C(\eta^{\la,\pi}_{\al t},0)|/\nl=|D_t^{\la,\pi}(0)|$, it suffices to check that for all $t\geq 3/2$ and all $B>0$, when $\la\to0$ and $\pi\to0$ in the regime $\cR(0)$,
\[\lim_{\la,\pi}\proba{|D_t^{\la,\pi}(0)|\geq B}=\proba{|D_t(0)|\geq B}.\]
This follows from Theorem \ref{converge}-2, which implies that $|D_t^{\la,\pi}(0)|$ goes in law to $\abs{D_t(0)}$ and from Lemma \ref{zunif0}-(ii).

\md

\noindent{\bf Point (a).} Due to Lemma \ref{zunif0}-(v) we only need that for all $0<a<b<1$, all $t\geq 5/2$, when $\la\to0$ and $\pi\to0$ in the regime $\cR(0)$,
\[\lim_{\la,\pi}\proba{|C(\eta^{\la,\pi}_{\al t},0)|\in\intervalleff{\la^{-a}}{\la^{-b}}}=\proba{Z_t(0)\in\intervalleff{a}{b}}.\]
But using Theorem \ref{converge}-3 and Lemma \ref{zunif0}-(i), we know that
\[\lim_{\la,\pi}\proba{\frac{\log(|C(\eta^{\la,\pi}_{\al t},0)|)}{\log(1/\la)}\indiq{|C(\eta^{\la,\pi}_{\al t},0)|\geq1}\in\intervalleff{ a}{b}}=\proba{Z_t(0)\in\intervalleff{a}{b}}\]
as $\la\to0$ and $\pi\to0$ in the regime $\cR(0)$. One immediately concludes.
\end{proof}

\clearpage
\renewcommand{\refname}{Bibliography}

\end{document}